\documentclass[a4paper]{amsart}

\usepackage{amsfonts}
\usepackage{amssymb}
\usepackage{amsmath}
\usepackage{hyperref}
\usepackage{mathrsfs}
\usepackage{centernot}
\usepackage{mathdots}
\usepackage{stmaryrd}
\usepackage[all]{xy}

\newtheorem{thm}{Theorem}[section]
\newtheorem{lem}[thm]{Lemma}
\newtheorem{prp}[thm]{Proposition}
\newtheorem{cor}[thm]{Corollary}
\newtheorem{dfn}[thm]{Definition}

\newtheorem{cnj}[thm]{Conjecture}

\newtheoremstyle{roman} 
    {8.0pt plus 2.0pt minus 4.0pt}                    
    {8.0pt plus 2.0pt minus 4.0pt}                    
    {\normalfont}                
    {}                           
    {\bfseries}                  
    {.}                          
    {5pt plus 1pt minus 1pt}     
    {}  

\theoremstyle{roman}

\newtheorem{example}[thm]{Example}
\newtheorem{remark}[thm]{Remark}

\theoremstyle{plain}

\newcommand{\rem}[1]{}

\newcommand{\C}{\mathbb{C}}

\newcommand{\F}{\mathbb{F}}

\newcommand{\N}{\mathbb{N}}
\newcommand{\Q}{\mathbb{Q}}
\newcommand{\R}{\mathbb{R}}
\newcommand{\Z}{\mathbb{Z}}

\newcommand{\frakCapital}{
\newcommand{\frakA}{{\mathfrak{A}}}
\newcommand{\frakB}{{\mathfrak{B}}}
\newcommand{\frakC}{{\mathfrak{C}}}
\newcommand{\frakD}{{\mathfrak{D}}}
\newcommand{\frakE}{{\mathfrak{E}}}
\newcommand{\frakF}{{\mathfrak{F}}}
\newcommand{\frakG}{{\mathfrak{G}}}
\newcommand{\frakH}{{\mathfrak{H}}}
\newcommand{\frakI}{{\mathfrak{I}}}
\newcommand{\frakJ}{{\mathfrak{J}}}
\newcommand{\frakK}{{\mathfrak{K}}}
\newcommand{\frakL}{{\mathfrak{L}}}
\newcommand{\frakM}{{\mathfrak{M}}}
\newcommand{\frakN}{{\mathfrak{N}}}
\newcommand{\frakO}{{\mathfrak{O}}}
\newcommand{\frakP}{{\mathfrak{P}}}
\newcommand{\frakQ}{{\mathfrak{Q}}}
\newcommand{\frakR}{{\mathfrak{R}}}
\newcommand{\frakS}{{\mathfrak{S}}}
\newcommand{\frakT}{{\mathfrak{T}}}
\newcommand{\frakU}{{\mathfrak{U}}}
\newcommand{\frakV}{{\mathfrak{V}}}
\newcommand{\frakW}{{\mathfrak{W}}}
\newcommand{\frakX}{{\mathfrak{X}}}
\newcommand{\frakY}{{\mathfrak{Y}}}
\newcommand{\frakZ}{{\mathfrak{Z}}}
}

\newcommand{\calCapital}{
\newcommand{\calA}{{\mathcal{A}}}
\newcommand{\calB}{{\mathcal{B}}}
\newcommand{\calC}{{\mathcal{C}}}
\newcommand{\calD}{{\mathcal{D}}}
\newcommand{\calE}{{\mathcal{E}}}
\newcommand{\calF}{{\mathcal{F}}}
\newcommand{\calG}{{\mathcal{G}}}
\newcommand{\calH}{{\mathcal{H}}}
\newcommand{\calI}{{\mathcal{I}}}
\newcommand{\calJ}{{\mathcal{J}}}
\newcommand{\calK}{{\mathcal{K}}}
\newcommand{\calL}{{\mathcal{L}}}
\newcommand{\calM}{{\mathcal{M}}}
\newcommand{\calN}{{\mathcal{N}}}
\newcommand{\calO}{{\mathcal{O}}}
\newcommand{\calP}{{\mathcal{P}}}
\newcommand{\calQ}{{\mathcal{Q}}}
\newcommand{\calR}{{\mathcal{R}}}
\newcommand{\calS}{{\mathcal{S}}}
\newcommand{\calT}{{\mathcal{T}}}
\newcommand{\calU}{{\mathcal{U}}}
\newcommand{\calV}{{\mathcal{V}}}
\newcommand{\calW}{{\mathcal{W}}}
\newcommand{\calX}{{\mathcal{X}}}
\newcommand{\calY}{{\mathcal{Y}}}
\newcommand{\calZ}{{\mathcal{Z}}}
}

\newcommand{\bbCapital}{
\newcommand{\bbA}{{\mathbb{A}}}
\newcommand{\bbB}{{\mathbb{B}}}
\newcommand{\bbC}{{\mathbb{C}}}
\newcommand{\bbD}{{\mathbb{D}}}
\newcommand{\bbE}{{\mathbb{E}}}
\newcommand{\bbF}{{\mathbb{F}}}
\newcommand{\bbG}{{\mathbb{G}}}
\newcommand{\bbH}{{\mathbb{H}}}
\newcommand{\bbI}{{\mathbb{I}}}
\newcommand{\bbJ}{{\mathbb{J}}}
\newcommand{\bbK}{{\mathbb{K}}}
\newcommand{\bbL}{{\mathbb{L}}}
\newcommand{\bbM}{{\mathbb{M}}}
\newcommand{\bbN}{{\mathbb{N}}}
\newcommand{\bbO}{{\mathbb{O}}}
\newcommand{\bbP}{{\mathbb{P}}}
\newcommand{\bbQ}{{\mathbb{Q}}}
\newcommand{\bbR}{{\mathbb{R}}}
\newcommand{\bbS}{{\mathbb{S}}}
\newcommand{\bbT}{{\mathbb{T}}}
\newcommand{\bbU}{{\mathbb{U}}}
\newcommand{\bbV}{{\mathbb{V}}}
\newcommand{\bbW}{{\mathbb{W}}}
\newcommand{\bbX}{{\mathbb{X}}}
\newcommand{\bbY}{{\mathbb{Y}}}
\newcommand{\bbZ}{{\mathbb{Z}}}
}

\newcommand{\bfCapital}{
\newcommand{\bfA}{{\mathbf{A}}}
\newcommand{\bfB}{{\mathbf{B}}}
\newcommand{\bfC}{{\mathbf{C}}}
\newcommand{\bfD}{{\mathbf{D}}}
\newcommand{\bfE}{{\mathbf{E}}}
\newcommand{\bfF}{{\mathbf{F}}}
\newcommand{\bfG}{{\mathbf{G}}}
\newcommand{\bfH}{{\mathbf{H}}}
\newcommand{\bfI}{{\mathbf{I}}}
\newcommand{\bfJ}{{\mathbf{J}}}
\newcommand{\bfK}{{\mathbf{K}}}
\newcommand{\bfL}{{\mathbf{L}}}
\newcommand{\bfM}{{\mathbf{M}}}
\newcommand{\bfN}{{\mathbf{N}}}
\newcommand{\bfO}{{\mathbf{O}}}
\newcommand{\bfP}{{\mathbf{P}}}
\newcommand{\bfQ}{{\mathbf{Q}}}
\newcommand{\bfR}{{\mathbf{R}}}
\newcommand{\bfS}{{\mathbf{S}}}
\newcommand{\bfT}{{\mathbf{T}}}
\newcommand{\bfU}{{\mathbf{U}}}
\newcommand{\bfV}{{\mathbf{V}}}
\newcommand{\bfW}{{\mathbf{W}}}
\newcommand{\bfX}{{\mathbf{X}}}
\newcommand{\bfY}{{\mathbf{Y}}}
\newcommand{\bfZ}{{\mathbf{Z}}}
}

\newcommand{\catCapital}{
\newcommand{\catA}{{\mathscr{A}}}
\newcommand{\catB}{{\mathscr{B}}}
\newcommand{\catC}{{\mathscr{C}}}
\newcommand{\catD}{{\mathscr{D}}}
\newcommand{\catE}{{\mathscr{E}}}
\newcommand{\catF}{{\mathscr{F}}}
\newcommand{\catG}{{\mathscr{G}}}
\newcommand{\catH}{{\mathscr{H}}}
\newcommand{\catI}{{\mathscr{I}}}
\newcommand{\catJ}{{\mathscr{J}}}
\newcommand{\catK}{{\mathscr{K}}}
\newcommand{\catL}{{\mathscr{L}}}
\newcommand{\catM}{{\mathscr{M}}}
\newcommand{\catN}{{\mathscr{N}}}
\newcommand{\catO}{{\mathscr{O}}}
\newcommand{\catP}{{\mathscr{P}}}
\newcommand{\catQ}{{\mathscr{Q}}}
\newcommand{\catR}{{\mathscr{R}}}
\newcommand{\catS}{{\mathscr{S}}}
\newcommand{\catT}{{\mathscr{T}}}
\newcommand{\catU}{{\mathscr{U}}}
\newcommand{\catV}{{\mathscr{V}}}
\newcommand{\catW}{{\mathscr{W}}}
\newcommand{\catX}{{\mathscr{X}}}
\newcommand{\catY}{{\mathscr{Y}}}
\newcommand{\catZ}{{\mathscr{Z}}}
}

\newcommand{\what}[1]{\widehat{#1}}

\newcommand{\veps}{\varepsilon}
\newcommand{\vphi}{\varphi}

\newcommand{\idealof}{\unlhd} 
\newcommand{\normalin}{\unlhd}

\newcommand{\derives}{\Longrightarrow}

\newcommand{\longto}{\longrightarrow}
\newcommand{\onto}{\twoheadrightarrow}

\newcommand{\suchthat}{\,:\,}
\newcommand{\where}{\,|\,}

\newcommand{\quo}[1]{\overline{#1}}

\newcommand{\smallSMatII}[4]{\left[\begin{smallmatrix} {#1} & {#2} \\ {#3} &
{#4} \end{smallmatrix}\right]}

\newcommand{\DotsArr}[4]{
\begin{array}{ccc} {#1} & \ldots & {#2} \\
\vdots & \ddots & \vdots \\
{#3} & \ldots & {#4} \end{array}}

\newcommand{\UTDotsArr}[3]{
\begin{array}{ccc} {#1} & \ldots & {#2} \\
 & \ddots & \vdots \\
 &  & {#3} \end{array}}

 \newcommand{\DDotsArr}[2]{
\begin{array}{ccc} {#1} &  &  \\
 & \ddots &  \\
 &  & {#2} \end{array}}

\newcommand{\Circs}[1]{\left( #1 \right)}

\newcommand{\Trings}[1]{\left< #1 \right>}

 \DeclareMathOperator{\Aut}{Aut}
\DeclareMathOperator{\Char}{char} %
\DeclareMathOperator{\End}{End} %
\DeclareMathOperator{\Hom}{Hom} %
\DeclareMathOperator{\id}{id} %
\DeclareMathOperator{\im}{im} %
\DeclareMathOperator{\Ind}{Ind} %
\DeclareMathOperator{\Nrd}{Nrd} %
\newcommand{\op}{\mathrm{op}} %
\DeclareMathOperator{\Span}{span} %
\DeclareMathOperator{\Spec}{Spec} %
\DeclareMathOperator{\Stab}{Stab} %
\DeclareMathOperator{\supp}{supp} %
\newcommand{\nGL}[2]{\mathrm{GL}_{#2}({#1})}
\newcommand{\nPGL}[2]{\mathrm{PGL}_{#2}({#1})}
\newcommand{\nSL}[2]{\mathrm{SL}_{#2}({#1})}

\newcommand{\nMat}[2]{\mathrm{M}_{#2}(#1)}

\newcommand{\uGL}{{\mathbf{GL}}}
\newcommand{\uPGL}{{\mathbf{PGL}}}
\newcommand{\uSL}{{\mathbf{SL}}}

\newcommand{\nGm}[1]{{\mathbf{G}}_{\mathbf{m},{#1}}}

\newcommand{\dirlim}{\underrightarrow{\lim}\,}
\newcommand{\invlim}{\underleftarrow{\lim}\,}

\newcommand{\units}[1]{{#1^\times}}

\newcommand{\lMod}[1]{{#1}\textrm{-}{\mathrm{Mod}}}

\usepackage[foot]{amsaddr}

\usepackage{xcolor}

\usepackage{dsfont} 

\makeatletter
\newenvironment{chapquote}[2][2em]
  {\setlength{\@tempdima}{#1}%
   \def\chapquote@author{#2}%
   \parshape 1 \@tempdima \dimexpr\textwidth-2\@tempdima\relax%
   \itshape}
  {\par\normalfont\hfill--\ \chapquote@author\hspace*{\@tempdima}\par\bigskip}
\makeatother

\calCapital %
\frakCapital %
\bbCapital %
\catCapital %
\bfCapital %

\newcommand{\Rep}[2][]{\mathrm{Rep}^{\mathrm{#1}}(#2)}
\newcommand{\Irr}[2][]{\mathrm{Irr}^{\mathrm{#1}}(#2)}

\newcommand{\Hecke}[2][\,]{{\catH}_{#1}(#2)}

\newcommand{\co}{\leq_{\textrm{c.o.}}}
\newcommand{\sm}[1]{#1_{\mathrm{sm}}}

\DeclareMathOperator{\Ball}{B}
\DeclareMathOperator{\dist}{d}
\DeclareMathOperator{\mult}{mult}

\newcommand{\leftmod}{ {\setminus} }

\newcommand{\vrt}[1]{{#1}_{\mathrm{vrt}}}

\newcommand{\ori}[1]{#1_{\mathrm{ori}}}
\newcommand{\ordr}[1]{#1_{\mathrm{ord}}}

\newcommand{\CSLC}[1]{C_c^{\infty}(#1)}

\newcommand{\Alg}[3][]{A^{#1}_{#3}({#2})}
\newcommand{\tAlg}[3][]{\what{A}^{#1}_{#3}({#2})}

\newcommand{\cconj}[1]{\quo{#1}} 
\newcommand{\charfunc}[1]{\mathds{1}_{#1}} 

\newcommand{\catSimp}{\mathtt{Sim}}
\newcommand{\catCov}{\mathtt{Cov}}
\newcommand{\catVec}{\mathtt{Vec}}

\newcommand{\catPHil}{\mathtt{pHil}}
\newcommand{\catSet}{\mathtt{Set}}

\newcommand{\udual}[1]{{\what{#1}}}

\newcommand{\norm}[1]{\|{#1}\|}
\newcommand{\Norm}[1]{\left\|{#1}\right\|}
\newcommand{\abs}[1]{|{#1}|}
\newcommand{\Abs}[1]{\left|{#1}\right|}

\newcommand{\wc}{\prec}
\newcommand{\sphere}[2][1]{\mathbb{S}^{#1}({#2})}
\newcommand{\fsubseteq}{\subseteq_f}

\newcommand{\cEnd}{{\mathrm{End}}^{0}}
\newcommand{\cHom}{{\mathrm{Hom}}^{0}}

\newcommand{\ids}[1]{\mathbb{I}(#1)}

\newcommand{\LL}[2][]{\mathrm{L}^{2}_{#1}({#2})}

\newcommand{\llf}{\tilde{\ell}^2}

\newcommand{\e}{\mathrm{e}}

\newcommand{\conv}{\star}

\newcommand{\sfx}{{\mathsf{x}}}
\newcommand{\sfy}{{\mathsf{y}}}

\newcommand{\gap}[1]{{\color{blue} {#1}}}
\newcommand{\change}[2]{#2}

\newcommand{\bGL}{\uGL}
\newcommand{\bPGL}{\uPGL}
\newcommand{\bSL}{\uSL}

\newcommand{\JL}{\mathrm{JL}}
\newcommand{\LJ}{\mathrm{LJ}}

\newcommand{\mSpec}{\mathrm{m\textrm{-}Spec}}

\newcommand{\Flag}{\mathrm{Flag}}
\newcommand{\llFlag}{\Omega_{\Flag}}

\numberwithin{equation}{section}

\title[The Ramanujan Property for Simplicial Complexes]{The Ramanujan Property for Simplicial Complexes}

\author{Uriya A.\ First$^*$}
\date{\today}
\address{$^*$Department of Mathematics, University of British Columbia}
\email{uriya.first@gmail.com}
\thanks{This research was supported by an ERC grant \#226135,
the Lady Davis Fellowship Trust, and the UBC Mathematics Department.}
\keywords{simplicial complex, Ramanujan complex, Ramanujan graph, idempotented $*$-algebra,
spectrum, affine building,
$\ell$-group, reductive group, automorphic form, Ramanujan--Petersson conjecture, Jacquet--Langlands correspondence}

\subjclass[2010]{
	05E18, 
	11F70, 
	22D10, 
	22D25, 
	46L05. 
}

\begin{document}

\maketitle



\begin{abstract}
    Let $G$ be a topological group acting on a simplicial complex $\calX$
    satisfying some mild assumptions. For example, consider a $k$-regular
    tree and its automorphism group, or more generally, a regular affine
    Bruhat-Tits building and its  automorphism group.
    We define and study various types of high-dimensional
    spectra of quotients of $\calX$ by subgroups of $G$.
    These spectra include the spectrum of many natural operators associated with
    the quotients, e.g.\ the high-dimensional Laplacians.

    We
    prove a theorem in the spirit of the Alon-Boppana Theorem, leading to a notion of \emph{Ramanujan quotients}
    of $\calX$.
    Ramanujan $k$-regular graphs
    and  Ramanuajn complexes in the sense of Lubotzky, Samuels and Vishne are \emph{Ramanujan in dimension $0$}
    according to our definition (for $\calX$, $G$ suitably chosen).
    We give a criterion for a quotient of $\calX$ to be Ramanujan which is phrased in terms of
    representations of $G$, and use it, together with deep results about automorphic
    representations, to show that affine buildings of inner forms of $\uGL_n$
    over local fields of positive characteristic  admit infinitely many quotients
    which are \emph{Ramanujan in all dimensions}.
    The Ramanujan (in dimension $0$) complexes constructed by
    Lubotzky, Samuels and Vishne arise as a special case of our construction.
    Our construction also gives rise to Ramanujan graphs which are apparently new.


    Other applications are also discussed.
    For example, we show that there are non-isomorphic simiplicial complexes which are \emph{isospectral in
    all dimensions}.
\end{abstract}

\setcounter{tocdepth}{1} 

\tableofcontents

\setcounter{tocdepth}{2} 

\section*{Introduction}

    Let $X$ be a connected $k$-regular graph. The spectrum of $X$ is defined
    to be the spectrum of its adjacency matrix. It is well-known that the spectrum affects
    many combinatorial properties of $X$ such as behavior of random walks,
    expansion and mixing properties,  and the chromatic number; see
    the survey
    \cite{Lub12A}, for instance.

    The graph $X$ is called \emph{Ramanujan} if all eigenvalues of its adjacency matrix,
    excluding $k$ and $-k$, lie in the interval $[-2\sqrt{k-1},2\sqrt{k-1}]$. Let us explain the motivation behind this definition:
    Denote by $\lambda(X)$ the maximal absolute value of an eigenvalue of the adjacency matrix, excluding
    $k$ and $-k$. Informally, the smaller $\lambda(X)$ is,
    the better the aforementioned combinatorial
    properties of the graph are.
    However, the Alon--Boppana Theorem \cite{Nilli91} states that for any $\veps>0$, there are only finitely
    many non-isomorphic $k$-regular graphs with $\lambda(X)<2\sqrt{k-1}-\veps$.
    Ramanujan graphs can therefore be thought of as having the smallest possible spectrum one can expect
    of an infinite family of graphs.
    In addition, the interval $[-2\sqrt{k-1},2\sqrt{k-1}]$ is the spectrum of the  $k$-regular
    tree (\cite[p.~252, Apx.~3]{Sunada88}, \cite[Th.~3]{Kest59}), which is the universal cover of any $k$-regular graph, so Ramanujan
    graphs can be regarded as finite approximations of  the infinite $k$-regular tree.

    The spectral properties of
    Ramanujan graphs make them
    into supreme \emph{expander graphs}. In addition, they have good mixing properties and large chromatic number, provided they are not bipartite. Some known
    constructions
    have large girth as well.
    Constructing infinite families of non-isomorphic $k$-regular Ramanuajan graphs  is considered difficult.
    The first such families were introduced by Lubotzky, Phillips and Sarnak \cite{LubPhiSar88} and
    independently  by
    Margulis \cite{Marg88},
    assuming $k-1$ is prime. Morgenstern \cite{Morg94} has extended this to the case $k-1$ is a prime power.
    These works rely on deep results of Delinge \cite{Deligne74} and Drinfeld \cite{Drinfel88} concerning
    the Ramanujan--Petersson conjecture for $\bGL_2$.
    The existence of infinitely many $k$-regular bipartite Ramanujan graphs for arbitrary $k$ was later shown  by Marcus, Spielman and
    Srivastava \cite{MarSpiSri14} using different methods; the non-bipartite case remains open.

\medskip

    A high-dimensional generalization of Ramanujan graphs, called \emph{Ramanujan complexes},
    was suggested
    by Cartwright, Sol\'{e} and \.{Z}uk \cite{CarSolZuk03}, and later
    refined by Lubotzky, Samuels and Vishne \cite{LubSamVi05}
    (see \cite{JorLiv00} 
    for another generalization of  Ramanujan graphs).
    These complexes are quotients of the affine Bruhat-Tits building of $\nPGL{F}{d}$,
    denoted  $\calB_d(F)$,
    where
    $F$ is a non-archimedean local field. The building $\calB_d(F)$ is a contractible
    simplicial complex of dimension $d-1$.
    The spectrum of a quotient
    of $\calB_d(F)$, i.e.\ a simplicial complex whose universal cover is $\calB_d(F)$,
    consists of the common spectrum of a certain
    family of $d-1$ linear operators associated with the quotient, called the \emph{Hecke operators}.
    According to Lubotzky, Samuels and Vishne  \cite{LubSamVi05}, a quotient of $\calB_d(F)$ is \emph{Ramanujan} if its spectrum,
    which is a subset of $\C^{d-1}$, is contained in the spectrum of the universal cover $\calB_d(F)$ together with a certain family
    of $d$ points in $\C^{d-1}$, called the \emph{trivial spectrum}.

    Li \cite[Thm.~4.1]{Li04} proved a theorem in the spirit of the Alon--Boppana Theorem for quotients of $\calB_d(F)$:
    If $\{X_n\}_{n\in\N}$ is a family of such quotients satisfying a mild assumption,
    then the closure of the union of the spectra of $\{X_n\}_n$ (in $\C^{d-1}$) contains
    the spectrum of the universal cover $\calB_d(F)$. Ramanujan complexes can therefore be thought of as having
    the smallest possible spectrum that can be expected of an
    infinite family, or  as  spectral approximations of the universal cover
    $\calB_d(F)$, similarly to Ramanujan graphs. When $d=2$, the complex $\calB_d(F)$ is a regular tree,
    and  Ramanujan complexes
    are just Ramanujan graphs in the previous sense.

    The
    existence of infinite families of Ramanujan complexes was shown by Lubotzky, Samuels and Vishne in \cite{LubSamVi05}
    (see also~\cite{LubSamVish05B}),
    using Lafforgue's proof of the Rama\-nu\-jan--Peters\-son conjecture for $\bGL_d$ in positive characteristic  \cite{Laff02}.\change{32}{\footnote{
    	The proof in \cite{LubSamVi05} assumed the global Jacquet--Langlands correspondence
    	for $\uGL_n$
    	in positive characteristic that was established later in \cite{BadulRoch14}.
    }
    Li \cite{Li04} has independently obtained very similar results using
    a special case of the conjecture established by Laumon, Rapoport and Stuhler \cite[Th.~14.12]{LaRaSt93}
    (the notion of Ramanujan complexes used in \cite{Li04} is slightly
    weaker than the one used in \cite{LubSamVi05}, but the constructions of  \cite{Li04} are in fact Ramanujan
    in the sense of \cite{LubSamVi05}).}
    As in the case of graphs, Ramanujan complexes enjoy various good combinatorial properties:
    They have high chromatic number \cite[\S6]{EvraGolLub14}, good mixing properties \cite[\S4]{EvraGolLub14},
    they
    satisfy Gromov's \emph{geometric expansion property} \cite{FoGrLaNaPa12} (see also \cite{Gro10},
    \cite{KauKazhLub14}), and the constructions of \cite{LubSamVi05} have high girth in addition \cite{LubMesh07}.

\medskip

    The Ramanujan property of quotients of $\calB_d(F)$ is measured with respect to the spectrum
    of the \emph{Hecke operators}. In a certain sense, to be made precise in Example~\ref{EX:zero-dim-spec-of-Bd}
    below,
    these operators
    capture all spectral information in dimension $0$. Therefore, we regard the spectrum of Lubotzky, Samuels and Vishne
    as the \emph{$0$-dimensional spectrum}. However, one can associate other operators with a simplicial complex such
    that their spectrum
    affect combinatorial properties. For example, this is the case for the high-dimensional Laplacians; see
    for instance \cite{ParRosTes13}, \cite{Golubev13},
    \cite{GolPar14}, \cite{Par15}.
    Other examples are adjacency operators
    between  various types of facets.
    These operators are high-dimensional in nature and so
    their spectrum is {a
    priori} not determined by the spectrum of the  Hecke operators.
    The  purpose of this work is to  treat  these and other high-dimensional operators, and
    to
    construct  examples of complexes which are Ramanujan relative to such operators.

\medskip

	In more detail, let $\calX$ be a simplicial complex and let $G$ be a group of automorphisms of $\calX$ satisfying certain mild assumptions.
    For example, one can take $\calX$ to be a  $k$-regular
    tree  $\calT_k$ and $G=\Aut(\calT_k)$, or $\calX=\calB_d(F)$ and $G=\nPGL{F}{d}$.
    Even more generally, $\calX$ can be an affine Bruhat-Tits building (see \cite{Buildings08AbramBrown}), and
    $G$ can be a group of automorphisms acting  on $\calX$ in a sufficiently transitive manner.
    With every quotient of $\calX$ by subgroups of $G$, called $G$-quotient for brevity,
    we will associate various types of spectra.
    Among them is the
    \emph{(non-oriented) $i$-dimensional spectrum}.
    When $\calX=\calB_d(F)$ and $G=\nPGL{F}{d}$, or when $\calX=\calT_k$
    and $G=\Aut(\calT_k)$, our $0$-dimensional spectrum coincides with
    the spectra of  quotients of regular graphs and quotients of $\calB_d(F)$  discussed earlier.

    We prove a theorem in the sprit of the Alon--Boppana Theorem, generalizing Li's aforementioned theorem,
    stating that if $\{X_n\}_{n\in\N}$ is a family of  $G$-quotients of $\calX$ satisfying a mild assumption, then
    the closure of $\bigcup_{n\in\N}\Spec(X_n)$ contains $\Spec(\calX)$ (Theorem~\ref{TH:Alon-Boppana-I}).
    This in turn leads to a notion of \emph{Ramanujan $G$-quotients} of $\calX$. In analogy with Ramanujan graphs and
    Ramanujan complexes, these complexes have the smallest possible spectrum one can expect of an infinite family of $G$-quotients
    of $\calX$. Alternatively, they can be regarded as spectral approximations of the covering complex $\calX$.
    When $\calX$ is a $k$-regular tree (resp.\ $\calB_d(F)$), the quotients of $\calX$ which
    are \emph{Ramanujan in dimension $0$} are precisely the Ramanujan graphs (resp.\ Ramanujan complexes).

    We proceed with establishing a criteria for a quotient of
    $\calX$ to be Ramanujan. We show that if $\Gamma\leq G$ is a subgroup such that $\Gamma\leftmod\calX$
    is a simplicial complex and $\calX\to \Gamma\leftmod\calX$ is a cover map,
    then the Ramanujan property of $\Gamma\leftmod\calX$ is equivalent to a certain condition on
    the $G$-representation $\LL{\Gamma\leftmod G}$ (Theorem~\ref{TH:Ramanujan-criterion}).
    This generalizes a similar criterion in \cite{LubSamVi05} for the case $\calX=\calB_d(F)$.
    More generally, we show that there is a one-to-one correspondence between  the spectrum $\Gamma\leftmod\calX$
    and a certain class of $G$-subrepresentations of $\LL{\Gamma\leftmod G}$ (Theorem~\ref{TH:spectrum-correspondence}).
    In case $\calX$ is the affine Bruhat-Tits building of an almost simple algebraic group $\bfG$ over $F$,
    $G=\bfG(F)$, and $\Gamma$ is an arithmetic cocompact lattice in $G$, we restate
    our criterion in terms of automorphic representations of $\bfG$
    (Theorem~\ref{TH:automorphic-ramanujan}).

    Finally, we apply  our automorphic criterion together with
    Lafforgue's proof of the Ramanujan--Petersson conjecture
    for $\uGL_d$ in positive characteristic \cite{Laff02} and the
    establishment of the global Jacquet--Langlands correspondence for $\uGL_d$ in positive characteristic
    by Badulescu and Roche \cite{BadulRoch14}, to give new
    examples of Ramanujan complexes.
    Specifically, let $F$ be a non-archimdean local field of positive characteristic, let $D$ be a central division
    $F$-algebra, let $G=\nPGL{D}{d}=\nGL{D}{d}/\units{F}$ and let $\calB_d(D)$ be the affine Bruhat-Tits
    building of $G$. Then $\calB_d(D)$ admits infinitely many $G$-quotients which
    are Ramanujan in all dimensions (Theorem~\ref{TH:ram-quo-exists}). (In fact, these quotients
    are \emph{completely Ramanujan}.)
    For example, the spectrum of the high dimensional Laplacians
    of these complexes is contained in the union of the
    spectrum of the high dimensional Laplacians of the  universal cover $\calB_d(D)$ together with the \emph{trivial spectrum}.
    When $D=F$, our Ramanujan complexes are the Ramanujan
    complexes constructed in \cite{LubSamVi05}. Thus, the Ramanujan complexes
    of Lubotzky, Samuels and Vishne \cite{LubSamVi05}, which are Ramanujan in dimension $0$ in our setting,
    are in fact Ramanujan on all dimensions.
    When $d=2$, our construction gives rise to  Ramanujan graphs, which seem to be new when $D\neq F$.

\medskip

    The machinery that we introduce has other applications:
    When combined with results from \cite{LubSamVi06},
    we obtain examples of non-isomorphic quotients of $\calB_d(F)$ which are isospectral in
    all dimensions, and more generally, \emph{completely isospectral} (Example~\ref{EX:isospectral-complexes}). In addition, using the classification
    of irreducible representations of $\nGL{F}{d}$ (resp.\ $\Aut(\calT_k)$),
    we show that if $\calX=\calB_3(F)$ (resp.\ $\calX=\calT_k$),
    then a quotient of $\calX$ is Ramanujan in dimension $0$ if and only if it is Ramanujan in all dimensions (Propositions~\ref{PR:Raman-graph-equiv-conds}
    and~\ref{PR:PGLiii-special-case}).
    In the case  $\calX=\calB_3(F)$, this agrees  with the related result \cite[Th.~2]{KaLiWa10}.
    In addition, the results of Marcus, Spielman and Srivastava on the existence of $k$-regular Ramanujan graphs
    \cite{MarSpiSri14}, and their generalizations by Hall, Puder and Sawin \cite{HallPudSaw16} can
    be translated into  representation-theoretic results about the automorphism group of a $k$-regular tree, which resemble
    the Ramanujan--Petersson conjecture for $\bGL_2$ (Corollary~\ref{CR:ramanujan-subgroups-of-tree} and
    the comment after Proposition~\ref{PR:Raman-graph-equiv-conds}).
    Finally, we show that in certain cases, the Ramanujan property is unaffected by replacing the group $G\subseteq \Aut(\calX)$ with a commensurable
    subgroup, even though it affects the way the spectrum is defined (Theorem~\ref{TH:finite-index-subgroup}).

\medskip

    We now give some brief details about how we define the spectrum of quotients of $\calX$.
    For a simplicial complex $X$ and $i\geq 0$, let $\Omega_i^+(X)$
    denote the vector space of
    $\C$-valued functions of finite support on
    the  $i$-dimensional cells in $X$, and let $\Omega_i^-(X)$ denote
    the space of \emph{$i$-dimensional forms}  on $X$ with finite support (see~\ref{subsec:orientation}).
    Let $F$ denote $\Omega_i^+$ or $\Omega_i^-$ for some fixed $i$; in general, $F$ can be taken to be a \emph{semi-elementary} functor
    from the category of $G$-quotients of $\calX$ to the category of pre-Hilbert spaces (Definition~\ref{DF:elementary-functor}).
    Every such $F$ gives rise to a certain kind of spectrum
    that can be associated with $G$-quotients of $\calX$, called the \emph{$F$-spectrum} (\ref{subsec:spectrum-of-simp}).
    A $G$-quotient of $\calX$ which is Ramanujan with respect to this spectrum is called \emph{$F$-Ramanujan} (\ref{subsec:Ramanujan-quotients}).
    The non-oriented (resp.\ oriented) $i$-dimensional spectrum is obtained by taking $F=\Omega_i^+$ (resp.\ $F=\Omega_i^-$),
    and
    the $G$-quotients of $\calX$ which are $\Omega_i^+$-Ramanujan are called \emph{Ramanujan in dimension $i$}.

    The action of $G$ on   $\calX$ induces
    an action on $F\calX$.
    Let $A$ denote the algebra of $G$-equivariant linear operators on $ F\calX$.
    It turns out that elements of $A$ act naturally on $F(\Gamma\leftmod\calX)$ for every
    subgroup $\Gamma\leq G$. The linear operators
    that our  $F$-spectrum takes into consideration are  the elements of
    $A$. For example, when $F=\Omega_i^-$,   the $i$-dimensional Laplacian
    can be regarded as an element of $A$, and when $F=\Omega_i^+$, the algebra $A$ contains many natural adjacency operators
    between $i$-dimensional cells. In addition, for $\calX=\calB_d(F)$ and $F=\Omega_0^+$, the Hecke operators live in $A$,
    and in fact generate it.
    We note that one can replace $A$ with a subalgebra of interest.

    Let $X$ be a  $G$-quotient of $\calX$. Naively, one can define the $F$-spectrum of $\calX$ as
    the {common} spectrum of the operators in $A$ on $FX$.
    This works well if $A$ is commutative (e.g.\ when $\calX=\calB_d(F)$ and $F=\Omega_0^+$), but this definition
    is unsatisfactory in general.
    Instead, we define
    the $F$-spectrum of $X$ as the irreducible $A$-submodules of $FX$.
    For such a definition to work, some complementary theory has to be developed.
    For example,
    the collection of irreducible $A$-modules has to be given a topology
    (because our generalization of Li's aforementioned theorem uses ``closure''),
    and one has to define a notion of a \emph{continuous spectrum} --- irreducible submodules of $FX$
    suffice when $X$ is finite, but this does not work in the infinite case, e.g.\ for  $X=\calX$.
    These and other technicalities are discussed in Chapter~\ref{sec:involutary-algebras}.
    They are resolved by introducing a canonical involution on $A$
    and invoking the spectral theory of \emph{$*$-algebras} (\cite{Palmer01}, for instance).

    We remark that in this general setting, some fundamental facts become  difficult to prove.
    For example,
    it is reasonable to expect that if $B$ is a subalgebra of $A$, then the spectrum of $FX$ with respect to $A$
    determines the spectrum of $FX$ with respect to $B$. This is easy to show
    when $X$ is finite, but the general case (Theorem~\ref{TH:subalgebra-spectrum-I})
    requires significant work.
    Another example of a fundamental fact with an involved proof is Theorem~\ref{TH:direct-sum-spectrum},
    which is used in the proof of our generalization of Li's Theorem.

\bigskip

    The paper is organized as follows: Chapter~\ref{sec:ramanujan-complexes} is preliminary and recalls  Ramanujan complexes
    as defined in \cite{LubSamVi05}.
    Chapter~\ref{sec:involutary-algebras} concerns with developing a spectral theory for \emph{idempotented $*$-algebras}.
    This chapter is long and technical, but the results and the definitions introduced  there are fundamental for the following chapters.
    Chapter~\ref{sec:simplicial-complexes} recalls simplicial complexes and certain facts about $\ell$-groups acting on them.
    In Chapter~\ref{sec:adj-operators}, we introduce our notion of  spectrum, give some examples, and discuss
    issues such as dependencies between different types of spectra.
    In Chapter~\ref{sec:optimal-spectrum}, we prove an generalization of Li's Theorem (reminiscent of the Alon-Boppana Theorem)
    and characterize the \emph{trivial spectrum}, which leads to the definition of Ramanujan quotients of $\calX$.
    Chapter~\ref{sec:spectrum-and-reps} gives a representation-theoretic criterion for a quotient of $\calX$ to be Ramanujan. Consequences of
    this criterion are also discussed.
    Finally, in Chapter~\ref{sec:existence}, we focus on the case where $\calX$ is the affine
    Bruhat-Tits building of an almost simple algebraic group $\bfG$ over a non-archimedean local field.
    We relate the Ramanujan property with properties of certain automorphic representations
    of $\bfG$, and use it to show that the affine building of $\nPGL{D}{d}$ admits infinitely
    many quotients which are Ramanujan in all dimensions.

\subsection*{Acknowledgements}

    We owe a debt of gratitude to Alex Lubotzky for presenting us with the theory of Ramanujan graphs
    and suggesting this research project. We are also in debt to Lior Silberman for many beneficial discussions.
    We further thank Anne-Marie Aubert, Alexandru Ioan Badulescu,
    David Kazhdan, Laurent Lafforgue and Dipendra Prasad for short-yet-crucial   correspondences,
    all concerning Chapter~\ref{sec:existence}. In addition, Amitay Kamber has given us several useful suggestions, for which are
    grateful.
    Finally, we thank the participants of the Ramanujan
    Complexes Seminar that took place at the Hebrew University in the winter of 2013.

\section*{Notation}

	All vector spaces and algebras are over $\C$. The complex conjugate
	of $z\in\C$ is denoted $\cconj{z}$.
    Algebras are always associative but not necessarily
    unital.

    Inner products of pre-Hilbert and Hilbert spaces
    will be denoted by triangular brackets $\Trings{~,~}$ with the convention that the left component is $\C$-linear.
    The unit sphere of a normed space $V$ is denoted by $\sphere{V}$. The completion of a pre-Hilbert
    $V$ is denoted by $\quo{V}$.
    If $T:V\to V'$ is a linear operator between pre-Hilbert spaces, a {dual} of $T$ is an operator $T^*:V'\to V$
    satisfying $\Trings{Tv,v'}=\Trings{v,T^*v'}$ for all $v\in V$, $v'\in V'$. The dual is unique if it exists,
    and it is guaranteed to
    exist when $V$ and $V'$ are Hilbert spaces and $T$ is bounded.
	
\medskip

    For a set $X$, we let $\llf(X)$ denote the set of functions $\vphi:X\to \C$ with finite
    support. We endow $\llf(X)$ with the inner product $\Trings{\vphi,\psi}=\sum_{x\in X}\vphi x\cdot\cconj{\psi x}$.
    This makes $\llf(X)$ into a pre-Hilbert space. Its completion is the Hilbert space of square-summable functions
    on $X$, denoted $\ell^2(X)$.
    The vector space $\llf(X)$ admits a standard basis $\{\e_x\}_{x\in X}$
    defined by
    \[
    \e_x(y)=\left\{\begin{array}{ll}
    1 & x=y \\
    0 & x\neq y
    \end{array}\right.\ .
    \]
    If $Y$ is another set and $f:X\to Y$ is any function, then we define $f_*:\llf(X)\to \llf(Y)$
    by $(f_*\vphi)y=\sum_{x\in f^{-1}\{y\}}\vphi(x)$ for all $\vphi\in\llf(Y)$, $y\in Y$.
    In particular, we have
    \[
    f_*\e_x=\e_{f(x)}\qquad\forall\, x\in X\ .
    \]
    This makes $X\mapsto \llf(X)$ into a \emph{covariant} functor from
    the category of sets to the category of pre-Hilbert spaces
    (non-continuous morphisms are allowed).

\medskip

	If $X$ is an $\ell$-space, i.e.\ a totally disconnected  locally compact Hausdorff
	topological space,
	we let $\CSLC{X}$ denote the vector space of compactly supported locally constant functions
	$\vphi:X\to\C$. If $X$ is equipped with a Borel measure $\mu$, which will always be obvious
	from the context, we
	make $\CSLC{X}$ into a pre-Hilbert space by setting
    \[\Trings{\vphi,\psi}=\int_{x\in X}\vphi x\cdot \cconj{\psi x}\,\mathrm{d}\mu\ .\]
	The completion of $\CSLC{X}$ is clearly $\LL{X}$, the space of square-integrable
	functions on $X$ (considered up to equivalence).

\medskip

    We write $F\fsubseteq Z$ to denote that $F$ is a finite subset of $Z$.
    If $G$ is an $\ell$-group, i.e.\ a totally disconnected locally compact
    topological group, we write $K\co G$
    to denote that $K$ is a compact open subgroup of $G$.
    It is well-known that the compact open subgroups of an $\ell$-group
    form a basis of neighborhoods at the identity.

\section{Ramanujan Complexes}
\label{sec:ramanujan-complexes}

    This  preliminary chapter recalls Ramanujan
    complexes as defined
    by Lubotzky Samuels and Vishne \cite{LubSamVi05}, basing on the work
    of Cartwright, Sol\'e and \.Zuk \cite{CarSolZuk03}.

\medskip



        Let $F$ be a non-archimedean local field with additive valuation $\nu$, let $\calO$ be the integer
        ring of $F$, and let $\pi\in \calO$ be a uniformizer, i.e.\ a generator of the maximal ideal
        of $\calO$. We assume that $\nu(\pi)=1$. Fix an integer $d\geq 2$
        and
        let
        \begin{align*}
        G&=\nPGL{F}{d}=\nGL{F}{d}/\units{F}\ ,\\
        K&=\nPGL{\calO}{d}:=\im(\nGL{\calO}{d}\to\nPGL{F}{d})\ .
        \end{align*}
        For $g\in\nGL{F}{d}$, denote by $\quo{g}$ the image of $g$ in $G$.
        The topology of $F$ induces a topology on $G$, making it into an
        $\ell$-group. The subgroup $K$ is compact and open
        in $G$, and
        the function
        \begin{eqnarray*}
            c\,:\,G &\to& \Z/d\Z\\
            \quo{g}&\mapsto& \nu(\det g)+d\Z
        \end{eqnarray*}
        is well-defined and satisfies $c(K)=0$.

\medskip

        The affine Bruhat-Tits building of $G$, denoted $\calB_d(F)$, is a
        $(d-1)$-dimensional simplicial
        complex constructed as follows:
        The vertices of $\calB_d(F)$ are the right cosets $G/K$.
        They admit a $d$-coloring $C_0:G/K\to \Z/d\Z$ given by
        \[
        C_0(gK)=c(g)\qquad\forall g\in G\ .
        \]
        To define the edges of $\calB_d(F)$, let
        \[
        g_1=\left[\begin{smallmatrix} \pi & & & & \\ & 1 & & & \\ & & 1 & & \\ & & & \ddots & \\ & & & & 1\end{smallmatrix}\right],\quad
        g_2=\left[\begin{smallmatrix} \pi & & & & \\ & \pi & & & \\ & & 1 & & \\ & & & \ddots & \\ & & & & 1\end{smallmatrix}\right],~ \dots~,\quad
        g_{d-1}=\left[\begin{smallmatrix} \pi & & & & \\ & \pi & & & \\ & & \ddots & & \\ & & & \pi & \\ & & & & 1\end{smallmatrix}\right]
        \in\nGL{F}{d}
        \]
        The double cosets $\{K\quo{g_i}K\}_{i=1}^{d-1}$ are distinct since $c(K\quo{g_i}K)=i+d\Z$.
        Two vertices ${g}K,{g'}K\in G/K$ are adjacent if
        \[
        {g^{-1}g'}\in K\cup K\quo{g_1}K\cup K\quo{g_2}K\cup\dots K\quo{g_{d-1}}K\ .
        \]
        The $i$-dimensional cells of $\calB_d(F)$ are the $(i+1)$-cliques, namely,
        they are
        sets $\{h_0K,\dots,h_{i+1}K\}\subseteq G/K$
        consisting of pairwise adjacent vertices.
        The resulting complex is indeed a {pure} $(d-1)$-dimensional contractible simplicial complex,
        which carries additional structure making it into an \emph{affine building}; see \cite[\S6.9]{Buildings08AbramBrown}
        for further details.
        (See also \cite{Tits79} for the definition of the affine building of a general reductive $p$-adic Lie group,
        and \cite{AbramNebe02} for an explicit description of buildings of classical groups.)
        Note that every $(d-1)$-simplex
        in $\calB_d(F)$ consists of $d$ vertices of different colors.

\medskip

        There is an evident left action of $G$ on $\calB_d(F)$ which respects the simplicial structure.
        Let $\Gamma\leq G$ be a discrete subgroup.
        Under mild assumptions (see Corollary~\ref{CR:distance-condition}
        below), one can form the quotient complex $\Gamma\leftmod \calB_d(F)$: Its vertices
        are the double cosets $\Gamma\leftmod G/K$, and its $i$-dimensional cells are obtained by projecting
        the $i$-dimensional cells of $\calB_d(F)$ pointwise into $\Gamma\leftmod G/K$.
        Since $\calB_d(F)$ is contractible (when viewed as a topological space), it is the universal
        cover of $\Gamma\leftmod \calB_d(F)$.
        Furthermore, when $\Gamma$ is cocompact in $G$,
        the simplicial complex $\Gamma\leftmod\calB_d(F)$ is finite.

        \begin{example}
            The complex $\calB_2(F)$ is a $q+1$ regular tree, where $q$ is the size of the residue
            field of $(F,\nu)$. The quotients $\Gamma\leftmod \calB_2(F)$ are therefore
            $(q+1)$-regular graphs.
        \end{example}

        The coloring $C_0$ of the vertices of $\calB_d(F)$ does not descend to $X:=\Gamma\leftmod\calB_d(F)$ in general.
        However,
        we can define a color function on the \emph{directed} edges of $X$ by
        \[
        C_1(\Gamma gK,\Gamma g'K):=c(g^{-1}g')\in \Z/d\Z
        \]
        It can be checked that $C_1$ is well-defined when $\Gamma\leftmod \calB_d(F)$ is a simplicial complex.
        Since $g^{-1}g'\in \bigcup_{i=1}^{d-1}K\quo{g_i}K$ whenever $(gK,g'K)$ is an edge of
        $\calB_d(F)$, we have
        \[C_1(\Gamma gK,\Gamma g'K)\in\{1+d\Z,\dots,(d-1)+d\Z\}\ .\]
        Write $\vrt{X}=\Gamma\leftmod G/K$ and define $a_1,\dots,a_{d-1}:\LL{\vrt{X}}\to \LL{\vrt{X}}$
        by
        \[
        (a_i\vphi)x=\sum_{\substack{y\in \vrt{X}\\C_1(x,y)=i}}\vphi y\qquad\forall \,\vphi\in \LL{\vrt{X}},\, x\in \vrt{X}\ .
        \]
        The operators $a_1,\dots,a_{d-1}$ are called the \emph{colored adjacency operators} or \emph{Hecke operators}
        of $X$.
        It turns out that $a_1,\dots,a_{d-1}$ commute with each other and $a_i^*=a_{d-i}$ for all $i$.
        We may therefore consider the {common} spectrum
        \[
        \Spec(a_1,\dots,a_{d-1})\subseteq \C^{d-1}\ .
        \]
        This set  is called the \emph{spectrum of $X$}
        and
        denoted
        $\Spec_0(X)$. The reason for the subscript $0$ will become clear later in the text
        (specifically, in Example~\ref{EX:zero-dim-spec-of-Bd}).

        \begin{example}
            In the case of $\calB_2(F)$, the operator $a_1$ is just the adjacency operator
            of the graph $\Gamma\leftmod\calB_2(F)$, and spectrum of $X=\Gamma\leftmod \calB_2(F)$
            is just $\Spec(a_1:\LL{\vrt{X}}\to\LL{\vrt{X}})$. Thus,
            $\Spec_0(X)$ is just the  spectrum of $X$  viewed as a regular graph.
        \end{example}

        Recall from the introduction that
        one is generally
        interested in finding quotients $\Gamma\leftmod \calB_d(F)$ whose spectrum
        is ``small'', as they are likely to have good combinatorial properties. More precisely,
        one is interested in constructing  infinite families of non-isomorphic quotients, all of which
        admit the same (nontrivial) bounds on the spectrum. In the case  $d=2$, this is equivalent to asking
        for an infinite family of $(q+1)$-regular expander  graphs.
        The question of how small can the union of the spectra
        of a   family of non-isomorphic quotients of $\calB_d(F)$ be
        was answered, to a certain extent, by Li \cite[Th.~4.1]{Li04}, who proved
        the following theorem:

        \begin{thm}[Li]\label{TH:Li-a}
            Let $\{X_n\}_{n\in \N}$ be a family of quotients of $\calB_d(F)$.
            Let $r_n$ be the maximal integer $r$ for which $X_n$ contains a copy
            of a ball  of radius $r$ in $\calB_d(F)$. If $\limsup r_n=\infty$,
            then the closure of $\bigcup_{n\in\N}\Spec_0(X_n)$ in $\C^{d-1}$
            contains $\Spec_0(\calB_d(F))$.
        \end{thm}

        The set $\Spec_0(\calB_d(F))$ was determined in \cite[Th.~2.11]{LubSamVi05}.
        In addition to the bound of Theorem~\ref{TH:Li-a},
        the spectrum of a \emph{finite} quotient of $\calB_d(F)$ always contains at least one
        of $d$ special points in $\C^{d-1}$ called the \emph{trivial eigenvalues}; we refer
        the reader to \cite[\S2.3]{LubSamVi05} or Example~\ref{EX:trivial-spec-complex} below for their description.

\medskip

        The following definition was suggested by Lubotzky, Samuels and Vishne \cite{LubSamVi05},
        following
        Carwright, Sol\'e and \.Zuk  \cite{CarSolZuk03}.

        \begin{dfn}[{\cite[Df.~1.1]{LubSamVi05}}]\label{DF:Ramanujan-Ad-complexes}
            The complex $\Gamma\leftmod\calB_d(F)$ is called \emph{Ramanujan} if
            $\Spec_0(\Gamma\leftmod\calB_d(F))$ is contained in the union of $\Spec_0(\calB_d(F))$
            with the set of trivial eigenvalues.
        \end{dfn}

        By the previous discussion,
        Ramanujan complexes can be regarded as quotients of $\calB_d(F)$
        whose spectrum is as small  as one dares to hope, or alternatively, as (finite) spectral estimations
        of their universal cover $\calB_d(F)$. Constructions of Ramanujan complexes were
        given by Lubotzky, Samuels and Vishne in \cite{LubSamVi05} and \cite{LubSamVish05B}.

        \begin{example}
            When $d=2$, we have $\Spec_0(\calB_2(F))=[-2\sqrt{q},2\sqrt{q}]$,
            and the trivial eigenvalues are $-q-1$ and $q+1$ (where $q=|\calO/\pi\calO|$).
            Therefore, a quotient $X=\Gamma\leftmod \calB_2$ is Ramanujan if
            $\Spec(X)\subseteq [-2\sqrt{q},2\sqrt{q}]\cup\{-q-1,q+1\}$.
            This agrees with the usual definition of Ramanujan $(q+1)$-regular graphs.
        \end{example}

        The present work extends the ideas of Lubotzky, Samuels, Vishne \cite{LubSamVi05} and Li
        \cite{Li04} to more general simplicial complexes, and to  operators different than the
        Hecke operators,
        e.g.\ the high-dimensional Laplacians.
		Loosely speaking, we will show how to define a Ramanujan property with respect to any
		prescribed family of \emph{associated operators}, and show that there are simplicial
		complexes which are Ramanujan with respect to any such family. For example, the Ramanujan complexes
		of \cite{LubSamVi05} will be shown to have this property.
		After we give our general definition of the Ramanujan property
		in \ref{subsec:Ramanujan-quotients}, we
		will address the Ramanujan complexes of Definition~\ref{DF:Ramanujan-Ad-complexes}
		as $G$-quotients of $\calB_d(F)$ which are  \emph{Ramamnujan in dimension $0$}.

\section{Idempotented $*$-Algebras}
\label{sec:involutary-algebras}

    As a preparation for the next chapters, this chapter develops a spectral theory
    for \emph{idempotented $*$-algebras}.
    This is an attempt to adapt the spectral theory of \emph{$C^*$-algebras}  (see \cite{Dixmier}, \cite[Ch.~14]{Wallach92})
    to the \emph{idempotented algebras} often used in the theory of $p$-adic Lie groups  (see \cite{Cartier79}).
    Of course, this is included in the general theory of $*$-algebras (\cite{Palmer01}, for instance), but the idempotented case is more tame in nature.

    The reader will find many similarities  with the theory of $C^*$-algebras.
    However, the absence of a topology causes certain differences that are pointed throughout.

\subsection{A Motivating Example}
\label{subsec:motivating-example}

    Recall that if $V$ is a Hilbert space
    and $T:V\to V$
    is a bounded linear operator, then the spectrum of $T$, denoted $\Spec(T)$, is
    the set of elements $\lambda\in\C$ for which $T-\lambda$ is not invertible.

    Suppose $T$ is \emph{normal}, namely $TT^*=T^*T$. Then $V$ can be viewed as a module over the free commutative algebra $A=\C[X,X^*]$
    where $X,X^*$ act as $T,T^*$ respectively. The algebra $A$ carries an involution $*$ taking $X$ to $X^*$ and
    extending the complex conjugation on $\C$. We clearly have
    \begin{equation}\label{EQ:unitary-rep-of-alg}
    \Trings{au,v}=\Trings{u,a^*v}\qquad\forall\, a\in A,\, u,v\in V\ .
    \end{equation}
    This is an example of a \emph{unitary representation} of $(A,*)$. We are interested
    in extracting the datum of $\Spec(T)$ from the action of $A$ on $V$.

\medskip

    In case $V$ is finite dimensional, this can be done as follows:
    Since $A$ is commutative and affine over $\C$, all irreducible $A$-modules
    are $1$-dimensional and are of the form $V_{\lambda,\mu}:=A/(X-\lambda,X^*-\mu)$ for
    $\lambda,\mu\in\C$  uniquely determined. There is essentially one way to make $V_{\lambda,\mu}$ into a Hilbert space.
    However, it is only when $\lambda=\cconj{\mu}$ that $V_{\lambda,\mu}$ becomes unitary (i.e.\
    \eqref{EQ:unitary-rep-of-alg} is satisfied).
    Thus, the irreducible \emph{unitary} representations of $A$ are $\{V_\lambda:=V_{\lambda,\cconj{\lambda}}\}_{\lambda\in \C}$.

    Observe  that for all $v\in V_\lambda$, we have $Xv=\lambda v$ and $X^*v=\lambda^*v$.
    Furthermore, if $\lambda$ is an eigenvalue of $T$ and $v\in V$ is a corresponding eigenvector,
    then $Tv=\lambda v$ and $T^*v=\cconj{\lambda}v$, because
    \begin{align*}
    \Trings{(T^*-\cconj{\lambda})v,(T^*-\cconj{\lambda})v}&=
        \Trings{v,(T-\lambda)(T^*-\cconj{\lambda})v}\\
        &=\Trings{v,(T^*-\cconj{\lambda})(T-\lambda)v}=\Trings{(T-\lambda)v,(T-\lambda v)}\ .
    \end{align*}
    Therefore, $\lambda$ is an eigenvalue of $T$ if and only if  $V_{\lambda}$
    is isomorphic to an $A$-submodule of $V$. Since $V$ is finite dimensional, $\Spec(T)$ consists
    entirely of eigenvalues.
    We thus get a one-to-one correspondence
    \[
    \Spec(T)\quad\longleftrightarrow\quad \left\{\begin{array}{c}\text{\footnotesize irreducible submodules}\\ \text{
    \footnotesize  of $V$, up to isomorphism}\end{array}\right\}
    \]
    given by $\lambda\longleftrightarrow V_\lambda$.

\medskip

    When $V$ is infinite dimensional, $\Spec(T)$ need not consist of eigenvalues, and a different approach has
    to be taken; it will be described in \ref{subsec:weak-containment} below.

\medskip

    The shifting from $T$ to the algebra $A$ allowed us to ``forget'' the special operator $T$
    among the elements of algebra $A$.
    This idea will be elaborated throughout this chapter,
    and ultimately manifest in the way we shall define spectrum of simplicial complexes. In the following chapters, we will encounter algebras
    of operators containing many elements of interest, but without a canonical set of generators.
    Moreover, the algebras will be non-commutative.
    Thus, rather than studying the spectrum of individual elements of the algebra, we shall consider irreducible
    submodules, which can be roughly thought of as the common spectrum of all operators in the algebra.

\medskip

    Before proceeding onward, we recall the following well-known characterization of the spectrum of
    normal linear operators
    $T:V\to V$.

    \begin{prp}\label{PR:normal-operator-spectrum}
        Assume $T:V\to V$ is a normal operator on a Hilbert space. Then $\lambda\in\Spec(T)$ if and only if $T-\lambda$ is not bounded
        from below, i.e.\ for all $\veps>0$, there is a unit vector $v\in \sphere{V}$ such that $\norm{(T-\lambda)v}<\veps$.
    \end{prp}

    \begin{proof}
        \gap{}We only verify the nontrivial direction. Suppose $\lambda\in\Spec(T)$.
        Replacing $T$ with $T-\lambda$, we may assume $\lambda=0$, so we need to show that $T$ is unbounded from
        below. This is clear if $\ker T\neq0$, so we may assume $T$ is injective. Since $\norm{T^*v}=\norm{Tv}$
        for all $v\in V$, $\ker T^*=0$ as well.
        Thus, if $v\in (TV)^\perp$, then $0=\Trings{TV,v}=\Trings{V,T^*v}$,
        and hence $v=0$.
        Therefore, $TV$ is dense in $V$.

        Assume by contradiction that $T$ is  bounded from below, namely, there
        is $c>0$ such that $c\norm{x}\leq \norm{Tx}$. We claim that $TV=V$.
        Indeed, let $y\in TV$. Since $TV$ is dense in $V$, there is a sequence
        $\{x_n\}_{n\in\N}\subseteq V$ such that $Tx_n\longto y$. Since $\norm{x_n-x_m}\leq c^{-1}\norm{Tx_n-Tx_m}$,
        the sequence $\{x_n\}$ is a Cauchy sequence, and its limit $x$ satisfies $Tx=y$, as required.
        Now, since $T$ is injective and surjective, it has an inverse $T^{-1}:V\to V$,
        and since $T$ is bounded from below, $T^{-1}$ is bounded. But this means $0=\lambda\notin\Spec(V)$,
        a contraction.
    \end{proof}

    Let $\{T_1,\dots,T_n\}$ be a  family of  operators on $V$ such that
    $T_1,\dots,T_n,T_1^*,\dots,T_n^*$
    commute.
    In analogy with Proposition~\ref{PR:normal-operator-spectrum}, the \emph{common spectrum} of
    $\{T_1,\dots,T_n\}$, denoted
    \[
    \Spec(T_1,\dots,T_n)\ ,
    \]
    is defined to be the set of tuples $(\lambda_1,\dots,\lambda_n)\in\C^n$ such
    that for all $\veps>0$, there is $v\in\sphere{V}$ with $\norm{T_iv-\lambda_iv}<\veps$
    for all $i$.
    In fact, since the definition makes sense for any set of operators, we will
    also use it  for  arbitrary families of  operators.

\subsection{Idempotented $*$-Algebras}
\label{subsec:idempotented-algebras}

    The term \emph{algebra} always refers to an associative $\C$-algebra which
    is not necessarily unital. In this context, the term \emph{module} also includes non-unital modules.
    For a left $A$-module $V$ and $a\in A$, denote by $a|_V$ the linear operator $[v\mapsto av]\in \End_{\C}(V)$.

    \medskip

    An involution on an algebra $A$ is a map $*:A\to A$ such that $a^{**}=a$, $(a+b)^*=a^*+b^*$, $(ab)^*=b^*a^*$
    and $(\alpha a)^*=\cconj{\alpha} a^*$ for all $a,b\in A$ and $\alpha\in\C$.
    A \emph{$*$-algebra} is an algebra $A$ equipped with an involution, which is always denoted $*$.


    A $*$-algebra $A$ is \emph{idempotented} if for every finite set $F\subseteq A$ there exists
    an idempotent $e\in A$ such that $e=e^*$ and $eae=a$ for all $a\in F$. The set of idempotents
    $e\in A$ with $e^*=e$ is denoted by $\ids{A}$.

    \begin{example}
        Any unital $*$-algebra is idempotented (take $e=1$).
    \end{example}

    Let $A$ be an idempotented $*$-algebra.
    A left $A$-module  $V$ is said to be \emph{smooth} if $AV=V$, or equivalently, if
    for all $v\in V$, there exists  $e\in \ids{A}$ with $ev=v$.
    The smooth left  $A$-modules together with $A$-homomorphisms form an abelian category denoted by $\lMod{A}$.
    The irreducible objects of $\lMod{A}$ are the nonzero $A$-modules $V$ for which $Av=V$ for all $0\neq v\in V$.

    \begin{example}
        Let $A$ be a \emph{unital} involutary algebra. Then
        the smooth $A$-modules are just the unital $A$-modules (i.e.\ modules $V$
        for which $1_A$ acts as $\id_V$).
    \end{example}

    \begin{example}\label{EX:smooth-part-def}
        Every $A$-module $V$ contains a unique maximal smooth submodule, namely, $\sm{V}:=AV$.
    \end{example}

    Another example of an  idempotented $*$-algebra is the \emph{Hecke algebra} of an $\ell$-group,
    which is usually not unital. This case is discussed  in   \ref{subsec:Hecke-algebra} below.

\subsection{Unitary Representations}
\label{subsec:unitary-reps}

    Let $A$ be an idempotented $*$-algebra.
    A \emph{unitary representation} of $A$ is a Hilbert space $V$ equipped with a left
    $A$-module structure such that
    \begin{enumerate}
        \item[(U1)] $\Trings{au,v}=\Trings{u,a^*v}$ for all $a\in A$ and $u,v\in V$,
        \item[(U2)]  $\sm{V}=AV$ is dense in $V$, and
        \item[(U3)] for all $a\in A$, the operator $a|_{V}:V\to V$ is  bounded.
    \end{enumerate}
    We say that $V$ is \emph{irreducible}\footnote{
        Some texts use the term \emph{topologically irreducible}.
    } if it is does not have a \emph{closed}  $A$-submodule.
    Let $\Rep[u]{A}$ denote the category whose objects are unitary representations of $A$
    and whose morphisms are continuous $A$-module homomorphisms.
    We further let $\Irr[u]{A}$ denote the class of irreducible unitary representations of $A$.

    Let $V_1,V_2\in\Rep[u]{A}$. We denote by $\cHom_A(V_1,V_2)$ the continuous
    $A$-homo\-mor\-phi\-sms from $V_1$ and $V_2$. Likewise, $\cEnd_A(V_1)$ denotes the continuous
    endomorphisms of $V_1$.
    Homomorphism which preserve the inner product are called \emph{unitary}.
    We write $V_1\leq V_2$ if there is a unitary injective $A$-homomorphism
    from $V_1$ to $V_2$. The image of $V_1$ in $V_2$ is easily seen to be closed,
    and hence $V_2=V_1\oplus V_1^\perp$.

    \begin{remark}
        If $V\in\Rep[u]{A}$ and $V_1$ is a closed $A$-submodule,
        then $V_1\in\Rep[u]{A}$. Conditions (U1) and (U3) are clear.
        To see (U2), let $P$ be the orthogonal projection onto $V_1$.
        Then $P$ is an $A$-homomorphism. Since $AV$ is dense in $V$, the space
        $P(AV)=AP(V)=AV_1$ is dense in $P(V)=V_1$, as required.
    \end{remark}


    Let $\{V_i\}_{i\in I}\subseteq\Rep[u]{A}$. The direct sum $\bigoplus_i V_i$
    admits an obvious inner-product making it into a pre-Hilbert space. The completion of  $\bigoplus_i V_i$
    is denoted $\hat{\bigoplus}_iV_i$.
    If for all $a\in A$, we
    have $\sup_{i}\norm{a|_{V_i}}<\infty$, then the diagonal action of $A$ on $\bigoplus_iV_i$
    extends to $\hat{\bigoplus}_iV_i$ and we may regard $\hat{\bigoplus}_iV_i$ as a unitary
    representation of $A$. We denote this by writing $\hat{\bigoplus}_iV_i\in\Rep[u]{A}$.
    When $I$ is finite, we always have $\bigoplus_iV_i=\hat{\bigoplus}_iV_i\in\Rep[u]{A}$.

\medskip

    We now recall several well-known facts about unitary representations.

    \begin{thm}[Schur's Lemma]\label{TH:Schur-lemma}
        Let $V\in \Irr[u]{A}$. Then $\cEnd_A(V)=\C\id_V$.
    \end{thm}

    \begin{proof}
        This is similar to the proof of \cite[Pr.~2.3.1]{Dixmier}. Alternatively, see
        \cite[Th.~9.6.1]{Palmer01}.
    \end{proof}

    \begin{cor}\label{CR:commutative-algs}
        If $A$ is commutative, then all irreducible unitary representations of $A$
        are $1$-dimensional.
    \end{cor}

    \begin{proof}
        Let $V\in\Irr[u]{A}$. Since $A$ is commutative, the map $a\mapsto a|_{V}:A\to\End_{\C}(V)$
        takes values in $\cEnd_A(V)$, which equals $\C \id_V$ by Schur's Lemma. This means that $Av$ is $1$-dimensional
        for all $0\neq v\in V$, hence $\dim V=1$.
    \end{proof}

    \begin{prp}\label{PR:iso-implies-unitary-iso}
        Let $V,V'\in\Rep[u]{A}$. If there exists a continuous $A$-module isomorphism
        $f:V\to V'$, then there exists a \emph{unitary} isomorphism
        $g:V\to V'$. (This holds even without assuming $V$ and $V'$ satisfy condition (U2).)
    \end{prp}

    \begin{proof}
        This is similar to the argument given in \cite[\S2.2.2]{Dixmier}.
        (Briefly, let $f=U|f|$ be the \emph{polar decomposition} of $f$,
        where $|f|$ is the positive square root of $f^*f$. One can show that $g=U$ is the required
        isomorphism.)
    \end{proof}

\subsection{States}
\label{subsec:states}

    In analogy with the theory of
    $C^*$-algebras and locally compact groups, we now define \emph{states}.
    Throughout, $A$ is an idempotented  $*$-algebra.
    We make
    \[\ids{A}:=\{e\in A\where e^2=e,\, e^*=e\}\]
    into
    a directed set by setting $e\leq e'$ when $e'ee'=e$.
    In addition, let $A^\vee=\Hom_{\C}(A,\C)$.
    We shall make repeated use of the fact that $\norm{e|_V}\leq 1$ for all $V\in\Rep[u]{A}$
    and $e\in\ids{A}$, which holds since $e|_V$ is an orthogonal projection.

    \begin{lem}\label{LM:approximation-lemma}
        For all $V\in\Rep[u]{A}$ and $v\in V$,
        the net $\{ev\}_{e\in\ids{A}}$  converges to $v$.
    \end{lem}

    \begin{proof}
        Let $\veps>0$. We need to find $e\in \ids{A}$
        such that $\norm{v-e'v}\leq \veps$ for all $e'\geq e$.
        Since $AV$ is dense in $V$, there  is a unit vector
        $u\in \sphere{AV}$ with $\norm{v-u}<\frac{\veps}{2}$. Choose $e\in \ids{A}$ with $eu=u$.
        Then  for all $e'\geq e$, we have $e'u=e'eu=eu=u$. Thus,
        $\norm{u-e'v}=\norm{e'(u-v)}\leq \norm{e'|_V}\norm{v-u}\leq \norm{v-u}=\frac{\veps}{2}$.
        Therefore, for all $e'\geq e$, we have $\norm{v-e'v}\leq\norm{v-u}+\norm{u-e'v}<\veps$, as required.
    \end{proof}

    With some work, the following theorem and its corollaries can be derived from results
    in \cite[\S9.4]{Palmer01}, which treat the more complicated
    situation of general $*$-algebras. However, we found it easier and more comprehensive to include proofs here
    than
    explaining how to derive everything from \cite[\S9.4]{Palmer01}.
    (See also \cite[\S2.4]{Dixmier}
    for the special case of $C^*$-algebras.)

    \begin{thm}\label{TH:states}
        Let $V\in\Rep[u]{A}$ and  $v\in \sphere{V}$. Define $\vphi=\vphi_{V,v}\in A^\vee$ by
        \[
        \vphi_{V,v}(a)=\Trings{av,v}\ .
        \]
        Then $\vphi$ satisfies:
        \begin{enumerate}
            \item[(S1)] $\vphi(a^*a)\in \R_{\geq 0}$ for all $a\in A$.
            \item[(S2)] For all $b\in A$, there is $r\in \R$
            such that ${\vphi(a^*b^*ba)}\leq r{\vphi(a^*a)}$ for all $a\in A$.
            \item[(S3)] $\sup\{\vphi(e)\where{e\in\ids{A}}\}=1$.
        \end{enumerate}
        Conversely, for any $\psi\in A^\vee$ satisfying conditions (S1)--(S3),
        there exist $V,v$ as above with $\psi=\vphi_{V,v}$ and $V=\quo{Av}$.
        The pair $(V,v)$ is unique up to unitary isomorphism preserving the  vector $v$.
    \end{thm}

    \begin{proof}
        Unfolding the definition of $\vphi$, we see that
        (S1) merely means $\Trings{av,av}\in\R_{\geq 0}$ for all $a\in A$,
        (S2) means that  $\norm{bav}\leq r^{1/2}\norm{av}$ for all $a\in A$, and (S3)
        means that $\sup_e\norm{ev}^2=1$. The conditions now follow from the fact that $\Trings{~,~}$
        is an inner product, $b|_V$ is continuous (take $r=\norm{b|_V}^2$) and Lemma~\ref{LM:approximation-lemma},
        respectively.

        Assume $\psi\in A^\vee$ satisfies conditions (S1)--(S3).
        Let $L=\{a\in A\suchthat \psi(a^*a)=0\}$. Condition (S2) implies that $L$ is a left
        ideal. Define $\Trings{~,~}:A/L\times A/L\to \C$ by $\Trings{x+L,y+L}=\psi(y^*x)$.
        Condition (S1) and the definition of $L$ imply that $\Trings{a+L,a+L}>0$ for all $a\in A - L$.
        Thus, $\Trings{~,~}$ is an inner product on $A/L$ which clearly
        satisfies $\Trings{ax,y}=\Trings{x,a^*y}$. Denote the induced norm
        on $A/L$ by $\norm{\,\cdot\,}$. Condition
        (S2) implies that every $b\in A$ admits an $r\in \R$ such that $\norm{ba}^2\leq r\norm{a}^2$,
        hence $b|_{A/L}$ is continuous. Let $V$ denote the completion of $A/L$ with respect to its norm.
        The left action of $A$ on $A/L$ extends to $V$, making $V$ into a unitary representation of $A$
        (note that $AV\supseteq A/L$, which is dense in $V$).

        We now construct $v\in \sphere{V}$ with $\psi=\vphi_{V,v}$.  For $e\in\ids{A}$,
        let $x_e=e+L$.
        Then $e\leq e'$ implies $\Trings{x_e,x_{e'}-x_e}=f(e'e-e)=0$, hence $\norm{x_{e'}}^2=\norm{x_e}^2+\norm{x_{e'}-x_e}^2$.
        We claim that $\{x_e\}_{e\in\ids{A}}$ is a Cauchy net in $V$.
        Indeed, by (S3), for all $\veps>0$, there is $e_0\in \ids{A}$ such that $\norm{x_{e_0}}^2=\psi(e_0)>1-\veps$,
        and for all $e\geq e_0$, we have $\norm{x_{e}-x_{e_0}}^2=\norm{x_{e}}^2-\norm{x_{e_0}}^2< 1-(1-\veps)=\veps$,
        as required. Similarly, one shows that $\lim\{\norm{x_e}^2\}_{e\in\ids{A}}=1$.
        Let $v=\lim \{x_e\}_{e\in\ids{A}}$. Then $\norm{v}=1$,
        and for all $e\in \ids{A}$, we have $ev=\lim\{ex_{e'}\}_{e'\geq e}=\lim\{x_e\}_{e'\geq e}=x_e$.
        Let $a\in A$. Then there exists $e\in \ids{A}$
        such that $eae=a$. Now, $\Trings{av,v}=\Trings{eaev,v}=\Trings{aev,ev}=\Trings{ax_e,x_e}=\psi(e^*(ae))=\psi(a)$.
        This means that $\vphi_{V,v}=\psi$, as required.

        To finish, observe that if $\psi=\vphi_{V',v'}$, then
        the morphism taking $a+L\in A/L$ to $av'\in V'$ is a well-defined
        $A$-homomorphism preserving the inner product and taking $x_e$ to $ev'$.
        Thus, it extends to a unitary isomorphism $V\to \quo{v'A}$ taking $v$ to $v'$
        (by Lemma~\ref{LM:approximation-lemma}).
    \end{proof}

    \begin{cor}\label{CR:coef-sharing}
        Let $V,V'\in\Rep[u]{A}$, and let $v\in \sphere{V}$, $v'\in \sphere{V'}$.
        If $\vphi_{V,v}=\vphi_{V',v'}$,
        then $\quo{Av}\cong \quo{Av'}$ as unitary representations.
    \end{cor}


    \begin{cor}\label{CR:sub-state}
        Let $V,V'\in\Rep[u]{A}$ and let $v\in\sphere{V}$, $v'\in\sphere{V'}$
        be vectors such that $V=\quo{Av}$ and $V'=\quo{Av'}$.
        Assume that there is $0< t$ such that $t\vphi_{V,v}(a^*a)\leq \vphi_{V',v'}(a^*a)$
        for all $a\in A$. Then there is a  continuous $A$-module homomorphism
        $f:V'\to V$ taking $v'$ to $v$.
    \end{cor}

    \begin{proof}
        Let $L=\{a\in A\suchthat \vphi_{V,v}(a^*a)=0\}$ and $L'=\{a\in A\suchthat\vphi_{V',v'}(a^*a)=0\}$.
        Condition (S1) and the assumption  $t\vphi_{V,v}(a^*a)\leq \vphi_{V',v'}(a^*a)$ imply that $L\supseteq L'$.
        By the proof of Theorem~\ref{TH:states}, the map $(x+L,y+L)\mapsto \vphi_{V,v}(y^*x)$
        is an inner product on $A/L$ and we have a unitary  isomorphism $\Phi:\quo{A/L}\to V$ taking $\lim\{e+L\}_{e\in\ids{A}}$
        to $v$. Similarly, we have a unitary isomorphism $\Phi':\quo{A/L'}\to V'$. Consider
        the $A$-homomorphism $\Psi:A/L'\to A/L$ given by $\Psi(a+L')=a+L$. Since   $t\vphi_{V,v}(a^*a)\leq \vphi_{V',v'}(a^*a)$,
        we have $\norm{a+L}^2\leq t^{-1}\norm{a+L'}^2$ for all $a\in A$. This means $\Psi$ is continuous
        and hence it extends uniquely to an $A$-homomorphism $\quo{A/L'}\to \quo{A/L}$,
        which clearly takes $\lim\{e+L'\}_{e\in\ids{A}}$ to $\lim\{e+L\}_{e\in\ids{A}}$.
        Now let $f=\Phi\Psi \Phi'^{-1}:V'\to V$. Then $f$ is continuous and takes $v$ to $v'$, as required.
    \end{proof}

    A linear functional $\psi\in A^\vee$ satisfying   conditions (S1)--(S3)
    of Theorem~\ref{TH:states} is called a \emph{state} of $A$.
    The function $\vphi_{V,v}$ is called the \emph{state associated with $V$ and $v$}.
    Denote by $E(A)$ the set of states of $A$. It is easy to see that $E(A)$ is convex.
    (Indeed, if $0\leq t\leq 1$, then $t\vphi_{V,v}+(1-t)\vphi_{V',v'}=\vphi_{V\oplus V',t^{1/2}v\oplus (1-t)^{1/2}v'}$.)

    \begin{remark}
        (i) Concerning condition (S1), any $\vphi\in A^\vee$ is completely
        determined by its values on the set $\{a^*a\where a\in A\}$. Indeed, for
        all $a\in A$, there is $e\in\ids{A}$ with $eae=a$ and then
        $a+a^*=(e+a)^*(e+a)-a^*a-e^*e$. Since for all $a\in A$,
        we have $a=\frac{1}{2}(a+a^*)-\frac{1}{2}i(ia+(ia)^*)$, it follows that  $\{a^*a\where a\in A\}$
        spans $A$ as a $\C$-vector space, hence our claim.

        (ii) Give $A^\vee$ the topology of pointwise convergence. Then $E(A)$ is not necessarily
        closed; the problem lies in condition (S2). It can be shown that this is indeed
        the case when  $A=\C[X,X^*]$  as in \ref{subsec:motivating-example}. In contrast, for $C^*$-algebras,
        the  set $E(A)$ is closed \cite[\S3.4]{Dixmier}.
    \end{remark}

    A state $\vphi$ is called \emph{pure} if it is an \emph{extremal point}
    of $E(A)$. That is, for all $\psi,\psi'\in E(A)$ and $0<t<1$ satisfying $\vphi=t\psi+(1-t)\psi'$,
    we have $\psi=\psi'=\vphi$.

    \begin{prp}\label{PR:pure-states}
        Let $V\in\Rep[u]{A}$ and let $v\in\sphere{V}$ be such that $V=\quo{vA}$.
        Then $\vphi_{V,v}$ is pure if and only if $V$ is irreducible.
    \end{prp}

    \begin{proof}
    	This follows from \cite[Th.~9.6.4]{Palmer01}. We give here a full proof for the sake of completeness.
    	
        Assume $V$ is reducible. Then there are nonzero closed $A$-submodules
        $U$ and $U'$ such that $V=U\perp U'$. Let $P$ and $P'$ be the orthogonal projections onto $U$
        and $U'$ respectively. If $Pv=0$, then $U=P(\quo{Av})=\quo{P(Av)}=\quo{A(Pv)}=0$, so
        $Pv\neq 0$, and likewise $P'v\neq 0$. Let $t=\norm{Pv}^2$ and $u=t^{-1/2}Pv$, and define $t'$ and $u'$
        similarly. Then $t,t'\in\R_{>0}$ and $t+t'=\norm{v}^2=1$. Now, for all $a\in A$, we have
        $\Trings{av,v}=\Trings{Pav,Pv}+\Trings{P'av,P'v}=\Trings{aPv,Pv}+\Trings{aP'v,P'v}=t\Trings{au,u}+t'\Trings{au',u'}$,
        hence $\vphi_{V,v}=t\vphi_{V,u}+(1-t)\vphi_{V,u'}$.
        We claim that $\vphi_{V,v}\neq \vphi_{V,u}$, and therefore $\vphi_{V,v}$ is not pure.
        Indeed, since $\quo{Av}=V$, for all $\veps>0$, there is $a\in A$ such that $\norm{av-u'}<\veps$.
        Now, $\norm{au}=\norm{P(av)-Pu'}=\norm{P(av-u')}\leq \norm{av-u'}<\veps$ and hence,
        $\abs{\Trings{au,u}}\leq \norm{au}\cdot\norm{u}<\veps$. On the other hand,
        \begin{align*}
        \abs{\Trings{av,v}}&\geq (1-t)\abs{\Trings{au',u'}}-t\abs{\Trings{au,u}}\geq(1-t)\abs{\Trings{av,u'}}-t\veps\\
        &\geq (1-t)(\abs{\Trings{u',u'}}-\abs{\Trings{av-u',u'}})-t\veps\\
        &\geq(1-t)-(1-t)\norm{av-u'}\cdot\norm{u'}-t\veps\\
        &\geq(1-t)-(1-t)\veps-t\veps=(1-t)-\veps\ .
        \end{align*}
        Taking $\veps<\frac{1-t}{2}$, we get $\Trings{av,v}\neq\Trings{au,u}$.

        Assume $V$ is irreducible. Let $\psi,\psi'\in E(A)$ and $0<t<1$ be such that $\vphi=t\psi+(1-t)\psi'$,
        and
        write $\psi=\vphi_{U,u}$ with $\quo{Au}=U$. Condition
        (S1) implies that $t\psi(a^*a)\leq \vphi(a^*a)$ for all $a\in A$.
        Thus, by Corollary~\ref{CR:sub-state}, there is a continuous homomorphism
        $f:V\to U$ taking $v$ to $u$.
        By Schur's Lemma, $f^*f:V\to V$ equals $\lambda\id_V$ for some $\lambda\in\C$, and it is easy
        to see that $\lambda\in\R_{>0}$ (notice that $f^*f\neq 0$ since $\Trings{f^*fv,v}=\norm{u}^2=1$).
        Now, for all $a\in A$, $\psi(a)=\Trings{au,u}=\Trings{afv,fv}=\Trings{f^*fav,v}=\lambda\Trings{av,v}=\lambda\vphi_{V,v}(a)$.
        Condition (S3)  implies $\lambda=1$, so $\psi=\vphi_{V,v}$. Likewise, $\psi'=\vphi_{V,v}$.
    \end{proof}

\subsection{Weak Containment}
\label{subsec:weak-containment}

    Let $A$ be an idempotented   $*$-algebra,
    and let $A^\vee=\Hom_{\C}(A,\C)$. We give $A^\vee$ the topology
    of point-wise convergence (i.e.\ the topology induced from the product topology on $\C^A$).
    Recall from Theorem~\ref{TH:states} that for $V\in \Rep[u]{A}$ and $v\in V$, we define $\vphi_{V,v}\in A^\vee$
    by $\vphi_{V,v}(a)=\Trings{av,v}$.
    We  write $F\fsubseteq A$ to denote that $F$ is a finite subset of $A$. For every
    such $F$, it is convenient to consider the seminorm  $\norm{\,\cdot\,}_F:A^\vee\to \R_{\geq 0}$ given by
    \[
    \norm{\vphi}_F=\max_{a\in F}\abs{\vphi(a)}\ .
    \]
    Notice that a net $\{\vphi_\alpha\}_{\alpha\in I}$ in $A^\vee$ converges
    to $\vphi$ if and only if $\lim_\alpha\norm{\vphi_\alpha-\vphi}_F=0$ for all $F\fsubseteq A$.

    \begin{lem}\label{LM:equiv-conds-of-weak-containment}
        Let $V\in\Irr[u]{A}$ and  $V'\in\Rep[u]{A}$. The following conditions are equivalent:
        \begin{enumerate}
            \item[(a)] For all $v\in\sphere{V}$, $\veps>0$ and $F\fsubseteq A$,
            there exists $v'\in\sphere{V'}$ such that $\norm{\vphi_{V,v}-\vphi_{V',v'}}_F<\veps$.
            \item[(b)] There exists $v\in\sphere{V}$ such that for all $\veps>0$ and $F\fsubseteq A$,
            there is $v'\in{V'}$ such that $\norm{\vphi_{V,v}-\vphi_{V',v'}}_F<\veps$.
        \end{enumerate}
        In part (a), we can always take $v'\in \sphere{AV'}$. Furthermore, if
        $ev=v$ for some $e\in\ids{A}$, then $v'$ can be chosen to be in $\sphere{eV'}$.
    \end{lem}

    \begin{proof}
        (a)$\derives$(b) is clear, so we only  prove (b)$\derives$(a).

        Let $v_1\in\sphere{V}$, $\veps>0$ and $F_1\fsubseteq A$ be given.
        Since $V$ is irreducible, for all $\delta>0$,
        there is $b\in A$ such that
        \[\norm{bv-v_1}<\delta\qquad\text{and}\qquad \norm{bv}=1\ .\]
        Choose $e\in \ids{A}$ such that $ebe=b$ and let $F:=\{b^*ab\where a\in F_1\cup\{e\}\}$.
        By assumption, for all $\delta'>0$, there exists $v'\in V'$ such that
        $\norm{\vphi_{V,v}-\vphi_{V',v'}}_F<\delta'$.
        Since $\vphi_{V,bv}(a)=\vphi_{V,v}(b^*ab)$ and $\vphi_{V',bv'}(a)=\vphi_{V',v'}(b^*ab)$.
        This is equivalent to
        \[\norm{\vphi_{V,bv}-\vphi_{V',bv'}}_{F_1\cup\{e\}}<\delta'\ .\]
        In particular, $\abs{1-\norm{bv'}^2}=\abs{\vphi_{V,bv}(e)-\vphi_{V',bv'}(e)}<\delta'$
        (since $ebe=b$ and $\norm{bv}=1$).
        Thus, for $\delta'\leq 1$, we have $0<\norm{bv'}^2<2$.

        Define $v'_1:=\norm{bv'}^{-1}bv'$ and let $a\in F_1$. Then
        \begin{align*}
        \abs{\vphi_{V',bv'}(a)-\vphi_{V',v'_1}(a)}&=
        \Abs{\vphi_{V',bv'}(a)-\norm{bv'}^2\vphi_{V',bv'}(a)}\\
        &=\abs{1-\norm{bv'}^2}\cdot\abs{\vphi_{V',bv'}(a)}\leq
        \delta'\abs{\Trings{abv',bv'}}\leq\delta'\norm{a|_{V'}}\norm{bv'}^2\\
        &\leq2\delta'\norm{a|_{V'}}\ .
        \end{align*}
        In addition,
        \begin{align}\label{EQ:approximation-of-mat-coef}
        \abs{\vphi_{V,v_1}(a)-\vphi_{V,bv}(a)}&=\abs{\Trings{av_1,v_1}-\Trings{abv,bv}}\\
        &\leq
        \abs{\Trings{av_1,v_1-bv}}+\abs{\Trings{av_1-abv,bv}}\nonumber\\
        &\leq \norm{a|_{V}}\cdot\norm{v_1}\cdot\norm{v_1-bv}+
        \norm{a|_{V}}\cdot\norm{v_1-bv}\cdot\norm{bv}\nonumber\\
        &\leq 2\delta\norm{a|_{V}}\ .\nonumber
        \end{align}
        Thus,
        \begin{align*}
        \abs{\vphi_{V,v_1}(a)-\vphi_{V',v'_1}(a)}\leq\,\, &\abs{\vphi_{V,v_1}(a)-\vphi_{V,bv}(a)}+
        \abs{\vphi_{V,bv}(a)-\vphi_{V',bv'}(a)}\\
        &+\abs{\vphi_{V',bv'}(a)-\vphi_{V',v'_1}(a)}\\
        <\,\, & 2\delta\norm{a|_{V}}+\delta'+ 2\delta'\norm{a|_{V'}}\ .
        \end{align*}
        Since $a\in F_1$ and $F_1$ is finite, we can choose $\delta$ and $\delta'$ small enough in advance
        to have $\abs{\vphi_{V,v_1}(a)-\vphi_{V',v'_1}(a)}<\veps$ for all $a\in F_1$. This proves (a).

        To finish, notice that the vector $v'_1$ lies $AV$. In addition, if $e_1v_1=v_1$ for some $e_1\in\ids{A}$,
        then we can
		apply the previous argument with $v=v_1$ and $b=e=e_1$.
        We then get $e_1v'_1=\norm{bv'}^{-1}e_1bv'=\norm{bv'}^{-1}bv'=v'_1$, so $v'_1\in\sphere{e_1V'}$.
    \end{proof}

    let $V\in\Irr[u]{A}$ and let $V'\in\Rep[u]{A}$.
    We say that $V$ is \emph{weakly contained} in $V'$ and write
    \[
    V\wc V'
    \]
    if the equivalent conditions of Lemma~\ref{LM:equiv-conds-of-weak-containment} are satisfied.
    For example, if $V\leq V'$, then $V\wc V'$. The converse is false in general.

    \medskip

    The next proposition relates weak containment with common spectrum of linear operators
    (see \ref{subsec:motivating-example} for the definition).

    \begin{lem}\label{LM:norm-approximation}
        Let $V\in\Irr[u]{A}$, $V'\in\Rep[u]{A}$ and $e\in \ids{A}$.
        Assume $V\wc V'$ and let $F\fsubseteq A$.
        Then for all $v\in \sphere{V}$ (resp.\ $v\in\sphere{eV}$) and $\veps>0$, there is $v'\in \sphere{V'}$
        (resp.\ $v'\in\sphere{eV'}$)
        such that $\Abs{\norm{av}-\norm{av'}}<\veps$ for all $a\in F$.
        In particular,  $\norm{a|_{V}}\leq \norm{a|_{V'}}$ for all $a\in A$.
    \end{lem}

    \begin{proof}
    	Observe that $\vphi_{V,v}(a^*a)=\norm{av}^2$.
    	The first assertion now follows from Lemma~\ref{LM:equiv-conds-of-weak-containment}
    	and the continuity of  the real function $x\mapsto x^{1/2}$.
        The last assertion follows from the first assertion
        because $\norm{a|_{V}}=\sup\{\norm{av}\where v\in\sphere{V}\}$
        and likewise for $a|_{V'}$.
    \end{proof}

    \begin{prp}\label{PR:specturm-vs-weak-containment}
        Let $V\in\Irr[u]{A}$, $V'\in\Rep[u]{A}$, and let $a_1,\dots,a_n\in A$.
        \begin{enumerate}
            \item[(i)] If $V\wc V'$, then $\Spec(a_1|_V,\dots,a_n|_V)\subseteq \Spec(a_1|_{V'},\dots,a_n|_{V'})$.
            \item[(ii)] Suppose $A$ is generated by $\{a_1,\dots,a_n\}$ (as
            a unital $\C$-algebra if $A$ has a unity) and $\dim V=1$.
            Write $a_i|_V=\lambda_i\id_V$ for $\lambda_i\in\C$.
            Then
            \[(\lambda_1,\dots,\lambda_n)\in\Spec(a_1|_{V'},\dots,a_n|_{V'})\quad\iff\quad V\wc V'\ .\]
        \end{enumerate}
    \end{prp}

    \begin{proof}
        (i)
        Let $(\lambda_1,\dots,\lambda_i)\in\Spec(a_1|_V,\dots,a_n|_V)$ and let $\veps>0$.
        Then there is $v\in\sphere{V}$ such that $\norm{(a_i-\lambda_i)v}<\frac{\veps}{2}$ for all $i$.
        Since $AV$ is dense in $V$, we may choose $v$ to be in $AV$. Choose  $e\in\ids{A}$
        with $ev=v$. Then $(a_i-\lambda_ie)v=(a_i-\lambda_i)v$.
        By Lemma~\ref{LM:norm-approximation}, there is $v'\in\sphere{eV'}$
        such that $\norm{(a_i-\lambda_i)v'}
        =\norm{(a_i-\lambda_ie)v'}<\norm{(a_i-\lambda_ie)v}+\frac{\veps}{2}<\veps$ for all $i$.
        This holds for all $\veps$, so $(\lambda_1,\dots,\lambda_n)\in\Spec(a_1|_{V'},\dots,a_n|_{V'})$.

        (ii) Since $A$ is idempotented, there is an idempotent
        $e\in A$ such that $ea_ie=a_i$ for all $i$. This idempotent is a necessarily a unity of $A$.
        Using this fact, we may
        replace $a_i$ with $a_i-\lambda_i$, and hence assume $\lambda_i=0$ for all $i$.

        The direction ($\Longleftarrow$) follows from (i), so we turn to prove ($\Longrightarrow$).
        Let $M$ be the free non-commutative monoid on $n$ letters
        $x_1,\dots,x_n$. Let $v\in \sphere{V}$, $\veps>0$ and $b_1,\dots,b_t\in A$, where
        $b_i=\sum_{w\in M}\alpha_{iw}w(a_1,\dots,a_n)$ for some $\{\alpha_{iw}\}_{i,w}\in\C$, all but finitely
        many are zero.
        We need to show that there exists $v'\in \sphere{V'}$ such that $\abs{\Trings{b_iv,v}-\Trings{b_iv',v'}}<\veps$
        for all $i$. Indeed, since $(0,\dots,0)\in\Spec(a_1|_{V'},\dots,a_n|_{V'})$,
        for all $\delta>0$, there is $v'\in \sphere{V'}$ such that
        $\norm{a_iv'}<r_i\delta$ where $r_i=\norm{a_i|_{V'}}$. Notice  that $b_iv=\alpha_{i1}v$
        since $a_1v=\dots=a_nv=0$.
        We now have
        \begin{eqnarray*}
        \abs{\Trings{b_iv',v'}-\Trings{b_iv,v}}
        &=&\Abs{\Trings{\alpha_{i1}v'+\sum_{w\neq 1}\alpha_{iw}w(a_1,\dots,a_n)v',v'}-\Trings{\alpha_{i1}v,v}} \\
        &=&
        \Abs{\Trings{\sum_{w\neq 1}\alpha_{iw}w(a_1,\dots,a_n)v',v'}} \\
        &\leq& \Norm{\sum_{w\neq 1}\alpha_{iw}w(a_1,\dots,a_n)v'}\leq
        \delta\sum_{w\neq 1}\abs{\alpha_{iw}}w(r_1,\dots,r_n)\ .
        \end{eqnarray*}
        Choosing $\delta$ small enough in advance, we get  $\abs{\Trings{b_iv,v}-\Trings{b_iv',v'}}<\veps$
        for all $i$.
    \end{proof}

    \begin{example}\label{EX:spectrum-I}
        Applying Proposition~\ref{PR:specturm-vs-weak-containment}
        in the setting of \ref{subsec:motivating-example} with $a_1=X$ and $a_2=X^*$
        implies that $(\lambda,\cconj{\lambda})\in \Spec(T,T^*)$
        $\iff$ $V_\lambda\wc V$. Since $T$ is normal, the former condition
        is equivalent to $\lambda\in\Spec(T)$. In particular, we get
        a one-to-one correspondence
        \[
        \Spec(T)\quad\longleftrightarrow\quad \left\{\begin{array}{c}\text{\footnotesize $U\in\Irr[u]{A}$ with $U\wc V$,}\\ \text{
        \footnotesize  up to isomorphism}\end{array}\right\}
        \]
        We have therefore shown that the spectrum of $T$
        can be recovered from the action of $A$ on $V$.
    \end{example}

    \begin{example}\label{EX:spectrum-II}
        Generalizing the previous example, let
        $A$ be a commutative  unital $*$-algebra generated
        by $\{a_1,\dots,a_n\}$.
        Since $A$ is commutative, every $V\in\Irr[u]{A}$ is $1$-dimensional
        (Corollary~\ref{CR:commutative-algs}) and hence operators on $V$ can be viewed as elements of $\C$.
        We associate with every such $V$ a vector $\lambda_V=(a_1|_V,\dots,a_n|_V)\in\C^n$.
        By Proposition~\ref{PR:specturm-vs-weak-containment}, for every
        $V'\in\Rep[u]{A}$, we have a one-to-one correspondence
        \[
        \Spec(a_1|_{V'},\dots,a_n|_{V'})\quad\longleftrightarrow\quad
        \left\{\begin{array}{c}\text{\footnotesize $V\in\Irr[u]{A}$ with $V\wc V'$,}\\ \text{
        \footnotesize  up to isomorphism}\end{array}\right\}
        \]
        induced by $\lambda_V\longleftrightarrow V$.
    \end{example}

\subsection{The Unitary Dual}
\label{subsec:unitary-dual}

    With Examples~\ref{EX:spectrum-I} and~\ref{EX:spectrum-II} in mind,
    we define a notion of spectrum for unitary representations of $A$.

\medskip

    For $V\in \Irr[u]{A}$, let $[V]$ denote the unitary isomorphism class of $V$, and let
    \[
    \udual{A}=\{[V]\where V\in\Irr[u]{A}\}\ .
    \]
    The set $\udual{A}$ is called the \emph{unitary dual} of $A$.
    It will be the space where points of the {spectrum}
    will live.
    For $V'\in\Rep[u]{A}$, we define the   \emph{$A$-spectrum of $V'$} to be\footnote{
        In the theory of $C^*$-algebras, the unitary dual is sometimes called
        the \emph{spectrum} of $A$, in which case the spectrum  defined above is called the \emph{support} of the representation.
        Similar terminology is used in algebraic geometry.
        However, since our goal is to demonstrate how the support generalizes
        the operator spectrum, we preferred to use the term ``spectrum''.
    }
    \[
    \Spec_A(V')=\{[V]\in \udual{A}\where V\wc V'\}\ .
    \]

    \begin{example}
        In the setting of \ref{subsec:motivating-example}, we
        have an isomorphism $\udual{A}\cong \C$ via $[V_\lambda]\leftrightarrow \lambda$,
        and under this isomorphism, $\Spec_A(V)$ coincides with $\Spec(T)$
        (Example~\ref{EX:spectrum-I}).
    \end{example}

    We now introduce a topology on $\udual{A}$: Let $V\in\Irr[u]{A}$, $v\in \sphere{V}$, $\veps>0$
    and  $F\fsubseteq A$.
    We define
    \[
    N_{V,v,\veps,F}\subseteq \udual{A}
    \]
    to be the set of all isomorphism classes $[U]\in\udual{A}$ for which
    there is $u\in {U}$ such that
    $
    \norm{\vphi_{V,v}-\vphi_{U,u}}_F<\veps
    $.
    Note that $u$ is not required to be a unit vector.\footnote{
        There is some room for variation here. Requiring $u$ to be  a unit vector
        affects parts (i) and (ii) of Proposition~\ref{PR:neighborhoods}(ii), but this
        makes little difference for many $*$-algebras
        by Remark~\ref{RM:bounded-dual}.
    }
    The possible sets $\{N_{V,v,\veps,F}\}$ form a subbasis for a topology on $\udual{A}$.

    \begin{remark}
        The common way to define a topology on the unitary dual of a $*$-algebra is
        by pulling back the hull-kernel topology on primitive $*$-ideals
        of $A$ along the map $[V]\mapsto \{a\in A\suchthat a|_V=0\}$.
        For $C^*$-algebras, this gives the topology we have just defined (see \cite[\S3.1]{Dixmier}
        and Proposition~\ref{PR:neighborhoods} below).
        However, these topologies are different in general, e.g.\ consider the case $A=\C[X,X^*]$ as in \ref{subsec:motivating-example}.
        Furthermore,
        Propositions~\ref{PR:unitary-dual-topology} and~\ref{PR:neighborhoods} below would fail if we were to give
        $\what{A}$ the pullback of the hull-kernel topology.
    \end{remark}

    It is convenient to call a subset $S\subseteq\udual{A}$ \emph{bounded} if
    $\sup_{[V]\in S}\norm{a|_V}<\infty$ for all $a\in A$, or equivalently,
    if $\hat{\bigoplus}_{[V]\in S} V\in\Rep[u]{A}$ (see \ref{subsec:unitary-reps}).

    \begin{prp}\label{PR:unitary-dual-topology}
        Assume $A$ is unital, commutative, and $a_1,\dots,a_t\in A$
        are elements such that $\{a_1,\dots,a_t,a_1^*,\dots,a_t^*\}$
        generate $A$ as a unital algebra. For $V\in \Irr[u]{A}$,
        denote denote by $\lambda_V\in\C^t$ the unique common eigenvalue
        of $\{a_1,\dots,a_t\}$ on $V$.
        Then the map
        \[
        [V]\mapsto \lambda_V: \udual{A}\to\C^t
        \]
        is a topological embedding. In addition, for all $V'\in\Rep[u]{A}$, we have
        \[
        \Spec(a_1|_{V'},\dots,a_t|_{V'})=\{\lambda_V\where [V]\in\Spec_A(V')\}\ .
        \]
        Finally, a subset $S\subseteq\udual{A}$ is bounded if and only if its image in $\C^t$
        is bounded.
    \end{prp}

    \begin{proof}
        For all $V'\in \Rep[u]{A}$, we have
        $(\lambda_1,\dots,\lambda_t)\in\Spec(a_1|_{V'},\dots,a_t|_{V'})$
        if and only if
        $(\lambda_1,\dots,\lambda_t,\cconj{\lambda_1},\dots,\cconj{\lambda_t})\in
        \Spec(a_1|_{V'},\dots,a_t|_{V'},a_1^*|_{V'},\dots,a_t^*|_{V'})$.
        Indeed, this easily follows from the fact that $\norm{(a_i-\lambda_i)v}=\norm{(a_i^*-\cconj{\lambda_i})v}$
        for all $v\in V'$ when $a^*_ia_i=a_ia_i^*$
        (see~\ref{subsec:motivating-example}).
        Since the map $(\lambda_1,\dots,\lambda_t)\mapsto (\lambda_1,\dots,\lambda_t,\cconj{\lambda_1},\dots,\cconj{\lambda_t}):\C^t\to\C^{2t}$
        is a topological embedding,
        we may replace $\{a_1,\dots,a_t\}$ with $\{a_1,\dots,a_t,a_1^*,\dots,a_t^*\}$, and  assume
        that $\{a_1,\dots,a_t\}$ generates $A$ as a unital algebra.

        Denote by $\theta$ the map $[V]\mapsto\lambda_V$. We have seen in Example~\ref{EX:spectrum-II} that $\theta$
        is injective and that $\Spec(a_1|_{V'},\dots,a_t|_{V'})=\{\lambda_V\where [V]\in\Spec_A(V')\}$ for all
        $V'\in\Rep[u]{A}$. It is therefore left to show that $\theta$ is a topological embedding.

        Let us check what is $N_{V,v,\veps,F}$:
        Write $F=\{b_1,\dots,b_m\}$ and choose polynomials
        $f_1,\dots,f_m\in\C[x_1,\dots,x_t]$ such that $b_i=f_i(a_1,\dots,a_t)$.
        Since $V$ is $1$-dimensional, for all $v\in\sphere{V}$,
        \[
        \Trings{b_iv,v}={f_i(\lambda_V)}\norm{v}^2=f_i(\lambda_V)\ .
        \]
        Likewise, if $[U]\in\udual{A}$, then
        $\Trings{au,u}={f_i(\lambda_U)}\norm{u}^2$ for all $u\in{U}$. Writing $r=\norm{u}^2$,
        it follows that
        \[
        N_{V,v,\veps,F}=\bigcup_{r\in\R_{\geq0}}\left\{[U]\in\udual{A}\suchthat \Abs{f_i(\lambda_V)-rf_i(\lambda_U)}<\veps\,\,\text{for all $i$}\right\}\ .
        \]
        Therefore, $N_{V,v,\veps,F}=\theta^{-1}(M_{\veps,f_1,\dots,f_m})$ where
        \[
        M_{\veps,f_1,\dots,f_m}:=\bigcup_{r\in \R_{\geq 0}}\left\{\mu\in\C^n\suchthat \Abs{f_i(\lambda_V)-rf_i(\mu)}<\veps\,\,\text{for all $i$}\right\}\ .
        \]
        The sets $M_{\veps,f_1,\dots,f_m}$ are  open in $\C^n$, and moreover, they form an open basis of
        neighborhoods of $\lambda_V$. (Indeed, take $m=n+1$ and let $f_i=x_i$ ($i\leq n$) and $f_{n+1}=1$; the details
        are left to the reader.) Thus,  $\theta:\udual{A}\to\C^n$ is a topological embedding.

        The final assertion is easy and is left to the reader.
    \end{proof}

    \begin{remark}\label{RM:unitary-dual-topology}
        In the setting of Proposition~\ref{PR:unitary-dual-topology}, the image of $\udual{A}$
        in $\C^t$ can be described as follows: Let $\phi:\C[x_1,\dots,x_t,y_1,\dots,y_t]\to A$ be the algebra homomorphism
        sending $x_i$ to $a_i$ and $y_i$ to $a_i^*$, and let $X$ be the affine scheme $\Spec A$.
        Then $\phi$ induces a closed embedding (of schemes) $\phi^*:X\to \mathbb{A}_{\C}^{2t}$,
        which in turn induces an injection $X(\C)\to \C^{2t}$. We thus view $X(\C)$
        as a subset of $\C^{2t}$.
        The image of $\udual{A}$ in $\C^t$ is $\{(\lambda_1,\dots,\lambda_t)\suchthat
        (\lambda_1,\dots,\lambda_t,\cconj{\lambda_1},\dots,\cconj{\lambda_t})\in X(\C)\}$.
        The technical and straightforward proof is left to the reader.
    \end{remark}

    The following example assumes basic knowledge of quivers and their representations.
    See \cite{AssemIbSim06Reps} for all relevant details.

    \begin{example}\label{EX:double-arrow-quiver}
        Let $A$ be the path algebra of the quiver
        \[
        Q:\qquad
        \xymatrix{
            0 \ar@/^/[r]|{a_{10}} & 1 \ar@/^/[l]|{a_{01}}
        }
        \]
        That is,
        $A$ is the free $\C$-algebra generated by $e_0,e_1,a_{01},a_{10}$ subject
        to the relations $e_0^2=e_0$, $e_1^2=e_1$, $e_0a_{01}e_1=a_{01}$ and $e_1a_{10}e_0=a_{10}$.
        We define $*:A\to A$ to be the only involution satisfying $e_0^*=e_0$, $e_1^*=e_1$, and
        $a_{01}^*=a_{10}$. It is easy to see that $A$ is a unital $*$-algebra ($e_0+e_1$ is the unity).

        Recall that $A$-modules are in correspondence with representations of $Q$ via sending $V\in\lMod{A}$
        to the vector space diagram
        \[
        \xymatrix{
            e_0V \ar@/^/[r]^{a_{10}|_{V}} & e_1V \ar@/^/[l]^{a_{01}|_V}
            }
        \]
        Any $V\in\Irr[u]{A}$ can be shown to be isomorphic to one of $\{V_r\}_{r\in\R_{>0}}$, $U_0$, $U_1$
        given by the following diagrams, respectively:\footnote{
            Here is an  ad-hoc  sketch of proof: Since $e_0+e_1=1_A$, at least one of $e_0V$, $e_1V$
            is nonzero.
            If $e_0V\neq 0$, then $e_0V$ is an irreducible unitary representation of $e_0Ae_0$
            (Lemma~\ref{LM:irreducible-corner-module} below), and since $e_0Ae_0$ is commutative,
            $e_0V$ must be $1$-dimensional (Corollary~\ref{CR:commutative-algs}).
            Likewise, if $e_1V\neq 0$, then $\dim e_1V=1$. The cases $e_0V=0$ and $e_1V=0$
            lead to $U_1$ and $U_0$, respectively. In all other cases, we may assume
            $e_0V=e_1V=\C$ and view $a_{10}|_{V}$, $a_{01}|_V$ as elements of $\C$.
            If $z=a_{10}|_V$, then $a_{01}|_V=\cconj{z}$ because $a_{10}^*=a_{01}$.
            One can show that the isomorphism type of $V$ determined by $s:=z\quo{z}$. If $s=0$, then $V\cong U_0\oplus U_1$.
            Otherwise, $V\cong V_r$ for $r=\sqrt{s}$.
        }
        \[
        \xymatrix{
            \C \ar@/^/[r]^{\cdot r} & \C \ar@/^/[l]^{\cdot r}
            }
            \qquad
            \xymatrix{
            \C \ar@/^/[r]^{0} & 0 \ar@/^/[l]^{0}
            }
            \qquad
            \xymatrix{
            0 \ar@/^/[r]^{0} & \C \ar@/^/[l]^{0}
            }
        \]
        (The inner products on $V_r=\C\oplus\C$,
        $U_0=\C\oplus0$ and $U_1=0\oplus \C$  are
        all given by $\Trings{\alpha\oplus \alpha',\beta\oplus\beta'}=\alpha\cconj{\beta}+\alpha'\cconj{\beta'}$.)
        Using arguments similar to those in the proof of Proposition~\ref{PR:unitary-dual-topology},
        one can show that $\udual{A}$ homeomorphic to two copies of $\R_{\geq 0}$ glued along $\R_{>0}$.
        The homeomorphism sends $[V_r]$ to $r\in\R_{>0}$
        and $U_0$ and $U_1$ to the two points lying over $0$.
        The details are left to the reader.

        This shows that $\udual{A}$ is not
        $\mathrm{T}_1$ in general.
    \end{example}

    \begin{remark}
        The unitary dual of a unital $C^*$-algebra is always quasi-compact \cite[Pr.~3.1.8]{Dixmier}.
        However, as follows from Proposition~\ref{PR:unitary-dual-topology} or Example~\ref{EX:double-arrow-quiver},
        this is not necessarily the case for unital $*$-algebras.
    \end{remark}

    We finally record the following useful facts about the topology of $\udual{A}$.

    \begin{prp}\label{PR:neighborhoods}
        Let $V\in \Irr[u]{A}$, $v\in\sphere{V}$ and let $S\subseteq \udual{A}$. Then:
        \begin{enumerate}
            \item[(i)] Then the sets $\{N_{V,v,\veps,F}\}_{\veps,F}$ ($\veps>0$, $F\fsubseteq A$) are
            a \emph{basis} of open neighborhoods of $[V]$ in $\udual{A}$.
            \item[(ii)] $[V]\in\overline{S}$ if and only if for all $\veps>0$ and
            $F\fsubseteq A$, there exist $[U]\in S$
            and $u\in U$ such that $\norm{\vphi_{V,v}-\vphi_{U,u}}_F<\veps$.
            \item[(iii)] When $S$ is bounded, $[V]\in\overline{S}$ if and only if for all $\veps>0$ and
            $F\fsubseteq A$, there exist $[U]\in S$
            and $u\in \sphere{U}$ such that $\norm{\vphi_{V,v}-\vphi_{U,u}}_F<\veps$.
        \end{enumerate}
    \end{prp}

    \begin{proof}
        (i) Since the sets $\{N_{V,v,\veps,F}\}_{\veps,F}$ already form a filter base, it is enough to prove that
        if $[V]\in N_{U,u,\delta,G}$, then there are $\veps>0$ and $F\fsubseteq A$
        such that $N_{V,v,\veps,F}\subseteq N_{U,u,\delta,G}$.

        Suppose $[V]\in N_{U,u,\delta,G}$. Then there is $v'\in V$
        such that $\norm{\vphi_{V,v'}-\vphi_{U,u}}_G<\delta$. If $v'=0$, replace it with a vector of sufficiently
        small magnitude.
        Fix some $\delta_0<\delta$ such that $\norm{\vphi_{V,v'}-\vphi_{U,u}}_G<\delta_0$,
        and let
        $\veps,\veps'>0$, to be chosen later.
        Since $\quo{Av}=V$, there is $b\in A$ such that $\norm{bv-v'}<\veps'$ and $\norm{bv}=\norm{v'}$
        (here we need $v'\neq 0$). Let $F=\{b^*ab\where a\in G\}$
        and let $[W]\in N_{V,v,\veps,F}$.
        Then there is $w\in W$ such that $\norm{\vphi_{V,v}-\vphi_{W,w}}_F<\veps$.
        As in the proof of Lemma~\ref{LM:equiv-conds-of-weak-containment},
        the latter implies $\norm{\vphi_{V,bv}-\vphi_{W,bw}}_G<\veps$,
        and by virtue of \eqref{EQ:approximation-of-mat-coef},
        $\abs{\vphi_{V,v'}(a)-\vphi_{V,bv}(a)}\leq 2\veps'\norm{a|_V}\cdot\norm{v'}$ for all $a\in G$.
        Thus, for all $a\in G$,
        \begin{align*}
            \abs{\vphi_{U,u}(a)-\vphi_{W,bw}(a)} \leq\,\,&
            \abs{\vphi_{U,u}(a)-\vphi_{V,v'}(a)}+\abs{\vphi_{V,v'}(a)-\vphi_{V,bv}(a)}\\
            &+\abs{\vphi_{V,bv}(a)-\vphi_{W,bw}(a)}\\
            < \,\,&\delta_0 + 2\veps'\norm{a|_V}\cdot\norm{v'}+\veps\ .
        \end{align*}
        Choose $\veps$ and $\veps'$ such that $\delta_0 + 2\veps'\norm{a|_V}\cdot\norm{v'}+\veps<\delta$
        for all $a\in G$. Then we have shown that $N_{V,v,\veps,F}\subseteq N_{U,u,\delta,G}$.

        (ii) The condition means that $N_{V,v,\veps,F}\cap S\neq \emptyset$ for all $\veps>0$ and $F\fsubseteq A$,
        and this is equivalent to $[V]\in \overline{S}$ by (i).

        (iii)
        Suppose $[V]\in\overline{S}$ and let $F_1\fsubseteq A$ and
        $\veps>0$.
        Since $S$ is bounded, there is $M\in\R$ such that $\norm{a|_U}<M$ for all $a\in F_1$, $[U]\in S$.
        Applying the proof of Lemma~\ref{LM:equiv-conds-of-weak-containment} with $M$ in place of $\norm{a|_{V'}}$ and $v_1=v$, we obtain
        $F\fsubseteq A$ and $\delta'>0$ such that for all $[U]\in S$
        and $u\in U$ with $\norm{\vphi_{U,u}-\vphi_{V,v}}_F<\delta'$, we have $\norm{\vphi_{V,v}-\vphi_{U,u'}}_{F_1}<\veps$
        for some $u'\in\sphere{U}$. The existence of $U$ and $u$
        with $\norm{\vphi_{U,u}-\vphi_{V,v}}_F<\delta$ follows from (ii), so this proves one
        direction. The other direction is immediate from (ii).
    \end{proof}

    \begin{remark}\label{RM:bounded-dual}
        Proposition~\ref{PR:neighborhoods}(iii) holds for any $S$ if $\udual{A}$ is locally bounded,
        i.e.\ if it admits a basis of bounded open sets. This can be shown to hold  when $A$ is
        a $C^*$-algebra, or when $A$ is the \emph{Hecke algebra}
        of an $\ell$-group (see~\ref{subsec:Hecke-algebra}).
    \end{remark}

\subsection{Properties of The Spectrum}
\label{subsec:properties-of-spectrum}

    We now prove several properties of the $A$-spectrum.

    \begin{prp}\label{PR:spectrum-is-closed}
        Let $V'\in\Rep[u]{A}$. Then $\Spec_A(V')$ is bounded and closed.
    \end{prp}

    \begin{proof}
        The set $\Spec_A(V')$ is bounded  by Lemma~\ref{LM:norm-approximation}.
        Assume $[V]\in\overline{\Spec_A(V')}$ and let $v\in\sphere{V}$.
        By Proposition~\ref{PR:neighborhoods}(iii),
        for all $\veps>0$ and $F\fsubseteq A$,
        there is $[U]\in\Spec_A(V')$ and $u\in \sphere{U}$
        satisfying $\norm{\vphi_{V,v}-\vphi_{U,u}}_F<\frac{\veps}{2}$.
        Since $U\wc V'$, there is $v'\in\sphere{V'}$ such
        that $\norm{\vphi_{U,u}-\vphi_{V',v'}}_F<\frac{\veps}{2}$, so
        $
        \norm{\vphi_{V,v}-\vphi_{V',v'}}_F\leq \norm{\vphi_{V,v}-\vphi_{U,u}}_F+\norm{\vphi_{U,u}-\vphi_{V',v'}}_F<\veps
        $. As this holds for all $F$ and $\veps$, it follows that $V\wc V'$.
    \end{proof}

    Recall that $A^\vee=\Hom_{\C}(A,\C)$ is endowed with the topology of point-wise convergence.
    Subsets of $A^\vee$ are given the induced topology.
    For $V'\in\Rep[u]{A}$, let $E(A;V')$ denote the \emph{closed} convex hull of $\{\vphi_{V',v'}\where v'\in\sphere{V'}\}$
    in $A^\vee$. In addition, let $P(A;V')$ be the set of extremal points of $E(A;V')$.
    Recall that $E(A)$ is the set of \emph{all} states of $A$. We also let $P(A)$ denote the set of \emph{pure} states of $A$
    (see \ref{subsec:states}).

    \begin{thm}\label{TH:convexity}
        Let $V'\in\Rep[u]{A}$. Then:
        \begin{enumerate}
            \item[(i)] $E(A;V')$ is compact in $A^\vee$.
            \item[(ii)] $E(A;V')$ is the closed convex hull of $P(A;V')$ and it is contained in $E(A)$.\footnote{
                The set $E(A)$ is convex, but since it is not closed
                in general, $\{\vphi_{V',v'}\}_{v'\in \sphere{V'}}\subseteq E(A)$ does not immediately imply $E(V';A)\subseteq E(A)$.
            }
            \item[(iii)] $P(A;V')=P(A)\cap E(A;V')=\{\vphi_{U,u}\where [U]\in\Spec_A(V'),\, u\in\sphere{U}\}$.
        \end{enumerate}
    \end{thm}

    \begin{proof}
        (i) $E(A;V')$ is closed in $A^\vee$ and hence in $\C^A$ (with the product topology), so it is enough to show that it is contained
        in a compact subset of $\C^A$. Observe that for all $v'\in \sphere{V'}$, we have $\abs{\vphi_{V',v'}(a)}=
        \abs{\Trings{av',v'}}\leq\norm{a|_{V'}}$.
        Thus, $E(A;V')\subseteq\{\vphi\in \C^A\suchthat \abs{\vphi(a)}\leq \norm{a|_{V'}}\,\text{for all}\,a\in A\}$.
        The right hand side is isomorphic to $\prod_{a\in A}\{z\in\C\suchthat \abs{z}\leq\norm{a|_{V'}}\}$, which is compact by Tychonoff's Theorem.

        (ii) The first assertion follows from (i) and the Krein--Milman Theorem (\cite[p.~138]{FunctionalAnalysis92Nikol},
        for instance).
        For the second assertion, observe that for all $a,b\in A$, $v'\in \sphere{V'}$,
        we have $\vphi_{V',v'}(a^*b^*ba)\leq\vphi_{V',v'}(a^*a)\norm{b|_{V'}}^2$.
        From this it is easy to see
        that  elements
        in the closed convex hull of $\{\vphi_{V',v'}\}_{v'\in\sphere{V'}}$
        satisfy conditions (S1)--(S3) of Theorem~\ref{TH:states}, and hence lie in $E(A)$.

        (iii)
        The definition of weak containment implies that $\vphi_{U,u}\in E(A;V')$ for all $[U]\in \Spec(A;V')$
        and $u\in\sphere{U}$. Together with Proposition~\ref{PR:pure-states},
        this implies
        $P(A;V')\supseteq P(A)\cap E(A;V')\supseteq\{\vphi_{U,u}\where [U]\in\Spec_A(V'),\, u\in\sphere{U}\}$, so
        it remains to prove that $P(A;V')\subseteq \{\vphi_{U,u}\where [U]\in\Spec_A(V'),\, u\in\sphere{U}\}$.

        Let $\vphi\in P(A;V')$. By (ii), we can write $\vphi=\vphi_{U,u}$
        with $U=\quo{Au}$ and $u\in\sphere{U}$.
        Since $E(A;V')$ is the closed convex hull of $\{\vphi_{V',v'}\where v'\in \sphere{V'}\}$
        and it is compact,  Milman's converse to the Krein--Milman Theorem (\cite[p.~139]{FunctionalAnalysis92Nikol}, for instance) implies
        that $\vphi_{U,u}\in\overline{\{\vphi_{V',v'}\where v'\in \sphere{V'}\}}$
        (because $\vphi_{U,u}$ is extremal in $E(A;V')$).
        This easily implies $\vphi_{U,au}\in\overline{\{\vphi_{V',av'}\where v'\in\sphere{V'}\}}$ for all $a\in A$,
        and
        since $U=\overline{Au}$, it follows that $E(A;U)\subseteq E(A;V')$.

        Suppose $U$ is reducible. Then $\vphi_{U,u}$ is not pure (Proposition~\ref{PR:pure-states}),
        and
        hence there are $U_1,U_2\in \Irr[u]{A}$, $u_1\in \sphere{U_1}$, $u_2\in\sphere{U_2}$
        and $0<t<1$ such that $\vphi_{U,u}=t\vphi_{U_1,u_1}+(1-t)\vphi_{U_2,u_2}$ and $\vphi_{U,u}\neq\vphi_{U_1,u_1}$.
        By Corollary~\ref{CR:sub-state}, $U_1,U_2\leq U$,
        and therefore $\vphi_{U_1,u_1},\vphi_{U_2,u_2}\in E(A;U)\subseteq E(A;V')$.
        This  contradicts
        our assumption that $\vphi_{U,u}$ is extremal in $E(A;V')$, so $U$ must be irreducible.
        Since $\vphi_{U,u}\in \overline{\{\vphi_{V',v'}\where v'\in \sphere{V'}\}}$, this means $U\wc V'$,
        as required.
\rem{
        Suppose $U$ is reducible. Then, as in the proof of Proposition~\ref{PR:pure-states},
        there are $u_1,u_2\in \sphere{U}$ and $0<t<1$ such that
        $\vphi_{U,u}=t\vphi_{U,u_1}+(1-t)\vphi_{U,u_2}$ and $\vphi_{U,u}\neq\vphi_{U,u_1}$.
        Since $\vphi_{U,u_1},\vphi_{U,u_2}\in E_A(U)\subseteq E_A(V')$, this contradicts
        our assumption that $\vphi_{U,u}$ is extremal in $E_A(V')$. Therefore, $U$ is irreducible,
        and $U\wc V'$ because $\vphi_{U,u}\in \overline{\{\vphi_{V',v'}\where v'\in \sphere{V'}\}}$.
}
    \end{proof}

    \begin{cor}\label{CR:non-empty-spectrum}
    	Let $0\neq V'\in\Rep[u]{A}$ and let $a\in A$. There
    	is $[V]\in\Spec_A(V')$ such that $\norm{a|_V}=\norm{a|_{V'}}$.
    	In particular, $\Spec_A(V')\neq\emptyset$.
    \end{cor}

    \begin{proof}
    	Observe that for all $v'\in \sphere{V'}$ we have $\vphi_{V',v'}(a^*a)=\norm{av'}^2$.
    	Since $\norm{a|_{V'}}=\sup\{\norm{av'}\where v'\in \sphere{V'}\}$,
    	we have
    	\begin{equation}\label{EQ:norm-bound}
    	\psi(a^*a)\leq\norm{a|_{V'}}^2 \qquad\forall \,\psi\in E(V';A)
    	\end{equation}
    	Let $S=\{\psi\in E(A;V')\suchthat \psi(a^*a)=\norm{a|_{V'}}^2\}$.
    	We first claim that $S\neq\emptyset$.
    	Since $V'\neq \emptyset$, there is a sequence $\{v'_n\}_{n\in\N}\subseteq\sphere{V'}$ such
    	that $\lim_{n\to\infty}\norm{av'_n}=\norm{a}$. Since $E(A;V')$ is compact,
    	$\{\vphi_{V',v'_n}\}$ has a subsequence converging to some
    	$\psi\in E(A;V')$, which clearly lies in $S$.
    	Now, $S$ is compact since it is closed in $E(A;V')$, so
    	the Krein--Milman Theorem implies that it has an extremal point $\psi$.
    	Using \eqref{EQ:norm-bound}, it is easy to see that $\psi$ is also extremal
    	in $E(A;V')$, so by Theorem~\ref{TH:convexity}(iii), there are $[V]\in\Spec_A(V')$
    	and $v\in \sphere{V}$ such that $\psi=\vphi_{V,v}$.
    	Since $\norm{av}^2=\psi(a^*a)=\norm{a|_{V'}}^2$,
    	we have $\norm{a|_V}\geq \norm{a|_{V'}}$. On the other hand, $\norm{a|_V}\leq \norm{a|_{V'}}$
    	by Lemma~\ref{LM:norm-approximation}.
\rem{
    	
    	Observe that for any $W\in\Rep[u]{A}$, we have $a|_W=0$ if and only
    	if $\vphi_{W,w}(a^*a)=\norm{aw}^2=0$ for all $w\in\sphere{W}$.
    	By assumption, there is $v'\in V'$ with $\vphi_{V',v'}(a^*a)\neq 0$.
    	Since $\vphi_{V',v'}\in E(A;V')$ and $E(A;V')$ is the closed convex hull of $P(A;V')$ (Theorem~\ref{TH:convexity}(ii)),
    	there must be $\psi\in P(A;V')$ with $\psi(a^*a)\neq 0$.
    	By Theorem~\ref{TH:convexity}(iii), there are $[V]\in\Spec_A(V')$ and $v\in\sphere{V}$ with $\psi=\vphi_{V,v}$,
    	so $a|_V\neq 0$.
}
    \end{proof}

    \begin{thm}\label{TH:direct-sum-spectrum}
        Let $\{V_i\}_{i\in I}$ be a family of unitary representations of $A$ such
        that $\hat{\bigoplus}_iV_i\in\Rep[u]{A}$. Then
        $
        \Spec_A({\hat{\bigoplus}_iV_i})=\overline{\bigcup_i\Spec_A(V_i)}
        $.
    \end{thm}

    \begin{proof}
        It is clear that $\Spec_A(V_i)\subseteq \Spec_A(\hat{\bigoplus}_iV_i)$ for all $i$. Since
        $\Spec_A(\hat{\bigoplus}_iV_i)$ is closed (Proposition~\ref{PR:spectrum-is-closed}),
        it is enough to show that $\Spec_A(\hat{\bigoplus}_iV_i)\subseteq \overline{\bigcup_i\Spec_A(V_i)}$.
        Assume $[V]\in \Spec_A(\hat{\bigoplus}_iV_i)$ and let $v\in \sphere{V}$.
        Then $\vphi_{V,v}\in P(A;\hat{\bigoplus}_iV_i)$ (Theorem~\ref{TH:convexity}(iii)).
        It is easy to see that $E(A;\hat{\bigoplus}_iV_i)$ is the closed convex hull
        of $\bigcup_i E(A;V_i)$, hence by Theorem~\ref{TH:convexity}(ii), it is the closed
        convex hull of $\bigcup_iP(A;V_i)$. Since $E(A;\hat{\bigoplus}_iV_i)$ is compact,
        Krein's Converse to the Krein--Milman Theorem \cite[p.~139]{FunctionalAnalysis92Nikol} implies that $P(A;V)\subseteq \overline{\bigcup_iP(A;V_i)}$,
        hence $\vphi_{V,v}\in \overline{\bigcup_iP(A;V_i)}$.
        By Proposition~\ref{PR:neighborhoods}(iii) and Theorem~\ref{TH:convexity}(iii), this means
        $[V]\in \overline{\bigcup_i\Spec_A(V_i)}$, as required.
    \end{proof}

\subsection{Subalgebras}
\label{subsec:subalgebras}

    Let $A$ be an  idempotented $*$-algebra.
    By an \emph{idempotented $*$-subalgebra} of $A$ we mean a  subalgebra $B\subseteq A$
    such that $B^*=B$ and $(B,*|_{B})$
    is an idempotented $*$-algebra.

    \begin{example}
        Let $a_1,\dots,a_t\in A$. Choose an idempotent
        $e\in \ids{A}$ such that $ea_ie=a_i$ for all $1\leq i\leq t$.
        Then the subalgebra $B$ generated by the elements $e,a_1,\dots,a_n,a_1^*,\dots,a_n^*$
        is an idempotented $*$-subalgebra of $A$.
    \end{example}

    Henceforth, $B$ is a fixed idempotented $*$-subalgebra of $A$.
    Observe that if $V'\in\Rep[u]{A}$, then $\overline{BV'}\in\Rep[u]{B}$.
    The purpose of this section is to show
    that $\Spec_A(V')$  determines $\Spec_B(\overline{BV'})$.
    Informally, this means that the $A$-spectrum   holds at least
    as much information as the $B$-spectrum.

\medskip

	We start by introducing notation.
	For any $V\in \Irr[u]{A}$, let
    $\calR_{A/B}(V):=\Spec_B(\overline{BV})\subseteq \udual{B}$.
    In general, $\calR_{A/B}(V)$ is not a singleton, so $\calR_{A/B}$ does not induce a function from $\udual{A}$ to $\udual{B}$.
    However,  $\calR_{A/B}$ induces a map $\calR_{A/B}:P(\udual{A})\to P(\udual{B})$ given by
    \[\calR_{A/B}(S)=\bigcup_{[V]\in S}\calR_{A/B}(V)\,.\]
	This map has properties resembling continuity and closeness in the following sense.
	
    \begin{prp}\label{PR:restriction-is-continuous}
        Suppose $S\subseteq \udual{A}$ is bounded (i.e.\
        $\hat{\bigoplus}_{[V]\in S} V\in\Rep[u]{A}$). Then
        $\calR_{A/B}(\overline{S})= \overline{\calR_{A/B}(S)}$.
    \end{prp}

    \begin{proof}
        Since $\hat{\bigoplus}_{[V]\in S} V\in\Rep[u]{A}$, we have $\overline{S}=\Spec_A(\hat{\bigoplus}_{[V]\in S} V)$
        and $\overline{\calR_{A/B}(S)}=\overline{\bigcup_{[V]\in S}\Spec_B(\overline{BV})}=\Spec_B(\hat{\bigoplus}_{[V]\in S} \overline{BV})$
        (Theorem~\ref{TH:direct-sum-spectrum}).
        Since $\overline{B(\hat{\bigoplus}_{[V]\in S} V)}=\hat{\bigoplus}_{[V]\in S} \overline{BV}$,
        this means $\calR_{A/B}(\overline{S})=\Spec_B(\hat{\bigoplus}_{[V]\in S} \overline{BV})=\overline{\calR_{A/B}(S)}$.
    \end{proof}

    \begin{example}
        The assumption $\hat{\bigoplus}_{[V]\in S} V\in\Rep[u]{A}$ in Proposition~\ref{PR:restriction-is-continuous} is essential.
        Let $M$ be the multiplicative submonoid of $\C[X]$ spanned by $\{X-a\where a\in\C\setminus \Q\}$,
        and let $A$ be the localization $\C[X]\cdot M^{-1}$ with the unique involution satisfying $X^*=X$.
        Elaborating the proof of Proposition~\ref{PR:unitary-dual-topology}, one sees that $\udual{A}\cong \Q$ as topological
        spaces, where the topology on $\Q$ is induced from $\R$. The isomorphism is given by sending $\lambda\in \Q$ to the
        class $[V_\lambda]\in\what{A}$ satisfying $X|_{V_\lambda}=\lambda\id_{V_\lambda}$.
        Let $B=\C[X]$, and identify $\udual{B}$ with $\R$ similarly. Then $\calR_{A/B}$ turns out to induce a function
        $\udual{A}\to \udual{B}$ (Corollary~\ref{CR:commutative-algs}; this holds whenever $A$ and $B$ are commutative and $A=AB$),
        which corresponds to the inclusion $\Q\to \R$. This map is clearly not closed, e.g.\ the set $[1,2]\cap \Q$
        is closed in $\Q$ but not in $\R$. Translating this to representations,
        we see that for $S=\{[V_\lambda]\suchthat \lambda\in [1,2]\cap \Q\}$, we have
        $\calR_{A/B}(\overline{S})\subsetneq \overline{\calR_{A/B}(S)}$.
        Indeed, $\{[V_\lambda]\where \lambda\in [1,2]\cap \Q\}$
        is not bounded since $A$ does not act continuously on $\bigoplus_{\lambda\in [1,2]\cap \Q}V_\lambda$ (consider the action
        of $(X-\sqrt{3})^{-1}\in A$, for instance).

        We believe that  the inclusion $\calR_B(\overline{S})\subseteq \overline{\calR_B(S)}$, which may be regarded
        as a form of continuity, should also fail
        when $S$ is not bounded.
    \end{example}

    The following theorem explains how $\Spec_A(V')$  determines $\Spec_B(\overline{BV'})$.

    \begin{thm}\label{TH:subalgebra-spectrum-I}
        For all $V'\in\Rep[u]{A}$, we have $\Spec_B(\quo{BV'})=\calR_{A/B}(\Spec_A(V'))$.
        Moreover, for any irreducible $U\wc \quo{BV'}$, there exists irreducible $V\wc V'$ such
        that $U\leq \quo{BV}$.
    \end{thm}

    \begin{proof}
    	We first check that  $\Spec_B(\quo{BV'})\supseteq \calR_{A/B}(\Spec_A(V'))$.
    	Suppose $[U]\in \calR_{A/B}(\Spec_A(V'))$. Then there is irreducible $V\wc V'$
    	such that $U\wc\quo{BV}$. Fix some $u\in\sphere{BU}$, and choose $e\in \ids{B}$
    	with $eu=u$. By Lemma~\ref{LM:equiv-conds-of-weak-containment}, for any $F\fsubseteq B$
    	and $\veps>0$,
    	there are $v\in \sphere{eV}\subseteq \sphere{BV}$ and $v'\in\sphere{eV'}\subseteq \sphere{BV'}$
        such that $\norm{\vphi_{U,u}-\vphi_{\quo{BV},v}}_F<\frac{\veps}{2}$
    	and $\norm{\vphi_{V,v}-\vphi_{V',v'}}_F<\frac{\veps}{2}$.
        This means
    	$\norm{\vphi_{U,u}-\vphi_{\quo{BV'},v'}}_F<\veps$,
    	so  $U\wc \quo{BV'}$.

    	Next, let
        $U\in \Irr[u]{B}$ be such that $U\wc \quo{BV}$.
        Fix $u\in\sphere{U}$, and let $S=\{\psi\in E(A;V')\suchthat \psi|_B=\vphi_{U,u}\}$.
        It is clear that $S$ is convex, closed, and hence compact (Theorem~\ref{TH:convexity}(i)).
        We claim that $S\neq\emptyset$. Indeed,  $U\wc \quo{BV'}$ implies $\vphi_{U,u}\in\overline{\{\vphi_{\quo{BV'},v'}\}_{ v'\in \quo{BV'}}}$.
        Choose a net $\{\vphi_{\quo{BV'},v'_\alpha}\}_{\alpha\in I}$ converging to $\vphi_{U,u}$.
        The set $E(A;V')$ is compact, so there is a subnet of $\{\vphi_{V',v'_\alpha}\}_{\alpha\in I}$
        converging to some $\psi\in E(A;V')$. Since $\psi|_B=\lim\{ \vphi_{V',v'_\alpha}|_B\}_\alpha=\lim\{ \vphi_{\quo{BV'},v'_\alpha}\}_\alpha
        =\vphi_{U,u}$,
        we have $\psi\in S$ and so $S\neq\emptyset$.
        Now, by the Krein-Milman Theorem \cite[p.~138]{FunctionalAnalysis92Nikol},
        there exists an extremal point $\psi\in S$. We claim that $\psi\in P(A;V')$.
        Indeed, if $\psi=t\psi'+(1-t)\psi''$ for $\psi',\psi''\in E(A;V')$ and $0<t<1$,
        then $\vphi_{U,u}=\psi|_B=t\psi'|_B+(1-t)\psi''|_B$. Since $\vphi_{U,u}$ is pure (Proposition~\ref{PR:pure-states}),
        we must have $\psi'|_B=\psi''|_B=\vphi_{U,u}$, which means $\psi',\psi''\in S$, and since $\psi$
        is extremal in $S$, we get $\psi=\psi'=\psi''$.
        Now, by Theorem~\ref{TH:convexity}(iii), there is $[V]\in\Spec_{A}(V')$
        and $v\in\sphere{V}$ such that $\psi=\vphi_{V,v}$.
        If $v\in \quo{BV}$, then Corollary~\ref{CR:coef-sharing} implies
        that $U=\quo{Au}\leq V$ and we are done. We therefore need to show that $v\in \quo{BV}$.

        Write $v=v_1+v_2$ with $v_1\in \quo{BV}$ and $v_2\in (BV)^\perp$. Then $Bv_2=0$,
        and hence $\vphi_{\quo{BV},v_1}=\vphi_{V,v_1}|_B=\vphi_{V,v}|_B=\vphi_{U,u}$.
        Thus, for all $e\in\ids{B}$, we have $\norm{ev_1}=\vphi_{\quo{BV},v_1}(e)=\vphi_{U,u}(e)=\norm{eu}$.
        By Lemma~\ref{LM:approximation-lemma}, this means
        $\norm{v_1}=\lim\{\norm{ev_1}\}_{e\in\ids{B}}=\lim\{\norm{eu}\}_{e\in\ids{B}}=\norm{u}=1$.
        Since $1=\norm{v}^2=\norm{v_1}^2+\norm{v_2}^2$,
        we get $\norm{v_2}=0$, and therefore $v=v_1\in \quo{BV}$.
    \end{proof}

    \begin{remark}
        In general, it is not true that for all $U\in\Irr[u]{B}$,
        there is $V\in\Irr[u]{A}$ such that $U\wc \quo{BV}$.
        Indeed, take $A=\C[X,X^*]$ and $B=\C[X^*X]$.
        Using Proposition~\ref{PR:unitary-dual-topology} and Remark~\ref{RM:unitary-dual-topology}, we identify $\hat{A}$
        with $\C$ (take $a_1=X$), and $\udual{B}$ with ${\R}$ (take $a_1=X^*X$).
        The map $\calR_{A/B}$ turns out to induce a function
        $\udual{A}\to\udual{B}$, which under the above identifications is given by $z\mapsto \cconj{z}z:\C\to \R$.
        This map is not onto, meaning that there are $U\in\Irr[u]{B}$ that are not weakly contained
        in $\quo{BV}$ for any $V\in \Irr[u]{A}$. This particularly holds for  $U\in\Irr[u]{B}$
        with $X^*X|_U=-\id_U$.
        An even wilder example is $A=\C(X)$ and $B=\C[X]$ with $X^*=X$. Then,
        $A$ has no unitary
        representations at all (Corollary~\ref{CR:commutative-algs}, Corollary~\ref{CR:non-empty-spectrum}),
        and $B$ has many.

        In contrast,
        when $A$ and $B$ are $C^*$-algebras,  every irreducible unitary
        representation of $B$ is contained in an irreducible unitary representation of $A$ \cite[Pr.~2.10.2]{Dixmier}.
    \end{remark}

    Let $\{a_1,\dots,a_t\}$  be a  family of elements in $A$ such that $a_ia_j=a_ja_i$
    and $a_i^*a_j=a_ja_i^*$ for all $i,j$.
    The following corollary
    implies that for all $V'\in\Rep[u]{A}$, one can recover the common spectrum $\Spec(a_1|_{V'},\dots,a_t|_{V'})$
    from $\Spec_A(V')$.

    \begin{cor}\label{CR:subalgebra-spectrum-II}
        Let $V'\in\Rep[u]{A}$, and let $a_1,\dots,a_t\in A$
        be elements generating a commutative $*$-subalgebra of $A$.
        Choose $e\in\ids{A}$ such that $ea_ie=a_i$ for all $1\leq i\leq t$ (e.g.\ $e=1$ if $A$ is unital).
        Then
        \[\Spec(a_1|_{eV'},\dots,a_t|_{eV'})=\bigcup_{[V]\in\Spec_A(V')}\Spec(a_1|_{eV},\dots,a_t|_{eV})\ .
        \]
        When $eV'\neq V'$, we further have $\Spec(a_1|_{V'},\dots,a_t|_{V'})=\Spec(a_1|_{eV'},\dots,a_t|_{eV'})\cup\{(0,\dots,0)\}$.
    \end{cor}

    \begin{proof}
        Let $B=\C[e,a_1,\dots,a_t,a_1^*,\dots,a_t^*]$. Then $\quo{BW}=BW=eW$ for any $W\in\Rep[u]{A}$.
        The first assertion now follows from Proposition~\ref{PR:unitary-dual-topology}
        and Theorem~\ref{TH:subalgebra-spectrum-I}.
        The second assertion holds since $B\cdot (eV')^\perp=0$. Indeed,
        for all $v\in (eV')^\perp=(BV')^\perp$, we have
        $\Trings{Bv,V'}=\Trings{v,BV'}=0$, so $Bv=0$.
    \end{proof}

\subsection{Corner Subalgebras}
\label{subsec:corner-subalgebras}

    Let $A$ be an idempotented involutary algebra.
    An idempotented $*$-subalgebra $B\subseteq A$ is called
    a \emph{corner subalgebra}, or just a corner, if $B=BAB$.
    Equivalently, $B$ is a corner subalgebra of $A$ if $B=\bigcup_{e\in\ids{B}}eAe$.

    \begin{example}
    	(i) Let $e\in\ids{A}$. Then $B=eAe$ is a corner subalgebra of $A$. When $A$ is unital,
    	$A$ has a  \emph{Pierce decomposition} $A\cong\smallSMatII{eAe}{eAf}{fAe}{fAf}$ with $f=1_A-e$.
    	This is the reason for the  name ``corner''. 
    	
    	(ii) Consider $A=\prod_{\N}\C$ with the involution $((a_n)_{n\in\N})^*=(\cconj{a_n})_{n\in\N}$
    	and let $B=\bigoplus_{\N}\C$. Then $B$ is a corner of $A$. However, there is no Pierce decomposition as in (i).
    \end{example}

    In the sequel, we shall make repeated usage of the following lemma, usually without comment.

    \begin{lem}
    	Let $V\in\Rep[u]{A}$ and let $W$ be a subspace of $V$.
    	Then $\quo{AW}=\quo{A\quo{W}}$.
    \end{lem}

    \begin{proof}
    	Since, $AW\subseteq A\quo{W}$, we have $\quo{AW}\subseteq\quo{A\quo{W}}$,
    	and since $a\quo{W}\subseteq \quo{AW}$ for all $a\in A$, we have $A\quo{W}=\sum_{a\in A}a\quo{W}\subseteq \quo{AW}$,
    	and hence $\quo{A\quo{W}}\subseteq \quo{AW}$.
    \end{proof}

    \begin{lem}\label{LM:corner-coefs}
        Let $B$ be a corner subalgebra of $A$,
        let $V\in\Rep[u]{A}$, and let $v\in {BV}$.
        Choose $e\in\ids{B}$ such that $v=ev$.
        Then $\vphi_{V,v}(a)=\vphi_{\quo{BV},v}(eae)$ for all $a\in A$.
        In particular, $\vphi_{\quo{BV},v}$ determines $\vphi_{V,v}$ and vice versa.
    \end{lem}

    \begin{proof}
        $\vphi_{V,v}(a)=\Trings{av,v}=\Trings{aev,ev}=\Trings{eaev,v}=\vphi_{\quo{BV},v}(eae)$.
    \end{proof}

    \begin{lem}\label{LM:irreducible-corner-module}
        Let $B$ be a corner subalgebra of $A$, and let $V,V'\in \Rep[u]{A}$.
        \begin{enumerate}
            \item[(i)] If $V$ is irreducible and $BV\neq 0$, then  $\quo{BV}\in\Irr[u]{B}$.
           	\item[(ii)] If $V,V'$ are irreducible and $BV,BV'\neq 0$,
            then $V\cong V'$ $\iff$ $\quo{BV}\cong \quo{BV'}$.
            \item[(iii)] If $V$ is irreducible and $BV\neq 0$,
            then $V\wc V'$ $\iff$ $\quo{BV}\wc \quo{BV'}$.
        \end{enumerate}
    \end{lem}

    \begin{proof}
        (i) Let $0\neq v\in \quo{BV}$. By Lemma~\ref{LM:approximation-lemma}, $v\in\quo{Bv}$. Since $V$ is irreducible,
         we have $\quo{Av}=V$ and hence $\quo{Bv}=\quo{BABv}= \quo{B\quo{A\quo{Bv}}}\supseteq \quo{B\quo{Av}}=\quo{BV}$.

        (ii)
        We only show the nontrivial direction. Let $f:\quo{BV}\to \quo{BV'}$ be a unitary
        isomorphism, let $v\in\sphere{BV}$ and let $v'=fv\in BV'$. Then
        $\vphi_{\quo{BV},v}=\vphi_{\quo{BV'},v'}$, and hence $\vphi_{V,v}=\vphi_{V',v'}$
        (Lemma~\ref{LM:corner-coefs}). By Corollary~\ref{CR:coef-sharing}, $V\cong V'$.

        (iii)
        Suppose $V\wc V'$ and let $v\in\sphere{BV}$. Choose $e\in \ids{B}$
        such that $ev=v$. By Lemma~\ref{LM:equiv-conds-of-weak-containment},
        for all $\veps>0$ and $F\fsubseteq B$, there is $v'\in eV'\subseteq BV'$ such that
        $\norm{\vphi_{V,v}-\vphi_{V',v'}}_F<\veps$, hence $\quo{BV}\wc \quo{BV'}$.
        Conversely, assume $\quo{BV}\wc \quo{BV'}$ and let $v\in\sphere{BV}$.
        Then there is a net $\{v'_\alpha\}_{\alpha\in I}\subseteq BV'$ such that
        $\lim\{\vphi_{\quo{BV'},v'_\alpha}\}_{\alpha\in I}=\vphi_{\quo{BV},v}$ in $B^\vee$.
        By Lemma~\ref{LM:corner-coefs}, this implies $\lim\{\vphi_{V',v'_\alpha}\}_\alpha=\vphi_{V,v}$,
        so $V\wc V'$.
    \end{proof}

    \begin{thm}\label{TH:corner-unitary-dual}
        Let $B$ be a corner subalgebra of $A$.
    	Denote by $\udual{A}^{(B)}$  the set of classes $[V]\in\udual{A}$
    	such that $BV\neq 0$. Then $\udual{A}^{(B)}$ is open in $\udual{A}$,
    	and the function $[V]\mapsto [\quo{BV}]:\udual{A}^{(B)}\to \udual{B}$
    	is a topological embedding.
    \end{thm}

    \begin{proof}
        Let $[W]\in\udual{A}^{(B)}$ and choose $u\in \sphere{BW}$
        and $e\in \ids{B}$ with $eu=u$. We claim that $N_{W,u,0.5,\{e\}}$ (see~\ref{subsec:unitary-dual})
        is contained in $\udual{A}^{(B)}$.
        Let $[V]\in N_{W,u,0.5,\{e\}}$. Then there is $v\in V$ such that
        $\abs{\Trings{ev,v}-\Trings{eu,u}}<0.5$. Since $\Trings{eu,u}=1$, this implies
        $\abs{\Trings{ev,v}}>0.5$. In particular, $ev\neq 0$, and hence $BV\neq 0$.
        This proves the first statement of the theorem.

        To proceed,
        let $S$ denote the image of the map $[V]\mapsto [\quo{BV}]:\udual{A}^{(B)}\to \udual{B}$.
        Fix some $[W]\in\udual{A}^{(B)}$, define $u$ and $e$ as above, and write $U=\quo{BW}$.
        By Proposition~\ref{PR:neighborhoods}(i) and the previous paragraph,
        the sets $\{N_{W,u,\veps,F}\}_{0<\veps<0.5,\,e\in F\fsubseteq A}$ form a basis of open neighborhoods of $[W]$ in $\udual{A}^{(B)}$,
        and
        the sets $\{N_{U,u,\veps,F}\}_{0<\veps,F\fsubseteq B}$ form a basis of open neighborhoods of $[U]$ in $\udual{B}$.
        Note that if $F\fsubseteq A$, then then $N_{W,u,\veps,F}=N_{W,u,\veps,eFe}$.
        Indeed, if $[V]\in N_{W,u,\veps,F}$, then there is $v\in V$ with $\norm{\vphi_{V,v}-\vphi_{W,u}}_F<\veps$.
        Since $\vphi_{V,ev}(a)=\vphi_{V,v}(eae)$ and $\vphi_{W,u}(a)=\vphi_{W,u}(eae)$ (cf.\ Lemma~\ref{LM:corner-coefs}),
        we get $\norm{\vphi_{V,ev}-\vphi_{W,u}}_{eFe}<\veps$, i.e.\ $[V]\in N_{W,u,\veps,eFe}$. The other direction is similar.
        Likewise, $N_{U,u,\veps,F}=N_{U,u,\veps,eFe}$ when $F\subseteq B$.
        Therefore,
        it is enough to show that for all and
        $0<\veps<\frac{1}{2}$ and $e\in F\fsubseteq eAe$, the map $[V]\mapsto [\quo{BV}]$ maps
        $N_{W,u,\veps,F}$ bijectively
        onto $N_{U,u,\veps,F}\cap S$.

        Suppose that $[V]\in N_{W,u,\veps,F}$.
        Then there is $v\in V$ with $\norm{\vphi_{V,v}-\vphi_{W,u}}_F<\veps$.
        Since $F\subseteq eAe$, we have $\vphi_{V,ev}|_F=\vphi_{V,v}|_F$, and hence $\norm{\vphi_{\quo{BV},ev}-\vphi_{U,u}}_F<\veps$,
        meaning that $[\quo{BV}]\in N_{U,u,\veps,F}$. Reversing the argument shows that any
        $[V]\in\udual{A}^{(B)}$ with $[\quo{BV}]\in N_{U,u,\veps,F}$ lies in $N_{W,u,\veps,F}$.
        Finally, the map $[V]\mapsto [\quo{BV}]$ is injective by Lemma~\ref{LM:irreducible-corner-module}(ii).
    \end{proof}

    \begin{remark}
        It is not true in general that any  unitary representation of a corner subalgebra $B$
        is  of the form $\quo{BV}$ for $V\in \Irr[u]{A}$. For example, take $A$ as in Example~\ref{EX:double-arrow-quiver}
        and
        let $B=e_0Be_0$. Then $B$ is the free unital algebra generated by $Y:=a_{01}a_{10}$ (with $Y^*=Y$), and
        by Remark~\ref{RM:unitary-dual-topology}, there is a homeomorphism $[U]\mapsto Y|_{U}:\udual{B}\to \R$.
        However, only $U$-s corresponding to $\R_{\geq 0}$ are of the form $\quo{BV}$ for $V\in \Irr[u]{A}$.
        This follows directly by checking the list of irreducible representations given in Example~\ref{EX:double-arrow-quiver}.
        Also, since $Y=a_{01}a_{10}=a_{10}^*a_{10}$, we must have $Y|_{\quo{BV}}\in\R_{\geq 0}$.
    \end{remark}

    The following theorem is a strengthening of Theorem~\ref{TH:subalgebra-spectrum-I} in case $B$ is a corner.

    \begin{thm}\label{TH:corner-subalgebra-spectrum-I}
        Let $B$ be a corner subalgebra of $A$, and let $V'\in \Rep[u]{A}$. Then:
        \begin{enumerate}
            \item[(i)] $\Spec_{B}(\quo{BV'})=\{[\quo{BV}]\where [V]\in\Spec_A(V')\cap\udual{A}^{(B)}\}$.
            \item[(ii)] $\Spec_{A}(V')\cap\udual{A}^{(B)}=\{[V]\in\udual{A}\suchthat [\quo{BV}]\in\Spec_{B}(\quo{BV'})\}$.
        \end{enumerate}
    \end{thm}

    \begin{proof}
        (i) That $\Spec_{B}(\quo{BV'})\supseteq \{[\quo{BV}]\where [V]\in\Spec_A(V')\cap\udual{A}^{(B)}\}$
        follows from parts (i) and (iii) of Lemma~\ref{LM:irreducible-corner-module}. Conversely, if $[U]\in\Spec_{B}(\quo{BV'})$,
        then by Theorem~\ref{TH:subalgebra-spectrum-I},
        there exists $V\in\Irr[u]{A}$ such that $U\leq \quo{BV}$ and $V\in \Spec_A(V')$.
        Since $\quo{BV}$ is irreducible (Lemma~\ref{LM:irreducible-corner-module}(i)), we have $[U]=[\quo{BV}]$.

        (ii) This follows from Lemma~\ref{LM:irreducible-corner-module}(iii).
    \end{proof}

    A family of corner subalgebras  $\{B_\alpha\}_{\alpha\in I}$ of $A$ is called
    \emph{full} if $\sum_{\alpha} AB_\alpha A=A$. In this case, it is easy
    to see that    for every $A$-module $V$ with $AV\neq 0$, there is
    $\alpha\in I$ such that $B_\alpha V\neq 0$.

    \begin{example}\label{EX:full-families-of-idems}
        Assume $A$ is unital and $\{e_1,\dots,e_t\}$ is a family of orthogonal idempotents
        in $\ids{A}$ such that $\sum_i e_i=1_A$. Then the family $\{e_iAe_i\}_{i=1}^t$ is full.
        More generally, if $\{e_\alpha\}_{\alpha\in I}\subseteq\ids{A}$,
        then the family $\{e_\alpha Ae_\alpha\}_{\alpha\in I}$ is full
        if and only if $\sum_{\alpha}Ae_\alpha A=A$. In this case,
        $\{e_\alpha\}_{\alpha\in I}$ is called a \emph{full} family of idempotents in $A$.
    \end{example}

    \begin{example}\label{EX:matrix-algebra}
        Suppose there is an idempotented involutary algebra $A'$ such that
        $A=\nMat{A'}{n}$ and the involution on $A$ is given by $((a_{ij})_{i,j})^*=(a_{ji}^*)_{i,j}$.
        Let $B$ denote the $*$-subalgebra of $A$ consisting of matrices
        of the form
        \[
        \left[\begin{smallmatrix}
        * & 0 & \dots\\
        0 & 0 & \dots\\[-3pt]
        \vdots & \vdots & \ddots
        \end{smallmatrix}\right]\subseteq \nMat{A'}{n}\ .
        \]
        Then $B$ is a corner subalgebra of $A$ and the family $\{B\}$ is full, i.e.\ $A=ABA$.
        In fact, in this case, the map $V\mapsto \quo{BV}:\Rep[u]{A}\to\Rep[u]{B}$
        is an equivalence of categories resembling a Morita equivalence (see \cite[\S17--18]{La99}).
        A mutual inverse of the functor $V\mapsto \quo{BV}$ can be constructed as follows:
        Identify $B$ with $A'$.
        For any $U\in \Rep[u]{B}$, view $U^n$ as column vectors and give it the obvious left $\nMat{A'}{n}$-module structure.
        Define an inner product on $U^n$ by $\Trings{(u_1,\dots,u_n),(u'_1,\dots,u'_n)}=\sum_{i=1}^n\Trings{u_i,u'_i}$.
        Then $U\mapsto U^n:\Rep[u]{B}\to\Rep[u]{A}$ is a mutual inverse of $V\mapsto \quo{BV}$. The details
        are left to the reader.
    \end{example}

    \begin{cor}\label{CR:corner-subalgebra-spectrum-II}
        Let $\{B_\alpha\}_{\alpha\in I}$ be a full family of corner subalgebras, and let $V'\in \Rep[u]{A}$.
        Then
        \[\Spec_A(V')=\{[V]\in\udual{A}\where [\quo{B_\alpha V}]\in\Spec_{B_\alpha}(\quo{B_\alpha V})~\text{for some $\alpha\in I$}\}\ .\]
    \end{cor}

    \begin{proof}
        This follows from Theorem~\ref{TH:corner-subalgebra-spectrum-I} and the fact that
        for all $V\in\Irr[u]{A}$, there is $\alpha\in I$ such that $B_\alpha V\neq 0$.
    \end{proof}

    Corollary~\ref{CR:corner-subalgebra-spectrum-II} and Theorem~\ref{TH:corner-subalgebra-spectrum-I}
    imply for any full family of corner subalgebras $\{B_\alpha\}_{\alpha\in I}$ and $V\in\Rep[u]{A}$,
    the set $\Spec_A(V)$ determines the sets $\{\Spec_{B_\alpha}(\quo{B_\alpha V})\}_{\alpha\in I}$
    and vice versa.
    As a consequence $\{\Spec_{B_\alpha}(\quo{B_\alpha V})\}_{\alpha\in I}$ determines $\Spec_{B'}(V)$
    for any idempotented $*$-subalgebra $B'$ of $A$ (Theorem~\ref{TH:subalgebra-spectrum-I}).

\subsection{Pre-Unitary Representations}
\label{subsec:pre-unitary-reps}

    Let $A$ be an  idempotented $*$-algebra.
    Some of the representations that we will encounter later in the text
    will not be unitary but rather {pre-unitary}.




\medskip

    A \emph{pre-unitary representation} of $A$ is a pre-Hilbert space $V$
    endowed with a left $A$-module structure such that
    \begin{enumerate}
        \item[(U1)\phantom{$'$}] $\Trings{au,v}=\Trings{u,a^*v}$ for all $a\in A$ and $u,v\in V$,
        \item[(U2$'$)]  $V$ is smooth (i.e.\ $V=AV$),
        \item[(U3)\phantom{$'$}] for all $a\in A$, the operator $a|_{V}:V\to V$ is  bounded.
    \end{enumerate}
    We denote by $\Rep[pu]{A}$ the category of pre-unitary representations of $A$.
    The morphisms are continuous $A$-homomorphisms.
    A pre-unitary representation is said to be irreducible if its underlying $A$-module
    is irreducible.

\medskip

    If $V$ is a pre-unitary representation, then the action of $A$ extends to the completion
    $\quo{V}$, which then becomes a unitary representation. Conversely, if $U$ is a unitary representation,
    then $\sm{U}=AU$ is a pre-unitary representation. We always have $\quo{\sm{U}}=U$, but in general $\sm{(\quo{V})}$
    may be larger than $V$. (However, we shall  see in \ref{subsec:admissible-representations} that
    $\sm{(\quo{V})}=V$ when $V$ is \emph{admissible}.)
    The maps $V\mapsto \quo{V}$ and $U\mapsto \sm{U}$ are in fact functorial,
    taking continuous $A$-homomorphisms to continuous $A$-homomorphisms.

\medskip

    We  extend the spectral theory of the previous sections to pre-unitary representations
    by defining
    \[\Spec_A(V'):=\Spec_A(\quo{V'}) \qquad\forall\,\, V'\in\Rep[pu]{A}\ .\]
    This means that $[V]\in\udual{A}$
    lies in $\Spec_A(V')$ if and only if $V\wc\quo{V'}$.
    By  Lemma~\ref{LM:equiv-conds-of-weak-containment}, this is
    equivalent to:
    \begin{itemize}
        \item
        For all $v\in \sphere{V}$, $\veps>0$ and $F\fsubseteq A$,
        there is $v'\in \sphere{V'}$ such that $\norm{\vphi_{V,v}-\vphi_{V',v'}}_F<\veps$
        for all $a\in F$.
    \end{itemize}
    Furthermore, it is enough to have this for just one $v\in\sphere{V}$, and the assumption
    $v'\in\sphere{V'}$ can be dropped.

\subsection{Admissible Representations}
\label{subsec:admissible-representations}

    We finish this chapter with recalling
    {admissible representations}.

\medskip

    Let $A$ be an idempoteneted $*$-algebra. A unitary (resp.\ pre-unitary)
    representation $V$ of $A$ is called
    \emph{admissible} if $\dim  eV<\infty$ for any  $e\in \ids{A}$.

    \begin{example}
        If $A$ is unital, then the admissible representations are precisely the finite-dimensional
        representations. For such representations, the terms unitary and pre-unitary coincide.
    \end{example}

    The following theorem summarizes the most important features of admissible representations.

    \begin{thm}\label{TH:admissible-pu-are-completely-red}
        Assume $V'\in \Rep[u]{A}$ is admissible. Then:
        \begin{enumerate}
            \item[(i)] $V'$ is completely reducible, namely, there are irreducible unitary
            representations $\{V_i\}_{i\in I}$ such that $V\cong \hat{\bigoplus}_{i\in I} V_i$.
            \item[(ii)] If $V\in \Irr[u]{A}$, then there are finitely many indices $i\in I$ (possibly none)
            for which $V\cong V_i$.
            \item[(iii)] For all $V\in \Irr[u]{A}$, we have $V\wc V'$ $\iff$ $V\leq V'$.
            In particular, $\Spec_A(V')=\{[V_i]\where i\in I\}$.
        \end{enumerate}
    \end{thm}

    \begin{proof}
        This is a well-known. We recall the proof for the sake of completeness.

        (i) We first claim that if $V'\neq 0$, then $V'$ contains an irreducible submodule.
        There is $e\in\ids{A}$ such that $eV'\neq 0$. Since $V'$ is admissible,
        $\dim eV'<\infty$. This easily implies that there is $U\in\Irr[u]{eAe}$
        such that $U\leq eV'$. By Theorem~\ref{TH:corner-subalgebra-spectrum-I}(i),
        there is $V\in\Irr[u]{A}$ such that $U=eV$. Let $u\in \sphere{U}$.
        We view $u$ as a vector of both $V$ and $V'$.
        By Lemma~\ref{LM:corner-coefs}, $\vphi_{V,u}=\vphi_{V',u}$
        and hence $V\leq V'$ (Corollary~\ref{CR:coef-sharing}).

        Now, using Zorn's Lemma, let $\{V_i\}_{i\in I}$ be a maximal collection
        of pairwise orthogonal irreducible subrepresentations of $V$. Let $W=(\bigoplus_i V_i)^\perp$.
        Then $W$ is admissible, and hence, if it is nonzero,
        it admits an irreducible subrepresentation. Since this contradicts the maximality
        of $\{V_i\}_{i\in I}$,  we must have $W=0$, so $V=\hat{\bigoplus}_{i\in I} V_i$.

        (ii) Choose $e\in \ids{A}$  such that $eV\neq 0$.
        Then $\dim \hat{\bigoplus}_i eV_i=\dim eV'<\infty$, and hence $V_i\cong V$
        is possible for only finitely many $i$-s.

        (iii) Assume $V\wc V'$ and choose $e\in \ids{A}$ such that $eV\neq 0$.
        Let $v\in \sphere{eV}$. By Lemma~\ref{LM:equiv-conds-of-weak-containment},
        there is a net $\{v'_\alpha\}_{\alpha\in J}\subseteq\sphere{eV'}$ such that $\lim\{\vphi_{V',v'_\alpha}\}_{\alpha}=\vphi_{V,v}$.
        Since $\dim eV'<\infty$, the unit sphere $\sphere{eV'}$ is compact, and hence
        there is a subnet $\{v'_\alpha\}_{\alpha\in J'}$ converging to $v'\in \sphere{V'}$.
        This means $\lim\{\vphi_{V',v'_\alpha}\}_{\alpha\in J'}=\vphi_{V',v'}$ and hence
        $\vphi_{V,v}=\vphi_{V',v'}$. By Corollary~\ref{CR:coef-sharing},  $V\leq V'$.
    \end{proof}

    Let $V'\in\Rep[u]{A}$ be admissible, let $[V]\in\udual{A}$, and write $V'=\hat{\bigoplus}_{i\in I}V_i$
    as in Theorem~\ref{TH:admissible-pu-are-completely-red}. The number of $i$-s with $V_i\cong V$
    is called the multiplicity of $[V]$ in $V'$ and is denoted $\mult_{V'}[V]$. It can be
    defined alternatively by
    \[
    \mult_{V'}[V]=\dim_\C\cHom_A(V,V')
    \]
    (cf.\ Theorem~\ref{TH:Schur-lemma}), so it is independent of the decomposition $V'=\hat{\bigoplus}_{i\in I}V_i$.
    This allows us to consider $\Spec_A(V')$ as a \emph{multiset}. For our purposes, a multiset cosists of a set $X$
    together with a map $\mult_X:X\to\N\cup\{0\}$ assigning each $x\in X$ its multiplicity; morphisms of multisets are
    defined in the obvious way.
    When $V'\in\Rep[u]{A}$ is admissible, we let
    \[
    \mSpec_A(V')
    \]
    denote the multiset obtained from $\Spec_A(V')$ by setting the multiplicity of $[V]$
    to be $\mult_V'[V]$.\footnote{
   		It is likely that the idea of ``spectrum with multiplicity'' can be extended
   		beyond admissible reprsentations by considering projection valued measures on $\udual{A}$,
   		rather than multisets;
   		developing the necessary theory
   		is beyond the scope of this work.
    }

    \begin{thm}\label{TH:admissible-category-equivalence}
    	Let $\Rep[u,ad]{A}$ (resp.\ $\Rep[pu,ad]{A}$) denote the category
    	of admissible unitary (resp.\ pre-unitary) representations of $A$.
    	Then $\Rep[u,ad]{A}$ and $\Rep[pu,ad]{A}$ are equivalent;
    	the equivalence is given by the functors $U\mapsto \sm{U}:\Rep[u,ad]{A}\to\Rep[pu,ad]{A}$ and
    	$V\mapsto \quo{V}:\Rep[pu,ad]{A}\to \Rep[u,ad]{A}$.
    	In particular,  $U\in \Rep[u,ad]{A}$ is irreducible if and only $\sm{U}$ is irreducible (as an $A$-module).
    \end{thm}

    \begin{proof}
    	We need to check that
    	$U\mapsto \sm{U}:\Rep[u,ad]{A}\to\Rep[pu,ad]{A}$ and
    	$V\mapsto \quo{V}:\Rep[pu,ad]{A}\to \Rep[u,ad]{A}$
    	are mutually inverse to each other.
    	The only detail that is not straightforward is showing that $\sm{(\quo{V})}=V$ for all $V\in\Rep[pu,ad]{A}$.
    	We prove this by showing that if $V^\perp$ is the orthogonal complement of $V$ in $V':=\sm{(\quo{V})}$, then
    	$V\oplus V^\perp=V'$. This  implies $\quo{V}=\quo{V'}=\quo{V}\oplus\quo{V^\perp}$, forcing $V^\perp=0$
    	and $V=V'=\sm{(\quo{V})}$.
    	
    	Let $v\in V'$. There is $e\in\ids{A}$
    	such that $ev=v$. Since $V$ is admissible, $\dim eV<\infty$, in which case
    	it is easy to check that $V'=eV\oplus(eV)^\perp$. Therefore, there are
    	unique  $v_1\in eV$ and $v_2\in (eV)^\perp$ with $v=v_1+v_2$.
    	Since $v=ev=ev_1+ev_2$ and $\Trings{ev_2,eV}=\Trings{v_2,eeV}=0$, we must
    	have $v_2=ev_2$. This means, $\Trings{v_2,V}=\Trings{ev_2,V}=\Trings{v_2,eV}=0$,
    	so $v_2\in V^\perp$. Therefore, $v=v_1+v_2\in V+V^\perp$, as required.
    \end{proof}

\section{Simplicial Complexes}
\label{sec:simplicial-complexes}

    This chapter recalls simplicial complexes. We also discuss topological groups
    acting on simplicial complexes.

\subsection{Simplicial Complexes}
\label{subsec:simp-complex}

    Recall that a  \emph{simplicial complex} consists of a nonempty  set $X$  of finite sets
    such that subsets of sets in $X$ are also in $X$.
    A partially ordered set $(Y,\leq)$ which is isomorphic to $(X,\subseteq)$ for some simplicial complex $X$ will also be called
    a simplicial complex. This holds  when:
    \begin{enumerate}
        \item[(1)] for all $y\in Y$, the partially ordered set $(\{y'\in Y\suchthat y'\leq y\},\leq)$ is isomorphic
        to $(P(\{1,\dots,n\}),\subseteq)$ for some $n\in\N$ and
        \item[(2)] for all $y,y'\in Y$, the set $\{y,y'\}$ has an infimum relative to $\leq $.
    \end{enumerate}

    The elements of a simplicial complex $X$ are called \emph{cells}.
    The complex $X$ is \emph{locally finite} if every cell in $X$ is contained in finitely many cells.
    We let $X^{(i)}$ denote the sets in $X$ of cardinality
    $i+1$. Elements of $X^{(i)}$ are called $i$-dimensional cells, or just
    $i$-cells.
    The \emph{vertex set} of $X$ is defined to be $\vrt{X}:=\bigcup_{x\in X}x$. By
    abuse of notation, we sometimes refer to elements of $X^{(0)}$ as vertices.
    A morphism of simplicial complexes $f:X\to Y$ consists of a function $f:\vrt{X}\to \vrt{Y}$
    such that for all $i$ and $x\in X^{(i)}$, we have $f(x):=\{f(v)\where v\in x\}\in Y^{(i)}$. The induced
    maps $X^{(i)}\to Y^{(i)}$ and $X\to Y$ are also denoted $f$.

    The dimension of $X$, denoted $\dim X$, is the maximal $i$ such that $X^{(i)}\neq \emptyset$.
    Cells of dimension $\dim X$ are called \emph{chambers}, and $X$ is \emph{pure} if every
    cell is contained in a chamber. If $x$ and $x'$ are distinct cells in $X$, then the \emph{combinatorial distance} of $x$ from $x'$,
    denoted $\dist(x,x')$, is the minimal $t\in\N$
    such that there exists a sequence of cells $y_1\dots,y_t$  with
    $x\subseteq y_1$, $x'\subseteq y_t$ and
    $y_i\cap y_{i+1}\neq \emptyset$
    for all $1\leq i<t$. We further set $\dist(x,x)=0$. (This agrees with the combinatorial distance in graphs.)
    The \emph{ball of radius $n$} around $x$,  $\Ball_X(x,n)$, consists of the
    cells in $X$ of distance $n$ or less from $x$. When $X$ is locally finite, all the balls $\Ball_X(x,n)$  ($x\neq\emptyset$)
    are  finite.
    We say that $X$ is \emph{connected} if $X=\bigcup_{n\geq 0} \Ball_X(x,n)$ for some (and hence all) $x\in X- \{\emptyset\}$,
    or equivalently, if $\dist(x,x')<\infty$ for all $x,x'\in X$.

    Every simplicial complex $X$ can be regarded as a topological space by replacing its abstract simplices
    with topological simplices, glued in the obvious way. If $X$ and $Y$ are connected simplicial complexes,
    then a morphism $f:X\to Y$ is  a \emph{cover map}
    if it  is a cover map when $X$ and $Y$ are regarded as topological spaces. This is equivalent to
    saying that $f:\vrt{X}\to \vrt{Y}$ is surjective and induces
    a bijection $f:\{x\in X\suchthat v\in x\}\to\{y\in Y\suchthat f(v)\in y\}$ for all $v\in\vrt{X}$.
    In this case, the \emph{deck transformations} of  $f:X\to Y$ are the
    automorphisms $h$
    of $X$ satisfying $f\circ h=f$.

    We let $\dot{X}$    denote $X- \{\emptyset\}$.

\medskip

    \emph{Throughout, all simplicial complexes are assumed to be connected and locally finite.}

\subsection{Orientation}
\label{subsec:orientation}

    Let $X$ be a simplicial complex. An \emph{ordered cell} in $X$ consists of a pair $(x,\leq)$ where $x\in X$ and $\leq$ is a full
    ordering of the vertices of $x$.
    Two orders on $x$ are equivalent if one can be obtained from the other by an even permutation. We denote
    by $[x,\leq]$ the equivalence class of $(x,\leq)$ and call it an \emph{oriented cell}.
    Oriented cells will usually be denoted by the letters $\mathsf{x},\mathsf{y},\dots$.
    Note that when $x\in X^{(i)}$ and $i>0$, there are exactly two possible orientations on $x$. In this
    case, we write $[x,\leq]^\op$ for $x$ endowed with the orientation different from the one induced by $\leq$.
    When $i\leq 0$,
    there is only one possible orientation on $x$, and we set $[x,\leq]^\op=[x,\leq]$ for convenience.
    In addition, for  $\{v_0,\dots,v_i\}\in X^{(i)}$, let
    \[
    [v_0v_1\dots v_i]=[\{v_0,\dots,v_i\}\, ,\,v_0<v_1<\dots<v_i]\ .
    \]
    We let $\ori{X}$ (resp.\ $\ori{X}^{(i)}$) denote the set of oriented cells in $X$
    (resp.\ $X^{(i)}$).

\medskip

    We now recall the construction of the \emph{$i$-dimensional Laplacian of $X$}
    (which will also be considered when $X$ is infinite):
    Recall that the space of \emph{$i$-dimensional forms} (or just \emph{$i$-forms}) of finite support on $X$ is defined by
    \[
    \Omega_i^-(X)=\left\{\begin{array}{ll}
    \Big\{\vphi\in\llf(\ori{X}^{(i)})\suchthat \vphi(\mathsf{x}^\op)=-\vphi(\mathsf{x})\,\text{for all}\,\mathsf{x}\in\ori{X}^{(i)}\Big\} & i>0 \\
    \llf(\ori{X}^{(i)}) & i= 0
    \end{array}\right.
    \]
    We make $\Omega_i^-(X)$ into a pre-Hilbert space by setting
    \begin{equation} \label{EQ:inner-prod-on-Omega-minus}
    \Trings{\vphi,\psi}=2\cdot \sum_{x\in X^{(i)}}{\vphi x}\cdot \cconj{\psi x}\qquad\forall \vphi,\psi\in\Omega_i^-(X)\ .
    \end{equation}
    In the summation over $X^{(i)}$, we pick an arbitrary orientation for $x$. Since both $\vphi$ and $\psi$ reverse sign when the orientation
    is changed, the expression $\cconj{\vphi x}\cdot \psi x$ is well-defined. When $i>0$,  the inner product
    agrees with the usual inner product on $\llf(\ori{X}^{(i)})$.

    Recall that the boundary map $\partial_{i+1}:\Omega_{i+1}^-(X)\to \Omega_{i}^-(X)$
    and the coboundary map
    $\delta_i:\Omega_i^-(X)\to\Omega_{i+1}^-(X)$
    are defined by
        \[
        (\partial_{i+1}\psi)[v_0\dots v_i]=\sum_{\substack{v\in\vrt{X} \\ \{v,v_0,\dots,v_i\}\in X^{(i+1)}}}\psi[vv_0\dots v_1]
        \]
        \[
        (\delta_i\vphi)[v_0\dots v_{i+1}]=\sum_{j=0}^{i+1}(-1)^j\vphi[v_0\dots \hat{v}_j\dots v_{i+1}]
        \]
    ($\hat{v}_j$ means omitting the $j$-th entry).
    It is easy to check that $\delta_i^*=\partial_{i+1}$.
    The upper, lower and total $i$-dimensional Laplacians are defined by
        \[
        \Delta_i^+=\partial_{i+1}\delta_i,\qquad\Delta_i^-=\delta_{i-1}\partial_i,\qquad\Delta_i=\Delta_i^++\Delta_i^-\, ,
        \]
    respectively (with the convention that $\Delta_0^-=0$). All three Laplacians take $\Omega_i^-(X)$ into itself.

\medskip

    Finally, for later usage, we also define
    \[
    \Omega_i^\pm(X)=\llf_i(\ori{X}^{(i)})\ ,
    \]
    and
    \[
    \Omega_i^+=\left\{\begin{array}{ll}
    \Big\{\vphi\in\llf(\ori{X}^{(i)})\suchthat \vphi(\mathsf{x}^\op)=\vphi(\mathsf{x})\,\text{for all}\,\mathsf{x}\in\ori{X}^{(i)}\Big\} & i>0 \\
    \llf(\ori{X}^{(i)}) & i= 0
    \end{array}\right.
    \]
    We give $\Omega_i^+(X)$ the inner product induced from   $\Omega_i^\pm(X)$.
    The space $\Omega_i^+(X)$ is sometimes called the space of \emph{$i$-dimensional anti-forms} of finite support on $X$.
    Notice that when $i>0$, we have $\Omega_i^\pm(X)=\Omega_i^+(X)\oplus \Omega_i^-(X)$ as pre-Hilbert spaces,
    whereas $\Omega_0^\pm(X)=\Omega_0^+(X)\cong\Omega_0^-(X)$.
    In addition, for $i>0$,
    the pre-Hilbert space
    $\Omega_i^+(X)$ is naturally isomorphic to $\llf(X^{(i)})$; the isomorphism is given by linearly extending
    $\e_{[x,\leq]}+\e_{[x,\leq]^\op}\mapsto \sqrt{2}\e_x$
    ($x\in X^{(i)}$). We will
    therefore often identify $\Omega^+_i(X)$ with $\llf(X^{(i)})$,
    usually without comment.

\subsection{Groups Acting on Simplicial Complexes}
\label{subsec:G-complex}

    Let $G$ be an $\ell$-group, i.e.\ a totally disconnected locally compact Hausdorff topological group.
    By a \emph{$G$-complex}, we mean a  (locally finite, connected) simplicial complex $\calX$ on which $G$ acts faithfully via automorphisms
    and such that for all $x\in\dot{\calX}$, the stablizer $\Stab_G(x)$ is a compact open subgroup
    of $G$. The latter is equivalent to   $\Stab_G(v)$ being a compact  open subgroup of $G$ for all $v\in\vrt{\calX}$.

    \begin{example}\label{EX:G-complex-building}
        The complex
        $\calB_d(F)$ of Chapter~\ref{sec:ramanujan-complexes} is a $G$-complex
        for $G=\nPGL{F}{d}$.
    \end{example}

    \begin{example}\label{EX:G-complex-k-regular-tree}
        Let $\calX$ be a $k$-regular tree and let $G=\Aut(\calX)$. We give $\calX$ the discrete topology
        and $G$ the topology of pointwise convergence. It is easy to check that $G$ is an $\ell$-group and $\calX$ is a $G$-complex
    \end{example}

    Generalizing Example~\ref{EX:G-complex-k-regular-tree}, we have:

    \begin{prp}\label{PR:characterization-of-G-comps}
        Let $\calX$ be a simplicial complex. Give $\calX$ the discrete topology
        and $G:=\Aut(\calX)$ the topology of pointwise convergence.
        Then:
        \begin{enumerate}
            \item[(i)] $G$ is an $\ell$-group and $\calX$ is a $G$-complex. Consequently, $\calX$ is an
            $H$-complex for any closed subgroup $H$ of $G$.
            \item[(ii)] If $\calX$ is an $H$-complex for an $\ell$-group  $H$, then the map $H\to \Aut(\calX)=G$
            is a closed embedding.
        \end{enumerate}
    \end{prp}

    \begin{proof}
        (i) Let $v\in\vrt{X}$.
        The set $\Stab_G(v)$
        is clearly open, so we need to show it is compact.
        Let $B_n=\Ball_\calX(v,n)$ and let $K_n$ be the  group of simplicial automorphisms of $B_n$ that fix $v$.
        Give $K_n$ the topology of point-wise convergence.
        Since automorphisms  preserve combinatorial distance,
        we have well-defined maps $g\mapsto g|_{K_{n-1}}:K_n\to K_{n-1}$ ($n\in\N$), and since $\calX$ is connected, we have
        $\invlim K_n\cong \Stab_{G}(v)$ as topological groups.
        By assumption, $\calX$ is locally finite and hence the groups $K_n$ are finite and discrete.
        It therefore follows from Tychonoff's Theorem that $\Stab_G(v)$ is compact.

        (ii) Observe first that the map $H\to G$ is continuous.
        Indeed, the collection $\{\Stab_{G}(v)\where v\in\vrt{\calX}\}$
        is a subbasis for the open neighborhoods of $1_G$, and $H\cap\Stab_{G}(v)=\Stab_{H}(v)$
        is open for all  $v\subseteq\vrt{\calX}$.
        Fix some $x\in \dot{\calX}$. Then $\Stab_{H}(x)$ is compact, hence the inclusion map
        $\Stab_{H}(x)\to \Stab_{G}(x)$ is  closed. This map is also  continuous and injective, so it is a topological embedding.
        Since $H$ and $G$ are disjoint unions of  cosets
        of the open subgroups $\Stab_{H}(x)$ and $\Stab_{G}(x)$ respectively,
        the map $H\to G$ is a topological embedding. Finally, $H$ is closed
        in $G$ since the complement $G- H$ is a union of translations  of
        the open sets $\Stab_{G}(x)$
        and $\Stab_{G}(x)-\Stab_{H}(x)$.
    \end{proof}

    We say that a $G$-complex $\calX$ is \emph{almost transitive} or \emph{almost $G$-transtive}
    when $G\leftmod \calX$ is  finite.
    The $G$-complexes of Examples~\ref{EX:G-complex-building} and~\ref{EX:G-complex-k-regular-tree} are
    almost  transitive.

    \begin{example}\label{EX:G-complex-building-III}
        Let $\bfG$ be an almost simple algebraic group over a non-archimedean
        local field $F$, let $\bfZ$ be the center of $\bfG$, and let $G=\bfG(F)/\bfZ(F)$.\footnote{
            In general, $\bfG(F)/\bfZ(F)$ is not the same as $(\bfG/\bfZ)(F)$. More precisely, one has
            an exact sequence $1\to \bfZ(F)\to \bfG(F)\to (\bfG/\bfZ)(F)\to \mathrm{H}^1_{\textit{fppf}}(F,\bfZ)$.
        }
        Let $\calB$ be the  \emph{affine Bruhat-Tits
        building} of $\bfG$ (see \cite{Tits79}, \cite{BruhatTits72I}; a more elementary
        treatment in the case $\bfG$ is  classical
        can be found in \cite{AbramNebe02}).
        The building $\calB$ is a pure simplicial complex carrying a faithful left $G$-action such that  stabilizers of chambers are compact
        open and $G$ acts transitively on chambers.
        Therefore, $\calB$ is an almost transitive $G$-complex.
        We further note that the vertices of $\calB$ have a canonical coloring, which
        is preserved by $G$ when  $\bfG$ is simply-connected.
        For example, when $\bfG=\bPGL_d$, this recovers the construction of $\calB_d(F)$ with
        the coloring $C_0$ in Chapter~\ref{sec:ramanujan-complexes}.
        When $\bfG=\uSL_d$, we also have $\calB=\calB_d(F)$, but in this case $G$ is $\im(\nSL{F}{d}\to \nPGL{F}{d})$ and it preserves
        the coloring $C_0$ (the group $G$ may be smaller than the group of color-preserving automorphisms).
    \end{example}

\subsection{Quotients of Simplicial Complexes}
\label{subsec:quotienst-of-simp-comps}

    We now recall several facts about quotients of simplicial complexes.
    The results to follow seem to be known, but we could not find explicit proofs in the literature.
    We have therefore included the proofs.
    Throughout, $G$ is an $\ell$-group and $\calX$ is a $G$-complex.
    Recall that if $x\in\calX$ and $x=\{v_1,\dots,v_t\}$ with $v_1,\dots,v_t\in\vrt{\calX}$,
    then $gx=\{gv_1,\dots,gv_t\}$.

\medskip

    Let $\Gamma$ be a subgroup of $G$.
    The partial order on $\calX$ induces a partial order on $\Gamma\leftmod \calX$ given
    by $\Gamma x\leq \Gamma y$ $\iff$ $\gamma x\subseteq y$ for some $\gamma\in\Gamma$.
    However, $\Gamma\leftmod \calX$ is not a simplicial complex in general, and even when this is the case,
    the projection map $x\mapsto\Gamma x:\calX\to \Gamma\leftmod \calX$ may not be a cover map.
    When both conditions hold, we call $\Gamma\leftmod \calX$  a \emph{$G$-quotient} of $\calX$
    and write $\Gamma\leq_\calX G$.

    More generally, a complex $X$ is a $G$-quotient
    of $\calX$ if there is a cover map $f:\calX\to X$ such that the group of deck transformations
    of $f$, call it $\Gamma$, is contained in $G$. In this case, there is a unique isomorphism $\Gamma\leftmod \calX\to X$
    such that the diagram
    \[
    \xymatrix{
    \calX \ar[d]_{x\mapsto\Gamma x} \ar[dr]^f & \\
    \Gamma\leftmod \calX \ar[r] & X
    }
    \]
    commutes, so $\Gamma\leq_\calX G$. We shall see below (Proposition~\ref{PR:quotient-cover-sufficient-nec-conds})
    that when $\Gamma\leq_\calX G$, the group of deck
    transformations of $\calX\to \Gamma\leftmod \calX$ is $\Gamma$, hence both definitions agree.

    \begin{example}\label{EX:simplicial-quotient}
    	(i) Let $\calX$ be a $k$-regular tree and let $G=\Aut(\calX)$. Then
        any $k$-regular graph (without multiple edges and loops) is a $G$-quotient of $\calX$.

        (ii) Let $\calX$ be a cyclic graph on $4$ vertices, let $\Z/4\Z$
        act on $\calX$ by cyclic rotations, and let $\Gamma=2\Z/4\Z$. Then $\Gamma\leftmod \calX$
        is not a simplicial complex. The problem is techincal and lies in the fact
        that our definition of simplicial complexes does not allow multiple edges (or cells). Indeed,
        $\Gamma\leftmod\calX$ consists of two vertices connected by two edges, so it is not a simplicial complex.
        However, when $\Gamma=\Aut(\calX)$ (which is isomorphic to the dihedral group of order $8$),
        and $\calX$ is regarded as a topological space,
        the map $\calX\to \Gamma\leftmod\calX$ is not a cover map of topological spaces,
        and the problem of giving $\Gamma\leftmod\calX$
        a decent simplicial structure is inherent.
    \end{example}

    \begin{example}\label{EX:simplicial-quotient-II}
        Let $G=\nPGL{F}{d}$ and $\calX=\calB_d(F)$ be as in Chapter~\ref{sec:ramanujan-complexes},
        and let
        $N=\im(\mathrm{SL}_d(F)\to \nPGL{F}{d})$.
        It is easy to check
        that $N\leftmod \calB_d(F)$ is a simplicial complex consisting of a single simplex of dimension $d-1$.
        However, the map $\calB_d(F)\to N\leftmod \calB_d(F)$ is not a cover map (e.g.\ by
        Proposition~\ref{PR:quotient-cover-sufficient-nec-conds} below).
    \end{example}

    \begin{prp}\label{PR:quotient-complex-sufficient-nec-conds}
        Let $\Gamma\leq G$. Then $\Gamma\leftmod \calX$ is a simplicial complex if and only if
        \begin{enumerate}
            \item[(C1)] $\{\Gamma u_1,\dots,\Gamma u_t\}=\{\Gamma v_1,\dots,\Gamma v_s\}$
            implies $\Gamma\{u_1,\dots,u_t\}=\Gamma\{v_1,\dots,v_s\}$ for all $\{u_1,\dots,u_t\},\{v_1,\dots,v_s\}\in \vrt{\calX}$.
        \end{enumerate}
        In this case, $\dim (\Gamma\leftmod \calX)=\dim \calX$ and the map $x\mapsto \Gamma x: \calX\to \Gamma\leftmod \calX$ is a morphism of
        simplicial
        complexes.
    \end{prp}

    \begin{proof}
        Assume (C1) holds. Define $Y$ to be the simplicial complex
        consisting of subsets $\{\Gamma v_1,\dots,\Gamma v_t\}\subseteq \Gamma\leftmod \vrt{X}$ such
        that $\{v_1,\dots,v_t\}\in X$. There is a morphism of partially ordered sets $\Phi:\Gamma\leftmod X\to Y$
        given by $\Phi(\Gamma\{v_1,\dots,v_t\})=\{\Gamma v_1,\dots,\Gamma v_t\}$. It is clear that $\Phi$ is onto, and
        $\Phi$ is injective by (C1). Thus, $\Phi$ is an isomorphism and $\Gamma\leftmod X$ is a simplicial complex.

        Suppose now that $\Gamma\leftmod X$ is a simplicial complex.
        We shall use the following notion of {height} in partially ordered sets:
        Let $(Y,\leq)$ be a partially ordered set and let $y\in Y$. The \emph{height} of $y$ in $Y$,
        denoted  $\mathrm{h}(y)=\mathrm{h}_Y(y)$, is the maximal $n\in\N\cup\{0,\infty\}$
        such that there exists a sequence $y_0<y_1<\dots<  y_{n}=y$ in $Y$.
        Automorphisms of $(Y,\leq)$ are easily seen to preserve height. Using this, it is easy to see that
        for all $x\in X$, we have $\mathrm{h}_X(x)=\mathrm{h}_{\Gamma\leftmod  X}(\Gamma x)$. Simplicial complexes $Y$
        satisfy the following special property, whose easy proof is left to the reader:
        \begin{enumerate}
            \item[($*$)] For all $y\in Y$, $y=\sup\{y'\in Y\suchthat y'\leq y,\,\mathrm{h}(y')=1\}$.
        \end{enumerate}

        Suppose now that $\{u_1,\dots,u_t\},\{v_1,\dots,v_s\}\in X$ are such that
        $\{\Gamma u_1,\dots,\Gamma u_t\}=\{\Gamma v_1,\dots,\Gamma v_s\}$.
        By the previous paragraph, $\mathrm{h}(\Gamma\{u_i\})=1$ for all $i$.
        In addition, it is easy to see that $\{\Gamma \{u_i\}\}_{i=1}^t$ consists of
        all elements $\Gamma x\in\Gamma\leftmod\calX$ with $\mathrm{h}(\Gamma x)=1$ and $\Gamma x\leq\Gamma\{u_1,\dots,u_t\}$.
        Therefore, by ($*$), $\Gamma\{u_1,\dots,u_t\}=\sup\{\Gamma\{u_1\},\dots,\Gamma\{u_t\}\}$
        and likewise
        $\Gamma\{v_1,\dots, v_s\}=\sup\{\Gamma \{v_1\},\dots,\Gamma\{v_s\}\}$.
        Since $\{\Gamma\{u_1\},\dots,\Gamma\{u_t\}\}=\{\Gamma \{v_1\},\dots,\Gamma\{v_s\}\}$, this means
        $\Gamma\{u_1,\dots,u_t\}=\Gamma\{v_1,\dots, v_s\}$, so we have shown (C1).

        The remaining assertions of the proposition follow easily from the fact that the map
        $\calX\to \Gamma\leftmod\calX$ preserves heights.
    \end{proof}

    \begin{cor}\label{CR:unique-image-in-cell}
        Assume $\Gamma\leftmod \calX$ is a simplicial complex. Then for all $\gamma,\gamma'\in \Gamma$
        and $v\in\vrt{\calX}$ with $\{\gamma v,\gamma' v\}\in X$, we have $\gamma v=\gamma'v$.
        In particular, if $\gamma x=x$ for $x\in X$, then $\gamma$ fixes the vertices of $x$.
    \end{cor}

    \begin{proof}
        We have $\{\Gamma \gamma v,\Gamma \gamma' v\}=\{\Gamma v\}$, so by (C1), $\Gamma\{\gamma v,\gamma' v\}=\Gamma\{v\}$.
        Thus, there is $\gamma''\in\Gamma$ such that $\gamma''\gamma v=\gamma''\gamma' v=v$, which implies
        $\gamma v=\gamma'v$. The last assertion follows by taking $\gamma'=1$.
    \end{proof}

    \begin{prp}\label{PR:quotient-cover-sufficient-nec-conds}
        Provided $\Gamma\leftmod \calX$ is a simplicial
        complex, the quotient map $\calX\to \Gamma\leftmod \calX$ is a cover map if and only if
        \begin{enumerate}
            \item[(C2)] $\Gamma\cap\Stab_G(v)=\{1_G\}$ for all $v\in\vrt{\calX}$.
        \end{enumerate}
        In this case, $\Gamma$ acts freely on $\dot{\calX}$ (i.e.\ $\Stab_\Gamma(x)=\{1_G\}$ for all $x\in \dot{\calX}$),
        $\Gamma$ is discrete in $G$, and
        the group of deck transformations of $\calX\to \Gamma\leftmod \calX$ is $\Gamma$.
    \end{prp}

    \begin{proof}
        Assume that  (C2) is satisfied. We need to check that for all $v\in \vrt{X}$,
        the map $x\mapsto\Gamma x$  induces a bijection  $\{x\in \calX\suchthat v\in x\}\to\{y\in\Gamma\leftmod \calX\suchthat \Gamma v\leq y\}$.
        The surjectivity is clear,
        so it is left to show that whenever $x,y\in \calX$ satisfy $v\in x,y$ and $\Gamma x=\Gamma y$, 
        we have $x=y$. Write $x=\gamma y$ with
        $\gamma\in \Gamma$. Then $\gamma v\in x$, hence $\{v,\gamma v\}\in X$. By Corollary~\ref{CR:unique-image-in-cell},
        $\gamma\in\Stab_G(v)$, so by (C2), $\gamma=1_G$ and $x=y$. Conversely, suppose $\calX\to \Gamma\leftmod \calX$
        is a cover map and let $v\in \vrt{X}$ and $\gamma\in\Gamma\cap\Stab_G(v)$.
        Let $u\in \vrt{X}$ be a neighbor of $v$, i.e.\ $\{u,v\}\in X$ and $u\neq v$.
        Then $\gamma u$ is a neighbor of $\gamma v=v$. Since $\{\Gamma u,\Gamma v\}=\{\Gamma \gamma u,\Gamma v\}$,
        condition (C1) of Proposition~\ref{PR:quotient-complex-sufficient-nec-conds}
        implies that $\Gamma\{u,v\}=\Gamma\{\gamma u,v\}$, and
        since $x\mapsto \Gamma x$ is a cover map, this means $\{u,v\}=\{\gamma u,v\}$, hence $\gamma u=u$ (because $u\neq v$).
        We have therefore showed that $\gamma$
        fixes all the neighbors of $v$. Proceeding by induction, we see
        that $\gamma$ fixes all vertices connected to $v$. Since $\calX$ is connected and $G$ acts faithfully, this
        means $\gamma=1_G$, as required.

        Suppose henceforth that $\calX\to \Gamma\leftmod\calX$ is a cover map.
        That $\Gamma$ acts freely on $\dot{X}$ follows from Corollary~\ref{CR:unique-image-in-cell} and (C2),
        and the discreteness of $\Gamma$ is immediate from (C2) (since $\Stab_G(v)$ is open in $G$).
        Let $h\in\Aut(X)$ be a deck transformation of $x\mapsto \Gamma x$. Then $\Gamma h x=\Gamma x$ for
        all $x\in \calX$.
        This means that for all $x\in \dot{X}$, there is $\gamma_x\in \Gamma$ such
        that $hx=\gamma_xx$, or equivalently $s_x:=\gamma_x^{-1} h\in\Stab_{\Aut(X)}(x)$.
        Therefore, for all $x\in X$, we can write $h=\gamma_xs_x$ with $\gamma_x\in\Gamma$ and $s_x\in\Stab_{\Aut(X)}(x)$,
        and since $\Gamma\cap\Stab_{\Aut(X)}(x)=\{1\}$ (because $\Gamma$ acts freely on $\dot{\calX}$),
        this decomposition is unique.
        We claim that $\gamma_x=\gamma_v$ for all  $v\in x$.
        Indeed,
		we have $\gamma_xs_x=\gamma_vs_v$ and hence
        $\{\gamma_x\gamma_v^{-1}v,v\}=\{s_x^{-1}s_vv,v\}=\{s_x^{-1}v,v\}\subseteq x\in\calX$, so
        $\gamma_x\gamma_v^{-1}v=v$ by Corollary~\ref{CR:unique-image-in-cell}, and by (C2), $\gamma_x=\gamma_v$.
        This means that $\gamma_x=\gamma_y$ whenever $x\cap y\neq \emptyset$. Since $\calX$ is connected,
        we have $h=\gamma s$ with $\gamma\in \Gamma$
        and $s$ stabilizing every cell in $\calX$, so $s=1$ and $h\in\Gamma$.
    \end{proof}

    The following corollary gives an elegant necessary and sufficient condition for having
    $\Gamma\leq_{\calX}G$. 

    \begin{cor}\label{CR:distance-condition}
        $\Gamma\leq_{\calX} G$ $\iff$ $\dist(v,\gamma v)>2$ for all $1\neq\gamma\in \Gamma$ and $v\in \vrt{\calX}$.
    \end{cor}

    \begin{proof}
        Suppose $\Gamma\leq_\calX G$ and assume $\dist(v,\gamma v)\leq 2$ for some $v\in \vrt{X}$ and $\gamma\in \Gamma$.
        Then there exists
        $u\in\vrt{X}$ such that $\{u,v\},\{u,\gamma v\}\in \calX$. By  (C1), there is $\gamma'\in\Gamma$
        such that $\{\gamma' u,\gamma' v\}=\{u,\gamma v\}$, and this implies $\{\gamma' v,\gamma v\},\{\gamma' u,u\}\in \calX$.
        By Corollary~\ref{CR:unique-image-in-cell}, $\gamma'u=u$ and $\gamma'v=\gamma v$, so by (C2)
        we get
        $\gamma'=\gamma=1$. 

        To see the converse, we
        verify conditions (C1) and (C2) above.
        Condition (C2) is straightforward, so we only  show (C1).
        Suppose $\{u_1,\dots,u_t\},\{v_1,\dots,v_s\}\in \calX$ satisfy $\{\Gamma u_1,\dots,\Gamma u_t\}=\{\Gamma v_1,\dots,\Gamma v_s\}$.
        We may assume that $t,s>0$ and
        there is $\gamma\in\Gamma$ such that $\gamma u_1=v_1$.
        Now, for all $i$, $\dist(\gamma u_i,v_1)=\dist(u_i,u_1)\leq 1$.
        On the other hand, there is $\gamma'\in\Gamma$ and $j$ such that $\gamma'  u_i=v_j$.
        Thus, $\dist(\gamma' u_i,v_1)\leq 1$. It follows that
        $\dist(\gamma^{-1}\gamma'u_i,u_i)=\dist(\gamma' u_i,\gamma u_i)\leq \dist(\gamma' u_i,v_1)+\dist(v_1,\gamma u_i)=2$,
        so by assumption, $\gamma^{-1}\gamma'=1$, meaning that $\gamma u_i=\gamma'u_i\in\{v_1,\dots, v_t\}$.
        Therefore, $\gamma\{u_1,\dots,u_t\}=\{\gamma u_1,\dots, \gamma u_t\}\subseteq \{v_1,\dots,v_s\}$
        and likewise, we have $\gamma^{-1}\{v_1,\dots, v_s\}\subseteq\{u_1,\dots,u_t\}$. Thus, $\Gamma\{u_1,\dots,u_t\}=\Gamma\{v_1,\dots,v_s\}$.
    \end{proof}

    \begin{cor}\label{CR:Gamma-is-discrete}
        If $\Gamma\leq_\calX G$, then $\Gamma'\leq_\calX G$ for all $\Gamma'\leq\Gamma$,
        and $g^{-1}\Gamma g\leq_{\calX}G$ for all $g\in G$.
    \end{cor}

    \begin{proof}
        Use Corollary~\ref{CR:distance-condition}.
    \end{proof}

    \begin{cor}\label{CR:finite-subsect-away}
        Let $\Gamma$ be a discrete subgroup of $G$ such that $\Gamma\leftmod\calX$ is finite.
        Then there is a finite subset $S\subseteq \Gamma-\{1\}$ such that any normal subgroup $\Gamma'\normalin\Gamma$
        with $\Gamma'\cap S=\emptyset$ satisfies $\Gamma'\leq_{\calX}G$.
    \end{cor}

    \begin{proof}
        Let $v_1,\dots,v_n$ be representatives for the $\Gamma$-orbits in $\vrt{\calX}$.
        For all $1\leq i\leq n$, let $S_i=\{\gamma\in \Gamma\suchthat \dist(v_i,\gamma v_i)\leq 2\}$.
        We claim that $S_i$ is finite. Indeed, $\Ball_\calX(v_i,2)$ is finite,
        and hence $K_i:=\{g\in G\suchthat \dist(v_i,gv_i)\leq 2\}$ is a finite union of right cosets
        of $\Stab_G(v_i)$, so it is compact. Therefore, $S_i:=\Gamma\cap K_i$ is discrete
        and compact, hence finite.
        Take $S=\bigcup_{i=1}^nS_i-\{1\}$ and suppose $\Gamma'\normalin\Gamma$ and $\Gamma'\cap S=\emptyset$.
        We use Corollary~\ref{CR:distance-condition} to show that $\Gamma'\leq_{\calX} G$.
        Let $v\in \vrt{\calX}$ and $\gamma'\in\Gamma'$ be such that $\dist(v,\gamma' v)\leq 2$.
        There is $\gamma\in \Gamma$ and $1\leq i\leq n$ such that $v=\gamma v_i$. Thus,
        $\dist(v_i,\gamma^{-1}\gamma'\gamma v_i)=\dist(\gamma v_i,\gamma'\gamma v_i)=\dist(v,\gamma' v)\leq 2$.
        Since $\Gamma'\normalin\Gamma$, this means $\gamma^{-1}\gamma'\gamma \in \Gamma'\cap S$,
        which is impossible unless $\gamma^{-1}\gamma'\gamma=1$, so $\gamma'=1$.
    \end{proof}

    \begin{prp}\label{PR:unimodularity-of-G}
        Let $\Gamma\leq_{\calX}G$. Then $\Gamma\leftmod\calX$ is finite
        if and only if  $\calX$ is almost $G$-transitive and $\Gamma$ is cocompact in $G$ (i.e.\ $\Gamma\leftmod G$ is compact).
        In this case, $G$ is unimodular.
    \end{prp}

    \begin{proof}
    	Assume $\Gamma\leftmod \calX$ is finite. Then $\calX$ is clearly almost $G$-transitive.
    	Let $x_1,\dots,x_n$ be representatives for the $G$-orbits in $\dot{\calX}$
    	and let $K_i=\Stab_G(x_i)$. Then each $K_i$ is compact and open in $G$, and $\Gamma\leftmod\dot{\calX}$ can be identified
    	with $\bigsqcup_{i=1}^n\Gamma\leftmod G/K_i$ via sending $\Gamma gK_i$ to $\Gamma gx_i$.
    	Since $\Gamma\leftmod\calX$ is finite, the set $\Gamma\leftmod G/K_1$ is finite, so $\Gamma\leftmod G$ is the (topological)
    	disjoint union of finitely many epimorphic images of $K_1$, which is compact. Thus, $\Gamma\leftmod G$ is compact.
    	
    	For the converse, identify $\Gamma\leftmod\dot{\calX}$ with $\bigsqcup_{i=1}^n\Gamma\leftmod G/K_i$ as above.
    	Then $\Gamma\leftmod G/K_i$ is a compact  discrete topological space, hence finite.
    	
    	Finally, any locally compact group admitting a discrete cocompact subgroup is unimodular
    	by \cite[Lm.~1]{Moore84} (for instance), and $\Gamma$ is discrete by
    	Proposition~\ref{PR:quotient-cover-sufficient-nec-conds}.
    \end{proof}

    \begin{remark}\label{RM:finite-quotients-vs-G-transitive}
        A $G$-complex $\calX$ may not have finite $G$-quotients even
        when it is almost transitive.
        For example, take $\calX=\calB_d(F)$ and $G=\nPGL{F}{d}$  as
        in Chapter~\ref{sec:ramanujan-complexes}, and let $B$ be the image
        of the group of invertible $d\times d$ upper-triangular matrices in $G$.
        The Iwasawa decomposition $G=BK$ with $K=\nPGL{\calO}{d}$ implies that $B$ acts transitively on $\calB_d(F)^{(0)}=G/K$
        and hence  almost transitively on $\calB_d(F)$. However, $B$ is not unimodular,
        so $\calB_d(F)$ cannot have finite $B$-quotients by Proposition~\ref{PR:unimodularity-of-G}.
    \end{remark}

    \begin{remark}
    	It is possible that $\Gamma\leftmod \calX$ can be understood as
    	a   simplicial complex
    	with multiple faces even when (C1) or (C2) do not hold, e.g.\ consider Example~\ref{EX:simplicial-quotient}(ii).
    	More generally, $\Gamma\leftmod\calX$ should make sense as a simplicial complex with multiple cells
    	when $\Gamma\cap\Stab_{G}(x)=1$ for all $x\in\dot{\calX}$.
    	Indeed, this guarantees that no cell in $\calX$ is glued to itself, and it can
    	be shown that in this case $\calX_{\mathrm{top}}\to\Gamma\leftmod\calX_{\mathrm{top}}$ is a cover map of topological spaces,
    	where $\calX_{\mathrm{top}}$ denotes the topological realization of $\calX$.
    	The condition  $\Gamma\cap\Stab_{G}(x)=1$ holds in particular when $\Gamma$ is a torsion-free
    	discrete subgroup, since for any compact open $K\leq G$, the group
    	$\Gamma\cap K$  is compact and discrete, hence finite. The torsion-free condition was
    	used in \cite{LubSamVi05}, for instance.
    	
    	Of course, this requires a  notion of complexes with multiple faces which is ample enough.
    	However, we do not know of such a notion.
    	For example,
    	regular cell complexes (in the sense \cite[Df.~A.22]{Buildings08AbramBrown})
    	of are not general enough since they
    	do not allow loops (consider the quotient
    	of a cyclic graph on $n$ vertices by the cyclic group of order $n$), and CW-complexes are too wild as there are very mild
    	assumptions on the gluing maps.
    	It seems likely that a decent definition exists, but this issue is out of the scope of this work.
		We also comment that with easy modifications,
		the theory in the following chapters extends to  other types of cell complexes which are not simplicial.
    \end{remark}

	Suppose that $\Gamma\leftmod\calX$ is a simplicial complex. We finish this chapter by showing that
	one can canonically identify $(\Gamma\leftmod \calX)^{(i)}$ with $\Gamma\leftmod (\calX^{(i)})$,
	$\Gamma\leftmod(\ori{\calX^{(i)}})$ with $\ori{(\Gamma\leftmod \calX^{(i)})}$, and
	$\Gamma\leftmod(\ordr{\calX^{(i)}})$ with $\ordr{(\Gamma\leftmod \calX^{(i)})}$, where $\ordr{X^{(i)}}$ denotes
	the ordered $i$-dimensional cells (see \ref{subsec:orientation}). The parenthesis in the previous expressions will be dropped henceforth.
	
	\begin{prp}\label{PR:orientation-of-quotient}
        Let $\Gamma\leq_\calX G$. The following maps are isomorphisms.
        \begin{enumerate}
        	\item[(i)] 
        	$\Gamma\{v_0,\dots,v_i\}\mapsto \{\Gamma v_0,\dots,\Gamma v_i\}:\Gamma\leftmod (\calX^{(i)})\to (\Gamma\leftmod\calX)^{(i)}$
        	\item[(ii)] 
        	$\Gamma[v_0\dots v_i]\mapsto [(\Gamma v_0)\dots(\Gamma v_i)]:\Gamma\leftmod(\ori{\calX^{(i)}})\to \ori{(\Gamma\leftmod \calX^{(i)})}$
        	\item[(iii)] 
        	$\Gamma(v_0<\dots< v_i)\mapsto (\Gamma v_0<\dots <\Gamma v_i):\Gamma\leftmod(\ordr{\calX^{(i)}})\to \ordr{(\Gamma\leftmod \calX^{(i)})}$
        \end{enumerate}
    \end{prp}

    \begin{proof}
    	(i) This is immediate from condition (C1) and the definition of $\Gamma\leftmod \calX$.

        (ii) It is clear that the map is well-defined and onto, so we only need to show injectivity.
        If $[(\Gamma v_0)\dots(\Gamma v_i)]=[(\Gamma u_0)\dots(\Gamma u_i)]$,
        then $\{\Gamma v_0,\dots,\Gamma v_i\}=\{\Gamma u_0,\dots,\Gamma u_i\}$,
        so by Proposition~\ref{PR:quotient-complex-sufficient-nec-conds},
        $\Gamma\{v_0,\dots,v_i\}=\Gamma\{u_0,\dots,u_i\}$. This means
        that there is $\gamma\in \Gamma$ and a permutation $\sigma\in S_{\{0,\dots,i\}}$
        such that $\gamma v_r=u_{\sigma r}$ for all $0\leq r\leq i$. In addition,
        since $[(\Gamma v_0)\dots(\Gamma v_i)]=[(\Gamma u_0)\dots(\Gamma u_i)]$,
        there is an \emph{even} permutation $\tau\in S_{\{0,\dots,i\}}$ and elements $\gamma_0,\dots,\gamma_i\in \Gamma$
        such that $\gamma_r v_r=u_{\tau r}$ for all $0\leq r\leq i$.
        Corollary~\ref{CR:unique-image-in-cell} now implies that $\gamma v_r=\gamma_r v_r$ for all $r$,
        so $\sigma=\tau$ and $\gamma[v_0\dots v_i]=[u_0\dots u_i]$, as required.

        (iii) This is similar to (ii).
    \end{proof}

\section{Spectrum in Simplicial Complexes}
\label{sec:adj-operators}

    When studying spectral  properties of a simplicial complex, one needs
    to decide what operators associated with the complex are inspected.
    For example, in the case of graphs, one  usually takes the vertex adjacency operator or the $0$-dimensional
    Laplacian (and not
    the  adjacency operator of the edges).
    However, in the case of  simplicial complexes, there are many relevant operators.

    In this chapter, we attempt to give a uniform approach in studying these various operators.
    This will lead to the  definition of high dimensional spectrum, which will ultimately generalize
    the spectrum of graphs and quotients of $\calB_d(F)$ (see Chapter~\ref{sec:ramanujan-complexes}).
	Our definition can also be adapted to hypergraphs, polysimplicial complexes,  regular cell complexes, and abstract partially
	ordered sets, although
	we do not discuss such structures in this work.

\subsection{Examples}
\label{subsec:examples-of-operators}

    Before we present our approach toward operators associated with simplicial complexes,
    let us first exhibit some examples that will arise as special cases. We will use
    the spaces $\Omega_i^+$, $\Omega_i^-$, $\Omega_i^{\pm}$ introduced in \ref{subsec:orientation}.

    \begin{example}\label{EX:adjacency-operator-tree}
        Let $X$ be a $k$-regular graph. The \emph{vertex adjacency
        operator} of $X$ is the operator $a_0:\Omega^+_0(X)\to \Omega^+_0(X)$
        given by
        \[(a_0\vphi)u=\sum_{v}\vphi u\qquad\forall\, \vphi\in\Omega^+_0(X),\, u\in X^{(0)}\,,\]
        where the sum is taken over all $v\in X^{(0)}$ connected by an edge to $u$.
        Likewise, the  \emph{edge adjacency operator} of $X$ is the operator
        $a_1:\Omega^+_1(X)\to \Omega^+_1(X)$ given by.
        \[(a_1\psi)x=\sum_{y}\psi y\qquad\forall\, \psi\in\Omega^+_1(X),\, x\in X^{(1)}\,,\]
        where the sum is taken over all edges $y\in X^{(1)}$ sharing exactly one vertex with $x$.
        In fact, the vertex and edge adjacency operators are defined
        for all locally finite graphs.

        The spectrum of $a_0$ affects many combinatorial properties of the graph $X$, particularly
        its \emph{Cheeger constant} and \emph{chromatic number}; see the survey
        \cite{Lubotzky14} for further information. The spectrum of $a_1$ is known to be almost equivalent
        with the spectrum of $a_0$; we shall recover this later in Example~\ref{EX:spectrum-dependency}.
    \end{example}

    The next example is a naive generalization of the vertex and edge adjacency operators to higher dimensions.

    \begin{example}\label{EX:higher-dimensional-adj-ops}
        Let $X$ be any simplicial complex and assume $i,j$ satisfy $0\leq i<j\leq 2i+1$.
        Define $a_{i;j}:\Omega^+_i(X)\to \Omega^+_i(X)$ by
        \[(a_{i;j}\psi)x=\sum_{\substack{y\in X^{(i)}\\ x \cup y\in X^{(j)}}}\vphi y\qquad\forall\, \vphi\in\Omega^+_i(X),\, x\in X^{(i)}\ .\]
        That is, the evaluation of $a_{i;j}\vphi$  at an $i$-cell $x$ sums the values of $\vphi$ on $i$-cells $y$ whose union with $x$ is a $j$-cell.
        In the notation of Example~\ref{EX:adjacency-operator-tree}, we have $a_{0}=a_{0;1}$ and $a_{1}=a_{1;2}$. However, on the level
        of edges, we also have $a_{1;3}$. In general, $a_{i;i+1},\dots,a_{i;2i+1}$ do not commute.
    \end{example}

    \begin{example}\label{EX:laplacians-revised}
        Let $X$ be any simplicial complex and let $i\geq 0$. The upper, lower and full
        Laplacians $\Delta_i^+$, $\Delta^-_i$, $\Delta_i$ (see~\ref{subsec:orientation}) take $\Omega_i^-(X)$ into itself.
        Recall from the introduction that $\Spec(\Delta_i)$ affects
        combinatorial properties of $X$ (see \cite{ParRosTes13}, \cite{GolPar14}, \cite{Golubev13} and related works).
        We also note that
        when $X$ is a $k$-regular graph,
        $\Delta_0=k-a_0$, so $\Delta_0$ and $a_0$ are spectrally equivalent.
    \end{example}

    If a simplicial complex $X$ admits further structure, e.g.\ a coloring of the vertices or the edges,
    then one may also consider operators taking this data into a account.
    The complex $\calX=\calB_d(F)$ of Chapter~\ref{sec:ramanujan-complexes} gives
    rise to such examples.

    \begin{example}\label{EX:adjacency-building}
        Let $\calX=\calB_d(F)$ and $G=\nPGL{F}{d}$ be as in Chapter~\ref{sec:ramanujan-complexes}.
        Recall that with every $G$-quotient $X:=\Gamma\leftmod\calB_d(F)$, we have
        associated $d-1$ operators $a_1,\dots,a_{d-1}:\LL{X^{(0)}}\to \LL{X^{(0)}}$
        which correspond to the $(d-1)$-coloring $C_1$ of the directed edges.
        The spectrum of these operators  affects various combinatorial properties of the the quotient $X$; see~\cite{EvraGolLub14},
        \cite{GolPar14}.
        Notice that $\Omega_0^+(X)$ is dense in $\LL{X^{(0)}}$
        (equality holds when $X$ is finite), so  the common
        specturm of $(a_1,\dots,a_{d-1})$ on $\LL{X^{(0)}}$ is the same
        as their spectrum on $\Omega_0^+(X)$.

		When $\Gamma$ is cotained in $\ker(c:\nPGL{F}{d}\to \Z/d\Z)$ (see Chapter~\ref{sec:ramanujan-complexes}),
		the $d$-coloring $C_0:\calB_d(F)^{(0)}\to \Z/d\Z$
		descends to $X$, giving rise to  operators on $\LL{X^{(0)}}$ or $\Omega_0^+(X)$ which take $C_0$ into account,
        e.g.\ restricting a function on $X^{(0)}$
		to vertices of a particular color. Such operators were used in
        \cite{EvraGolLub14}.
    \end{example}

\rem{
	The following example requires certain knowledge of buildings; we refer the reader to
	\cite{Buildings08AbramBrown}  for all relevant definitions.

	\begin{example}
		Let $\calB$ be a $d$-dimensional locally finite \emph{building} equipped with its maximal set of apartments (see \cite{Buildings08AbramBrown} for
		the definitions). Suppose a group $G$ acts faithfully on $\calB$
		such that the action is \emph{type preserving} and \emph{strongly transitive}. Suppose further
		that $G$ is closed in $\Aut(\calB)$ with respect to point-wise convergence.
		Let $W$ be the Weyl group associated with $\calB$.
		Then there is a distance function on the chambers $d:\calB^{(d)}\times \calB^{(d)}\to W$
		(see \cite[Chp.~5]{Buildings08AbramBrown}), which is preserved by $G$.
	\end{example}
}

\subsection{Associated Operators}
\label{subsec:adjacency-operators}

    We now introduce a notion
    of  operators associated with a simplicial complex that includes the examples of \ref{subsec:examples-of-operators}.
    Until we give the definition below, we shall
    informally address such operators as \emph{associated operators}.

\medskip

    Let $\catSimp$ denote the class of all finite-dimensional locally finite connected simplicial complexes.
    We first
    observe that an associated operator $a$ is in fact a family of operators
    $\{a_X\}_{X\in\catC}$ ranging on some class of simplicial complexes $\catC\subseteq\catSimp$.
    This class can be all simplicial complexes (Examples~\ref{EX:higher-dimensional-adj-ops}
    and~\ref{EX:laplacians-revised}), or  a  restricted family of quotients
    of some fixed universal cover (Examples~\ref{EX:adjacency-operator-tree}
    and~\ref{EX:adjacency-building}).

    The domain and range of an operator $a_X$ in the family $\{a_X\}_{X\in\catC}$ is
    a pre-Hilbert space  varying with  $X$;
    denote it by $FX$. Then $F$ is an assignment from $\catC$ to $\catPHil$, the class of pre-Hilbert
    spaces. In the examples of \ref{subsec:examples-of-operators}, the assignment $F$ was $\Omega^+_i$ or $\Omega_i^-$.
    To conclude, the datum of an associated operator consists of a family of linear operators
    $\{a_X:FX\to FX\}_{X\in\catC}$ where $\catC$ is a subclass of $\catSimp$ and $F:\catSimp\to \catPHil$
    is an assignment. Of course, this is still too general, and some conditions should be imposed.

\medskip

    We proceed by observing that $\catSimp$ and $\catPHil$ can be made into categories; the morphisms of
    $\catSimp$ are morphisms of simplicial complexes (\ref{subsec:simp-complex}) and the morphisms
    of $\catPHil$ are $\C$-linear maps (not-necessarily continuous).
    Actually,
    it is usually more convenient to consider the subcategory of $\catSimp$ whose objects are
    those of $\catSimp$ and whose morphisms are cover maps. We denote it by
    \[
    \catCov\ .
    \]
    Now, the assignments $X\mapsto \Omega^+_i(X)$, $X\mapsto \Omega_i^-(X)$ and $X\mapsto \Omega_i^\pm(X)$ (\ref{subsec:orientation})
    can be extended to
    functors $\catSimp\to \catPHil$: If $f:X\to Y$ is a simplicial map, then we define $f_*:\Omega^+_i(X)\to \Omega^+_i(Y)$
    by
    \[
    (f_*\vphi)\sfy=\sum_{\sfx\in f^{-1}\{\sfy\}}\vphi (\sfx)\qquad\forall\, \vphi\in \Omega^+_i(X),\, \sfy\in \ori{Y^{(i)}}\ .
    \]
    This makes $X\mapsto \Omega^+_i(X)$ into a \emph{covariant}  functor.
    Abusing the notation, we also define $f_*:\Omega_i^-(X)\to \Omega_i^-(Y)$
    and $f_*:\Omega_i^\pm(X)\to \Omega_i^\pm(X)$ in the same way.
    This agrees with the notation section.
    Notice that $f_*$ is continuous
    if and only if the size of the fibers of $f:X^{(i)}\to Y^{(i)}$ is uniformly bounded, so in order
    to include infinite simplicial complexes, we must allow non-continuous
    morphisms in $\catPHil$.

    View $\Omega_i^+$ and $\Omega_i^-$ as functors from $\catCov$ to $\catPHil$.
    It straightforward to check that
    $a_{i;j}=\{a_{i;j,X}\}_{X\in\catSimp}$ of Example~\ref{EX:higher-dimensional-adj-ops}
    is a {natural transformation} from $\Omega^+_i$ to itself,
    and $\Delta_i=\{\Delta_{i,X}\}_{X\in\catSimp}$
    of Example~\ref{EX:laplacians-revised}
    is a natural transformation from $\Omega^-_i$ to itself. That is, for  any cover map $f:X\to Y$
    in $\catCov$, the following diagrams commute:
    \[
    \xymatrix{
    \Omega^+_i(X)\ar[r]^{a_{i;j,X}} \ar[d]^{f_*} & \Omega^+_i(X) \ar[d]^{f_*} \\
    \Omega^+_i(Y)\ar[r]^{a_{i;j,Y}} & \Omega^+_i(Y)
    }
    \qquad
    \xymatrix{
    \Omega_i^-(X)\ar[r]^{\Delta_{i,X}} \ar[d]^{f_*} & \Omega^-_i(X) \ar[d]^{f_*} \\
    \Omega_i^-(Y)\ar[r]^{\Delta_{i,Y}} & \Omega^-_i(Y)
    }
    \]
    This suggests that a general associated operator $\{a_X:FX\to FX\}_{X\in\catC}$ should be a natural
    transformation, in which case $\catC$ has a to be a category and $F$ has to be a functor.
    For reasons to become clear later, we also require $a$ to have a \emph{dual}, that is,
    a natural transformation $\{a^*_X\}_{X\in\catC}:F\to F$ such that
    $a^*_X$ is a dual of  $a_X$.
    (Notice that $a_X$ may not have a dual since $FX$ is not a Hilbert space. Also,
    even when $a_X^*$ exists for all $X$,  it can happen  that $\{a_X^*\}_{X\in\catC}$ is not a natural transformation
    from $F$ to itself.)

\medskip

    We finally conclude with the definition:

    \begin{dfn}
        Let $\catC$ be a subcategory of $\catSimp$ and $F:\catC\to \catPHil$ a covariant
        functor. An \emph{associated operator} of $(\catC,F)$, or just a $(\catC,F)$-operator, is a natural transformation $a:F\to F$
        that admits a dual.
    \end{dfn}
    
    Interesting examples usually arise when $\catC$ is a subcategory of $\catCov$.
    Of particular interest is the \emph{category of quotients of a $G$-complex $\calX$},
    which we now define.
    It will be our main example, and will play a major role in the sequel.


    \begin{dfn}\label{DF:G-complex-category}
        Let $\calX$ be a $G$-complex (\ref{subsec:G-complex}).  We define the subcategory
        \[
        \catC=\catC(G,\calX)\subseteq\catCov
        \]
        as follows: The objects of $\catC$ are $\{\Gamma\leftmod\calX\where \Gamma\leq_{\calX} G\}$
        (see~\ref{subsec:quotienst-of-simp-comps}), where $1\leftmod\calX$ is identified with $\calX$.
        The morphisms of $\catC$ are given as follows:
        \begin{itemize}
            \item For all $\Gamma\leq_{\calX}G$, set $\Hom_{\catC}(\calX,\Gamma\leftmod\calX)=\{p_\Gamma\circ g\where
            g\in G\}$, where $p_\Gamma$ is the quotient map $x\mapsto \Gamma x:\calX\to \Gamma\leftmod\calX$.
            \item For all $1\neq\Gamma'\leq\Gamma\leq_{\calX} G$, set $\Hom_{\catC}(\Gamma'\leftmod\calX,\Gamma\leftmod\calX)=\{p_{\Gamma',\Gamma}\}$
            where
            $p_{\Gamma',\Gamma}$ is the quotient map $\Gamma'x\mapsto \Gamma x:\Gamma'\leftmod\calX\to \Gamma\leftmod\calX$.
            \item All other $\Hom$-sets are empty.
        \end{itemize}
        In particular, $\End_{\catC}(\calX)=G$.
    \end{dfn}

    \begin{example}
        In  Example~\ref{EX:adjacency-operator-tree}, take $\catC=\catC(G,\calX)$ where
        $\calX$ is a $k$-regular tree and  $G=\Aut(\calX)$.
        The objects of $\catC$ are $k$-regular graphs, and the operator  $a_0=\{a_{0,X}\}_{X\in\catC}$
        (resp.\ $a_1=\{a_{1,X}\}_{X\in\catC}$) is associated with $(\catC,\Omega^+_0)$
        (resp.\ $(\catC,\Omega^+_1)$).

        In Example~\ref{EX:adjacency-building},
        take 
        $\catC=\catC(\nPGL{F}{d},\calB_d(F))$ (see Chapter~\ref{sec:ramanujan-complexes}).
        The operators $a_1,\dots,a_{d-1}$ are associated with $(\catC,\Omega^+_0)$.
    \end{example}

	We finish with introducing another functor from $\catSimp$ to $\catPHil$ that may be
    considered in applications: For $X\in \catSimp$, let $\Flag(X)$ denote the
	set of \emph{maximal flags} in $X$, namely, the set of {maximal} chains of cells in $X$. If
	$X$ is pure of dimension $d$, then $\Flag(X)$ is canonically isomorphic to
	the set of pairs $(x,\leq)$ such that $x\in\calX^{(d)}$ and $\leq$ is a full ordering of the vertices of $x$.
    Let
	\[\llFlag(X)=\llf(\Flag(X))\ .\]
	We make $\llFlag$ into  a (covariant) functor from $\catSimp$ to $\catPHil$ in the same
	way we made $\Omega^+_i$, $\Omega_i^-$, $\Omega_i^\pm$ into functors.

\subsection{Spectrum}
\label{subsec:spectrum-of-simp}

    Let $\catC$ be a subcategory of $\catSimp$ and let $F:\catC\to\catPHil$ be a functor.
    Denote by
    $\Alg{\catC,F}{}$ be the collection of all
    $(\catC,F)$-operators.
    We make $\Alg{\catC,F}{}$ into a unital  ``$*$-algebra'' (this is a priori not a set) by defining
    \begin{align*}
    &\{a_X\}+\{a'_X\}=\{a_X+a'_X\},\\
    &\{a_X\}\cdot\{a'_X\}=\{a_X\circ a'_X\},\\
    &\{a_X\}^*=\{a_X^*\},\\
    &\alpha\{a_X\}=\{\alpha a_X\}
    \end{align*}
    for all $\{a_X\},\{a'_X\}\in\Alg{\catC,F}{}$ and $\alpha\in\C$.
    The collection $\Alg{\catC,F}{}$ is a set when $\catC$ is skeletally small, which we tacitly
    assume throughout.
    An  idempotented $*$-subalgebra of $\Alg{\catC,F}{}$ is called
    an  \emph{algebra of $(\catC,F)$-operators}.

\medskip

    Let $A$ be an algebra of $(\catC,F)$-operators.
    We can regard $FX$ as a left $A$-module by setting $a\cdot v=a_Xv$
    for all $v\in FX$ and $a=\{a_X\}_{X\in\catC}\in A$. We clearly
    have
    \[
    \Trings{au,v}=\Trings{u,a^*v}\qquad\forall u,v\in FX,\, a\in\Alg{\catC,F}{}\ .
    \]
    However, a priori, $a|_{FX}$ is not continuous.
    When this holds for all $a\in A$,
    we  say that $A$ acts \emph{continuously} on $FX$. (For example,
    this is the case when $FX$ is finite dimensional.)

    To simplify the discussion, assume henceforth  that
    \begin{enumerate}
    	\item[(1)] every $(\catC,F)$-operator acts continuously on $FX$ for all $X\in\catC$, and
    	\item[(2)] $FX$ is finite-dimensional when $X$ is finite.
    \end{enumerate}
    In this case, for all $X\in\catC$, the smooth left $A$-module
    $AFX=A\cdot FX$ can be regarded as a pre-unitary representation of $A$ (\ref{subsec:pre-unitary-reps}), and this representation
    is admissible (\ref{subsec:admissible-representations}) when $X$ is finite.
    (If the identity transformation $\id:F\to F$ is in $A$, then $AFX=FX$.)
    We may therefore apply the spectral theory developed in Chapter~\ref{sec:involutary-algebras} to
    $AFX$ and $A$. In particular, we define the \emph{$A$-spectrum} of
    $X$ to be
    \[
    \Spec_{A}(X):=\Spec_{A}(A FX)\,,
    \]
    and when $X$ is finite, we also define the \emph{$A$-spectrum multiset}
    \[
    \mSpec_A(X):=\mSpec_A(A FX)\,.
    \]
    When $A=\Alg{\catC,F}{}$, we will also address these (multi-)sets as the \emph{$(\catC,F)$-spectrum},
    or just \emph{$F$-spectrum} when $\catC$ is clear from the context. We then write
    $\Spec_{\catC,F}$ or $\Spec_F$ instead of $\Spec_{\Alg{\catC,F}{}}$.
    Recall from~\ref{subsec:unitary-dual}, \ref{subsec:properties-of-spectrum} and~\ref{subsec:admissible-representations}
    that $\Spec_{A}(X)$ is a closed subset of the unitary dual of $A$, denoted
    $\udual{A}$, and, when defined, $\mSpec_{A}(X)$  a multiset of elements in $\udual{A}$.

    We  stress that the spectrum of \emph{any}  $a\in A$
    on $AFX$ can be recovered from $\Spec_{A}(X)$ as follows from Corollary~\ref{CR:subalgebra-spectrum-II}.
    We further note that if $B$ is an algebra of $(\catC,F)$-operators contained in $A$,
    then
    $\Spec_{B}(X)$ can be recovered from $\Spec_{A}(X)$ (Theorem~\ref{TH:subalgebra-spectrum-I}).
    In particular, we can recover $\Spec_{A}(X)$ from $\Spec_{\catC,F}(X)$. Similar claims hold for the spectrum
    multiset when $X$ is finite.

	\medskip
	
	When condition (1) (resp.\ (2)) above does not hold, one case still define $\Spec_A(X)$ (resp.\
	$\mSpec_A(X)$) when $A$ acts continuously on $FX$ (resp.\ $AFX$ is an admissible $A$-module).

    \begin{example}\label{EX:k-regular-graphs-basic-spectrum}
        Let $\calX$ be a $k$-regular tree, let $G=\Aut(\calX)$, $\catC=\catC(\calX,G)$
        (Definition~\ref{DF:G-complex-category}),
        and write $A_i=\Alg{\catC,\Omega^+_i}{}$ for $i=0,1$.
        We shall see in Example~\ref{EX:tree-spectrum} below that
        \[
        A_0=\C[a_0]\qquad\text{and}\qquad A_1=\C[a_1]
        \]
        where $a_0$ and $a_1$ are the vertex and edge adjacency operators of Example~\ref{EX:adjacency-operator-tree}.
        By Proposition~\ref{PR:unitary-dual-topology}, we can identify $\udual{A}_0$
        with a subset of $\C$ (this set is $\R$, in fact) such
        that for all $V\in\Rep[u]{A_0}$, the set $\Spec_{A_0}(V)$
        corresponds to $\Spec(a_0|_V)$. Therefore, for all $X\in\catC$, the datum of $\Spec_{A_0}(X)$
        is equivalent to the spectrum of the vertex adjacency operator of $X$, so
        the $A_0$-spectrum is essentially the same as the usual spectrum of $k$-regular graphs.
        Likewise, $\Spec_{A_1}(X)$ is equivalent to the spectrum of the edge adjacency
        operator of $X$.
    \end{example}

    \begin{example}\label{EX:zero-dim-spec-of-Bd}
        Let $\calX=\calB_d(F)$ and $G=\nPGL{F}{d}$ be as in Chapter~\ref{sec:ramanujan-complexes},
        and let $\catC=\catC(\calB_d(F),G)$. We shall see below (Example~\ref{EX:computation-of-Hecke-algebra-for-Bd}) that
        $A_0:=\Alg{\catC,\Omega^+_0}{}$ is in fact the free commutative unital algebra generated
        by the natural transformations $a_1,\dots,a_{d-1}$ (cf.\
        Chapter~\ref{sec:ramanujan-complexes}, Example~\ref{EX:adjacency-building}). By Proposition~\ref{PR:unitary-dual-topology},
        we see that  $\udual{A}_0$ can be embedded in $\C^{d-1}$
        in such a way that $\Spec_{A_0}(\Gamma\leftmod \calB_d(F))$ corresponds to
        the common spectrum of $(a_1,\dots,a_{d-1})$ on $\Omega^+_0(\Gamma\leftmod\calB_d(F))$ for all
        $\Gamma\leq_{\calB_d(F)}G$.
        Thus, $\Spec_{A_0}(\Gamma\leftmod\calB_d(F))$ is
        equivalent to $\Spec_0(\Gamma\leftmod\calB_d(F))$ in the sense of Chapter~\ref{sec:ramanujan-complexes}.
    \end{example}

    \begin{remark}\label{RM:summand-transition}
        Let $F,F':\catC\to \catPHil$ be functors. Then $\Alg{\catC,F}{}$
        embeds as a $*$-subalgebra of $\Alg{\catC,F\oplus F'}{}$ via $a\mapsto a\oplus 0:=\{a_X\oplus 0_{F'X}\}_{X\in\catC}$.
        As a result, we can view any algebra $A$ of $(\catC,F)$-operators as an algebra of $(\catC,F\oplus F')$-operators.
        This transition  does not affect the definition of $\Spec_A(X)$ and $\mSpec_A(X)$ (when they exist),
        since $A(FX\oplus F'X)=AFX\oplus 0$.
    \end{remark}

    Let $\catC\subseteq \catSimp$ and $i\geq 0$. We denote the spectrum
    taken with respect to $\Alg{\catC,\Omega_i^+}{}$, $\Alg{\catC,\Omega_i^-}{}$, $\Alg{\catC,\Omega_i^{\pm}}{}$ and
    $\Alg{\catC,\llFlag}{}$ by
    \[
    \Spec_{i,\catC},\quad\Spec_{-i,\catC},\quad\Spec_{\pm i,\catC},\quad\Spec_{\Flag,\catC},
    \]
    respectively.
    We will drop $\catC$ when it is obvious from the context.
    We call $\Spec_i$ (resp.\ $\Spec_{-i}$, $\Spec_{\pm i}$) the \emph{non-oriented} (resp.\ \emph{oriented}, \emph{full})
    \emph{$i$-dimensional spectrum}, and we call $\Spec_{\Flag}$ the \emph{flag spectrum}. Notice that $\Spec_{+0}$, $\Spec_{-0}$ and $\Spec_{\pm 0}$
    are in fact identical (since $\Omega_0^+=\Omega_0^-=\Omega_0^{\pm}$),
    so we simply write $\Spec_0$.

\rem{
    When $\catC=\catC(G,\calX)$ (Definition~\ref{DF:G-complex-category}), we will replace
    the subscript $\catC$ with ``$G,\calX$'' writing $\Spec_{i,G,\calX}$, $\Spec_{\pm i,G,\calX}$, etcetera.
    In this case, we also let
    \[
    \Alg{G,\calX}{i},\quad \Alg{G,\calX}{-i},\quad \Alg{G,\calX}{\pm i},\quad \Alg{G,\calX}{\Flag}
    \]
    denote the algebras $\Alg{\catC,\Omega_i^+}{}$, $\Alg{\catC,\Omega_i^-}{}$, $\Alg{\catC,\Omega_i^{\pm}}{}$,
    $\Alg{\catC,\llFlag}{}$, respectively.
}

\subsection{Elementary Functors}
\label{subsec:case-of-Alg-G-X}

    Let $G$ be an $\ell$-group,  let
    $\calX$ be an almost transitive $G$-complex (\ref{subsec:G-complex}), and let $\catC=\catC(G,\calX)$
    (Definition~\ref{DF:G-complex-category}). In this section we study in detail the algebra of all $(\catC,F)$-operators
    when $F:\catC\to\catPHil$ satisfies certain assumptions satisfied
    by  $\Omega_i^+$, $\Omega_i^-$, $\Omega_i^\pm$ and $\llFlag$.

	\begin{dfn}\label{DF:elementary-functor}
		A functor $F:\catC(G,\calX)\to \catPHil$ is \emph{elementary}
		if there exists a covariant functor $S:\catC(G,\calX)\to\catSet$ such
		that
		\begin{enumerate}
			\item[{\rm(E1)}] There is a unitary natural
			isomorphism $\llf\circ S\cong F$.
			\item[{\rm(E2)}] For all $x\in S\calX$, the group
			$\Stab_G(x)$ is compact open in $G$ (the action of $G$ on $S\calX$ is via $S$).
			Furthermore, $\Stab_G(x)$ is contained in the stabilizer of a nonempty cell
			in $\calX$.
			\item[{\rm(E3)}] For all $\Gamma\leq_{\calX} G$, the map $\Gamma\leftmod S\calX\to S(\Gamma\leftmod\calX)$ given
            $\Gamma x\mapsto (Sp_\Gamma)x$
			(notation as in Definition~\ref{DF:G-complex-category}) is an isomorphism.
			\item[{\rm(E4)}] $G\leftmod S\calX$ is finite.
		\end{enumerate}
		The functor $F$ is called \emph{semi-elementary} if there is another
		functor $F':\catC(\calX,G)\to \catPHil$ such that $F\oplus F'$ is
		elementary.
	\end{dfn}
	
	\begin{example}
		The functors $\Omega_i^\pm$, $\Omega_i^+$ and $\llFlag$ are elementary.
		Indeed,
		take $S$ to be $X\mapsto \ori{X^{(i)}}$, $X\mapsto X^{(i)}$ and
		$X\mapsto\Flag(X)$, respectively; condition (E3) follows from Proposition~\ref{PR:orientation-of-quotient}.
		The functor  $\Omega_0^-$ is also elementary.

		For $i>0$, the
		functor $\Omega_i^-$ is semi-elementary since
		$\Omega_i^\pm=\Omega_i^+\oplus\Omega_i^-$.
		More generally, any functor $F:\catC(G,\calX)\to \catPHil$
		which is an orthogonal summand of a finite direct sum of functors
		of the form $\Omega_i^+,\Omega_i^-,\llFlag$ is semi-elementary.
	\end{example}
	
	\begin{remark}
		The second part of condition (E2) will not be used until
		\ref{subsec:algebra-A-revised}. Furthermore,
		if the topological realization of $\calX$ is a complete metric ${\mathrm{CAT}(0)}$-space
		on which $G$ acts by isometries,
		then any compact subgroup of $G$ is contained in the stabilizer of a cell in $\dot{\calX}$.
		This follows from the Bruhat--Tits Fixed Point Theorem (\cite[Th.~11.23]{Buildings08AbramBrown};
		notice that orbits of compact groups
		are are compact and hence bounded). When $\calX$ is an \emph{affine building}, e.g.\
		when $\calX=\calB_d(F)$ as in Chapter~\ref{sec:ramanujan-complexes}, the complex
		$\calX$ can be
		realized as a complete metric ${\mathrm{CAT}(0)}$-space \cite[Th.~11.16(b)]{Buildings08AbramBrown}.
	\end{remark}

    For the next proposition, recall that if $\Gamma$ is a group and $V$
    is a complex representation of $\Gamma$, then the space
    of coinvariants $V_\Gamma$ is $V/V(\Gamma)$, where $V(\Gamma):=\Span\{v-\gamma v\where v\in V,\,\gamma\in\Gamma\}$.
    Also notice that if $F:\catC(G,\calX)\to\catPHil$, then $G$ acts on $F\calX$ via $F$.
    Recall from the notation section that
	for a set $X$, we let $\{\e_x\}_{x\in X}$ denote the standard
	basis of $\llf(X)$.

    \begin{prp}\label{PR:S-quotient-check}
        (i) Let $S:\catC(G,\calX)\to \catSet$ be a functor satisfying condition (E3) of Definition~\ref{DF:elementary-functor}.
        Define $S_1:\catSimp\to\catSet$ by
        \begin{align*}
        S_1(\Gamma\leftmod\calX)&=\Gamma\leftmod S\calX,\\
        S_1(p_\Gamma\circ g)&=[x\mapsto \Gamma gx:S\calX\to\Gamma\leftmod S\calX],\\
        S_1(p_{\Gamma', \Gamma})&=[\Gamma'x\mapsto \Gamma x:\Gamma'\leftmod S\calX\to\Gamma\leftmod  S\calX]
        \end{align*}
        for all $\Gamma'\leq\Gamma\leq_{\calX}G$, $g\in G$ (notation as in Definition~\ref{DF:G-complex-category}).
        Then $S\cong S_1$.

        (ii) Let $F:\catC(G,\calX)\to \catPHil$ be a semi-elementary functor viewed
        as a functor into $\catVec$, the category of vector spaces over $\C$. Write $V=F\calX$ and define
        $F_1:\catC(G,\calX)\to\catVec$ by
        \begin{align*}
        F_1(\Gamma\leftmod\calX)&=V_{\Gamma},\\
        F_1(p_\Gamma\circ g)&=[v\mapsto gv+V(\Gamma):V\to V_\Gamma],\\
        F_1(p_{\Gamma',\Gamma})&=[v+V(\Gamma')\mapsto v+V(\Gamma):V_{\Gamma'}\to V_{\Gamma}]
        \end{align*}
        for all $\Gamma'\leq\Gamma\leq_{\calX}G$, $g\in G$ (notation as in Definition~\ref{DF:G-complex-category}).
        Then $F\cong F_1$.
    \end{prp}

    \begin{proof}
        (i) For all  $\Gamma\leq_{\calX}G$, let $u_{\Gamma\leftmod\calX}:S_1(\Gamma\leftmod\calX)=\Gamma\leftmod S\calX\to S(\Gamma\leftmod\calX)$
        be the map $\Gamma x\mapsto (Sp_\Gamma)x$ (this is well-defined since $p_\Gamma\circ \gamma=p_\Gamma$
        for all $\gamma\in\Gamma$).
        Then $u_{\Gamma\leftmod\calX}$ is an isomorphism by (E3), and it is routine to check that $u=\{u_X\}_{X\in\catC}:S_1\to S$
        is a natural transformation.

        (ii)
        For every $X=\Gamma\leftmod\calX\in\catC(G,\calX)$, define
        $u_{F,X}:F_1\calX=(F\calX)_\Gamma\to FX$
        by $u_{F,X}(\vphi+(F\calX)(\Gamma))=(Fp_{\Gamma})\vphi$.
        This is well-defined since for all $\vphi\in F\calX$ and $\gamma\in\Gamma$,
        we have $(Fp_{\Gamma})(\vphi-\gamma\vphi)=(Fp_{\Gamma})\vphi-(F(p_\Gamma\circ \gamma))\vphi=
        (Fp_{\Gamma})\vphi-(Fp_{\Gamma})\vphi=0$.
        It is routine to check that $u_F=\{u_{F,X}\}_{X\in \catC}:F_1\to F$
        is a natural transformation. It is left to show that $u_F$ is an isomorphism.

        Suppose first that $F$ is elementary and write $F= \llf\circ S$ as in Definition~\ref{DF:elementary-functor}.
        By (i), we may assume that $S=S_1$, and hence
        $F_1(\Gamma\leftmod\calX)=\llf(S\calX)_\Gamma$,
        $F(\Gamma\leftmod\calX)=\llf(\Gamma\leftmod S\calX)$,
        and  $u_{F,\Gamma\leftmod\calX}$ is given by sending
        $\e_x+(F\calX)(\Gamma)$ to $(Sp_{\Gamma})_*(\e_x)=\e_{\Gamma x}$ for all $x\in S\calX$
        (see the notation section). Since $(F\calX)(\Gamma)=\Span \{\e_x-\gamma\e_x\where
        x\in S\calX,\,\gamma\in\Gamma\}$, this means $u_{F,\Gamma\leftmod\calX}$ is an isomorphism.
        For general $F$, choose $F'$ such that $F\oplus F'$ is elementary.
        Then $u_{F\oplus F'}$ is an isomorphism by what we have shown.
        Since $u_{F\oplus F'}=u_{F}\oplus u_{F'}$,
        it follows that $u_F$ is an isomorphism as well.
    \end{proof}

    \begin{thm}\label{TH:Adj-algebra-descrition}
        Let $\calX$ be an almost transitive $G$-complex, let $\catC=\catC(G,\calX)$, let $F:
        \catC\to \catPHil$ be semi-elementary,
        and write $A=\Alg{\catC,F}{}$. Then:
        \begin{enumerate}
            \item[(i)] The map $\{a_X\}_{X\in\catC}\mapsto a_{\calX}:A\to \End_G(F\calX)$ is an isomorphism of unital
            $*$-algebras, where the involution on $\End_G(F\calX)$ is given by taking the dual with respect to the inner product on $F\calX$.
            \item[(ii)] For every $a\in A$, there is $M=M(a)\in\R_{\geq0}$ such that $\norm{a|_{FX}}\leq M$
            for all $X\in \catC$.
        \end{enumerate}
    \end{thm}

    When (ii) holds, we  say that $A$ acts \emph{uniformly continuously} on $\catC$.	
	We shall need several lemmas for the proof.

    \begin{lem}\label{LM:Jensen}
        Let $r_1,\dots,r_n\in\R$. Then $(r_1+\dots+r_n)^2\leq n(r_1^2+\dots +r_n^2)$.
    \end{lem}

    \begin{proof}
        This is well-known.
        The function $x\mapsto x^2:\R\to \R^2$
        is convex, so by  Jensen's Inequality, $(\frac{r_1}{n}+\dots+\frac{r_n}{n})^2\leq
        \frac{r_1^2}{n}+\dots+\frac{r_n^2}{n}$. Now multiply by $n^2$.
    \end{proof}

    \begin{lem}\label{LM:boundedness-test}
        Let $X$ and $Y$ be sets, let $a:\llf(X)\to\llf(Y)$ be a linear
        operator, and let $\{\alpha_{xy}\}_{x\in X,y\in Y}\subseteq\C$
        be the unique complex numbers satisfying
        $a\e_x=\sum_y\alpha_{xy}\e_y$. Assume that there are $M\in\N$ and $N\in\R$
        such that
        \begin{enumerate}
            \item[(B1)] for all $x\in X$, $\norm{a\e_x}^2=\sum_y\abs{\alpha_{xy}}^2\leq N$ and
            \item[(B2)] for all $y\in Y$, $\#\{x\in X\suchthat \alpha_{xy}\neq 0\}\leq M$
        \end{enumerate}
        Then $\norm{a}\leq \sqrt{MN}$ and $a$ admits a dual $a^*:\llf(Y)\to\llf(X)$.
    \end{lem}

    \begin{proof}
        Define $a^*:\llf(Y)\to \llf(X)$ by $a^*\e_y=\sum_{x}\cconj{\alpha_{xy}}\e_x$. This is well-defined
        by (B2). It is easy to check that $\Trings{a\e_x,\e_y}=\alpha_{xy}=\Trings{\e_x,a^*\e_y}$,
        so $a^*$ is indeed a dual of $a$.
        Let $\vphi=\sum_x\beta_x\e_x\in\llf(X)$. Then
        \[
        a\vphi=\sum_x\beta_x\sum_y\alpha_{xy}\e_y=\sum_y\Circs{\sum_x\beta_x\alpha_{xy}}\e_y\ .
        \]
        Using Lemma~\ref{LM:Jensen},
        (B1) and (B2), we have
        \begin{eqnarray*}
        \norm{a\vphi}^2
        &=& \sum_y\Big|\sum_x\beta_x\alpha_{xy}\Big|^2
        \leq \sum_y\Big(\sum_x|\beta_x\alpha_{xy}|\Big)^2\\
        &\leq & \sum_y M\sum_x|\beta_x\alpha_{xy}|^2=M\sum_x\abs{\beta_x}^2\sum_y\abs{\alpha_{xy}}^2
        \leq MN\norm{\vphi}^2\ .
        \end{eqnarray*}
        This means that $\norm{a}\leq\sqrt{MN}$.
    \end{proof}

    \begin{lem}\label{LM:pass-to-summand}
        Let $F_1,F_2:\catC\to \catPHil$ be two functors.
        If the conclusion of Theorem~\ref{TH:Adj-algebra-descrition}(i) (resp.\
        Theorem~\ref{TH:Adj-algebra-descrition}(ii)) holds for $F_1\oplus F_2$, then it also
        holds for
        $F_1$ and $F_2$.
    \end{lem}

    \begin{proof}
        It is enough to prove this for $F_1$.
        For every $X\in\catC$, let $e_{X}$ denote the orthogonal projection
        of $(F_1\oplus F_2)X=F_1X\oplus F_2X$ onto the summand $F_1X$.
        Then $e:=\{e_X\}_{X\in\catC}$ is a self-dual natural transformation from $F_1\oplus F_2$
        to itself, and hence lives in $\Alg{\catC,F_1\oplus F_2}{}$. It is easy to see that the map
        \begin{eqnarray*}
        \Alg{\catC,F_1}{} &\to & e\Alg{\catC,F_1\oplus F_2}{}e \\
        \{a_X\}_{X\in\catC} & \mapsto & \{a_X\oplus 0_{F_2X}\}_{X\in\catC}
        \end{eqnarray*}
        is an isomorphism of $*$-algebras. Since $\norm{a_X}=\norm{a_X\oplus 0_{F_2X}}$,
        we see that
        if $\Alg{F_1\oplus F_2,\catC}{}$ acts uniformly continuously on $\catC$, then so does $\Alg{F_1,\catC}{}$.

        Suppose now that $F_1\oplus F_2$ satisfies the conclusion of Theorem~\ref{TH:Adj-algebra-descrition}(i) and let $a_{\calX}\in\End_{G}(F_1\calX)$.
        Then $a_{\calX}\oplus 0_{F_2\calX}\in \End_{G}(F_1\calX\oplus F_2\calX)$, so by assumption,
        there is unique $b=\{b_X\}_{X\in\catC}\in\Alg{\catC,F_1\oplus F_2}{}$ with
        $b_{\calX}=a_{\calX}\oplus 0_{F_2\calX}$. Since we also have $(ebe)_{\calX}=a_{\calX}\oplus 0_{F_2\calX}$, we must have
        $b=ebe$. Thus, there is unique $c=\{c_X \}_{X\in\catC}\in\Alg{\catC,F}{}$
        such that $b_X=c_X\oplus 0_{F_2X}$ for all $X\in\catC$. In particular, $c_{\calX}=a_{\calX}$,
        and $c$ is unique by the uniqueness of $b$.
    \end{proof}

    Let us describe several cases where Lemma~\ref{LM:pass-to-summand} can be applied.

    \begin{example}\label{EX:summand-functor}
        (i) By definition, $\Omega^\pm_i=\Omega_i^+\oplus\Omega_i^-$ for $i>0$.

        (ii) One can  embed $\Omega^\pm_i$ as an orthogonal summand of $\llFlag$.
        Observe first that any maximal flag $f=(\emptyset=x_{-1}\subsetneq x_0\subsetneq x_1\subsetneq\dots\subsetneq x_d)$
        in $X\in\catC$
        induces a partial order on $x_i$, namely, for $u,v\in x_i$, set $u<v $ if and only if there is $0\leq j< i$
        such that $u\in x_j$ and $v\notin x_j$. Denote this order by $\leq$ and write $\zeta(f)=[x_i,\leq ]\in\ori{X^{(i)}}$.
        Also, for all $\sfx\in \ori{X^{(i)}}$, let $\eta(\sfx)=\{f\in\Flag(X)\suchthat \sfx=\zeta(f)\}$.
        Notice that $\eta(\sfx)$ is finite since $X$ is locally finite.
        Now, the map $j_X:\Omega_i^\pm(X)\to \llFlag(X)$ sending $\e_{\sfx}$ to $|\eta(\sfx)|^{-1/2}\sum_{f\in\eta(\sfx)}e_f$
        is a unitary injection, and $j=\{j_X\}_{X\in\catC}$ is a natural transformation from $\Omega_i^\pm$ to $\llFlag$.
        The image of $j_X$ is indeed an orthogonal summand --- its orthogonal complement is
        $FX:=\{\vphi\in\llFlag(X)\suchthat \text{$\sum_{f\in\eta(\sfx)}\vphi(f)=0$ for all $\sfx\in\ori{X^{(i)}}$}\}$.
        Thus, $\llFlag\cong \Omega_i^\pm\oplus F$.
    \end{example}

    \begin{proof}[Proof of Theorem~\ref{TH:Adj-algebra-descrition}]
        By Lemma~\ref{LM:pass-to-summand}, it is enough
        to prove the theorem when $F$ is elementary.
        Write $F=\llf\circ S$ where $S:\catC\to \catSet$
        satisfies conditions (E1)--(E4).
        By Proposition~\ref{PR:S-quotient-check}(i), we may assume
        $S(\Gamma\leftmod\calX)=\Gamma\leftmod S\calX$ for all $\Gamma\leq_{\calX}G$.

\smallskip

        (i) Let $a=\{a_X\}\in A$ and let $p_\Gamma:\calX\to \Gamma\leftmod\calX$
        be the quotient map. Then $a_{\Gamma\leftmod \calX}\circ (Sp_\Gamma)_*=(Sp_\Gamma)_*\circ a_{\calX}$.
        The map
       	$Sp_\Gamma$ is the quotient map $S\calX\to \Gamma\leftmod S\calX$, and hence  $(Sp_{\Gamma})_*$ is surjective.
       	This means $a_{\calX}$ determines $a_{\Gamma\leftmod \calX}$
        for all $\Gamma\leq_{\calX}G$, so the map $\{a_X\}\mapsto a_{\calX}$ is injective.

        Suppose now that we are given $a_{\calX}\in \End_G(F\calX)=\End_G(\llf(S\calX))$.
        For all $\Gamma\leq_{\calX}G$, we define
        $a_{\Gamma\leftmod \calX}:S(\Gamma\leftmod\calX)\to S(\Gamma\leftmod\calX)$
        as follows: If $\Gamma x \in \Gamma\leftmod S\calX$ and $a\e_x=\sum_{y\in S\calX}
        \alpha_{y} \e_{y}$
        (in $\llf(S\calX)$), then set
        \begin{equation}\label{EQ:Adj-action-on-quotients}
        a_{\Gamma\leftmod \calX}\e_{\Gamma x}=\sum_{y\in S\calX}\alpha_{y}\e_{\Gamma y}\ ,
        \end{equation}
        and extend $a_{\Gamma\leftmod\calX}$ linearly to $\llf(\Gamma\leftmod S\calX)$.
        This is well-defined because if we replace
        $x$ with $\gamma x$ for $\gamma\in \Gamma$, then
        $a\e_{\gamma x}=a\gamma \e_x=\gamma a\e_x=\gamma\sum_{y}\alpha_{y} \e_{y}
        =\sum_{y}\alpha_{y}\e_{\gamma y}$
        and we get $a\e_{\Gamma\gamma x}=\sum_{y}\alpha_{y}\e_{\Gamma\gamma y}
        =\sum_{y}\alpha_{y}\e_{\Gamma y}$.
        It is  routine to check that $\{a_{X}\}_{X\in\catC}$ is a natural transformation
        from $F$ to itself.

        Suppose now that $a_\calX\in\End_G(F\calX)$ admits a dual $b_{\calX}$.
        We claim that $(a_{\Gamma\leftmod\calX})^*=b_{\Gamma\leftmod\calX}$,
        and hence $a:=\{a_X\}_{X\in\catC}$ has a dual.
        For all $x\in S\calX$ write $a_{\calX}\e_x=\sum_{y\in S\calX}\alpha_{xy}\e_y$ with $\{\alpha_{xy}\}_{x,y}\subseteq\C$.
        Then $b_{\calX}\e_y=\sum_{x}\cconj{\alpha_{xy}}\e_x$. Since
        $a_{\calX}\e_{gx}=a_{\calX}g\e_x=ga_{\calX}\e_x$,
        for all $g\in G$, $x,y\in S\calX$, we have $\alpha_{(gx)y}=\alpha_{x(g^{-1}y)}$.
        Now, by \eqref{EQ:Adj-action-on-quotients}, for all $z,w\in S\calX$ and $\Gamma\leq_{\calX}G$, we have
        $\Trings{a_{\Gamma\leftmod\calX}\e_{\Gamma z},\e_{\Gamma w}}
        =\Trings{\sum_y\alpha_{zy}\e_{\Gamma y},\e_{\Gamma w}}=\sum_{\gamma\in\Gamma}\alpha_{z(\gamma w)}=
        \sum_{\gamma\in\Gamma}\alpha_{(\gamma^{-1}z)w}
        =\Trings{\e_{\Gamma z},\sum_x\cconj{\alpha_{xw}}\e_{\Gamma x}}=\Trings{\e_{\Gamma z},b_{\Gamma\leftmod\calX}\e_{\Gamma w}}$,
        so $(a_{\Gamma\leftmod\calX})^*=b_{\Gamma\leftmod\calX}$.

        The previous paragraphs imply that
        $\{a_X\}\mapsto a_{\calX}:A\to \End_G(F\calX)$
        is injective and surjective,
        provided that any $a_{\calX}\in\End_G(F\calX)$ admits a dual.
        We shall verify the latter in the proof of (ii).

\smallskip

        (ii) Let $a\in\End_G(F\calX)$ and define
        $\{a_X\}_{X\in\catC}$ as above.
        We will apply Lemma~\ref{LM:boundedness-test} to $a_{\Gamma\leftmod \calX}$ with constants
        $N$, $M$ which are independent of $\Gamma$, thus showing that
        $a_{\calX}$ has a dual and  $\norm{a_{\Gamma\leftmod  \calX}}$ can be bounded uniformly in $\Gamma$.

        Let $\{x_1,\dots,x_r\}$ be representatives for
        the $G$-orbits in $S\calX$ (there are finitely many by (E4))
        and write
        $a\e_{x_j}=\sum_{k=1}^{s_{j}}\alpha_{jk}\e_{z_{jk}}$
        where $\{z_{jk}\}_{j,k}\subseteq S\calX$.
        Then for all $g\in G$, we have
        $a\e_{gx_j}=ag\e_{x_j}=ga\e_{x_j}=\sum_{k}\alpha_{jk}\e_{gz_{jk}}$.
        Thus,
        by \eqref{EQ:Adj-action-on-quotients}, 
        \[
        a_{\Gamma\leftmod\calX}\e_{\Gamma gx_j}=\sum_{k=1}^{s_{j}} \alpha_{jk}\e_{\Gamma gz_{jk}}\ .
        \]
        The cosets $\{\Gamma gz_{jk}\}_{k=1}^{s_j}$ may coincide, hence
        there is a partition $\pi=\pi(g,j)$ of the set $\{1,\dots,s_j\}$ such that
        \[
        \norm{a_{\Gamma\leftmod \calX}\e_{\Gamma gx_j}}^2=\sum_{S\in \pi}\Big|\sum_{k\in S}\alpha_{jk}\Big|^2\ .
        \]
        The number of partitions of  $\{1,\dots,s_j\}$ is finite, as well as the number of possible \mbox{$j$-s},
        hence condition (B1) of Lemma~\ref{LM:boundedness-test} holds for $a_{\Gamma\leftmod \calX}$ with $N$ which
        is independent of $\Gamma$.

        To show (B2), we need to find $M\in\N$, independent of $\Gamma$, such that for all
        $\Gamma y\in \Gamma\leftmod S\calX$,
        there are at most $M$ orbits $\Gamma x\in \Gamma\leftmod S\calX$ with
        $\Gamma y\in\supp(a_{\Gamma\leftmod \calX}\e_{\Gamma x})$.
        Let $K_j=\Stab_{G}(x_j)$,
        and let $I_j$ be a set of representatives
        for the double cosets $\Gamma\leftmod G/K_j$.
        Then every $\Gamma x\in \Gamma\leftmod S\calX$ can be written as $\Gamma g x_j$
        for unique $j$ and $g\in I_j$.
        In this case, $\supp(a_{\Gamma\leftmod \calX}\e_{\Gamma x})\subseteq \{\Gamma g z_{jk}\where 1\leq k\leq s_j\}$.
        Therefore, it is enough to show that for every $1\leq j\leq r$ and $1\leq k\leq s_j$, there
        is a  bound on $\#\{g\in I_j\suchthat \Gamma y=\Gamma gz_{jk}\}$ which is independent of $y$ and $\Gamma$.
        We may assume that $y=hz_{jk}$ for some $h\in G$, since otherwise the quantity in question is $0$.
        Let $L=\Stab_G(z_{jk})$ and put $K=K_j$. Then $\Gamma y=\Gamma gz_{jk}$ implies
        that there is $\gamma\in\Gamma$ such that $\gamma hz_{jk}=gz_{jk}$,
        hence $h^{-1}\gamma^{-1}g\in L$, or  $\gamma^{-1}g\in hL$. This implies
        that $\Gamma g K\cap hL\neq \emptyset$.
        By (E2), $L$ and $K$ are compact and open in $G$, so
        the index $[L:K\cap L]$ is finite. Since $\Gamma g K\cap hL$
        is  a union of right $(K\cap L)$-cosets, no more than
        $[L:K\cap L]$ double cosets $\Gamma g K$ can intersect $hL$ non-trivially, which proves
        that $\#\{g\in I_j\suchthat \Gamma y=\Gamma gz_{jk}\}\leq [L:K\cap L]$. This
        completes the proof.
    \end{proof}

    \begin{example}\label{EX:tree-spectrum}
        Let $\calX$ be a $k$-regular tree, let $G=\Aut(\calX)$, and let $\catC=\catC(G,\calX)$. We now use
        Theorem~\ref{TH:Adj-algebra-descrition} to show that
        $A_0:=\Alg{\catC,\Omega_0^+}{}$ is freely generated by $a_0$, the vertex adjacency operator,
        and $A_1:=\Alg{\catC,\Omega_1^+}{}$ is freely generated by $a_1$, the edge adjacency operator.
        As mentioned in Example~\ref{EX:k-regular-graphs-basic-spectrum}, this implies that:
        \begin{enumerate}
        \item[(i)] $\Spec_{0,\catC}(\Gamma\leftmod \calX)$, the  $0$-dimensional spectrum of $\Gamma\leftmod\calX$,
        is equivalent
        to the spectrum of the {vertex} adjacency operator of $\Gamma\leftmod\calX$.
        \item[(ii)] $\Spec_{1,\catC}(\Gamma\leftmod \calX)$, the non-oriented $1$-dimensional spectrum of $\Gamma\leftmod\calX$, is equivalent
        to the spectrum of the {edge} adjacency operator of $\Gamma\leftmod\calX$.
        \end{enumerate}
        (These assertions are false
        for general $\calX$ and $G$.)

        Using Theorem~\ref{TH:Adj-algebra-descrition}, we identify
        $A_0$ with $\End_G(\Omega^+_0(\calX))$.
        Fix a vertex $x\in \calX^{(0)}$, let $b\in \End_G(\Omega^+_0(\calX))=A_0$, and write
        $b\e_x=\sum_{y\in\calX^{(0)}} \alpha_y \e_y$.
        Set $R_n=\{y\in \calX^{(0)}\where \dist(x,y)=n\}$. Then $K_x:=\Stab_G(x)$ acts transitively
        on $R_n$ for all $n\geq 0$, and for all $g\in K_x$, we have
        \[\sum_{y} \alpha_y \e_y=b\e_x=b(g\e_x)=g(b\e_x)=g\sum_{y} \alpha_y \e_y=\sum_y \alpha_y \e_{gy}\ .\]
        Therefore, $\alpha_{gy}=\alpha_y$, and the coefficients $\alpha_y$ are identical on the set $R_n$.
        We may therefore write $b\e_x=\sum_{n\geq 0}\sum_{y\in R_n}\alpha_n \e_y$. Let $m=m(b)\geq 0$ be the least
        integer such that $\alpha_m=0$. We prove by induction
        on $m$ that $b$ is a polynomial in $a_0$. Indeed, if $m=0$, then $b\e_x=0$, hence
        for all $g\in G$, we have $b\e_{gx}=g(b\e_x) =0$, so $b=0$.
        If $m>0$, then $b':=b-\alpha_m a^{m-1}_0$ satisfies $m(b')<m(b)$, and by induction, $b'$ is a polynomial in $a_0$.
        Finally, it is easy to show that $m(a^n_0)=n+1$, hence $\{\e_x,a_0\e_x,a^2_0\e_x,\dots\}$
        are linearly independent in $\Omega^+_0(\calX)$, which means $\{1,a_0,a_0^2,\dots\}$ are linearly independent in $A_0$.
        Thus, $A_0$ is
        freely generated by $a_0$.

        Similarly, $A_1$ is freely generated by $a_1$.
    \end{example}

    \begin{example}\label{EX:Alg-of-buildings-general}
        The previous example can be generalized to buildings (definitions can be
        found in \cite{Buildings08AbramBrown}):
        Let $\calB$ be a locally finite building with Coxeter system
        $(W,S)$ and write $d:=\dim \calB=|S|-1$.
        There is a $W$-valued distance function $\delta$ taking two chambers of $\calB$ to an element
        of $W$ and satisfying certain axioms (see \cite[Def.~5.1]{Buildings08AbramBrown}).
        Let $G$ be a subgroup of $\Aut(\calB)$ such that $G$ is closed under pointwise
        convergence, $\delta(gx,gy)=\delta(x,y)$ for all $g\in G$, $x,y\in\calB^{(d)}$,
        and $G$  acts
        transitively on pairs $(x,y)\in\calB^{(d)}\times\calB^{(d)}$ with $\delta(x,y)=w$ for all $w\in W$.
        Then $\calB$ is an almost transitive $G$-complex (Proposition~\ref{PR:characterization-of-G-comps}(i)).
        For every $w\in W$, define $a_w:\Omega_d^+(\calB)\to\Omega_d^+(\calB)$
        by
        \[
        (a_w\vphi)x=\sum_{\substack{y\in\calB^{(d)}\\ \delta(x,y)=w}}\vphi y\qquad\forall\, \vphi\in\Omega_d^+(\calB),\,x\in\calB^{(d)}\ .
        \]
        It is easy to check that $a_w\in \End_G(\Omega_d^+(\calB))$, and hence $a_w$ can be viewed
        as an element of $A_d:=\Alg{\catC(G,\calB),\Omega_d^+}{}$. Furthermore, using the transitivity property of $G$,  the properties of $\delta$,
        and induction on the length of elements in  $W$,  one can show that the elements $\{a_w\}_{w\in W}$
        form a basis of $A_d$.
        

\rem{
        \gap{In fact, it is possible
        to describe $A_i:=\Alg{\catC(G,\calB),\Omega_i^+}{i}$ in a similar fashion using double cosets in $W$.
        This
        was communicated to us by Amitay Kamber and will appear, together with other relevant results,
        in a forthcoming paper by him;
        see also his thesis.}
}

        We finally note that if $\bfG$ is a \emph{simply-connected} almost simple  algebraic group over  a local non-archimedean field $F$,
        then its  affine Bruhat-Tits building $\calB$ satisfies the above assumptions  with $G=\im(\bfG(F)\to\Aut(\calB))$
        (cf.\ Example~\ref{EX:G-complex-building-III}).
        For example, when $\bfG=\uSL_d$, we get $\calB=\calB_d(F)$ as in Chapter~\ref{sec:ramanujan-complexes} and
        $G=\im(\nSL{F}{d}\to\nPGL{F}{d})$ (the group $\nPGL{F}{d}$ does not preserve $\delta$-distance).
    \end{example}

    We finish this section with a proposition relating the $A$-spectrum of two $G$-quotients
    of $\calX$.

    \begin{prp}\label{PR:subgroup-quotient}
    	Let $\Gamma'\leq\Gamma\leq_{\calX}G$ with $[\Gamma:\Gamma']<\infty$,
    	let $F:\catC\to\catPHil$ be semi-elementary, and let $A$ be an algebra
    	of $(\catC,F)$-operators.
    	Then $\Spec_A(\Gamma\leftmod\calX)\subseteq\Spec_A(\Gamma'\leftmod\calX)$.
    	When
    	$\Gamma'\leftmod\calX$ is finite, we also have $\mSpec_A(\Gamma\leftmod\calX)\subseteq
        \mSpec_A(\Gamma'\leftmod\calX)$ (as multisets).
    \end{prp}

    \begin{proof}
    	Choose a functor $F':\catC\to \catPHil$ such that $F\oplus F'$ is elementary.
        We may view $A$ as an algebra of $\Alg{\catC,F\oplus F'}{}$-operators (Remark~\ref{RM:summand-transition}),
        so assume henceforth that $F$ is elementary. Write $F=\llf\circ S$ where $S$ is as in Definition~\ref{DF:elementary-functor}.
        By Proposition~\ref{PR:S-quotient-check}(i), we may also assume $S(\Delta\leftmod\calX)=\Delta\leftmod S\calX$ for all
        $\Delta\leq_{\calX} G$,
        and that $Sp_{\Gamma',\Gamma}: \Gamma'\leftmod S\calX\to \Gamma\leftmod S\calX$ is the quotient map
        $\Gamma'x\mapsto \Gamma x$. The fibers of $Sp_{\Gamma',\Gamma}$
        have size at most   $[\Gamma':\Gamma]<\infty$, and
        hence the map $Fp_{\Gamma',\Gamma}=(Sp_{\Gamma',\Gamma})_*:\llf(\Gamma'\leftmod S\calX)\to \llf(\Gamma\leftmod S\calX)$
        is  continuous. Therefore, this map extends
        to a continuous map on the completions $P=\quo{Fp_{\Gamma',\Gamma}}:\ell^2(\Gamma'\leftmod S\calX)\to\ell^2(\Gamma\leftmod S\calX)$,
        which is  surjective (because $P\e_{\Gamma' x}=\e_{\Gamma x}$ for all $x\in S\calX$).
        Since $A$ acts continuously on $\llf(\Gamma'\leftmod S\calX)$ and $\llf(\Gamma\leftmod S\calX)$
        (Theorem~\ref{TH:Adj-algebra-descrition}(ii)), the action of $A$ extends
        to the relevant completions, and $P$ is an $A$-homomorphism.
        The map $P$ restricts to a continuous isomorphism of $A$-modules $(\ker P)^\perp\to\ell^2(\Gamma\leftmod S\calX)=
        \quo{F(\Gamma\leftmod\calX)}$.
        By Proposition~\ref{PR:iso-implies-unitary-iso}, there exists a unitary isomorphism of $A$-modules
        $(\ker P)^\perp \to \ell^2(\Gamma\leftmod S\calX)=\quo{F(\Gamma\leftmod\calX)}$.
        This means that $\Spec_A(\Gamma\leftmod\calX)=\Spec_A(\quo{AF(\Gamma\leftmod\calX)})= \Spec_A(\quo{A(\ker P)^\perp})\subseteq
        \Spec_A(A\llf(\Gamma'\leftmod S\calX))=\Spec_A(\Gamma'\leftmod\calX)$. The statement about the multiset spectrum
        is proved similarly.
    \end{proof}

    The rest of this chapter discusses dependencies between different kinds of spectra,
    equivalence of oriented and non-oriented spectra (in some cases), and the fact that the spectra
    of a complex may not be determined from its isomorphism class.
    The reader can skip this without loss of continuity.

\subsection{Dependencies Between Spectra}
\label{subsec:dependencies}

    Fix a skeletally small subcategory $\catC\subseteq\catSimp$
    and let $F,F':\catC\to\catPHil$ be two functors (e.g.\ $F=\Omega^+_i$ and $F'=\Omega^+_j$).
    For brevity, we write
    \[
    A_F=\Alg{\catC,F}{},\qquad
    \Spec_{F}(X)=\Spec_{\Alg{\catC,F}{}}(X),
    \]
    and likewise for $F'$.

    The spectra of $X\in\catC$ with respect  to $F$ and $F'$ are usually not independent in the sense that
    the presence of certain points from $\what{A}_F$
    in $\Spec_{F}(X)$ may imply the presence
    of certain points from $\what{A}_{F'}$ in $\Spec_{F'}(X)$.
    In fact, in some cases, $\Spec_{F}(X)$ determines $\Spec_{F'}(X)$ completely.
    Such dependencies can sometimes be explained by analyzing $F\oplus F'$.

\medskip

    Write $A_{F\oplus F'}:=\Alg{\catC,F\oplus F'}{}$.
    For all $X\in\catC$, let $e_{F,X}$ and $e_{F',X}$ be the orthogonal projections
    of $FX\oplus F'X$ onto $FX$ and $F'X$, respectively.
    It is easy to see that $e_F=\{e_{F,X}\}_{X\in\catC}$ and
    $e_{F'}=\{e_{F',X}\}_{X\in\catC}$ live in $A_{F\oplus F'}$, and moreover, that
    $e_F$ and $e_{F'}$ are idempotents, $e_F+e_{F'}=1$, $e_F^*=e_F$ and $e_{F'}^*=e_{F'}$.
    There are obvious isomorphisms
    \begin{align*}
    &A_F \cong
    e_FA_{F\oplus F'}e_F, &
    &A_{F'} \cong
    e_{F'}A_{F\oplus F'}e_{F'},\\
    & FX \cong e_F(FX\oplus F'X), & & F'X \cong e_{F'}(FX\oplus F'X),
    \end{align*}
    which are compatible with the relevant module structures for all $X\in\catC$.
    Provided $A_{F\oplus F'}$ acts continuously on $(F'\oplus F)X$,
    Theorem~\ref{TH:corner-subalgebra-spectrum-I} and Corollary~\ref{CR:corner-subalgebra-spectrum-II} imply
    that the datum of $\Spec_{F\oplus F'}(X)$
    is equivalent to the data of $\Spec_{F}(X)$ and $\Spec_{F'}(X)$.
    However, $\Spec_{F\oplus F'}(X)$ holds this data
    in a more efficient way which avoids certain dependencies. More formally:

    \begin{prp}\label{PR:summand-spectrum}
        In the previous setting, let $\what{A}_{F\oplus F'}^{(A_F)}=\{[V]\in \what{A}_{F\oplus F'}\where e_FV\neq 0\}$
        and $\what{A}_{F\oplus F'}^{(A_{F'})}:=\{[V]\in \what{A}_{F\oplus F'}\where e_{F'}V\neq 0\}$,
        and identify the open subset $\what{A}_{F\oplus F'}^{(A_F)}$ (resp.\ $\what{A}_{F\oplus F'}^{(A_{F'})}$)
        with a subspace of $\what{A}_{F}$ (resp.\ $\what{A}_{F'}$) via $[V]\mapsto [e_FV]$
        (resp.\ $[V]\mapsto [e_{F'}V]$) using Theorem~\ref{TH:corner-unitary-dual}.
        Then
        $
        \what{A}_{F\oplus F'}=\what{A}_{F\oplus F'}^{(A_{F})}\cup \what{A}_{F\oplus F'}^{(A_{F'})}
        $, and under the previous identifications,
        for all $X\in\catC$ such that $\Spec_{F\oplus F'}(X)$ is defined, we have
        \begin{align*}
        \Spec_F(X)&=\Spec_{F\oplus F'}(X)\cap \what{A}_{F\oplus F'}^{(A_F)}\\
        \Spec_{F'}(X)&=\Spec_{F\oplus F'}(X)\cap
        \what{A}_{F\oplus F'}^{(A_{F'})}\nonumber
        \end{align*}
        In particular:
        \begin{enumerate}
            \item[(i)] When defined, the $F\oplus F'$-spectrum determines the $F$-spectrum and $F'$-spectrum,
            and vice versa.
            \item[(ii)] For any $\eta \in\what{A}_{F\oplus F'}^{(A_F)}\cap \what{A}_{F\oplus F'}^{(A_{F'})}$,
            we have $\eta\in \Spec_{F}(X)$ $\iff$ $\eta\in\Spec_{F'}(X)$ whenever $\Spec_{F\oplus F'}(X)$ is defined.
            More precisely, if $\eta=[V]$ for $V\in\Irr[u]{A_{F\oplus F'}}$, then
            $[e_FV]\in\Spec_F(X)$ $\iff$ $[e_{F'}V]\in\Spec_{F'}(X)$.
            \item[(iii)] If $\what{A}_{F\oplus F'}^{(A_{F})}=\what{A}_{F\oplus F'}$, then the $F$-spectrum
            determines the $F'$-spectrum whenever the $F\oplus F'$-spectrum is defined.
        \end{enumerate}
    \end{prp}

    \begin{proof}
        Notice that $e_F+e_{F'}=1$ in $A_{F\oplus F'}$, hence $\{e_F,e_{F'}\}$ is a full
        family of idempotents in the sense of Example~\ref{EX:full-families-of-idems}.
        Thus, everything follows from Theorem~\ref{TH:corner-subalgebra-spectrum-I} and
        Corollary~\ref{CR:corner-subalgebra-spectrum-II}. We have $
        \what{A}_{F\oplus F'}=\what{A}_{F\oplus F'}^{(A_{F})}\cup \what{A}_{F\oplus F'}^{(A_{F'})}
        $ by the comment preceding Example~\ref{EX:full-families-of-idems}.
    \end{proof}

    \begin{remark}
        When defined, the $F\oplus F'$-spectrum actually determines the $B$-spectrum for any algebra $B$ of $(\catC,F\oplus F')$-operators,
        by Theorem~\ref{TH:subalgebra-spectrum-I}.
        Proposition~\ref{PR:summand-spectrum} makes this process more concrete in the cases $B=A_F$ and $B=A_{F'}$.
    \end{remark}

    \begin{example}\label{EX:summand-spectrum-equiv}
        Example~\ref{EX:summand-functor} and Proposition~\ref{PR:summand-spectrum} imply that for all $i\geq0$, the flag spectrum determines
        the full $i$-dimensional spectrum, which in turn determines the non-oriented and the oriented $i$-dimensional spectra, provided all are defined.
        Also, the non-oriented and the oriented  $i$-dimensional spectra determine the full  $i$-dimensional spectrum, when defined.
    \end{example}

    \begin{example}
        For all $n\in\N$,
        the $F$-spectrum determines the $F^n$-spectrum and vice versa.
        More generally, if $G:\catC\to \catPHil$ is a functor such that $G$ is a summand of $F^n$ for some $n$,
        then, when defined, the $F$-spectrum determines the $G$-spectrum.
    \end{example}

    \begin{example}\label{EX:spectrum-dependency}
        Let $\calX$ be a $k$-regular tree, let $G=\Aut(\calX)$,
        and let $\catC=\catC(G,\calX)$ (Definition~\ref{DF:G-complex-category}).
        For $i=0,1$, write $A_i=\Alg{\catC,\Omega_i^+}{}$, and
        recall from Example~\ref{EX:tree-spectrum} that $A_i=\C[a_i]$,
        where $a_0$ is the vertex adjacency operator and $a_1$ is the edge adjacency operator
        (cf.\ Example~\ref{EX:adjacency-operator-tree}).

        Write $\Omega^+_{0,1}=\Omega^+_0\oplus\Omega^+_1$ and let $A_{0,1}=\Alg{\catC,\Omega_{0,1}^+}{}$.
        We take $F=\Omega_0^+$ and $F'=\Omega_1^+$ in the previous discussion
        and write $e_0=e_F$ and $e_1=e_{F'}$. Then $A_{0,1}=A_{F\oplus F'}$,
        $A_0=A_F\cong e_0A_{0,1}e_0$ and $A_1=A_{F'}\cong e_1A_{0,1}e_1$.
        Define natural transformations $a_{01}:\Omega^+_1\to\Omega^+_0$ and $a_{10}:\Omega^+_0\to \Omega^+_1$ by
        \begin{align*}
            &(a_{01,X}\vphi)x=\sum_{x\subseteq y\in X^{(1)}}\vphi y & &\forall\, x\in X^{(0)},\, \vphi\in\ \Omega^+_1(X) \\
            &(a_{10,X}\psi)y=\sum_{y\supseteq x\in X^{(0)}}\psi x & &\forall\, y\in X^{(1)},\, \psi\in\ \Omega^+_0(X)
        \end{align*}
        We may view $a_{01,X}$ and $a_{10,X}$ as operators on $\Omega^+_{0,1}(X)=\Omega^+_0(X)\oplus\Omega^+_1(X)$ by setting them to be
        $0$ on $\Omega^+_0(X)$ and $\Omega^+_1(X)$, respectively. It is easy
        to check that $a_{01},a_{10}\in A_{0,1}$ and $a_{01}=a_{10}^*$.
        Arguing as in
        Example~\ref{EX:tree-spectrum}, one can show that $A_{0,1}$ is isomorphic
        to the path algebra $A$ of the quiver
        \[
            \xymatrix{
            0 \ar@/^/[r]|{a_{10}} & 1 \ar@/^/[l]|{a_{01}}
            }
        \]
        discussed in Example~\ref{EX:double-arrow-quiver}
        (the elements $e_0,e_1,a_{01},a_{10}\in A_{0,1}$
        correspond to the elements with the same name in $A$).

        The unitary dual of $A_{0,1}$ is described in  Example~\ref{EX:double-arrow-quiver} as
        the gluing of two copies of $\R_{\geq 0}$ along $\R_{>0}$. Since $A_0= e_0A_{0,1}e_0$ and $A_1= e_1A_{0,1}e_1$, the sets
        $\what{A}_{0,1}^{(A_0)}$ and $\what{A}_{0,1}^{(A_1)}$ correspond to the two copies of $\R_{\geq 0}$,
        and so $\what{A}_{0,1}^{(A_0)}\cap \what{A}_{0,1}^{(A_1)}$ corresponds to $\R_{>0}$,
        which,
        in the notation of Example~\ref{EX:double-arrow-quiver},   $\R_{>0}$ corresponds to the irreducible
        representations $\{[V_r]\where r\in\R_{>0}\}$. Therefore, Proposition~\ref{PR:summand-spectrum}  implies that
        for any $k$-regular graph $X\in\catC$ and $r\in \R_{>0}$, we have
        $
        [e_0V_r]\in \Spec_{0}(X)$ if and only if  $[e_1V_r]\in\Spec_{1}(X)$.

        Recall from Example~\ref{EX:k-regular-graphs-basic-spectrum} that for $i=0,1$ we can identify
        $\what{A}_i$ with a subset of $\C$ such that
        $\Spec_{i}(X)$ corresponds to $\Spec(a_{i,X})$ for all $X\in \catC$. It is easy to check
        that
        \begin{equation}\label{EQ:spectrum-dep-in-the-tree-II}
        a_{01}a_{10}=a_0+ke_0\qquad\text{and}\qquad a_{10}a_{01}=a_1+2e_1
        \end{equation}
        in $A_{0,1}$.
        Using this, one sees that
        for all $r\in\R_{>0}$, we have $\Spec(a_0|_{e_0V_r})=r^2-k$ and $\Spec(a_1|_{e_1V_r})=r^2-2$.
        Since $\Spec_{\Omega^+_{0,1}}(X)$ is always defined (Theorem~\ref{TH:Adj-algebra-descrition}(ii)), this means
        that for all $\lambda\in \R$, we have
        \begin{equation}
        \lambda\in\Spec(a_{0,X})-\{-k\}
        \quad\iff\quad
        \lambda+k-2\in\Spec(a_{1,X})-\{-2\}\ .
        \end{equation}
        This is a well-known dependency between the spectrum of  the vertex and  edge adjacency operators.
        An alternative way to exhibit   this dependency by noting  that
        $a_{10,X}a_{01,X}$ and $a_{01,X}a_{10,X}$
        always have the same spectrum
        except maybe  the multiplicity of $0$, and then applying \eqref{EQ:spectrum-dep-in-the-tree-II}.
    \end{example}

    \begin{example}
        Let  $G=\nPGL{F}{d}$ and $\calX=\calB_d(F)$ be as in Chapter~\ref{sec:ramanujan-complexes}.
        In the case $d=3$, dependencies between the $\Spec_{0,\catC(G,\calX)}(\Gamma\leftmod\calB_d(F))$ (which is equivalent
        to the spectrum defined in Chapter~\ref{sec:ramanujan-complexes} by Example~\ref{EX:zero-dim-spec-of-Bd}) and the spectrum
        of some other higher dimensional associated operators were shown in \cite{KaLiWa10}
        and \cite{GolPar14}. These dependencies can be explained using Proposition~\ref{PR:summand-spectrum};
        they essentially follow
        from the analysis carried in \cite{KaLiWa10}
        and \cite{GolPar14} and Theorem~\ref{TH:spectrum-correspondence} below.
    \end{example}

    The previous ideas    can be extended to arbitrary families of functors
    $\{F_\alpha:\catC\to\catPHil\}_{\alpha\in I}$
    as follows: Let $F=\bigoplus_{\alpha\in I}F_\alpha$. When $I$ is finite,
    the $F$-spectrum determines the $F_\alpha$-spectrum for every $\alpha\in I$
    and vice versa.
    However,
    for infinite $I$, one has to consider
    \[A_{\{F_\alpha\}}=\Alg{\catC,\{F_\alpha\}}{}:=
    \dirlim \{A_{\bigoplus_{\alpha\in S}F_\alpha}\}_{S\fsubseteq I}\ .\]
    Here, $S$ ranges over the finite subsets
    of $I$, and for $S\subseteq T\fsubseteq I$, $A_{\bigoplus_{\alpha\in S}F_\alpha}$ is embedded as a (non-unital)
    subalgebra of $A_{\bigoplus_{\alpha\in T}F_\alpha}$ in the obvious way. The algebra $A_{\{F_\alpha\}}$
    is a  non-unital idempotented $*$-subalgebra of
    $A_F=\Alg{\catC,F}{}$, and we define
    $\Spec_{\{F_\alpha\}}(X):=\Spec_{A_{\{F_\alpha\}}}(FX)$.
    If $e_{\alpha,X}$ denotes the orthogonal projection
    from $FX$ onto $F_\alpha X$, then $\{e_\alpha:=\{e_{\alpha,X}\}_{X\in\catC}\}_{\alpha\in I}$
    form a full system
    of idempotents in $A_{\{F_\alpha\}}$ (see Example~\ref{EX:full-families-of-idems})
    and $e_\alpha A_{\{F_\alpha\}}e_\alpha\cong A_{F_\alpha}$.
    Thus, by Theorem~\ref{TH:corner-subalgebra-spectrum-I} and Corollary~\ref{CR:corner-subalgebra-spectrum-II}, $\Spec_{\{F_\alpha\}}(X)$
    determines $\{\Spec_{F_\alpha}(X)\}_{\alpha\in I}$ and vices versa, when both are defined.
    Moreover, an analogue of Proposition~\ref{PR:summand-spectrum} holds in this case.
    (In contrast, $\{e_\alpha\}_{\alpha\in I}$ is \emph{not} a full system of idempotents in $A_{F}{}$, so a priori $\Spec_{F}(X)$
    contains more information than $\{\Spec_{F_\alpha}(X)\}_{\alpha\in I}$, if it exists.)

    Interesting examples of families $\{F_\alpha\}_{\alpha\in I}$ could be, for instance, $\{\Omega_i^+\}_{i\geq0}$, $\{\Omega_i^-\}_{i\geq 0}$
    or $\{\Omega_i^{\pm}\}_{i\geq 0}$. The resulting spectra can be considered as a variation
    of the non-oriented (resp.\ oriented, full)
    spectrum, seeing all dimensions at once. We shall not investigate such examples in this work, although they are
    relevant to studying maps relating cells of different dimensions such as the boundary and coboundary maps
    (\ref{subsec:orientation}), which  live in $A_{\{\Omega_i^-\}_{i\geq 0}}$, or adjacency operators
    between cells of different dimensions, which live in $A_{\{\Omega_i^+\}_{i\geq 0}}$.

\subsection{More Dependencies Between Spectra}
\label{subsec:equivalence-of-ori-and-non-ori}

    This section continues discussing dependencies between spectra,
    but here we take a more explicit approach, focusing on the case
    $\catC=\catC(G,\calX)$ where $\calX$ is a fixed almost transitive $G$-complex
    (\ref{subsec:G-complex}, Definition~\ref{DF:G-complex-category}).

\medskip

    Observe first that if   $F,F':\catC\to \catPHil$ are functors and  $u:F\to F'$ is a unitary natural isomorphism, then $u$ induces an isomorphisms
    \begin{align*}
            a\mapsto uau^{-1}~:~ & A_F\to A_{F'}\\
            \vphi\mapsto u_X\vphi\,~:~ & FX \to F'X & (X\in\catC(G,\calX))
    \end{align*}
    where the second isomorphism is
    compatible with the relevant module structures. This in turn induces
    a homeomorphism $\what{u}:\what{A}_F\to\what{A}_{F'}$ under which $\Spec_F(X)$ corresponds
    to $\Spec_{F'}(X)$
    for all $X\in\catC$. We conclude that $u$ induces an equivalence between
    the $F$-spectrum and the $F'$-spectrum of $G$-quotients of $\calX$.
    In particular, when $F=F'$, the map $\what{u}$ is an automorphism of $\what{A}_F$, and
    the spectrum of any $G$-quotient of $\calX$ is stable under $\what{u}$.

\medskip

    We say that the action of $G$ on $\calX$ \emph{preserves orientation} if for all $\mathsf{x}\in\ori{\calX}$
    with $\dim\mathsf{x}>0$, $\mathsf{x}$ and $\mathsf{x}^\op$ are not in the same $G$-orbit. Equivalently,
    this means that whenever $g\in G$, $x\in \calX$ satisfy $gx=x$, the element $g$ induces an even permutation on the vertices of $x$.
    We say that the action of $G$ on $\calX$ \emph{preserves order} if for any $x\in X$ and $g\in\Stab_G(x)$,
    we have $g v=v$ for any vertex $v\in x$.

    \begin{prp}\label{PR:spectrum-dependencies}
        In the previous setting:
        \begin{enumerate}
            \item[(i)] If $G$ preserves orientation, then for all $i>0$, there exists a (non-canonical)
            unitary natural isomorphism $\Omega_i^+\cong\Omega_i^-$. In particular, the non-oriented, oriented, and full
            $i$-dimensional spectra determine each other.
            \item[(ii)] If $\calX$ is pure of dimension $d$ and $G$ preserves ordering, then
            there is a (non-canonical) unitary natural  isomorphism $(\Omega_d^+)^{(d+1)!}\cong\llFlag$. In particular,
            the non-oriented $d$-dimensional spectrum determines the flag spectrum,
            and hence also the non-oriented, oriented, and full spectra in all dimensions $0\leq i\leq d$.
        \end{enumerate}
    \end{prp}

    \begin{proof}
        (i) Choose representatives $\{x_\alpha\}_{\alpha\in I}$
        for the $G$-orbits in $X^{(i)}$ and choose an arbitrary
        orientation $\mathsf{x}_\alpha\in\ori{\calX}^{(i)}$ for $x_\alpha$.
        For all $\Gamma\leq_{\calX}G$, define $u_{\Gamma\leftmod \calX}:\Omega_i^+(\Gamma\leftmod\calX)\to
        \Omega_i^-(\Gamma\leftmod\calX)$ by linearly extending
        $
        \e_{\Gamma g\sfx_\alpha}+\e_{\Gamma g\sfx_\alpha^\op}\mapsto \e_{\Gamma g\mathsf{x}_\alpha}-\e_{\Gamma g\mathsf{x}_\alpha^\op}
        $
        ($\alpha\in I$, $g\in G$). This is well defined since if $\gamma gx_\alpha= g'x_{\alpha'}$ ($\gamma\in\Gamma$, $g,g'\in G$),
        then $\alpha'=\alpha$, and $g'^{-1}\gamma g\in\Stab_G(x_\alpha)$.
        Since $G$ preserves orientation, $\Stab_G(\mathsf{x}_\alpha)=\Stab_G(x_\alpha)$, and hence $\gamma g\mathsf{x}_\alpha=g'\mathsf{x}_{\alpha'}$.
        It is straightforward to check that $u=\{e_X\}_{X\in \catC}$ is a unitary natural isomorphism from $\Omega_i^+$ to $\Omega_i^-$.
        All other assertions follow from the discussion before the proposition and Proposition~\ref{PR:summand-spectrum}(i), since
        $\Omega_i^\pm=\Omega_i^+\oplus \Omega_i^-$.

        (ii)
        Choose representatives  $\{x_\alpha\}_{\alpha\in I}$
        for the $G$-orbits in $X^{(d)}$. Since $\calX$ is pure of dimension $d$, for every $\alpha\in I$, there
        are are $(d+1)!$-maximal flags $(x_{-1}\subset x_0\subset \dots\subset x_d)$ with $x_d=x_\alpha$,
        call them $\{f_{\alpha,j}\}_{j=1}^{(d+1)!}$. Now, for all $\Gamma\leftmod \calX\in \catC$,
        define $u_{\Gamma\leftmod\calX}:\llFlag(\Gamma\leftmod\calX)\to \Omega_d^+(\Gamma\leftmod\calX)^{(d+1)!}$
        by sending $\e_{\Gamma gf_{\alpha,j}}$ to the image of $\e_{\Gamma gx_\alpha}$ in the $j$-th copy
        of $\Omega_d^+(\Gamma\leftmod\calX)$ in  $ \Omega_d^+(\Gamma\leftmod\calX)^{(d+1)!}$.
        This is well-defined since if $\gamma gf_{\alpha,j}= g'f_{\beta,k}$
        ($\gamma\in\Gamma$, $g,g'\in G$),
        then   $g'^{-1}\gamma gx_\alpha=x_\beta$, and so $\alpha=\beta$ and $g'^{-1}\gamma g\in\Stab_{G}(x_\alpha)$,
        which means $\gamma gx_\alpha=g'x_\beta$. To see that $j=k$, notice that
        $G$ is order preserving and hence $g'^{-1}\gamma g$ fixes the vertices of $x_\alpha$. Thus,
        $f_{\alpha,j}=g'^{-1}\gamma g f_{\alpha,j}=f_{\beta,k}$, so $j=k$. It is easy
        to check that $u=\{e_X\}_{X\in \catC}$ is a unitary natural isomorphism from $\llFlag$ to $(\Omega_d^+)^{(d+1)!}$.
        The other assertions follow from Proposition~\ref{PR:summand-spectrum}(i) and Example~\ref{EX:summand-functor}.
    \end{proof}

    \begin{example}
        Let $G=\nPGL{F}{d}$ and $\calX=\calB_d(F)$ be as in Chapter~\ref{sec:ramanujan-complexes}.
        Then $G$ preserves orientation if and only if $d$
        is odd.
        Indeed, recall that there is color
        function $C_0:\calB_d(F)^{(0)}=G/K\to \Z/d\Z$  given by $C_0(\quo{g}K)=\nu_F(\det g)+d\Z$ for all $g\in\nGL{F}{d}$.
        Since adjacent vertices have different colors, $C_0$ is injective
        on any cell $x=\{v_0,\dots,v_i\}\in \calB_d(F)$.
        Suppose now that $g\in \nGL{F}{d}$ satisfies $\quo{g}x=x$ and let $r=\nu_F(\det g)+d\Z$.
        Then $C_0(\quo{g}v_s)=C_0(v_s)+r$ for all $0\leq s\leq i$, and hence $\quo{g}^d$ fixes $v_0,\dots,v_i$.
        When $d$ is odd, the orbits that $\quo{g}$ induces on $\{v_0,\dots,v_i\}$ must have odd sizes,
        hence $\quo{g}$ induces an even permutation on $\{v_0,\dots,v_i\}$.
        On the other hand, if $d$ is even, we can take $x$ to be a $(d-1)$-cell
        and $g$ to be a matrix of determinant $\pi$. The permutation $\quo{g}$ induces
        on  $\{v_0,\dots,v_{d-1}\}$ is a cycle of size $d$, which is odd. (The existence
        of $x$ and $g$ as above is easy to see.)
    \end{example}

    \begin{example}
        Let $\calX=\calB_d(F)$ be as in Chapter~\ref{sec:ramanujan-complexes} and let $G=\im(\nSL{F}{d}\to\nPGL{F}{d})$.
        Then $G$ preserves order because it preserves the coloring $C_0$ of the vertices (see Chapter~\ref{sec:ramanujan-complexes}).
        More generally, if
        $\bfG$ is a simply-connected almost simple algebraic group over $F$ with affine building $\calB$
        and $G=\im(\bfG(F)\to\Aut(\calB))$ (cf.\ Example~\ref{EX:G-complex-building-III}),
        then $G$ preserves
        order because it preserves the type of the vertices.
    \end{example}

\subsection{A Note On Isomorphic Complexes}

    Let $\catC\subseteq\catSimp$ be a skeletally small subcategory,
    let $F:\catC\to\catPHil$ be a functor, and let $A$ be an algebra of
    $(\catC,F)$-operators (e.g.\ $\Alg{\catC,F}{}$).
    The reader should take notice that if $X,X'\in\catC$ are isomorphic
    simplicial complexes, then $\Spec_{A}(X)$ and $\Spec_{A}(X')$ may still differ,
    because $X$ and $X'$ need not be isomorphic in $\catC$. However, when $\catC$
    is a \emph{full} subcategory of $\catSimp$, and $F$ sends simplicial isomorphisms to
    continuous isomorphisms (e.g.\ if $F\in\{\Omega^+_i, \Omega_i^-, \Omega_i^\pm,\llFlag\}$), isomorphic simplicial complexes
    have the same $A$-spectrum (cf.\ Proposition~\ref{PR:iso-implies-unitary-iso}).

\medskip

    Consider now the case  $\catC=\catC(G,\calX)$ and $A=\Alg{\catC,F}{}$ where $\calX$ is a $G$-complex.
    The category $\catC$ is usually not full in $\catSimp$, but under certain assumptions,
    one can show that isomorphic complexes in $\catC$ have the same $A$-spectrum.

    \begin{prp}\label{PR:independence-in-cover-map}
        Assume $\calX$ is  simply-connected when realized as a topological space
        and $G=\Aut(\calX)$. Let $F:\catSimp\to\catPHil$ be a functor
        such that:
        \begin{enumerate}
        \item[(1)] $Fp_\Gamma:F\calX\to F(\Gamma\leftmod\calX)$ is onto for all
        $\Gamma\leq_{\calX} G$ (notation as in Definition~\ref{DF:G-complex-category}).
        \item[(2)] $Fu$ is continuous for every isomorphism $u$ in $\catSimp$.
        \end{enumerate}
        (For example, this holds when $F\in\{\Omega^+_i, \Omega_i^-, \Omega_i^\pm,\llFlag\}$.)
        Then isomorphic complexes in $\catC:=\catC(G,\calX)$ have the same
        $\Alg{\catC,F}{}$-spectrum.
    \end{prp}

    \begin{proof}
        Let $\Gamma,\Gamma'\leq_{\calX}G$, let $X=\Gamma\leftmod\calX$,
        $X'=\Gamma'\leftmod\calX$, and let $u:X\to X'$ be an
        isomorphism. It is enough to prove that $Fu:FX\to FX'$
        is a homomorphism of $\Alg{\catC,F}{}$-modules, since then, by (2) and Proposition~\ref{PR:iso-implies-unitary-iso},
        $FX\cong FX'$ as unitary representations of $\Alg{\catC,F}{}$.
        Let $a\in\Alg{\catC,F}{}$.
        Since $\calX$ is simply-connected, there is $g\in \Aut(\calX)=G$ such
        that $p_{\Gamma'}\circ g=u\circ p_\Gamma$. In particular, $u\circ p_\Gamma$
        is a morphism in $\catC$, and hence $a_{X'}\circ F(u\circ p_\Gamma)=F(u\circ p_\Gamma)\circ a_{\calX}$.
        Now, let $\psi\in FX$.
        By (1), there is $\xi\in F\calX$ such that $(Fp_\Gamma) \xi=\psi$.
        Thus,
        \begin{align*}
        (Fu)(a_X\psi)&=(Fu)(a_X((Fp_\Gamma)\xi))=(Fu)((Fp_\Gamma)(a_{\calX}\xi))
        =(F(u\circ p_\Gamma))(a_\calX\xi)\\
        &=
        a_{X'}((F(u\circ p_\Gamma))\xi)=a_{X'}((Fu)((Fp_\Gamma) \xi))=a_{X'}((Fu)\psi),
        \end{align*}
        as required.
    \end{proof}

    \begin{remark}\label{RM:auto-group-of-Bd}
    The conditions of Proposition~\ref{PR:independence-in-cover-map} are \emph{not}
    satisfied when $\calX=\calB_d(F)$ and $G=\nPGL{F}{d}$ (see Chapter~\ref{sec:ramanujan-complexes}), because $\Aut(\calX)$ is strictly larger
    than $G$. \gap{}%
    More precisely, when $d>2$, $G$ is of index $2$ in $\Aut(\calB_d(F))$.
    A representative for the non-trivial $G$-coset
    in $\Aut(\calB_d(F))$ is the automorphism $\tau$ given by $\tau(gK)=(g^{-1})^{\mathrm{T}}K$, where $K=\nPGL{\calO}{d}$
    and $\mathrm{T}$ denotes matrix transposition.
    A simple modification of the proof Proposition~\ref{PR:independence-in-cover-map}
    shows that every isomorphism class of $\catSimp$ with a representative in $\catC(G,\calB_d(F))$
    gives rise to two possible spectra\footnote{
        Interestingly,  this mild issue seems to be unaddressed in all
        papers we have seen.
    } (which may coincide in special cases).
    In more detail,
    let $\Gamma\leq_{\calX}G$.
    Using Corollary~\ref{CR:distance-condition}, it easy to check that $\tau^{-1}\Gamma \tau\leq_{\calX}G$,
    and we  have $\Gamma\leftmod\calB_d(F)\cong \tau\Gamma \tau^{-1}\leftmod\calB_d(F)$ via $\Gamma v\mapsto \tau\Gamma \tau^{-1}(\tau v)$.
    The spectra of $\Gamma\leftmod\calB_d(F)$ and $\tau\Gamma \tau^{-1}\leftmod\calB_d(F)$ are the two spectra associated
    with the isomorphism class of $\Gamma\leftmod\calB_d(F)$. If $S_1:=\Spec_{0}(\Gamma\leftmod\calB_d(F))$
    and $S_2:=\Spec_{0}(\tau\Gamma \tau^{-1}\leftmod\calB_d(F))$
    are realized as subsets of $\C^{d-1}$ as in Example~\ref{EX:zero-dim-spec-of-Bd},
    then $S_1=\{(z_{d-1},\dots,z_1)\where (z_1,\dots,z_{d-1})\in S_2\}$. The technical proof is omitted.

    When $d=2$, one can show by arguing as in  Example~\ref{EX:tree-spectrum} that $\Alg{G,\calB_d(F)}{0}=\C[a_0]$
    and $\Alg{G,\calB_d(F)}{1}=\C[a_1]$, so the oriented $0$- and $1$-dimensional spectrum are in fact determined
    by the isomorphism class of the quotient of $\calB_2(F)$.
    \end{remark}

    \begin{remark}
        One way to overcome the problem that the spectrum of a $G$-quotient
        of  $\calX$ is not determined by its isomorphism class is by specifying the cover map.
        Indeed, for every $G$-quotient $Y$ and a cover map $f:\calX\to Y$, there
        exists  unique $\Gamma\leq_{\calX} G$
        and isomorphism $s:\Gamma\leftmod\calX\to Y$ such that $s(\Gamma v)=f(v)$ for all $v\in\vrt{\calX}$
        (cf.\ \ref{subsec:quotienst-of-simp-comps}).
        We may therefore define the spectrum of $(Y,f)$ to be the spectrum of $\Gamma\leftmod\calX$.
    \end{remark}

\section{Optimal Spectrum}
\label{sec:optimal-spectrum}

    Recall that
    a connected $k$-regular graph $X$ is called Ramanujan if the spectrum of its adjacency operator,
    denoted $\Spec_0(X)$, is contained in $[-2\sqrt{k-1},2\sqrt{k-1}]\cup\{\pm k\}$.
    As explained in the introduction, the definition is motivated by the Alon-Boppana Theorem (\cite{Nilli91}, see also \cite[Pr.~4.2]{LubPhiSar88}),
    stating that for any
    $\veps>0$, only finitely many isomorphism classes of $k$-regular graphs $X$ satisfy
    the tighter bound $\Spec_0(X)\subseteq [-2\sqrt{k-1}+\veps,2\sqrt{k-1}-\veps]\cup\{\pm k\}$.
    Ramanujan graphs can therefore be considered as having optimal, or smallest possible, spectrum. Alternatively, Ramanujan
    graphs can be regarded as spectral approximations of their universal cover, the $k$-regular tree $\calT_k$,
    since $\Spec_0(\calT_k)=[-2\sqrt{k-1},2\sqrt{k-1}]$ (see \cite[p.~252, Apx.~3]{Sunada88}; this result
    seems to go back as far as \cite[Th.~3]{Kest59}).

    The Alon-Boppana Theorem has been extended from $k$-regular graphs to further simplicial structures,
    particularly
    for non-regular graphs (\cite{GrigZuk99}) and quotients of the building $\calB_d(F)$ of Chapter~\ref{sec:ramanujan-complexes}
    (\cite[Th.~4.1]{Li04}, quoted as Theorem~\ref{TH:Li-a} above).
    The general flavor is always that the spectrum of the universal cover is the smallest possible
    in a certain sense which depends on the exact statement.

    We remind  the reader that reason for wanting spectrum ``as small as possible'' is because
    ``small'' spectrum usually manifests in good combinatorial properties; see \cite{Lubotzky14} (for graphs)
    and \cite{EvraGolLub14},
    \cite{ParRosTes13}, \cite{GolPar14}, \cite{Golubev13} (for simplicial complexes), for instance.

\medskip

    In this chapter, we prove a theorem in the spirit of the Alon-Boppana Theorem
    for $G$-quotients of an almost transitive $G$-complex $\calX$  (see~\ref{subsec:G-complex}).
    Our theorem applies to spectrum taken
    with respect to any semi-elementary functor $F$ (see~\ref{subsec:case-of-Alg-G-X}).
    In particular, it applies to the high dimensional spectra introduced at the end of \ref{subsec:spectrum-of-simp}.
    When $G=\nPGL{F}{d}$ and $\calX=\calB_d(F)$ as in Chapter~\ref{sec:ramanujan-complexes}
    and one considers the $0$-dimensional spectrum, our theorem is just Li's Theorem (Theorem~\ref{TH:Li-a} above).
    The proof relies on
    results from
    Chapter~\ref{sec:involutary-algebras}; the main novelty is treating  spectrum points
    corresponding to irreducible representations of dimension greater than $1$.

    With a notion of an optimal $F$-spectrum at hand,
    we proceed with identifying the \emph{trivial $F$-spectrum} of $G$-quotients of $\calX$.
    This  leads to  a notion of  $G$-quotients of $\calX$ which are Ramanujan with respect to $F$.
    Taking $F$ to be $\Omega_i^+$ (see~\ref{subsec:orientation}),
    we arrive at the notion $G$-quotients of $\calX$ which are  Ramanujan in dimension $i$.
    In
    case $\calX=\calT_k$ and $G=\Aut(\calT_k)$ as above, or $\calX=\calB_d(F)$ and $G=\nPGL{F}{d}$ as in Chapter~\ref{sec:ramanujan-complexes},
    \emph{Ramanujan in dimension $0$} coincides with
    the usual meaning of Ramanujan in the literature (see \cite{Lubotzky14}, \cite{LubSamVi05} for instance).

\subsection{Lower Bounds on The Spectrum}
\label{subsec:alon-boppana-theorem}

    Throughout,
    $\calX$ is an almost transitive $G$-complex (\ref{subsec:G-complex}) and
    $\catC=\catC(G,\calX)$ (Definition~\ref{DF:G-complex-category}).

    \begin{thm}\label{TH:Alon-Boppana-I}
        Let $F:\catC\to \catPHil$ be a semi-elementary functor
        (e.g.\ $\Omega_i^+$, $\Omega_i^-$, $\Omega_i^\pm$ or $\llFlag$; cf.~\ref{subsec:case-of-Alg-G-X}),
        and let $A$ be an algebra of $(\catC,F)$-operators
        (cf.~\ref{subsec:spectrum-of-simp}).
        Let  $\{\Gamma_\alpha\}_{\alpha\in I}$ be a family
        of subgroups of $G$ such that $\Gamma_\alpha\leq_{\calX} G$ for all $\alpha\in I$,
        and one of the following
        conditions, which are equivalent, is satisfied:
        \begin{enumerate}
        	\item[(1)] For every compact    $C\subseteq G$ with $1\in C$, there exist $\alpha\in I$
        	and $g\in G$ such that
        	such that $C\cap g^{-1}\Gamma_\alpha g=1$
        	\item[(2)] For every $n\in\N$,
        	there exists $\alpha\in I$ and $v\in\vrt{\calX}$
        	such that quotient map $\calX\to\Gamma_\alpha\leftmod\calX$
        	is injective on the combinatorial ball $\Ball_\calX(v,n)$ (see \ref{subsec:simp-complex}).
        \end{enumerate}
        Then
        \[
        \overline{\bigcup_{\alpha}\Spec_{A}(\Gamma_\alpha\leftmod \calX)}\supseteq\Spec_{A}(\calX)\ .
        \]
    \end{thm}

    We first prove the following lemma.

    \begin{lem}\label{LM:injective-image}
        Suppose $\{\Gamma_\alpha\}_{\alpha\in I}$ satisfies
        condition (1) of  Theorem~\ref{TH:Alon-Boppana-I},
        and let $S:\catC\to \catSet$ be a functor satisfying conditions (E1)--(E4) of
        Definition~\ref{DF:elementary-functor}.
        Then for any finite subset $T\subseteq S\calX$, there exist $\alpha\in I$
        and $g\in G$
        such that $Sp_{\Gamma_\alpha}:S\calX\to S(\Gamma_\alpha \leftmod\calX)$ is injective on $gT$.
    \end{lem}

    \begin{proof}
        We may assume $T\neq\emptyset$.
        Let $C=\bigcup_{x,y\in T}\{g\in G\suchthat gx=y\}$. Since $T$ is finite and stabilizers
        of elements in $S\calX$ are compact, $C$ is compact, and $1\in C$ since $T\neq\emptyset$.
        Therefore, there is $\alpha\in I$
        and $g\in G$ such that $g^{-1}\Gamma_\alpha g \cap C=1$. Now, if $x,y\in T$ and $\gamma\in\Gamma_\alpha$
        satisfy $\gamma gx=gy$,
        then $g^{-1}\gamma g\in g^{-1}\Gamma_\alpha g\cap C=1$, so $gx=gy$. This proves that the quotient map
        $S\calX\to \Gamma_\alpha\leftmod S\calX$
        is injective on $gT$.
    \end{proof}

    \begin{proof}[Proof of Theorem~\ref{TH:Alon-Boppana-I}]
		 We first prove the equivalence of (1) and (2).
		 If (1) holds, then (2) follows by applying Lemma~\ref{LM:injective-image}
		 with $S:\catSimp\to \catSet$ being the forgetful functor
		 and $T=\Ball_\calX(v,n)$.
		 Suppose now that (2) holds, and let $v_1,\dots,v_t$ be representatives
		 for $G\leftmod\vrt{\calX}$ (which is finite since $\calX$ is almost $G$-transitive).
		 Let $C\subseteq G$ be  compact with $1\in C$. Since $G$ acts continuously on $\vrt{\calX}$
		 (when given the discrete topology), $Cv_i$ is finite for all $i$. Choose
		 $n\in\N$ such that $Cv_i\subseteq \Ball_\calX(v_i,n)$ for all $1\leq i\leq t$.
		 There exists $\alpha\in I$ and $u\in\vrt{\calX}$
		 such that
		 the quotient map $\calX\to\Gamma_\alpha\leftmod\calX$ is injective on $\Ball_\calX(u,n)$.
		 By construction, there is $g\in G$ and $1\leq i\leq t$ such that
		 $u=gv_i$. Now, if $h\in g^{-1}\Gamma_\alpha g\cap C$, then $\Gamma_\alpha gv_i=
		 \Gamma_\alpha ghv_i$ and $gv_i,ghv_i\in g\Ball_\calX(v_i,n)=\Ball_\calX(u,n)$,
		 hence $gv_i=ghv_i$. It follows that $h\in\Stab_G(v_i)\cap g^{-1}\Gamma_\alpha g$,
		 so $h=1$ by Proposition~\ref{PR:quotient-cover-sufficient-nec-conds}
		 and Corollary~\ref{CR:Gamma-is-discrete}.

		We now prove the main statement. Let $F':\catC\to \catPHil$ be a functor
        such that $F\oplus F'$ is elementary. By Remark~\ref{RM:summand-transition}, we may view
        $A$ as an algebra of $(F\oplus F',\catC)$-operators. We may therefore assume
        that $F$ is elementary, and write
        $F=\llf\circ S$
        where $S$ satisfies conditions (E1)--(E4) in Definition~\ref{DF:elementary-functor}.

        Let $[V]\in\Spec_{A}(\calX)$. We need to show that
        $[V]$ belongs to  $\overline{\bigcup_{\alpha}\Spec_{A}(\Gamma_\alpha\leftmod \calX)}$.
        Theorem~\ref{TH:Adj-algebra-descrition}(ii) implies that $A$ acts continuously
        on $\bigoplus_{\alpha\in I}F(\Gamma_\alpha\leftmod\calX)$. Therefore,
        by Theorem~\ref{TH:direct-sum-spectrum},
        \[
        \Spec_{A}\Big(\bigoplus_{\alpha\in I}AF(\Gamma_\alpha\leftmod \calX)\Big)
        =\overline{\bigcup_{\alpha}\Spec_{A}(AF(\Gamma_\alpha\leftmod\calX))}\ ,
        \]
        so we only need to show that $V\wc \bigoplus_{\alpha\in I}AF(\Gamma_\alpha\leftmod \calX)$.
        Let $v\in\sphere{V}$, $\veps>0$ and $Y\fsubseteq A$.
        Then there is $\psi\in AF\calX=A\llf(S\calX)$
        such that $\abs{\Trings{av,v}-\Trings{a\psi,\psi}}<\veps$ for all $a\in Y$.
        Let $T=\big(\bigcup_{a\in Y}\supp(a\psi)\big)\cup\supp(\psi)$.
        By Lemma~\ref{LM:injective-image}, there is $\alpha\in I$
        such that $Sp_{\Gamma_\alpha}|_{T}$ is injective.
        Write $p=p_{\Gamma_\alpha}$.
        Then $\Trings{(Fp)(a\psi),(Fp)\psi}=\Trings{a\psi,\psi}$ for all $a\in Y$.
        Furthermore, since any $a\in Y$ is a natural transformation from $F$ to
        itself, $(Fp)(a\psi)=a((Fp)\psi)$ and $(Fp)\psi\in AF(\Gamma_\alpha\leftmod\calX)$.
        Therefore,
        $\abs{\Trings{av_j,v_j}-\Trings{a((Fp)\psi),(Fp)\psi}}=\abs{\Trings{av,v}-\Trings{a\psi,\psi}}<\veps$
        for all $a\in Y$. Viewing $(Fp)\psi$ as an element of $\bigoplus_\alpha AF(\Gamma_\alpha\leftmod\calX)$,
        we see that
        $V\wc \bigoplus_\alpha F(\Gamma_\alpha\leftmod\calX)$, as required.
    \end{proof}

    \begin{example}
        Let be $\Gamma_1\supseteq \Gamma_2\supseteq\Gamma_3\supseteq\dots$ be subgroups of $G$ satisfying
        $\bigcap_n\Gamma_n=1$ and $\Gamma_1\leq_{\calX} G$. Then Theorem~\ref{TH:Alon-Boppana-I} applies
        to $\{\Gamma_n\}_{n\in\N}$. Indeed, by Corollary~\ref{CR:Gamma-is-discrete},
        $\Gamma_n\leq_{\calX}G$ for all $n\in\N$,
        and $\Gamma_1$ is discrete by Proposition~\ref{PR:quotient-cover-sufficient-nec-conds}.
        We now show condition (1): Since $\Gamma_1$ is discrete
        and $G$  is an $\ell$-group, there is
        $K\co G$ such that $K\cap \Gamma_1=1$.
        Let $C\subseteq G$ be a compact  set with $1_G\in C$.
        We may assume $C\subseteq K\Gamma_1$
        without loss of generality. Since $K\cap \Gamma_1=1$, every $c\in C$ can be written
        as $k\gamma$ for unique $k\in K$ and $\gamma\in \Gamma$. In addition, the compactness of $C$ implies that
        there are $\gamma_1,\dots,\gamma_r\in\Gamma-\{1\}$ such that
        $C-K\subseteq \gamma_1K\cup\dots\cup \gamma_rK$.
        Choosing $n$ such that $\gamma_1,\dots,\gamma_r\notin\Gamma_n$, we get $\Gamma_n\cap C\subseteq
        \Gamma_1\cap K=1$.
    \end{example}

    \begin{cor}\label{CR:Alon-Boppana-comm}
        Let $G$, $\calX$, $\catC$, $F$ and $\{\Gamma_\alpha\}_{\alpha\in I}$ be as
        in Theorem~\ref{TH:Alon-Boppana-I},
        and suppose $a_1,\dots,a_t\in\Alg{\catC,F}{}$ satisfy $a_ia_j=a_ja_i$ and
        $a_ia_j^*=a_j^*a_i$ for all $i,j$.
        Write $V_\alpha=F(\Gamma_\alpha\leftmod \calX)$ and $V=F\calX$. Then
        \[
        \overline{\bigcup_{\alpha}\Spec(a_1|_{V_\alpha},\dots,a_t|_{V_\alpha})}\supseteq\Spec(a_1|_{V},\dots,a_t|_{V})\ .
        \]
    \end{cor}

    \begin{proof}
        Let $A$ be the  unital $*$-subalgebra of $\Alg{\catC,F}{}$ generated
        by $a_1,\dots,a_t$. Then $A$ is commutative.
        The corollary therefore follows from
        Proposition~\ref{PR:unitary-dual-topology} and Theorem~\ref{TH:Alon-Boppana-I}.
    \end{proof}

    When $\calX=\calB_d(F)$, $G=\nPGL{F}{d}$ and $F=\Omega_0^+$,
    Corollary~\ref{CR:Alon-Boppana-comm} is a result of Li \cite[Th.~4.1]{Li04} (quoted as Theorem~\ref{TH:Li-a} above);
    this follows from Example~\ref{EX:zero-dim-spec-of-Bd}.

\medskip

	The argument in the proof of Theorem~\ref{TH:Alon-Boppana-I} can be applied in a  broader setting.
	For example, one can similarly prove:

    \begin{thm}\label{TH:Alon-Boppana-II}
        Let $\catC\subseteq \catSimp$ be a subcategory,
        let $F\in\{\Omega_i^+,\Omega_i^-,\Omega_i^\pm,\llFlag\where i\in\N\cup\{0\}\}$,
        and let $A$ be an algebra of $(\catC,F)$-operators.
        Let $\{X_\alpha\}_{\alpha\in I}\subseteq\catC$, $\calX\in\catC$ and
        assume that:
        \begin{enumerate}
        \item[(1)] $A$ acts continuously on $F\calX$ and $\bigoplus_\alpha FX_\alpha$.
        \item[(2)] For any finite set $T\subseteq\vrt{\calX}$, there exists
        $\alpha\in I$ and a morphism $p:\calX\to X_\alpha$ in $\catC$ such that $p|_{T}$ is injective.
        \end{enumerate}
        Then
        $
        \overline{\bigcup_\alpha\Spec_{A}(X_\alpha)}\supseteq\Spec_{A}(\calX)
        $.
    \end{thm}

\subsection{Trivial Spectrum}
\label{subsec:trivial-spec}

    Let $\calX$ be an almost transitive $G$-complex, let
    $\catC=\catC(G,\calX)$, let $F:\catC\to \catPHil$ be semi-elementary (Definition~\ref{DF:elementary-functor}),
    and let $A$ be an algebra of $(\catC,F)$-operators.
    It of interest
    to construct an infinite family of non-isomorphic $G$-quotients $\{\Gamma_n\leftmod\calX\}_{n\in\N}$
    such that $\quo{\bigcup_n\Spec_A(\Gamma_n\leftmod\calX)}$ is as small as possible.
    Theorem~\ref{TH:Alon-Boppana-I} implies that under mild assumptions, this set  must contain at least $\Spec_A(\calX)$.
    However, if we consider  finite $G$-quotients, then
    this is not the only limitation.  The $A$-spectrum of
    a finite $G$-quotient must contain at least some points
    from a family called the \emph{trivial $A$-spectrum},
    which we now describe.


\medskip

    Let $N\leq G$ be a finite-index open subgroup, and let $\Gamma\leq_{\calX}G$
    be such that $\Gamma\leq N$.
    Then $(F\calX)(N)=\{\vphi-g\vphi\where \vphi\in F\calX,\,g\in N\}$
    is an $A$-submodule of $F\calX$, and hence $(F\calX)_N=F\calX/(F\calX)(N)$
    has a
    structure
    of a left $A$-module.

    \begin{lem}\label{LM:coinvariants-module-structure}
        There is an inner product on $A(F\calX)_N=(AF\calX)_N$ such that $A(F\calX)_N$ becomes
        a unitary representation of $A$. The unitary isomorphism class of $A(F\calX)_N$
        is independent of the inner product.
    \end{lem}

    \begin{proof}
        Choose $F'$ such that $F\oplus F'$ is elementary.
        By Remark~\ref{RM:summand-transition}, we may replace $F$ with $F\oplus F'$ and assume $F$ is elementary.
        Write $F=\llf\circ S$  as in Definition~\ref{DF:elementary-functor}.
        By Theorem~\ref{TH:Adj-algebra-descrition}(i), we may identify $A$ with an involutary
        subalgebra of $\End_G(F\calX)$. The proof of this theorem implies that
        for \emph{any} $H\leq G$, one can endow $\llf(H\leftmod S\calX)$
        with a left $A$-module structure given by $a\e_{Hx}=\sum_{y\in S\calX}\alpha_y\e_{Hy}$
        where $\{\alpha_y\}_{y}\subseteq\C$ are determined by $a\e_x=\sum_y\alpha_y\e_y$.
        Moreover, $\llf(H\leftmod S\calX)$ is a pre-unitary representation of $A$.
        Now, as in the proof of Proposition~\ref{PR:S-quotient-check}(ii), there is an isomorphism $(F\calX)_N\cong \llf(N\leftmod S\calX)$
        given by linearly extending $\e_x+(F\calX)(N)\mapsto \e_{N x}$,
        and it is straightforward to check that this isomorphism is a homomorphism of $A$-modules.
        This shows the existence of an inner product on $A(F\calX)_N$ such that $A(F\calX)_N\in\Rep[pu]{A}$.
        Since $[G:N]$ and $G\leftmod S\calX$ are finite, $\dim (F\calX)_N=\dim \llf(N\leftmod S\calX)<\infty$,
        and hence $A(F\calX)_N\in\Rep[u]{A}$. The
        unitary isomorphism class of $A(F\calX)_N$
        is uniquely determined by Proposition~\ref{PR:iso-implies-unitary-iso}.
    \end{proof}

    Write $\frakT_{A,N}=\Spec_A(A(F\calX)_N)$.

    \begin{prp}\label{PR:trivial-spectrum-containment}
        Let $\Gamma\leftmod\calX$ be a finite $G$-quotient
        of $\calX$ such that $\Gamma\leq N$. Then $\frakT_{A,N}\subseteq \Spec_A(\Gamma\leftmod\calX)$.
        Furthermore, if $N'\leq N$ is open of finite index, then $\frakT_{A,N}\subseteq \frakT_{A,N'}$.
    \end{prp}

    \begin{proof}
        Identify $F(\Gamma\leftmod\calX)$ with $(F\calX)_{\Gamma}$ using
        Proposition~\ref{PR:S-quotient-check}(ii), and consider
        the surjective $A$-homomorphism $P:A(F\calX)_{\Gamma}\to A(F\calX)_N$ given by $P(\vphi+(F\calX)(\Gamma))=
        \vphi+(F\calX)(N)$. Since $\Gamma\leftmod\calX$ is finite and $F$ is semi-elementary,
        $A(F\calX)_\Gamma$ is finite dimensional, and hence $(\ker P)^\perp\cong A(F\calX)_N$
        as $A$-modules. By Proposition~\ref{PR:iso-implies-unitary-iso},
        $(\ker P)^\perp\cong A(F\calX)_N$ as unitary representations. Since $AF(\Gamma\leftmod\calX)\cong\ker P\oplus (\ker P)^\perp$,
        it follows that $\frakT_{A,N}=\Spec_A((\ker P)^\perp)\subseteq\Spec_A(AF(\Gamma\leftmod\calX))=\Spec_A(\Gamma\leftmod\calX)$.

        The second part of the proposition is shown similarly.
    \end{proof}
    
    We shall prove a partial converse to Proposition~\ref{PR:trivial-spectrum-containment}
    in Chapter~\ref{sec:spectrum-and-reps} (Proposition~\ref{PR:converse-trivial-spectrum-containment}).

    \begin{cor}
        For any finite $G$-quotient $\Gamma\leftmod\calX$, we have $\frakT_{A,G}\subseteq\Spec_A(\Gamma\leftmod\calX)$.
    \end{cor}

    Proposition~\ref{PR:trivial-spectrum-containment} motivates the following definition.

    \begin{dfn}
        The \emph{trivial $A$-spectrum} is the set $\frakT_A=\bigcup_{N}\frakT_{A,N}$, where $N$
        ranges over all finite index open subgroups of $G$.
    \end{dfn}

    \begin{remark}\label{RM:operator-trivial-spectrum}
        Assume $a_1,\dots,a_t\in \Alg{\catC,F}{}$ satisfy $a_ia_j=a_ja_i$ and $a_ia_j^*=a_j^*a_i$ for all
        $i,j$. Then one can set $\frakT_{a_1,\dots,a_t,N}=\Spec(a_1|_{A(F\calX)_N},\dots,a_t|_{A(F\calX)_N})$
        and define the trivial common spectrum of $(a_1,\dots,a_t)$ to be
        $\frakT_{a_1,\dots,a_t}:=\bigcup_N\frakT_{a_1,\dots,a_t,N}$.
        The  the unital $*$-subalgebra $A$ of $\Alg{\catC,F}{}$ spanned by $a_1,\dots,a_t$ is commutative,
        and if
        one identifies the $A$-spectrum with the common spectrum of $a_1,\dots,a_t$
        as in Proposition~\ref{PR:unitary-dual-topology},
        then $\frakT_A$ corresponds to $\frakT_{a_1,\dots,a_t}$.
    \end{remark}

    \begin{remark}
        We actually did not use the fact that $N$ is open anywhere.
        However, restricting to open subgroups makes no difference, since for any finite-index subgroup
        $N\leq G$, we have $(F\calX)_N=(F\calX)_{\quo{N}}$, where $\quo{N}$ is the closure of $N$ in $G$.
        It is enough to check this when $F$ is elementary, in which case the claim easily boils down
        to showing that $N\leftmod S\calX=\quo{N}\leftmod S\calX$, where $F\cong \llf\circ S$ as in Definition~\ref{DF:elementary-functor}.
        It is well-known that $\quo{N}=\bigcap_U NU$ where $U$ ranges over all neighborhoods of $1_G$.
        Since
        stabilizers of elements of $S\calX$ are open, every $x\in S\calX$ is stabilized by some neighborhood
        $U$ of $1_G$, and thus $N\leftmod S\calX=\quo{N}\leftmod S\calX$.
    \end{remark}

    \begin{example}\label{EX:trivial-spec-tree}
        Let $\calX$ be a $k$-regular tree, let $G=\Aut(\calX)$, let $\catC=\catC(G,\calX)$,
        and let  $a_0\in A_0:=\Alg{\catC,\Omega_0^+}{}$ be the vertex
        adjacency operator. Using Example~\ref{EX:k-regular-graphs-basic-spectrum}, we
        identify the $0$-dimensional spectrum
        with the spectrum of $a_0$.
        It is well-known (\cite{Tits70} or \cite{Moller91}) that
        $G$ has only one proper finite-index subgroup, denoted $H$,
        which is of index $2$ and consists of the elements preserving  the canonical $2$-coloring of $\calX^{(0)}$.
        The trivial spectrum $\frakT_{A_0}$ is therefore the spectrum of $a_0$ on $\Omega_0^+(\calX)_H\cong \llf(H\leftmod \calX^{(0)})$.

        Let $u$ and $v$ be representatives for
        the $H$-orbits of $\calX^{(0)}$. Then $\{\e_{Hu},\e_{Hv}\}$ is a basis of $\llf(H\leftmod \calX^{(0)})$
        and the action of $a_0$ with respect to this basis is easily seen to be given by $\smallSMatII{0}{k}{k}{0}$
        (one can regard $H\leftmod\calX$ as a graph with two vertices connected by $k$ edges).
        Thus, $\frakT_{A_0}=\Spec(a_0|_{\llf_0(H\leftmod \calX)})=\{\pm k\}$, and we have recovered the usual
        definition of trivial spectrum for $k$-regular graphs.

        If one identifies the $1$-dimensional oriented spectrum with the spectrum of
        the edge adjacency operator $a_1\in A_1:=\Alg{\catC,\Omega_1^+}{}$ as in Example~\ref{EX:tree-spectrum},
        then
        one similarly finds that $\frakT_{A_1}=\{2k-2\}$.
    \end{example}

    \begin{example}\label{EX:trivial-spec-complex}
        Let $G=\nPGL{F}{d}$ and $\calX=\calB_d(F)$ be as in Chapter~\ref{sec:ramanujan-complexes},
        and let $\catC=\catC(G,\calX)$.
        By Example~\ref{EX:zero-dim-spec-of-Bd},
        $A_0:=\Alg{\catC,\Omega_0^+}{}=\C[a_1,\dots,a_{d-1}]$, and hence
        we can identify the $0$-dimensional spectrum with the common spectrum
        of $a_1,\dots,a_{d-1}\in A_0$. 
        It is well-known that any normal subgroup of $\nGL{F}{d}$ either contains $\nSL{F}{d}$ or is contained in the center
        of $\nGL{F}{d}$. Since all finite index subgroups contain
        a normal finite index subgroup, any  finite index subgroup of $G=\nPGL{F}{d}$ contains $N:=\im(\nSL{F}{d}\to\nPGL{F}{d})$.
        The trivial spectrum $\frakT_{A_0}$ is therefore the common spectrum of $a_0,\dots,a_{d-1}$
        on $\Omega_0^+(\calB_d(F))_N \cong\llf(\calB_d(F)^{(0)})_N\cong \llf(N\leftmod \calB_d(F)^{(0)})$.

        By Example~\ref{EX:simplicial-quotient-II},
        $N\leftmod\calB_d(F)$ is   a single simplex of dimension $d-1$.
        In fact, the vertex coloring $C_0:\calB_d(F)\to \Z/d\Z$
        descends to $N\leftmod\calB_d(F)$, so we can write $\vrt{N\leftmod\calB_d(F)}=\{v_0,\dots,v_{d-1}\}$
        with $C_0(v_i)=i+d\Z$. Let $\e_i:=\e_{v_i}\in \llf(N\leftmod\calB_d(F)^{(0)})\cong \Omega_0^+(N\leftmod\calB_d(F))$.
        An easy computation shows that  $a_i\e_j=\alpha_{i}\e_{j+i}$,
        where $\alpha_i$ is the number of vertices in $\calB_d(F)$ with color $(j+i)+d\Z$
        which are adjacent to some vertex of color $j+d\Z$ (this is independent of the vertex and $j$ since $G$
        acts transitively on $\calB_d(F)^{(0)}$). Let $q$ denote the cardinality
        of the residue field of $(F,\nu)$. It is well-known that $\alpha_i$ equals $\binom{d}{i}_q$,
        the number of $\F_q$-subspaces of dimension $i$ in $\F_q^d$
        (see the description of $\calB_d(F)$ through $\calO$-lattices in \cite[\S2.1]{LubSamVi05}, for instance).
        Let $\zeta\in\C$ be a primitive $d$-th root of unity.
        Then
        \[\{\vphi_k:=\zeta^0\e_0+\zeta^k\e_1+\zeta^{2k}\e_2+\dots+\zeta^{(d-1)k}\e_{d-1}\}_{0\leq k< d}\]
        is a basis of $\Omega_0^+(N\leftmod\calB_d(F))$ consisting of common eigenvectors of $(a_1,\dots,a_{d-1})$.
        The multi-eigenvalue corresponding to $\vphi_k$ is
        \[\textstyle
        \left(\binom{d}{1}_q\zeta^{-k},\binom{d}{2}_q\zeta^{-2k},\dots,\binom{d}{d-1}_q\zeta^{-(d-1)k}\right)\ .
        \]
        Replacing $\zeta$ with $\zeta^{-1}$, we get
        \[
        \frakT_{A_0}=\left\{{\textstyle\left(\binom{d}{1}_q\zeta^k,\binom{d}{2}_q\zeta^{2k},\dots,\binom{d}{d-1}_q\zeta^{(d-1)k}\right)}\where 0\leq k<d
        \right\}\ .
        \]
        Notice that for $d=2$, the complex $\calB_2(F)$ is a $(q+1)$-regular tree, and $\frakT_{A_0}=\{-q-1,q+1\}$,
        so this agrees with Example~\ref{EX:trivial-spec-tree}.

        According to Lubotzky, Samuels and Vishne \cite[\S2.3]{LubSamVi05}, the trivial spectrum is given by
        \[\{(\zeta^ks_1,\dots,\zeta^{(d-1)k} s_{d-1})\where 0\leq k<d\}\ ,\]
        where $s_i=q^{k(d-k)/2}\sigma_k(q^{(1-d)/2},q^{(3-d)/2},\dots,q^{(d-1)/2})$ and $\sigma_k$ denotes
        the $k$-th elementary symmetric function on $n$ variables. The argument in the proof
        of \cite[Pr.~2.1]{LubSamVi05} shows that $s_i=\binom{d}{i}_q$, so our trivial $A_0$-spectrum
        coincides with trivial spectrum in the sense of \cite{LubSamVi05}.
    \end{example}

\subsection{Ramanujan Quotients}
\label{subsec:Ramanujan-quotients}

    Let $\calX$, $G$, $\catC$, $F$, $A$ be as in \ref{subsec:trivial-spec}.

    \begin{dfn}\label{DF:Ramanujan-quotient}
        A $G$-quotient $\Gamma\leftmod\calX$ is \emph{$A$-Ramamnujan} if
        \[
        \Spec_A(\Gamma\leftmod\calX)\subseteq \frakT_A\cup \Spec_A(\calX)\ .
        \]
    \end{dfn}

    In view of Theorem~\ref{TH:Alon-Boppana-I} and Proposition~\ref{PR:trivial-spectrum-containment},
    $A$-Ramanujan $G$-quotients of $\calX$ can be regarded as $G$-quotients whose spectrum
    is as small as one might expect of a decent infinite
    family of $G$-quotient. Alternatively, they can be regarded as spectral approximations
    of $\calX$.

    Here are several possible specializations of Definition~\ref{DF:Ramanujan-quotient}:
    \begin{itemize}
        \item $\Gamma\leftmod\calX$ is \emph{$F$-Ramanujan} if it is $\Alg{F,\catC}{}$-Ramanujan (cf.\ \ref{subsec:spectrum-of-simp}).
        \item $\Gamma\leftmod\calX$ is \emph{Ramanujan in dimension $i$} if it is $\Omega_i^+$-Ramanujan (cf.\ \ref{subsec:orientation},
            \ref{subsec:adjacency-operators}).
        \item $\Gamma\leftmod\calX$ is \emph{flag Ramanujan} if it is $\llFlag$-Ramanujan (cf.~\ref{subsec:adjacency-operators}).
        \item $\Gamma\leftmod\calX$ is \emph{completely Ramanujan} if it is $F$-Ramanujan for any
        semi-elementary functor $F:\catC\to \catPHil$.
        (By Proposition~\ref{PR:Ramanujan-well-behaved} below, it is enough
        to check this when $F$ is elementary.)
    \end{itemize}

    \begin{example}\label{EX:classical-Ramamanujan}
        (i) Let $\calX$ be a $k$-regular tree, let $G=\Aut(\calX)$,
        let $\catC=\catC(G,\calX)$, and let  $a_0\in A_0:=\Alg{\catC,\Omega_0^+}{}$ be the vertex
        adjacency operator. As usual, we  identify the $0$-dimensional spectrum
        with the spectrum of $a_0$ (cf.\ Example~\ref{EX:k-regular-graphs-basic-spectrum}).
        It is well-known that $\Spec_0(\calX)=[-2\sqrt{k-1},2\sqrt{k-1}]$ (\cite[p.~252, Apx.~3]{Sunada88},
        for instance), and
        in  Example~\ref{EX:trivial-spec-tree}, we saw that the trivial spectrum is $\frakT_{A_0}=\{\pm k\}$.
        Thus, $\Spec_0(\calX)\cup\frakT_{A_0}=[-2\sqrt{k-1},2\sqrt{k-1}]\cup\{\pm k\}$,
        and
        it follows that a $k$-regular graph is Ramanujan in dimension $0$ (or $\Omega_0^+$-Ramanujan)
        if and only if it is Ramanujan in the classical sense.
        In fact, we will see in \ref{subsec:being-Ramanujan}
        that $k$-regular graphs  are Ramanujan in dimension $0$ if and only if they are completely Ramanujan.

        (ii) Let $G=\nPGL{F}{d}$ and $\calX=\calB_d(F)$ be as in Chapter~\ref{sec:ramanujan-complexes}.
        We identify the $0$-dimensional spectrum with the common spectrum
        of $a_1,\dots,a_{d-1}\in A_0:=\Alg{\catC(G,\calX),\Omega_0^+}{}$ as explained in
        Example~\ref{EX:zero-dim-spec-of-Bd}.
        According Lubotzky, Samuels and Vishne \cite[Def.~2.6]{LubSamVi05},
        a $G$-quotient $\Gamma\leftmod\calB_d(F)$ is called Ramanujan if $\Spec_0(\Gamma\leftmod\calB_d(F))$
        is contained in the union of $\Spec_0(\calB_d(F))$ and $\frakT_{A_0}$ (cf.\ Example~\ref{EX:trivial-spec-complex}).
        Thus, $\Gamma\leftmod\calB_d(F)$ is Ramanujan
        in dimension $0$ if and only if it is  Ramanujan in the sense of  \cite{LubSamVi05}.

        We shall see in Chapter~\ref{sec:existence}, that the Ramanujan quotients constructed in \cite{LubSamVi05}
        are in fact completely Ramanujan. We will also see in {\ref{subsec:consequences}} that
        when $d=3$, a $G$-quotient $\Gamma\leftmod\calB_3(F)$ is Ramanujan in dimension $0$ if and only if it is flag Ramanujan.
    \end{example}

    The following proposition shows that  $A$-Ramanujan property behaves well as $A$
    and $\Gamma$ vary.

    \begin{prp}\label{PR:Ramanujan-well-behaved}
        Let $F,F':\catC\to\catPHil$ be semi-elementary functors,  let $A$ be an algebra of $(\catC,F)$-operators,
        and let  $\Gamma'\leq\Gamma\leq_{\calX}G$. Then:
        \begin{enumerate}
            \item[(i)] $\Gamma\leftmod\calX$ is $F\oplus F'$-Ramanujan if and only if
            $\Gamma\leftmod\calX$ is $F$-Ramanujan and $F'$-Ramanujan.
            \item[(ii)] If $\Gamma\leftmod\calX$ is $A$-Ramanuajan, then $\Gamma\leftmod\calX$
            is $B$-Ramanujan for any idempotented $*$-subalgebra $B\subseteq A$.
            \item[(iii)] If $[\Gamma:\Gamma']<\infty$ and $\Gamma'\leftmod\calX$
            is $A$-Ramanujan, then $\Gamma\leftmod\calX$ is $A$-Ramanujan.
        \end{enumerate}
    \end{prp}

    \begin{proof}
        (i) We use the notation of Proposition~\ref{PR:summand-spectrum} and the discussion preceding it.
        Everything follows from this proposition if we show that $\frakT_{A_{F\oplus F'}}=\frakT_{A_F}\cup \frakT_{A_F'}$.
        Let $N\leq G$ be open of finite index.
        We have $e_F((F\oplus F')\calX)_N=(F\calX)_N$
        and $e_{F'}((F\oplus F')\calX)_N= (F'\calX)_N$, so, as in the proof of Proposition~\ref{PR:summand-spectrum}, Corollary~\ref{CR:corner-subalgebra-spectrum-II} implies that $\frakT_{A_{F\oplus F'},N}=\frakT_{A_F,N}\cup \frakT_{A_{F'},N}$.
        Taking union over all possible $N$-s, we get $\frakT_{A_{F\oplus F'}}=\frakT_{A_F}\cup \frakT_{A_F'}$.

        (ii) This follows from the definition and from Theorem~\ref{TH:subalgebra-spectrum-I}.

        (iii)
        This is immediate from Proposition~\ref{PR:subgroup-quotient}.
    \end{proof}

    \begin{remark}\label{RM:ramanujan-covers}
        At this level of generality, finite $A$-Ramanujan  $G$-quotients are not guaranteed to exist;
        this follows implicitly from \cite{LubNagn98}. Instead, one may consider
        \emph{$A$-Ramanujan covers}: Let $\Gamma'\leq\Gamma\leq_{\calX} G$ be subgroups such that
        $\Gamma'\leftmod\calX$ and $\Gamma\leftmod\calX$ are \emph{finite}, and let $p=p_{\Gamma',\Gamma}$
        denote the cover map $\Gamma'x\mapsto \Gamma x:\Gamma'\leftmod\calX\to\Gamma\leftmod\calX$. Then,
        as in the proof of Proposition~\ref{PR:subgroup-quotient},
        $Fp:F(\Gamma'\leftmod\calX)\to
        F(\Gamma\leftmod\calX)$ is continuous
        and surjective, and
        $F(\Gamma'\leftmod\calX)\cong\ker Fp \oplus F(\Gamma\leftmod\calX)$
        as unitary representations of $A$.
        In particular, $\Spec_A(\Gamma\leftmod\calX)\subseteq\Spec_A(\Gamma'\leftmod\calX)$.
        We
        call $\Gamma'\leftmod\calX$ an \emph{$A$-Ramanujan cover} of $\Gamma\leftmod\calX$
        if all the $A$-spectrum points of $\Gamma'\leftmod\calX$ that do not come from $\Gamma\leftmod\calX$
        are in $\Spec_A(\calX)$, namely, if
        \[\Spec_{A}(\ker Fp_{\Gamma',\Gamma})\subseteq \Spec_A(\calX)\ .\]
        One can likewise define covers which are $F$-Ramanujan, Ramanujan in dimension $i$, flag Ramanujan, and completely Ramanujan.
        
        It is more reasonable to believe that under certain assumptions, all finite $G$-quotients of $\calX$
        will have an $A$-Ramamnujan cover. Indeed, Marcus, Spielman and Srivastava \cite{MarSpiSri14} 
        showed that any bipartite $k$-regular graph admits a $\C[a_0]$-Ramanujan $2$-cover
        ($a_0$ is the vertex adjacency operator; cf.\ Example~\ref{EX:k-regular-graphs-basic-spectrum}). 
        This was extended to covers of any prescribed rank by Hall, Puder and Sawin \cite{HallPudSaw16}.\footnote{
            In fact, both \cite{MarSpiSri14}  and
            \cite{HallPudSaw16} treat the more general case of bipartite $(c,d)$-biregular,
            showing that for any such
            graph and any $r\in\N$, there is an $r$-cover such that none of the new eigenvalues of the vertex adjacency operator
            of the cover exceed in absolute value the spectral radius of the infinite $(c,d)$-biregular tree, which equals
            $\sqrt{c-1}+\sqrt{d-1}$.
            However, when $c\neq d$, such covers are a priori not $\C[a_0]$-Ramanujan because the spectrum of the $(c,d)$-biregular
            tree is a proper subset of $[-\sqrt{c-1}-\sqrt{d-1},\sqrt{c-1}+\sqrt{d-1}]$;
            see Question~6.3 in \cite{HallPudSaw16}. 
            See also \cite{BallCiub11} for constructions of $\C[a_0]$-Ramanujan $(q+1,q^3+1)$-biregular graphs when $q$ is a prime power.
        } 
    \end{remark}

\section{Representation Theory}
\label{sec:spectrum-and-reps}

    Let $G$ be a unimodular $\ell$-group, let $\calX$ be an almost transitive $G$-complex (\ref{subsec:G-complex}),
    let $F:\catC(G,\calX)\to \catPHil$ be an elementary functor (\ref{subsec:case-of-Alg-G-X}), and
    let $\Gamma\leq_{\calX}G$. In this chapter,
    we give a criterion to determine when $\Gamma\leftmod\calX$ is $F$-Ramanujan (\ref{subsec:Ramanujan-quotients})
    which is phrased in terms of the unitary $G$-representation $\LL{\Gamma\leftmod G}$.

    Note that by Proposition~\ref{PR:unimodularity-of-G}, $G$ is unimodular whenever $\calX$ has a finite
    $G$-quotient. Therefore, the assumption that $G$ is unimodular is not unreasonable.

\medskip

    Throughout, $G$ is a {unimodular} $\ell$-group (i.e.\ a totally disconnected locally
    compact group) and $\mu$ is a Haar measure on $G$.
    A (left) $G$-module is a $\C$-vector space on which $G$ acts (on the left) via
    $\C$-linear automorphisms.
    We write $K\co G$ to denote that $K$ is a compact open subgroup of $G$. It is well-known
    that   the collection $\{K\suchthat K\co G\}$ is basis of neighborhoods of $1_G$.
    Many computational facts recalled in this
    chapter can be easily verified using the following identities, which hold
    whenever the integrals make sense:
    \[
    \int_{x\in G}\psi(gx)\,d\mu = \int_{x\in G}\psi(xg)=\int_{x\in G} \psi(x^{-1})\,d\mu = \int_{x\in G}\psi (x) \,d\mu
    \]
    The characteristic function of $S\subseteq G$ is denoted by $\charfunc{S}$.

\subsection{Unitary and Pre-Unitary Representations}
\label{subsec:unitary-reps-G}

    As usual, a \emph{unitary representation} of $G$ is a Hilbert space
    $V$ carrying a $G$-module structure such that the action
    $G\times V\to V$ is continuous and
    $\Trings{gu,gv}=\Trings{u,v}$ for all $u,v\in V$, $g\in G$.
    The representation $V$ is \emph{irreducible} if $V$ does not contain closed
    $G$-submodules other than $0$ and $V$.
    The category of unitary representations of $G$ with continuous $G$-homomorphisms
    is denoted $\Rep[u]{G}$,
    and the class of irreducible representations is denoted $\Irr[u]{G}$.

    \begin{example}\label{EX:right-regular-rep}
        The \emph{right regular representation} of $G$ is $\LL{G}$
        endowed with the \emph{left} $G$-action given by
        \[
        (g\vphi)x=\vphi(xg)\qquad\forall\, g,x\in G,\, \vphi\in\LL{G}\ .
        \]
        The \emph{left regular representation} of $G$ is $\LL{G}$ endowed
        with the left $G$-action given by $(g\vphi)x=\vphi(g^{-1}x)$.

        To avoid ambiguity between the right and left regular representations, we shall
        denote them by $\LL{1\leftmod G}$ and $\LL{G/1}$, respectively.
    \end{example}

    \begin{example}\label{EX:LL-Gamma-leftmod-G}
        Let $\Gamma$ be a discrete subgroup of $G$.
        It is well-known
        that there exists a unique measure $\mu_{\Gamma\leftmod G}$ on $\Gamma\leftmod G$ such that
        for all $\vphi\in \CSLC{G}$, one has
        \[
        \int_{x\in G}\vphi x\,d\mu=\int_{\Gamma x\in\Gamma\leftmod G}\hat{\vphi} (\Gamma x)\,d\mu_{\Gamma\leftmod G}\ ,
        \]
        where $\hat{\vphi}\in\CSLC{\Gamma\leftmod G}$ is defined by $\hat{\vphi}(\Gamma x)=\sum_{\gamma\in\Gamma}\vphi(\gamma x)$.
        We make $\LL{\Gamma\leftmod G}$ into a unitary representation by
        setting $(g\vphi)x=\vphi(xg)$ for all $g,x\in G$ and $\vphi\in\LL{\Gamma\leftmod G}$.
    \end{example}

    Let $K\leq G$. For any $G$-module $V$, let
    \[
    V^K:=\{v\in V\suchthat kv=v~\text{for all}~ k\in K\}\ .
    \]
    Recall that $V$ is \emph{smooth} if $V=\bigcup_{K\co G} V^K$.
    This is equivalent to saying that the action $G\times V\to V$ is continuous when $V$
    is endowed with the discrete topology.

    A pre-unitary representation of $G$ is a pre-Hilbert space $V$
    carrying a $G$-module structure such that
    $V$ is smooth and
    $\Trings{gu,gv}=\Trings{u,v}$ for all $u,v\in V$, $g\in G$.
    The category of  pre-unitary representations with continuous $G$-homomorphisms is denoted by $\Rep[pu]{G}$.
	A pre-unitary representation is irreducible it is irreducible as a $G$-module
	and the class of irreducible pre-unitary representation is denoted $\Irr[pu]{G}$.

\medskip

    Every unitary representation $U\in\Rep[u]{G}$ contains a maximal pre-unitary
    subrepresentation  $\sm{U}:=\bigcup_{K\co G} U^K$. It is well-known that $\sm{U}$ is dense in $U$.
    Conversely, if $V$ is a pre-unitary
    representation, then the action of $G$ extends uniquely to
    a continuous action on the completion $\quo{V}$,
    which becomes a unitary representation.
    We always have $\quo{(\sm{V})}=V$, but in general $\sm{(\quo{V})}$ may be larger than $V$.

\medskip

	Finally, recall that a $G$-module $V$ is \emph{admissible} if $\dim V^K<\infty$ for
	all $K\co G$.
	
	\begin{example}\label{EX:LL-Gamma-mod-G-admissible}
		Let $\Gamma$ be a cocompact discrete subgroup of $G$. Then $\LL{\Gamma\leftmod G}$
        is admissible. Indeed,
		for all $K\co G$, we have $\LL{\Gamma\leftmod G}^K\cong \LL{\Gamma\leftmod G/K}$
		and $\Gamma\leftmod G/K$ is a compact discrete topological space, so it is finite.
	\end{example}

\subsection{The Hecke Algebra}
\label{subsec:Hecke-algebra}

    With every unimodular $\ell$-group $G$ one can associated an idempotented $*$-algebra, $\Hecke{G}$,
    called the \emph{Hecke algebra} of $G$. It is constructed as follows: As a vector space,
    $\Hecke{G}$ is  $\CSLC{G}$, the set of compactly supported locally constant functions from
    $G$ to $\C$. The multiplication in $G$ is the convolution product given by
    \[
    (\psi\conv\vphi)g=\int_{x\in G} \psi (x)\cdot \vphi (x^{-1}g)\, d\mu \qquad
    \forall\,\psi,\vphi\in \Hecke{G},\,  g\in G\ .\]
    Note that $\Hecke{G}$ is not unital unless $G$ is discrete.
    The Hecke algebra admits an  involution $*:\Hecke{G}\to\Hecke{G}$ given by
    \[
    \psi^*(g)=\cconj{\psi(g^{-1})} \qquad\forall \,\psi\in\Hecke{G},\,g\in G\ .
    \]
    For $K\co G$, we define
    \[
    e_K:=\mu(K)^{-1}\charfunc{K}\in\Hecke{G}\ ,
    \]
    where $\charfunc{K}$ denotes the characteristic function of $K$. It is easy to check that
    $e_K$ is an idempotent and $e_K^* =e_K$. The idempotent $e_K$ also has the following properties:
    \begin{enumerate}
        \item[(i)] $\psi\in\Hecke{G}$ is left $K$-invariant if and only if $\psi=e_K\conv\psi$.
        \item[(ii)] $\psi\in\Hecke{G}$ is right $K$-invariant if and only if $\psi=\psi\conv e_K$.
    \end{enumerate}
    Since the compact open subgroups form a basis of
    neighborhoods of $1_G$, every function $\psi\in\Hecke{G}$ is bi-$K$-invariant for some $K\co G$. Thus,
    $\Hecke{G}$ is an  idempotented $*$-algebra (\ref{subsec:idempotented-algebras}).
    The subalgebra
    \[
    \Hecke[K]{G}:=e_K\Hecke{G}e_K=\{\psi\in\Hecke{G}\where \psi(kgk')=\psi(g)~\text{for all}~k,k'\in K\, g\in G\}
    \]
    is a unital $*$-subalgebra of $\Hecke{G}$ called the \emph{Hecke algebra of $(G,K)$}.

    \begin{remark}
        The convolution product $\vphi\conv\psi$ is in fact defined
        under the milder assumption that one of $\vphi$, $\psi$ is compactly supported and the other is measurable.
    \end{remark}

    The algebra $\Hecke{G}$ admits  another involution given by
    \[
    \psi^\#(g)={\psi(g^{-1})} \qquad\forall \,\psi\in\Hecke{G},\,g\in G\ .
    \]
    This involution is in fact $\C$-linear, hence it is not an involution in the sense
    of \ref{subsec:idempotented-algebras}.

\subsection{Representations of $G$ vs.\ Representations of $\Hecke{G}$}
\label{subsec:G-vs-Hecke}

    Keep the  notation of \ref{subsec:Hecke-algebra}.
    Every $V\in\Rep[u]{G}$ can be made into a unitary representation of $\Hecke{G}$
    by defining
    \begin{equation}\label{EQ:G-to-H}
    \psi\cdot v:=\int_{x\in G} \psi(x)\cdot x v\, d\mu \qquad\forall \, \psi\in\Hecke{G}\, v\in V\ .
    \end{equation}
    The integral is defined because the function $x\mapsto \psi(x)\cdot xv$ is continuous and compactly
    supported. That $V$ is indeed a unitary representation of $\Hecke{G}$ follows by  computation.

    Conversely, every unitary representation of $\Hecke{G}$ can
    be made into a unitary representation of $G$ as follows:
    For $g\in G$ and $\psi\in\Hecke{G}$, define $g\conv \psi\in\Hecke{G}$ by $(g\conv \psi)x=\psi(g^{-1}x)$.
    Now, for $v\in V$, let
    \begin{equation}\label{EQ:H-to-G}
    g\cdot v:=\lim\{(g\conv e_K)v\}_{K\co G}\ .
    \end{equation}
    Here, $\{e_K\}_{K\co G}$ is viewed as a subnet of $\ids{\Hecke{G}}$
    (see~\ref{subsec:states}).
    That $V$ is indeed a unitary representation of $G$ is routine.

    It is also worth noting that for all $K\co G$ and $v\in V$, one has $e_Kv=v$ if and only if $v\in V^K$.

\medskip

    We summarize the previous paragraphs in  the following well-known proposition.

    \begin{prp}
        The maps $\Rep[u]{G}\to \Rep[u]{\Hecke{G}}$
        and $\Rep[u]{\Hecke{G}} \to \Rep[u]{G}$ described above are inverses of each other,
        and they induce an isomorphism
        of categories (morphisms are mapped to themselves). 
    \end{prp}

    Likewise,
    there is an isomorphism between
    $\Rep[pu]{G}$ and $\Rep[pu]{\Hecke{G}}$, and an isomorphism between the category
    of smooth $G$-modules and the category of smooth $\Hecke{G}$-modules.
    (The isomorphisms are defined using the same formulas used in the unitary case. When $V$ is smooth,
    the integral
    in \eqref{EQ:G-to-H} is just a finite sum, and the net in \eqref{EQ:H-to-G} is eventually constant.)

\medskip

    The previous discussion implies that some of the results of Chapter~\ref{sec:involutary-algebras}
    also apply to unitary and pre-unitary representations of $G$ (by taking $A=\Hecke{G}$).
    In the sequel, we will freely apply these results to representations
    of $G$.

    \begin{example}\label{EX:right-regular-smooth-rep}
        Consider the right regular representation $\LL{1\leftmod G}$ of Example~\ref{EX:right-regular-rep}.
        The induced left $\Hecke{G}$-module structure on $\LL{1\leftmod G}$ is  given by
        \[\vphi\cdot \psi=\int_{x\in G}\vphi (x)\cdot (x\cdot \psi)\,d\mu \ .\]
        Applying both sides to $g\in G$
        yields
        \begin{align*}
        (\vphi\cdot\psi)g&=\Circs{\int_{x\in G}\vphi x\cdot x\psi\,d\mu}g=
        \int_{x\in G}\vphi x\cdot (x\psi)g\,d\mu=
        \int_{x\in G}\vphi x\cdot \psi(gx)\,d\mu\\
        &=\int_{x\in G} \psi(gx)\cdot \vphi^\# (x^{-1})\,d\mu=
        \int_{x\in G} \psi(x)\cdot \vphi^\# (x^{-1}g)=(\psi\conv \vphi^\#)(g)
        \end{align*}
        Thus, the action of $\Hecke{G}$ on $\LL{1\leftmod G}$ is given by
        \[
        \vphi\cdot\psi=\psi \conv \vphi^\#\ .
        \]
        Likewise, the induced action of $\Hecke{G}$ on $\LL{\Gamma\leftmod G}$ (see Example~\ref{EX:LL-Gamma-leftmod-G}), is also
        given by $\vphi\cdot \psi=\psi\conv \vphi^\#$ (when we view $\psi\in\LL{\Gamma\leftmod G}$ as a function on $G$).
    \end{example}

    \begin{example}\label{EX:left-regular-smooth-rep}
        Consider $\LL{G/1}$, the left regular representation of $G$ (Example~\ref{EX:right-regular-rep}).
        Then the  action
        of $\Hecke{G}$ on $\LL{G/1}$ is given by
        \[
        \vphi\cdot \psi=\vphi\conv \psi\ .
        \]
        The computation is similar to Example~\ref{EX:right-regular-smooth-rep} and is left to the reader.
    \end{example}

    \begin{remark}
        Let $V$ be a unitary representation of $G$.
        We can define $\sm{V}$ by considering $V$ as a $G$-module (\ref{subsec:unitary-reps-G})
        and by considering $V$ as an $\Hecke{G}$-module (\ref{subsec:idempotented-algebras}).
        However, there is no ambiguity because $\Hecke{G}V=\bigcup_{K\co G} e_KV=\bigcup_{K\co G} V^K$.
        Likewise, $V$ is admissible as a representation of $G$ if and only if it is admissible as a representation
        of $\Hecke{G}$ because $e_KV=V^K$ for all $K\co G$.
    \end{remark}

\subsection{Elementary Functors Revised}
\label{subsec:algebra-A-revised}

    Let $\calX$ be an almost transitive $G$-complex (\ref{subsec:G-complex}),
    let $\catC=\catC(G,\calX)$
    (Definition~\ref{DF:G-complex-category}),
    and let $F:\catC\to \catPHil$ be elementary (e.g.\ $\Omega_i^+$, $\Omega_i^\pm$,
    $\llFlag$; see \ref{subsec:case-of-Alg-G-X}).
    We further write $F=\llf \circ S$ where $S$ is as in Definition~\ref{DF:elementary-functor}.
    In this section, we  describe $\Alg{\catC,F}{}$
    and $FX$ ($X\in\catC$) in terms of $\Hecke{G}$.
    For brevity, write
    \[\catH=\Hecke{G}\]
    We shall freely identify
    $S(\Gamma\leftmod\calX)$ with $\Gamma\leftmod S\calX$ using
    Proposition~\ref{PR:S-quotient-check}(i).

\medskip

    We set some general notation:
    Fix   representatives $x_1,\dots,x_t$ for the $G$-orbits in $S\calX$, and
    for all $1\leq n \leq t$ let
    \[
    K_n=\Stab_G(x_n)\qquad\text{and} \qquad e_n=e_{K_n}\in\catH
    \]
    ($K_n\co G$ by condition (E2) in Definition~\ref{DF:elementary-functor}).
   	Let $\Gamma\leq_{\calX} G$.
   	The canonical measure $\mu_{\Gamma\leftmod G}$ on $\Gamma\leftmod G$ (see~Example~\ref{EX:LL-Gamma-leftmod-G})
    induces a measure $\mu_{\Gamma\leftmod G/K_n}$ on the discrete topological space $\Gamma\leftmod G/K_n$ by pushing forward.

    \begin{lem}\label{LM:double-coset-measure}
        For all $g,h\in G$, we have (i) $\charfunc{\Gamma\leftmod\Gamma g K_n}(\Gamma h)=\sum_{\gamma\in \Gamma}\charfunc{gK_n}(\gamma h)$
        and (ii) $\mu_{\Gamma\leftmod G/K_n}(\{\Gamma gK_n\})=\mu_{\Gamma\leftmod G}(\Gamma\leftmod\Gamma gK_n)=\mu(K_n)$.
    \end{lem}

    \begin{proof}
        (i) If $h\notin\Gamma gK_n$,
        then $\charfunc{\Gamma\leftmod\Gamma g K_n}(\Gamma h)=0=\sum_{\gamma\in \Gamma}\charfunc{gK_n}(\gamma h)$,
        so assume $h\in\Gamma gK_n$.
        By condition (E2) in Definition~\ref{DF:elementary-functor}, $K_n$ is contained
        in the stabilizer of a cell $y\in \dot{\calX}$. Clearly $gK_ng^{-1}\subseteq \Stab_G(gy)$,
        so by Proposition~\ref{PR:quotient-cover-sufficient-nec-conds}, $\Gamma\cap gK_ng^{-1}=1$.
        This means that there are unique $\gamma\in \Gamma$ and $k\in K_n$ such that $h=\gamma g k$
        (if $\gamma gk=\gamma'gk'$, then $\gamma'^{-1}\gamma =gk'k^{-1}g^{-1}\in \Gamma\cap gK_ng^{-1}=1$, so $\gamma'=\gamma$
        and $k=k'$).
        It follows that $\charfunc{\Gamma\leftmod\Gamma g K_n}(\Gamma h)=1=\sum_{\gamma\in \Gamma}\charfunc{gK_n}(\gamma h)$

        (ii) In the notation of Example~\ref{EX:LL-Gamma-leftmod-G},
        part (i) implies that $\charfunc{\Gamma\leftmod\Gamma g K_n}=\hat{\charfunc{}}_{gK_n}$.
        Thus,
        by the definition of the measure on $\Gamma\leftmod G$,
        we have
        \[\mu_{\Gamma\leftmod G}(\Gamma\leftmod \Gamma gK_n)=\int_{\Gamma\leftmod G}\charfunc{\Gamma\leftmod\Gamma g K_n}\,
        \mathrm{d}\mu_{\Gamma\leftmod G}=\int_G\charfunc{gK_n}\,\mathrm{d}\mu=\mu(K_n)\ .\]
        That $\mu_{\Gamma\leftmod G/K_n}(\{\Gamma gK_n\})=\mu_{\Gamma\leftmod G}(\Gamma\leftmod\Gamma gK_n)$ is immediate.
    \end{proof}

    Recall  from  the notation section
    that $\CSLC{\Gamma\leftmod G/K_n}$ denotes the set of locally constant compactly
    supported functions from $\Gamma\leftmod G/K_n$ to $\C$. This  a subspace of $\LL{\Gamma\leftmod G/K_n}$,
    and equality holds if and only if
    $\Gamma\leftmod G/K_n$ is finite.
    Also notice that $\CSLC{\Gamma\leftmod G/K_n}=\llf(\Gamma\leftmod G/K_n)$ as $\C$-vectors spaces,
    but these spaces may differ as pre-Hilbert spaces, because
    \begin{equation}\label{EQ:inner-product-ratio}
    \Trings{\vphi,\psi}_{\CSLC{\Gamma\leftmod  G/K_n}}=\mu(K_n)\Trings{\vphi,\psi}_{\llf({\Gamma\leftmod  G/K_n})}\ ,
    \end{equation}
    thanks to Lemma~\ref{LM:double-coset-measure}. In the sequel, we shall  view $\CSLC{\Gamma\leftmod G/K_n}$ as a subspace of
    $\CSLC{\Gamma\leftmod G}$ in obvious way.

    Let $\Gamma'\leq \Gamma$. Then the map $p:\Gamma'\leftmod G /K_n\to
    \Gamma\leftmod G /K_n$ given by $p(\Gamma'gK_n)= \Gamma gK_n$ induces a map
    $p_*:\llf(\Gamma'\leftmod G/K_n)\to\llf(\Gamma\leftmod G/K_n)$. We  shall view
    view $p_*$ as a map from  $\CSLC{\Gamma'\leftmod G/K_n}$ to $\CSLC{\Gamma\leftmod G/K_n}$,
    and  denote it by
    \[
    P_{\Gamma',\Gamma}^{(n)} \ .
    \]
    Explicitly, $P^{(n)}_{\Gamma',\Gamma}\charfunc{\Gamma'gK_n}=\charfunc{\Gamma gK_n}$.
    Furthermore, by Lemma~\ref{LM:double-coset-measure}(i), when viewing $\CSLC{G/K_n}$ as a subspace of $\CSLC{G}$, we have
    \begin{equation}\label{EQ:P-Gamma-computation}
    (P_{1,\Gamma}^{(n)}\vphi)g=\sum_{\gamma\in\Gamma}\vphi(\gamma g)\qquad\forall\, \vphi\in\CSLC{G/K_n},\,g\in G\ .
    \end{equation}


    We can now phrase the main result of this section:

    \begin{thm}\label{TH:Hecke-description-of-A}
        Define an algebra
        \[
        B=\left[
        \DotsArr{e_1\catH e_1}{e_1\catH e_t}{e_t\catH e_1}{e_t\catH e_t}
        \right]\subseteq \nMat{\catH}{t}
        \]
        and an involution $*:B\to B$ by $((\vphi_{ij})_{i,j})^*=(\vphi_{ji}^*)_{i,j}$.
        For every $\Gamma\leq_{\calX} G$, we make $\bigoplus_{i=1}^t\CSLC{\Gamma\leftmod G/K_n}$
        into a left $B$-module by setting
        \[
        (\vphi_{ij})\cdot (\psi_j)_{j=1}^t=\Big(\sum_{j=1}^t\psi_j\conv \vphi_{ij}^{\#}\Big)_{i=1}^t
        \]
        (see \ref{subsec:Hecke-algebra} for the definition of $\#$).
        Then there is an isomorphism of algebras with involutions
        \[
        \Alg{\catC,F}{}\cong B
        \]
        and  isomorphisms of pre-Hilbert spaces
        \[
        F(\Gamma\leftmod\calX)\cong\bigoplus_{n=1}^t\CSLC{\Gamma\leftmod G/K_n}
        \qquad \forall\,\Gamma\leq_{\calX}G
        \]
        which are compatible with the relevant $\Alg{\catC,F}{}$-module and $B$-module structures.
    \end{thm}

    The proof is somewhat technical. We first prove the following two lemmas.

    \begin{lem}\label{LM:Hecke-description-of-A}
        Define a functor $F_1:\catC(G,\calX)\to \catPHil$ as follows (notation as in Definition~\ref{DF:G-complex-category}):
        For all  $\Gamma'\leq \Gamma\leq_{\calX} G$ and $g\in G$, let
        \begin{itemize}
            \item $F_1(\Gamma\leftmod \calX)=\bigoplus_{n=1}^t\CSLC{\Gamma\leftmod G/K_n}$,
            \item $F_1(p_\Gamma\circ g)=\bigoplus_{n=1}^t P_{1,\Gamma}^{(n)}\circ (g|_{\CSLC{G/ K_n}})$,
            where $g$ acts on $\CSLC{G/ K_n}$ via the left regular representation $\LL{G/1}$ (Example~\ref{EX:right-regular-rep}).
            \item $F_1p_{\Gamma',\Gamma}=\bigoplus_{n=1}^t P^{(n)}_{\Gamma',\Gamma}$.
        \end{itemize}
        Then there is a unitary natural isomorphism
        $
        F\cong F_1
        $.
    \end{lem}

    \begin{proof}
        Note that $F(\Gamma\leftmod\calX)=\llf(\Gamma\leftmod S\calX)$,
        so we need to construct a natural isomorphism between $\llf(\Gamma\leftmod S\calX)$
        and $\bigoplus_{n=1}^t\CSLC{\Gamma\leftmod G/K_n}$.
        There is an isomorphism $\Gamma\leftmod S\calX\cong \bigsqcup_{n=1}^t \Gamma\leftmod G/K_n$
        given by sending $\Gamma g K_n$ to $\Gamma gx_n$.
        This
        induces a unitary isomorphism
        $\Phi: \llf(\Gamma\leftmod S\calX)\to \bigoplus_{n=1}^t\llf(\Gamma\leftmod G/K_n)$.
        Next, by \eqref{EQ:inner-product-ratio}, we have a unitary isomorphism
        \[
        \Psi:=\bigoplus_{n=1}^t \mu(K_n)^{-1/2}\id_{\llf(\Gamma\leftmod  G/K_n)}~~:~~
        \bigoplus_{n=1}^t\llf(\Gamma\leftmod G/K_n)~~\to~~
        \bigoplus_{n=1}^t\CSLC{\Gamma\leftmod G/K_n}\ .
        \]
        Now, $\Psi\circ \Phi:\llf(\Gamma\leftmod S\calX)\to\bigoplus_{n=1}^t\CSLC{\Gamma\leftmod G/K_n}$ is a unitary
        isomorphism. That $\Psi\circ \Phi$ is natural is routine; use \eqref{EQ:P-Gamma-computation}.
    \end{proof}

    \begin{lem}\label{LM:endo-ring-of-idempotent}
        Let $R$ be a ring, possibly non-unital,  and let $e,f\in R$ be idempotents.
        Consider $Re$ and $Rf$ as left $R$-modules.
        Then for every $\phi\in \Hom_R(Re,Rf)$, there exists unique $a\in eRf$
        such hat $\phi(x)=xa$ for all $x\in Re$. Conversely, $[x\mapsto xa]\in\Hom_R(Re,Rf)$
        for all $a\in eRf$.
    \end{lem}

    \begin{proof}
        The last assertion is easy, and the uniqueness of $a$ holds because $\phi(e)=ea=a$.
        To see the existence, take $a=\phi(e)$, and note that for all $r\in Re$,
        we have $\phi(r)=\phi(re)=r\cdot \phi(e)=ra$.
    \end{proof}

    \begin{proof}[Proof of Theorem~\ref{TH:Hecke-description-of-A}]
        By Lemma~\ref{LM:Hecke-description-of-A}, we may assume $F=F_1$.
        By Theorem~\ref{TH:Adj-algebra-descrition}(i), we  have
        \[
        \Alg{\catC,F}{}\cong \End_G(F\calX)=\End_G\left(\bigoplus_{n=1}^t\CSLC{G/K_n}\right).
        \]
        Under this isomorphism, the involution on $\Alg{\catC,F}{}$ coincides with taking the adjoint
        operator with respect to the inner product of $\bigoplus_{n=1}^t\CSLC{G/K_n}$.

        View $\CSLC{G/K_n}$ as a subspace of $\CSLC{G}=\catH$.
        Then
        $
        \CSLC{G/K_n}=\catH e_n
        $, and it  is a pre-unitary $G$-subrepresentation of $\LL{G/1}$. We may therefore view $\catH e_n$
        as a pre-unitary representation of the $*$-algebra $\catH$;
        by Example~\ref{EX:left-regular-smooth-rep}, the (left) action of $\catH$ on $\catH e_n$ is given
        by  $\vphi\cdot \psi=\vphi\conv \psi$ ($\vphi\in\catH$, $\psi\in\catH e_n$).
        We now have
        \[
        \End_G\left(\bigoplus_{n=1}^t\CSLC{G/K_n}\right)=\End_\catH\left(\bigoplus_{n=1}^t{\catH e_n}\right)\ .
        \]
        Viewing $\bigoplus_{n=1}^t\catH e_n$ as column vectors, this gives
        \[
        \Alg{\catC,F}{}\cong
        \left[
        \DotsArr{\Hom_{\catH}(\catH e_1, \catH e_1)}{\Hom_{\catH}(\catH e_t, \catH e_1)}{\Hom_{\catH}(\catH e_1, \catH e_t)}{\Hom_{\catH}(\catH e_t, \catH e_t)}
        \right]\ .
        \]
        By Lemma~\ref{LM:endo-ring-of-idempotent}, we can identify $\Hom_{\catH}(\catH e_n,\catH e_m)$
        with $e_m\catH e_n$ by letting $\vphi\in e_m\catH e_n$ act on $\catH e_n$ from
        the \emph{left} via $\vphi\cdot \psi=\psi\conv \vphi^\#$ (note that
        $(e_m\catH e_n)^\#=e_n\catH e_m$). This identification turns composition
        into convolution product
        and so it gives rise to an isomorphism of (non-involutary) algebras
        \[
        \Alg{\catC,F}{}\cong\left[
        \DotsArr{e_1\catH e_1}{e_1\catH e_t}{e_t\catH e_1}{e_t\catH e_t}
        \right]=B\ .
        \]
        From the way we defined this isomorphism, it is immediate that the action of
        $B$ on
        \[
        F \calX= \bigoplus_{n=1}^t\CSLC{G/K_n}
        \]
        via the isomorphism $\Alg{\catC,F}{}\cong B$
        is given by $(\vphi_{ij})\cdot (\psi_j)_{j=1}^t=(\sum_{j=1}^t\psi_j\conv \vphi_{ij}^{\#})_{i=1}^t$.

        Next, we need to check
        that the action of $B$ on $F(\Gamma\leftmod \calX)=\bigoplus_{n=1}^t\CSLC{\Gamma\leftmod G/K_n}$
        via the isomorphism
        $\Alg{\catC,F}{}\cong B$
        is given by $(\vphi_{ij})\cdot (\psi_j)_{j=1}^t=(\sum_{j=1}^t\psi_j\conv \vphi_{ij}^{\#})_{i=1}^t$.
        By Theorem~\ref{TH:Adj-algebra-descrition}(i),
        it is enough to show
        that for all $(\vphi_{ij})\in B$,  the collection $\{a_{X}\}_{X\in\catC}$
        given by $a_{\Gamma\leftmod\calX}=[(\psi_j)\mapsto (\sum_{j=1}^t\psi_j\conv \vphi_{ij}^{\#})_{i=1}^t]:
        \bigoplus_{n=1}^t\CSLC{\Gamma\leftmod G/K_n}\to \bigoplus_{n=1}^t\CSLC{\Gamma\leftmod G/K_n}$
        is a natural transformation
        from $F$ to itself. In fact,  the proof of Theorem~\ref{TH:Adj-algebra-descrition}(i) shows that it follows
        from the weaker condition
        $Fp_{\Gamma}\circ a_{\calX}=a_{\Gamma\leftmod\calX}\circ Fp_{\Gamma}$. Working coordinate-wise, this boils
        down to showing that for all $1\leq i,j\leq t$, $\vphi\in e_i\catH e_j$ and $\psi\in \CSLC{G/K_j}$,
        we have $P_{1,\Gamma}^{(i)}(\psi\conv \vphi^\#)=(P_{1,\Gamma}^{(j)} \psi)\conv \vphi^\#$.
        Using \eqref{EQ:P-Gamma-computation}, for
        all $g\in G$, we have
        \begin{align*}
            (P_{1,\Gamma}^{(i)}(\psi\conv \vphi^\#))g
            &=\sum_{\gamma\in\Gamma}(\psi\conv \vphi^\#)(\gamma g)
            =\sum_{\gamma\in\Gamma}\int_{x\in G}\psi x\cdot \vphi^\#(x^{-1}\gamma g)\,\mathrm{d}\mu\\
            &=\sum_{\gamma\in\Gamma}\int_{x\in G}\psi (\gamma x)\cdot \vphi^\#(x^{-1} g)\,\mathrm{d}\mu
            =\int_{x\in G}\sum_{\gamma\in\Gamma}\psi (\gamma x)\cdot \vphi^\#(x^{-1} g)\,\mathrm{d}\mu\\
            &=\int_{x\in G}(P_{1,\Gamma}^{(j)}\psi) x\cdot \vphi^\#(x^{-1} g)\,\mathrm{d}\mu
            =((P_{1,\Gamma}^{(j)}\psi)\conv \vphi^\#)g\ .
        \end{align*}
        The sum and the integral can be exchanged because
        the function $(\gamma,x)\mapsto\vphi^\#(\gamma x)\cdot\psi(x^{-1}g):\Gamma\times G\to\C$
        is compactly supported.

        It is left to show that the isomorphism $\Alg{\catC,F}{}\cong B$
        is an isomorphism of algebras with involution. This amounts to checking
        that $\bigoplus_{n=1}^t\CSLC{G/K_n}$ is a unitary representation of $B$,
        and
        after unfolding the definitions, this boils down to showing that for all $\psi\in \catH e_n$, $\psi'\in \catH e_m$
        and $\vphi\in e_m\catH e_n$, we have
        $
        \Trings{\psi\conv \vphi^{\#},\psi'}_{\LL{G}}=\Trings{\psi,\psi'\conv (\vphi^*)^{\#}}_{\LL{G}}\
        $.
        Indeed,
        \begin{align*}
        \Trings{\psi\conv \vphi^{\#},\psi'}
        &= \int_{x\in G}(\psi\conv \vphi^\#)x\cdot \cconj{\psi'x}\,d\mu
        =\int_{x\in G}\int_{y\in G}\psi y\cdot\vphi^\#(y^{-1}x)\cdot \cconj{\psi'x}\,d\mu\, d\mu
        \\
        &=\int_{x\in G}\int_{y\in G}\psi y\cdot \cconj{\psi'x}\cdot\cconj{(\vphi^*)^\#(x^{-1}y)}\,d\mu\, d\mu
        \\
        &=\int_{y\in G}\psi y\cdot \cconj{(\psi'\conv(\vphi^*)^\#)y}\,d\mu
        =\Trings{\psi,\psi'\conv (\vphi^*)^{\#}}\ ,
        \end{align*}
        as required.
    \end{proof}

    \begin{example}
    	Suppose $F=\Omega_{i}^+$. Then $F=\llf\circ S$ where $S:\catC\to \catPHil$ is given
    	by $SX=X^{(i)}$. Taking $x_1,\dots,x_t$ to be representatives
    	for the $G$-orbits in $\calX^{(i)}$ and setting $K_n=\Stab_G(x_n)$
    	and $e_n=e_{K_n}$, Theorem~\ref{TH:Hecke-description-of-A} gives an alternative,
    	more explicit, description of $\Alg{\catC,\Omega_i^+}{}$.
    	
    	Likewise, when $F=\Omega_i^\pm$, we can take representatives
    	$\sfx_1,\dots,\sfx_s$ for the $G$-orbits in $\ori{\calX^{(i)}}$ and get an alternative
    	description of $\Alg{\catC,\Omega_i^\pm}{}$
    \end{example}

    \begin{example}\label{EX:computation-of-Hecke-algebra-for-Bd}
        Let $G=\nPGL{F}{d}$, $K=\nPGL{\calO}{d}$ and $\calX=\calB_d(F)$ be as in Chapter~\ref{sec:ramanujan-complexes},
        and take $F=\Omega_0^+$. Since $\calB_d(F)^{(0)}=G/K$, the group $G$ acts transitively on $\calB_d(F)^{(0)}$.
        Taking  $x_1:=K$ as a representative, we have $K_1=\Stab_{G}(x_1)=K$.
        We may therefore identify
        $\Alg{\catC,\Omega_0^+}{}$ with $\Hecke[K]{G}=e_K\catH e_K$ and $\Omega^+_0(\Gamma\leftmod\calX)$ with $\CSLC{\Gamma\leftmod G/K}$
        for all $\Gamma\leq_{\calX}G$.
        Under this identification, the action of $\Hecke[K]{G}$ on $\CSLC{\Gamma\leftmod G/K}$ is given by
        $\vphi\cdot\psi=\psi\conv \vphi^\#$ ($\vphi\in \Hecke[K]{G}$, $\psi\in \CSLC{\Gamma\leftmod G/K}$).
        Moreover, if we choose the Haar measure $\mu$ on $G$ such that $\mu(K)=1$, then the isomorphism
        $\Omega^+_0(\Gamma\leftmod\calX)\cong\CSLC{\Gamma\leftmod G/K}$ is the identity map.

        Define $\quo{g_1},\dots,\quo{g_{d-1}}\in G$ as in Chapter~\ref{sec:ramanujan-complexes}.
        It is well known that
        $\Hecke[K]{G}$ is  a   commutative unital algebra freely generated
        by the  operators $\charfunc{K\overline{g_1}K},\dots,\charfunc{K\overline{g_{d-1}}K}$, called the \emph{Hecke operators}
        (see \cite[Ch.~V]{Macdonald95}, for instance).
        An easy computation shows that under the isomorphism $\Omega^+_0(\Gamma\leftmod\calX)\cong\CSLC{\Gamma\leftmod G/K}$, the operator of
        $\charfunc{K\overline{g_j}K}$ corresponds to the operator $a_j$ of Chapter~\ref{sec:ramanujan-complexes}.
        It follows that $\Alg{\catC,\Omega_0^+}{}$ is a commutative unital algebra freely generated
        by  $a_1,\dots,a_{d-1}$.
    \end{example}

    \begin{remark}
        One can extend the analysis of this section to semi-elementary functors
        $F:\catC\to \catPHil$ as follows: Choose a functor $F':\catC\to\catPHil$
        such that $F\oplus F'$ is elementary and write $F\oplus F'\cong \llf\circ S$
        with $S$ is as in Definition~\ref{DF:elementary-functor}. Take representatives
        $x_1,\dots,x_t$ for the $G$-orbits in $S\calX$ and define $e_1,\dots,e_t\in \catH$
        and the algebra $B$ as in Theorem~\ref{TH:Hecke-description-of-A}.
        Define $e_F=\{e_{F,X}\}_{X\in\catC}\in \Alg{\catC,F\oplus F'}{}$ as in \ref{subsec:dependencies},
        namely, $e_{F,X}$ is the orthogonal projection $F X\oplus F'X\to FX$,
        and let $\hat{e}_F$ be the corresponding idempotent in $B$.
        Then the discussion in \ref{subsec:dependencies} and Theorem~\ref{TH:Hecke-description-of-A}
        imply that there is an isomorphism of involutary unital algebras
        $\Alg{\catC,F}{}\cong \hat{e}_FB\hat{e}_F$ and isomorphisms
        of pre-Hilbert spaces $F(\Gamma\leftmod\calX)\cong \hat{e}_F\cdot \prod_{n=1}^t\CSLC{\Gamma\leftmod G/K_n}$
        ($\Gamma\leq_{\calX}G$)
        which are compatible with the  actions of $\Alg{\catC,F}{}$ and $\hat{e}_FB\hat{e}_F$. (Here,
        $\hat{e}_FB\hat{e}_F$ acts on $\hat{e}_F\cdot \prod_{n=1}^t\CSLC{\Gamma\leftmod G/K_n}$
        via the action of $B$ on $\prod_{n=1}^t\CSLC{\Gamma\leftmod G/K_n}$ described in Theorem~\ref{TH:Hecke-description-of-A}.)
    \end{remark}

\subsection{Weak Containment}
\label{subsec:weak-containment-for-groups}

    In order to apply Theorem~\ref{TH:Hecke-description-of-A}, we need to recall the notion of weak containment
    for unitary representations of $G$, and relate it to weak containment of unitary representations of $\Hecke{G}$
    as defined in \ref{subsec:weak-containment}.

\medskip

    Let $V\in \Irr[u]{G}$ and $V'\in\Rep[u]{G}$.
    Recall that $V$ is \emph{weakly contained} in $V'$, denoted $V\wc V'$, if
    for all $v\in \sphere{V}$, $\veps>0$ and compact $C\subseteq G$, there exists $v'\in \sphere{V'}$
    such that
    \[
    \abs{\Trings{gv,v}-\Trings{gv',v'}}<\veps\qquad\forall g\in C\ .
    \]
    In fact, it is easy to see that $v'$ can be taken to be in $\sphere{\sm{V'}}$, or any
    prescribed dense subset of $\sphere{V'}$.

    The representation $V$ is called \emph{tempered} if it weakly contained in $\LL{1\leftmod G}$, the right regular
    representation of $G$ (Example~\ref{EX:right-regular-rep}).

    \begin{remark}
        Let $\bfG$ be a  reductive algebraic group
        over a non-archimedean local field $F$.
        When $G=\bfG(F)$, there are other definitions of temperedness in the literature.
        These definitions agree with our definition by   \cite[\S2.4]{Oh02}, for instance.
    \end{remark}

    \begin{prp}\label{PR:equivalent-defs-of-weak-cont}
        Let $V\in\Irr[u]{G}$ and $V'\in \Rep[u]{G}$.
        Then $V\prec V'$ as unitary representations of $G$ $\iff$ $V\prec V'$ as unitary representations
        of $\Hecke{G}$.
    \end{prp}

    \begin{proof}
        Write $\catH=\Hecke{G}$.
        Assume $V\wc V'$ as representations of $G$, and let $v\in \sphere{\sm{V}}$,  $\veps>0$ and
        $F\fsubseteq\catH$.
        Let $C=\bigcup_{\vphi\in F}\supp(\vphi)$. Then $C$ is compact and open.
        For all $\vphi\in \catH$, let
        $\norm{\vphi}_1:=\int_{x\in G}\abs{\vphi(x)}\,d\mu$,
        and
        let $M=\max\{\norm{\vphi}_1\suchthat \vphi\in F\}$.
        By assumption, there is $v'\in \sphere{\sm{V'}}$ such that
        $|\Trings{gv,v}-\Trings{gv',v'}|<\veps M^{-1}$ for all $g\in C$.
        Since $v\in\sm{V}$ and $v'\in\sm{V'}$, there is $K\co G$ such that $v\in V^K$ and $v'\in V'^K$.
        We can  choose $K$ small enough such that
        each $\vphi\in F$ can be written
        as a finite sum $\sum_j \alpha_{j}\charfunc{g_jK}$ with $\{g_j\}_j\subseteq C$
        and such that $\{g_jK\}_j$ are disjoint. A straightforward computation
        shows that $\vphi v=\mu(K)\sum_j\alpha_jg_j v$ and
        $\vphi v'=\mu(K)\sum_j\alpha_j g_j v'$. Thus,
        \begin{eqnarray*}
        |\Trings{\vphi v, v }-\Trings{\vphi v', v'}|&=& \left|\mu(K)\sum_j\alpha_{j}
        \Circs{\Trings{g_jv,v }-\Trings{g_jv',v' }}\right| \\
        &<&
        \mu(K)\sum_j|\alpha_{j}|\veps M^{-1}=\|\vphi\|_1 \veps M^{-1}\leq \veps. \
        \end{eqnarray*}
        This shows that $V\wc V'$ as representations of $\catH$.

\smallskip

        Conversely, assume $V\prec V'$ as representations of $\catH$, and
        let $v\in \sphere{V}$, $\veps>0$ and $C\subseteq G$ be compact.
        Suppose first that $v\in\sm{V}$.
        Then there is $K\co G$ such that $v\in V^K=e_KV$.
        Write $CK=\bigsqcup_{i=1}^n g_iK$
        with   $g_1,\dots g_n\in G$, and let $\vphi_i=\mu(K)^{-1}\charfunc{g_iK}\in\catH$.
        By Lemma~\ref{LM:equiv-conds-of-weak-containment}, there is $v'\in e_KV'$
        such that
        $|\Trings{\vphi_i v,v}-\Trings{\vphi_i v',v'}|<\veps$ for all $i$.
        For all $g\in C$, there are  $x\in K$ and $i$ such that $g=g_ix$.
        Since $v\in V^K$ and $v'\in V'^K$, we have
        $\vphi_iv=g_iv=g_ixv=gv$ and
        $\vphi_iv'=g_iv'=g_ixv'=gv'$, hence $|\Trings{gv,v}-\Trings{gv',v'}|<\veps$ for all $g\in C$.

        When $v\notin\sm{V}$, we argue as follows: For all $v_1\in\sphere{V}$ and $g\in G$,
        we have
        \begin{align*}\abs{{\Trings{gv,v}-\Trings{gv_1,v_1}}}&= \abs{\Trings{gv,v-v_1}+\Trings{gv-gv_1,v_1}}\\
        &\leq
        \norm{gv}\cdot \norm{v-v_1}+\norm{g(v-v_1)}\cdot\norm{v_1}=2\norm{v-v_1}\ .
        \end{align*}
        Take $v_1\in\sphere{\sm{V}}$ close enough to $v$ to have  $\abs{{\Trings{gv,v}-\Trings{gv_1,v_1}}}<\frac{\veps}{2}$.
        By the  previous paragraph,  there is $v'\in \sphere{V'}$
        such that $\abs{\Trings{gv_1,v_1}-\Trings{gv',v'}}<\frac{\veps}{2}$ for all $g\in C$.
        Then $\abs{\Trings{gv,v}-\Trings{gv',v'}}<\veps$ for all $g\in C$.
    \end{proof}

    \begin{prp}\label{PR:compact-quotient-temperedness}
        Let $H$ be a normal compact subgroup of $G$ and let $U\in\Irr[u]{G/H}$, $V\in \Rep[u]{G}$.
        Then $U\wc V^H$ as representations of $G/H$ if and only if $U\wc V$ as representations of $G$.
    \end{prp}

    \begin{proof}
        We choose the Haar measure on $G/H$ to be the push-forward of the Haar measure
        on $G$. This allows us to view $\Hecke{G/H}$ as a $*$-subalgebra of $\Hecke{G}$.
        It is easy to check that $\Hecke{G/H}$ is a corner (and also an ideal) of $\Hecke{G}$, and for all
        $W \in\Rep[u]{G}$, we have $\quo{\Hecke{G/H}\cdot W}=W^H$ (as unitary representations of $\Hecke{G/H}$). The Proposition therefore follows
        from Proposition~\ref{PR:equivalent-defs-of-weak-cont} and Lemma~\ref{LM:irreducible-corner-module}(iii).
    \end{proof}

\subsection{A Criterion for Being Ramanujan}
\label{subsec:being-Ramanujan}

    Let $\calX$ be an almost transitive $G$-complex, let $\catC=\catC(G,\calX)$ (Definition~\ref{DF:G-complex-category}),
    and let $F:\catC\to \catPHil$ be an elementary functor (Definition~\ref{DF:elementary-functor}).
    We assume that $G$ is unimodular, which is automatically the case if $\calX$ has finite $G$-quotients, by Proposition~\ref{PR:unimodularity-of-G}.
    In this section, we phrase and prove a representation-theoretic criterion for $G$-quotients of $\calX$
    to be $F$-Ramanujan (\ref{subsec:Ramanujan-quotients}). In particular,
    we get criteria for being \emph{Ramanujan in dimension $i$} and \emph{completely Ramanujan}. 
    When $G=\nPGL{F}{d}$, $\calX=\calB_d(F)$ and $F=\Omega_0^+$, we recover the criterion given in 
    \cite[Prp.~1.5]{LubSamVi05}.

    \medskip

    Write $\catH=\Hecke{G}$. We shall freely view unitary representations of $G$
    as unitary representations of $\catH$ and vice versa (\ref{subsec:G-vs-Hecke}).
    The relation of weak containment is not affected by these transitions thanks to Proposition~\ref{PR:equivalent-defs-of-weak-cont}.

    \begin{lem}\label{LM:finite-factorization}
        Let $V\in\Irr[u]{G}$ (resp.\ $V\in\Irr[pu]{G}$). The following conditions are equivalent.
        \begin{enumerate}
            \item[(a)] The image of $G$ in the unitary group of $V$ is finite.
            \item[(b)] There is an open subgroup of finite index $N\leq G$ such that $V=V^N$.
            \item[(c)] There is an open subgroup of finite index $N\leq G$ such that
            $V\leq \llf(N\leftmod G)$.
            (Here, $G$ acts on $\llf(N\leftmod G)$ via $(g\vphi)x=\vphi(xg)$.)
        \end{enumerate}
        In this case, $V$ is finite dimensional.
    \end{lem}

    \begin{proof}
        The implications, (a)$\derives$(b), (c)$\derives$(a), and the final assertion are straightforward,
        so we
        only prove  (b)$\derives$(c).
        Fix some $0\neq u\in\sphere{V}$. A continuous embedding of $V$ in $\llf(N\leftmod G)$ is given by
        and sending $v\in V$ to $[Ng\mapsto \Trings{gv,u}]\in\llf(N\leftmod G)$. That this is indeed a nonzero $G$-homomorphism
        is routine. Since $V$ is irreducible, this homomorphism is injective, and since $\llf(N\leftmod G)$ is finite-dimensional,
        it is also continuous.
        Now, by Proposition~\ref{PR:iso-implies-unitary-iso} (applied with  $A=\catH$),
        there is   a unitary embedding $V\to \llf(N\leftmod G)$.
    \end{proof}

    An irreducible representation satisfying the equivalent conditions of Lemma~\ref{LM:finite-factorization}
    is said to have  \emph{finite action}. In general, there may be irreducible finite dimensional representations without  finite
    action; consider the case $G=\Z$ with the discrete topology.
    However, in certain cases, finite dimension and finite action are equivalent.

    \begin{lem}\label{LM:finite-dim-reps-of-G}
        Let $\Gamma\leq G$ be cocompact discrete subgroup.
        Then for all
        $V\in \Irr[u]{G}$ with $V\wc \LL{\Gamma\leftmod G}$,
        we have $\dim V<\infty$ if and only if $V$ has  finite action.
    \end{lem}

    \begin{proof}
    	By Example~\ref{EX:LL-Gamma-mod-G-admissible}, $\LL{\Gamma\leftmod G}$ is admissible, so
    	$V\leq \LL{\Gamma\leftmod G}$
    	by Theorem~\ref{TH:admissible-pu-are-completely-red}(iii).
        Suppose $V$ is finite dimensional. Since $\sm{V}$ is dense in $V$ (Lemma~\ref{LM:approximation-lemma}),
        we must have $V=\sm{V}$, so $V$ is smooth.
        Let $v_1,\dots,v_t$ be a basis of $V$, and choose $K_1,\dots,K_t\co V$
        with $v_i\in V^{K_i}$. Then $V=V^K$ for $K=\bigcap_{n=1}^t K_n$.
        Let $H$ be the smallest normal subgroup of $G$ containing $K$. Then $H$ is open and
        $V^H=V$. Since $\Gamma$ is cocompact, $\Gamma\leftmod G/K$ is finite,
        and hence so is $\Gamma\leftmod G/H=\Gamma H\leftmod G$.
        It follows that $V\subseteq\LL{\Gamma H\leftmod G}$ (viewed as a subspace of $\LL{\Gamma\leftmod G}$).
        The image of $G$ in the unitary group of $\LL{\Gamma H\leftmod G}$ is clearly finite, so $V$ has finite action.
    \end{proof}

    \begin{thm}\label{TH:spectrum-correspondence}
        Write $F=\llf\circ S$ where $S$ is as in Definition~\ref{DF:elementary-functor}.
        Let $x_1,x_2,\dots,x_t$ be  representatives of the $G$-orbits in $S \calX$,
        and let $K_n=\Stab_G(x_n)$ ($1\leq n\leq t$). For any subset $T$ of the unitary dual
        $\udual{\catH}$,
        define
        \[T^{(K_1,\dots,K_t)}=\{[V]\in T\suchthat V^{K_1}+\dots+V^{K_t}\neq 0\}. \]
        Then there exists an additive functor
		\[
		\calF:\Rep[u]{\catH}\to \Rep[u]{\Alg{\catC,F}{}}
		\]
        with the following properties:
        \begin{enumerate}
            \item[(i)]
				$\calF$ induces a topological embedding
				$[V]\mapsto [\calF V]:\udual{\catH}^{(K_1,\dots,K_t)}\to \tAlg{\catC,F}{}$,
				denoted $\what{\calF}$.
			\item[(ii)]
				For all $[V]\in\udual{\catH}^{(K_1,\dots,K_t)}$ and
				$V'\in\Rep[u]{\catH}$, we have $V\wc V'$ $\iff$ $\calF V\wc \calF V'$.
			\item[(iii)]
				$\calF(\LL{\Gamma\leftmod G})=\overline{F(\Gamma\leftmod\calX)}$
				for all $\Gamma\leq_{\calX}G$.
			\item[(iv)]
				$\what{\calF}(\Spec_{\catH}(\LL{\Gamma\leftmod G})^{(K_1,\dots,K_t)})=\Spec_{\Alg{\catC,F}{}}
				(\Gamma\leftmod\calX)$ for all $\Gamma\leq_{\calX}G$.
				Furthermore, when $\Gamma\leftmod\calX$ is finite, or equivalently,
				when $\Gamma$ is cocompact in $G$,
            	the map $\what{\calF}$
            	induces an isomorphism of multisets $\mSpec_{\catH}(\LL{\Gamma\leftmod G})^{(K_1,\dots,K_t)}
            	\cong\mSpec_{\Alg{\catC,F}{}}(\Gamma\leftmod\calX)$.
            \item[(v)] Let $[V]\in\udual{\catH}^{(K_1,\dots,K_t)}$.
            	Then $\what{F}[V]$ is in $\frakT_{\Alg{\catC,F}{}}$, the trivial $\Alg{\catC,F}{}$-spectrum
            	(\ref{subsec:trivial-spec}), if and only if  $V$ has  finite action.
            More precisely, we have $\what{\calF}(\Spec_{\catH}(\llf(N\leftmod G))^{(K_1,\dots,K_t)})=\frakT_{\Alg{\catC,F}{},N}$,
            for any open finite-index subgroup $N\leq G$.
        \end{enumerate}
    \end{thm}

    \begin{proof}
        Define an involution $*:\nMat{\catH}{t}\to \nMat{\catH}{t}$
        by $(\vphi_{ij})^*=(\vphi_{ji}^*)$. Then $\nMat{\catH}{t}$ is an idempotented involutary algebra.
        As explained in Example~\ref{EX:matrix-algebra}, there is an equivalence of categories
        $\Rep[u]{\catH}\sim \Rep[u]{\nMat{\catH}{t}}$. Explicitly, the $\nMat{\catH}{t}$-representation
        corresponding to
        $V\in\Rep[u]{\catH}$ is  $V^t$ (viewed as column vectors) endowed with the standard
        left $\nMat{\catH}{t}$-action.
        In the other direction,
        the equivalence is given by $U\mapsto \overline{\catH'U}$
        where
        \[\catH'=
        \left[\begin{smallmatrix}
        \catH & 0 & \ldots \\
        0 & 0 & \ldots \\
        \vdots & \vdots & \ddots
        \end{smallmatrix}\right]\subseteq \nMat{\catH}{t}\]
        and $\catH'$ is identified with $\catH$ in the obvious manner.
        There is a similar equivalence between the corresponding categories of pre-unitary representations.

        We
        identify $\Alg{\catC,F}{}$ with the algebra $B$
        defined in Theorem~\ref{TH:Hecke-description-of-A}.
        Observe that
        \[
        B=e\nMat{\catH}{t}e
        \]
        where $e\in \nMat{\catH}{t}$ is the diagonal matrix $\mathrm{diag}(e_{K_1},\dots,e_{K_t})$.
        Define the functor
        $
        \calF:\Rep[u]{\catH}\to \Rep[u]{B}
        $
        by
        \[
        \calF V:= e \cdot V^t = \bigoplus_{n=1}^t e_{K_n}V=\bigoplus_{n=1}^t V^{K_n}\ .
        \]
        (the functor $\calF$ acts on morphisms is the obvious way).
        We  similarly define $\calF:\Rep[pu]{\catH}\to\Rep[pu]{B}$. It is easy
        to see that there is a natural isomorphism
        $\quo{\calF V}\cong \calF(\quo{V})$ for all $V\in\Rep[pu]{\catH}$.
        Observe also that for all unitary or pre-unitary $V$, we have
        \begin{equation}\label{EQ:equiv-nullity-cond}
        \calF V=eV^t\neq 0\qquad\iff \qquad V^{K_1}+\dots+V^{K_t}\neq 0\ .
        \end{equation}
        We now prove properties (i)--(v).

        (i) Let $[V]\in\udual{\catH}^{(K_1,\dots,K_t)}$.
        Then $V^t\in \Irr[u]{\nMat{\catH}{t}}$ by Example~\ref{EX:matrix-algebra},
        and
        by Lemma~\ref{LM:irreducible-corner-module}(i) and \eqref{EQ:equiv-nullity-cond}, $\calF V=eV^t\in\Irr[u]{B}$.
        Furthermore, by Theorem~\ref{TH:corner-unitary-dual},
        the maps $[U]\mapsto [\catH'U]:\nMat{\catH}{t}\what{~}\to\udual{\catH}$
        and $[U]\mapsto [eU]:(\nMat{\catH}{t}\what{~})^{(B)}\to\udual{B}$
        are topological embeddings. Since
        $[U]\mapsto [\catH'U]:\nMat{\catH}{t}\what{~}\to\udual{\catH}$ is bijective
        with inverse $[V]\mapsto [V^t]:\udual{\catH}\to \nMat{\catH}{t}\what{~}$,
        and since the image of $\udual{\catH}^{(K_1,\dots,K_t)}$ in $\nMat{\catH}{t}\what{~}$
        under $[V]\mapsto [V^t]$ is $(\nMat{\catH}{t}\what{~})^{(B)}$,
        it follows that the map $[V]\mapsto [eV^t]=[\calF V]: \udual{\catH}^{(K_1,\dots,K_t)}\to \udual{B}$
        is a topological embedding.

        (ii)
        By Lemma~\ref{LM:irreducible-corner-module}(iii), if $V\in \Irr[u]{\catH}$ satisfies
        $V^{K_1}+\dots+V^{K_t}\neq 0$ and $V'\in\Rep[u]{\catH}$, then $V\wc V'$ $\iff$ $V^t\wc V'^t$ (as
        representations of $\nMat{\catH}{t}$) $\iff$ $\calF(V)=eV^t\wc eV'^t=\calF(V')$.

        (iii)
		We identify
		$F(\Gamma\leftmod \calX)$ with $\bigoplus_{n=1}^t\CSLC{\Gamma\leftmod G/K_n}$
		as in Theorem~\ref{TH:Hecke-description-of-A}.
		By Example~\ref{EX:right-regular-smooth-rep},
		we have $\calF(\CSLC{\Gamma\leftmod G})=
		\bigoplus_{n=1}^t\CSLC{\Gamma\leftmod G/K_n}$  as pre-unitary representations of $B$,
		so $\calF(\LL{\Gamma\leftmod G})=\calF(\overline{\CSLC{\Gamma\leftmod G}})\cong
		\quo{\calF(\CSLC{\Gamma\leftmod G})}=\quo{F(\Gamma\leftmod\calX)}$.

        (iv) This follows from (ii) and (iii). That $\Gamma\leftmod\calX$ is finite
        if and only if $\Gamma$ is cocompact in $G$ and $\Gamma\leq_{\calX}G$ follows from
        Proposition~\ref{PR:unimodularity-of-G}.

        (v)
        Recall from \ref{subsec:trivial-spec} that $\frakT_{\Alg{\catC,F}{},N}$
        is defined to be $\Spec_{\Alg{\catC,F}{}}((F\calX)_N)$, where $(F\calX)_N$
        is the space of $N$-coinvariants of $F\calX$. 
        By (ii), it is enough to prove that $\calF(\llf(N\leftmod G))\cong (F\calX)_N$ as pre-unitary
        representations of $B=\Alg{\catC,F}{}$. In fact, by
        Lemma~\ref{LM:coinvariants-module-structure}, it is enough to show that $\calF(\llf(N\leftmod G))\cong (F\calX)_N$
        as $B$-modules.

        Let $C^\infty(N\leftmod G)$ be the space of functions $G\to\C$ which are $N$-invariant on the left.
        We view $C^\infty(N\leftmod G)$ as a smooth $G$-module by setting $(g\cdot \vphi)x=\vphi(xg)$.
        Then $\llf(N\leftmod G)\cong C^\infty(N\leftmod G)$  as  $G$-modules via
        $\e_{Ng}\mapsto\charfunc{Ng}$. As in Example~\ref{EX:right-regular-smooth-rep},
        the corresponding left $\catH$-module structure of $C^\infty(N\leftmod G)$
        is given by $\vphi\cdot \psi=\psi\conv \vphi^\#$. It follows that
        $\calF(C^\infty(N\leftmod G))=\bigoplus_{n=1}^tC^\infty(N\leftmod G/K_n)$,
        where the action of $B$ on the left hand side is given by
        $(\vphi_{ij})_{i,j}\cdot(\psi_j)_j=(\sum_j\psi_j\conv \vphi_{ij}^\#)_i$.
        By Theorem~\ref{TH:Hecke-description-of-A}, the action of $B$
        on $F\calX=\bigoplus_{n=1}^t \CSLC{G/K_n}$ is also given by
        $(\vphi_{ij})_{i,j}\cdot(\psi_j)_j=(\sum_j\psi_j\conv \vphi_{ij}^\#)_i$,
        and this action induces the action of $B$ on $(F\calX)_N=\bigoplus_{n=1}^t\CSLC{G/K_n}_N$
        (notice that $G$ acts on $\CSLC{G/K_n}$ via $(g\vphi)x=\vphi(g^{-1}x)$).

        Define $\Phi_n:\CSLC{G/K_n}\to C^\infty(N\leftmod G /K_n)$ by
        $(\Phi_n\vphi)x=\int_{y\in N}\vphi(yx)\,\mathrm{d}\mu$.
        Since
        $\Phi_n(g\vphi)=\Phi_n(\vphi)$ for all $g\in N$,
        this induces a map $\CSLC{G/K_n}_N\to C^\infty(N\leftmod G /K_n)$,
        which we also denote by $\Phi_n$.
        We claim that $\Phi:=\bigoplus_{n=1}^t\Phi_n:\bigoplus_{n=1}^t\CSLC{G/K_n}_N\to \bigoplus_{n=1}^tC^\infty(N\leftmod G/K_n)$
        is an isomorphism of $B$-modules, which will complete the proof.
        Working component-wise, this amounts to showing that $\Phi_n$ is bijective for all $n$,
        and  for all $n$, $m$, $\vphi\in e_{K_m}\catH e_{K_n}$ and $\psi\in\CSLC{G/K_n}$, we have
        $\Phi_m(\psi\conv \vphi^\#)=(\Phi_n\psi)\conv \vphi^\#$.

        It easy to see that $\Phi_n(\charfunc{gK})$ is a nonzero multiple of $\charfunc{NgK}$, so $\Phi_n$ is surjective.
        The injectivity follows since $ \CSLC{G/K_n}_N\cong \llf(N\leftmod G/K_n)$
        as vector spaces, and hence $\dim \CSLC{G/K_n}_N=\abs{N\leftmod G/K_n}=\dim C^\infty(N\leftmod G/K_n)$.
        Finally, for all $g\in G$, we have
        \begin{align*}
        (\Phi_m(\psi\conv \vphi^\#))g
        &=\int_{y\in N}(\psi\conv \vphi^\#)(yg)\,\mathrm{d}\mu
        =\int_{y\in N}\int_{x\in G}\psi(x)\cdot \vphi^\#(x^{-1}yg)\,\mathrm{d}\mu \,\mathrm{d}\mu\\
        &=\int_{y\in N}\int_{x\in G}\psi(yx)\cdot \vphi^\#(x^{-1}g)\,\mathrm{d}\mu \,\mathrm{d}\mu\\
        &=\int_{x\in G}(\Phi_n\psi)x\cdot \vphi^\#(x^{-1}g)\,\mathrm{d}\mu=((\Phi_n \psi)\conv \vphi^\#)g\ ,
        \end{align*}
        so $\Phi_m(\psi\conv \vphi^\#)=\Phi_n(\psi)\conv \vphi^\#$.
    \end{proof}

    As a corollary, we get the following criterion for checking the $F$-Ramanujan
    and completely Ramanujan properties  (\ref{subsec:Ramanujan-quotients}).

    \begin{thm}\label{TH:Ramanujan-criterion}
        Keep the notation of Theorem~\ref{TH:spectrum-correspondence}
        and let
        $\Gamma\leq_{\calX} G$. Then:
        \begin{enumerate}
            \item[(i)] $\Gamma\leftmod\calX$ is $F$-Ramanujan
            if and only if
            every irreducible unitary representation  $V\wc \LL{\Gamma\leftmod G}$
            satisfying $V^{K_1}+\dots+V^{K_t}\neq 0$
            is tempered (i.e.\ $V\wc \LL{1\leftmod G}$) or has finite action.
            \item[(ii)] $\Gamma\leftmod\calX$ is completely Ramanujan if and only if
            every irreducible unitary representation  $V\wc \LL{\Gamma\leftmod G}$
            is tempered or has finite action.
        \end{enumerate}
        When $\Gamma\leftmod \calX$ is finite, one can replace ``finite action'' with ``finite dimension''
        and ``$V\wc \LL{\Gamma\leftmod G}$'' with ``$V\leq \LL{\Gamma\leftmod G}$''.
    \end{thm}

    \begin{proof}
        (i) The complex $\Gamma\leftmod\calX$ is $F$-Ramanujan
        precisely when  $\Spec_{\Alg{\catC,F}{}}(\Gamma\leftmod \calX)\subseteq \Spec_{\Alg{\catC,F}{}}(\calX)\cup\frakT_{\Alg{\catC,F}{}}$.
        Taking inverse images relative to $\what{\calF}$ in
        Theorem~\ref{TH:spectrum-correspondence} yields the equivalence.

        (ii) Suppose $\Gamma\leftmod \calX$ is completely Ramanujan. For any $K\co G$
        such that $K$ stabilizes some nonempty cell of $\calX$,
        define a functor $S:\catC\to \catSet$ by setting $S(\Gamma\leftmod\calX)=\Gamma\leftmod G/K$,
        $S(p_{\Gamma}\circ g)=[xK\mapsto \Gamma gxK]$, $Sp_{\Gamma',\Gamma}=[\Gamma'xK\mapsto \Gamma xK]$ (notation
        as in Definition~\ref{DF:G-complex-category}), and let $F=\llf\circ S$.
        It is easy to check that $S$ satisfies
        conditions (E1)--(E4) of Definition~\ref{DF:elementary-functor},
        and hence $F$ is elementary.
        Furthermore,  $S\calX=G/K$ consists of a single $G$-orbit. Taking $x_1=1_GK\in S\calX$,
        and $K_1=\Stab_G(x_1)=K$ in part (i), we get that any
        irreducible unitary representation  $V\wc \LL{\Gamma\leftmod G}$
        with $V^K\neq 0$ is
        tempered or has finite action.
        Since $K$ can be taken to be any sufficiently small
        compact open subgroup of $G$, it follows that any irreducible unitary representation
        $V\wc \LL{\Gamma\leftmod G}$ is tempered or has finite action.
        The other direction is immediate from (i).

        The final assertion follows from Lemma~\ref{LM:finite-dim-reps-of-G} and
        Theorem~\ref{TH:admissible-pu-are-completely-red}(iii)
        (cf.\
        Proposition~\ref{PR:unimodularity-of-G}, Example~\ref{EX:LL-Gamma-mod-G-admissible}).
    \end{proof}

    \begin{remark}\label{RM:non-faithful-action}
        It sometimes convenient to consider groups acting non-faithfully on $\calX$, e.g.\
        a non-adjoint almost simple algebraic group over a local non-archimedean field acting
        on its affine building (Example~\ref{EX:G-complex-building-III}). Theorem~\ref{TH:Ramanujan-criterion}
        can be adjusted to this case as follows:
        Suppose $\tilde{G}$ acts on $\calX$ such that stabilizers of nonempty cells
        are compact open. Let $G$ be the image of $\tilde{G}$ in $\Aut(\calX)$,
        and write $H=\ker(\tilde{G}\to G)=\Stab_{\tilde{G}}(\calX)$.
        Let $\tilde{\Gamma}$ be a lattice
        in $\tilde{G}$ such that $\Gamma:=\im(\tilde{\Gamma }\to G)\leq_\calX G$.
        In the notation of Theorem~\ref{TH:spectrum-correspondence}, write $\tilde{K}_n=\Stab_{\tilde{G}}(x_n)$.
        Then $\Gamma\leftmod \calX=\tilde{\Gamma}\leftmod\calX$ is $F$-Ramanujan
        if and only if any $V\wc \LL{\tilde{\Gamma}\leftmod \tilde{G}}$ with $V^{\tilde{K}_1}+\dots +V^{\tilde{K}_t}\neq 0$
        is tempered or has finite action, and $\Gamma\leftmod \calX=\tilde{\Gamma}\leftmod\calX$ is  completely Ramanujan
        if and only if any $V\wc \LL{\tilde{\Gamma}\leftmod \tilde{G}}$ with $V^H\neq 0$
        is tempered or has finite action. Indeed, notice that $\LL{\tilde{\Gamma}\leftmod \tilde{G}}^H\cong\LL{\Gamma\leftmod G}$
        as representations of $G=\tilde{G}/H$, so the previous statements can be translated to
        the statements of Theorem~\ref{TH:Ramanujan-criterion}  by taking $H$-invariants, thanks
        to Proposition~\ref{PR:compact-quotient-temperedness}.
    \end{remark}

    \begin{remark}
    	The functor $F=\llf\circ S$ used in the proof of Theorem~\ref{TH:Ramanujan-criterion}(ii)
    	has no a priori combinatorial interpretation. However,
    	the proof can be done using  functors defined by combinatorial means. For example,
    	consider the family of functors
    	$\{S_{n,m}:\catSimp\to\catSet\}_{n,m\in\N}$ defined by letting
    	$S_{n,m}X=\{(v_1,\dots,v_n)\in\vrt{\calX}\suchthat \text{$\dist(v_i,v_j)\leq m$ for all $i,j$}\}$
    	($S_{n,m}$ acts on morphisms in
    	the obvious way). It is easy to see that the collection
    	$\{\Stab_G(x)\where x\in S_{n,m}\calX,\, n,m\in\N\}$ is a basis of neighborhoods
    	of $1_G$,  so we can use $F_{n,m}:=\llf\circ S_{n,m}$ in the proof.
    \end{remark}

    \begin{example}\label{EX:Ramanujan-in-dim-i-equiv-cond}
    	Let $x_1,\dots,x_t$ be representatives for the $G$-orbits in $\calX^{(i)}$,
    	and let $K_n=\Stab_G(x_n)$ ($1\leq n\leq t$). Then by Theorem~\ref{TH:Ramanujan-criterion}(i),
        applied with $F=\Omega_i^+=\llf \circ[X\mapsto X^{(i)}]$,
        a
    	$G$-quotient $\Gamma\leftmod\calX$
    	is Ramanujan in dimension $i$ (i.e.\ $\Omega_i^+$-Ramanujan) if and
    	only if any irreducible unitary representation
    	$V\wc \LL{\Gamma\leftmod G}$ with  $V^{K_1}+\dots+V^{K_t}\neq 0$ is tempered or has finite action.
    	
    	Likewise, letting   $\sfx_1,\dots,\sfx_s$ be a set of representatives for the
        $G$-orbits in $\ori{\calX^{(i)}}$ and setting $L_n=\Stab_G(\sfx_n)$,
        we get that $\Gamma\leftmod\calX$ is $\Omega_i^\pm$-Ramanujan if and only
        if any irreducible unitary representation
    	$V\wc \LL{\Gamma\leftmod G}$ with  $V^{L_1}+\dots+V^{L_s}\neq 0$ is tempered or has finite action.
    	In this case, the spectrum of the $i$-dimensional Laplacian (\ref{subsec:orientation})
    	of $\Gamma\leftmod\calX$ is contained in the union of the spectrum of the $i$-dimensional Laplacian
    	of $\calX$ and the trivial spectrum $\frakT_{ \Delta_i}$ (cf.\ Remark~\ref{RM:operator-trivial-spectrum}).
    \end{example}

    \begin{example}\label{EX:ramanujan-equivalence-tree}
        Assume $\calX$ is a $k$-regular tree and let $G=\Aut(\calX)$.
        Choose
        a vertex $v\in\calX^{(0)}$ and write $K=\Stab_G(v)$.
        Then ${\calX^{(0)}}=Gv$, and, as in Example~\ref{EX:Ramanujan-in-dim-i-equiv-cond},
        we get that a finite graph $\Gamma\leftmod \calX$ is Ramanujan in dimension $0$
        if and only if every  infinite dimensional irreducible subrepresentation $V\leq \LL{\Gamma\leftmod G}$
        with $V^K\neq 0$ is tempered. (Recall from
        Example~\ref{EX:classical-Ramamanujan}(i) that being Ramanujan in dimension $0$ is just being a
        Ramanujan $k$-regular graph in the
        classical sense.)
    \end{example}

    \begin{example}\label{EX:ramanujan-equivalence-Bd}
        Let $G=\nPGL{F}{d}$, $K=\nPGL{\calO}{d}$ and $\calX=\calB_d(F)$ be as in Chapter~\ref{sec:ramanujan-complexes}.
        There is only one $G$-orbit in $\calB_d(F)^{(0)}=G/K$, represented by  $x_1:=K$,
        and we have $\Stab_G(x_1)=K$.
        As in Example~\ref{EX:Ramanujan-in-dim-i-equiv-cond},
        a {finite}
        $G$-quotient $\Gamma\leftmod \calB_d(F)$ is Ramanujan in dimension $0$ if and only if
        every infinite dimensional irreducible subrepresentation $V\leq \LL{\Gamma\leftmod G}$
        with $V^K\neq 0$
        is tempered. This statement is \cite[Pr.~1.5]{LubSamVi05}
        (recall from Example~\ref{EX:classical-Ramamanujan}(ii) that being Ramanujan
        in dimension $0$ is equivalent to being Ramanujan in the sense of \cite{LubSamVi05}).

        Consider now the case of $G$-quotients which are flag Ramanujan, i.e.\ $\llFlag$-Ramanujan.
        Notice that $\llFlag=\llf \circ \Flag$. Since $\calB_d(F)$ is pure,
        maximal flags in $\calB_d(F)$ correspond to chambers equipped with a full
        order on its vertices.
        Thus,
        the stabilizer of any maximal flag in $\calB_d(F)$
        is just the pointwise stabilizer of some chamber $x\in\calB_d(F)^{(d-1)}$. The
        group $I_x:=\bigcap_{v\in x}\Stab_G(v)$ is called an \emph{Iwahori subgroup} of $G$.
        For example, the Iwahori group corresponding to the fundamental chamber
        $x_0:=\{K,\quo{g_1}K,\dots,\quo{g_{d-1}}K\}$ (notation as in Chapter~\ref{sec:ramanujan-complexes})
        is $I_0:=K\cap \quo{g_1}K\quo{g_1}^{-1}\cap \dots\cap \quo{g_{d-1}}K\quo{g_{d-1}}^{-1}$,
        and an easy computation shows that $I_0$ is the image of
        \[
        I:=\left[\begin{matrix}
        \units{\calO} & \pi\calO & \ldots & \pi\calO \\
        \calO & \units{\calO} & \ddots & \vdots \\
        \vdots & \ddots  &\ddots & \pi\calO \\
        \calO & \ldots & \calO & \units{\calO}
        \end{matrix}\right]\subseteq \nGL{\calO}{d}
        \]
        in $G$.
        Since $G$
        acts transitively on $\calB_d(F)^{(d-1)}$, any Iwahori group $I_x$
        is conjugate to $I_0$, and hence  any $[V]\in \udual{\catH}$ with $V^{I_x}\neq 0$ also satisfies
        $V^{I_0}\neq 0$. By Theorem~\ref{TH:Ramanujan-criterion}(i), this means that a finite quotient $\Gamma\leftmod\calB_d(F)$
        is flag Ramanujan if and only if any irreducible infinite-dimensional subrepresentation
        $V\leq \LL{\Gamma\leftmod G}$ with $V^{I_0}\neq 0$ is tempered.
        In this case, $\Gamma\leftmod \calB_d(F)$ is also Ramanujan
        in all dimensions (Proposition~\ref{PR:Ramanujan-well-behaved}(i), Example~\ref{EX:summand-functor}).
    \end{example}

\subsection{Consequences}
\label{subsec:consequences}

    We now give some consequences of Theorems~\ref{TH:spectrum-correspondence} and~\ref{TH:Ramanujan-criterion}.

\medskip

    Let $\calT_k$ denote the $k$-regular tree and let $G_k=\Aut(\calT_k)$.
    In the sequel, we shall freely consider $k$-regular graphs $X$ (with no double edges or loops) as $G_k$-quotients
    $\Gamma\leftmod \calT_k$. (The implicit choice of the  cover map $\calT_k\to X$
    determining $\Gamma$ will not affect the discussion.)
    Recall from Example~\ref{EX:classical-Ramamanujan}(i) that a $k$-regular graph is Ramanujan in the classical sense if and only if it is Ramanujan
    in dimension $0$, or $\Omega_0^+$-Ramanujan.
    We shall henceforth use Ramanujan in dimension $0$ to avoid ambiguity.

    Marcus, Spielman and Srivastava \cite{MarSpiSri14} proved   that every bipartite finite $k$-regular graph $X$
    admits an  $2$-cover $X'\to X$ which is Ramanujan in dimension $0$
    (cf.\ Remark~\ref{RM:ramanujan-covers}). That is, the  eigenvalues of the vertex adjacency operator of $X'$ which do not
    arise from $X$  lie in the interval $[-2\sqrt{k-1},2\sqrt{k-1}]$. This was extended covers of any prescribed
    rank by Hall, Puder and Sawin \cite{HallPudSaw16}. By
    applying Theorem~\ref{TH:spectrum-correspondence} with $\calX=\calT_k$, $G=G_k$ and $F=\Omega_0^+$,
    these results can be restated  in a representation-theoretic manner:

    \begin{cor}\label{CR:ramanujan-subgroups-of-tree}
		Let $H_k$ be the index-$2$ subgroup of $G_k$ consisting of
        automorphisms preserving the canonical $2$-coloring of $\calX^{(0)}$.
        Let $K$ be the stabilizer of a vertex of $\calX$, and let $\Gamma\leq H_k$ be a cocompact lattice.
        Then for any $r\in \N$, there exists a sublattice $\Gamma'\leq \Gamma$ of index $r$ such that
        every irreducible unitary subrepresentation $V$ of the orthogonal complement of
        $\LL{\Gamma\leftmod G_k}$ in $\LL{\Gamma'\leftmod G_k}$
        satisfying $V^K\neq 0$ is tempered.
    \end{cor}

    We do not know of a representation-theoretic proof of this result.

    \begin{remark}
        Corollary~\ref{CR:ramanujan-subgroups-of-tree} suggests the following definition:
        Let $G$ be an $\ell$-group, let $K_1,\dots,K_t\co G$, and let $\Gamma$ be a lattice in $G$.
        Call a subgroup $\Gamma'\leq \Gamma$  a  \emph{$(K_1,\dots,K_t)$-Ramanujan subgroup}
        if $[\Gamma:\Gamma']<\infty$ and every irreducible unitary representation $V$ with $V^{K_1}+\dots+V^{K_t}\neq 0$ that is weakly contained in
        the orthogonal
        complement of $\LL{\Gamma\leftmod G}$ in $\LL{\Gamma'\leftmod G}$  is tempered.
        If $\calX$ is an almost transitive $G$-complex, $F:\catC(G,\calX)\to \catPHil$ is an elementary
        functor, $\Gamma\leq_{\calX}G$, and $K_1,\dots,K_t$ are as in Theorem~\ref{TH:spectrum-correspondence},
        then $\Gamma'$ is a $(K_1,\dots,K_t)$-Ramanujan subgroup of $\Gamma$
        if and only if $\Gamma'\leftmod \calX\to \Gamma\leftmod\calX$ is an  $F$-Ramanujan cover in the sense of
        Remark~\ref{RM:ramanujan-covers}.

        It is also reasonable to call $\Gamma'\leq \Gamma$  a \emph{completely Ramanujan} subgroup if
        $[\Gamma:\Gamma']<\infty$ and
        any irreducible unitary representation   that is weakly contained in the orthogonal
        complement of $\LL{\Gamma\leftmod G}$ in  $\LL{\Gamma'\leftmod G}$ is tempered.
    \end{remark}

	The following results concern equivalences of different types of the Ramanujan property.
    They follow from representation-theoretic properties of the relevant  groups. We expect
    similar results should hold whenever $\calX$ is an affine building of dimension $1$ or $2$.

    \begin{prp}\label{PR:Raman-graph-equiv-conds}
        A $k$-regular graph is Ramanujan in dimension $0$
        if and only if it is completely Ramanujan (\ref{subsec:Ramanujan-quotients}).
        Likewise, a finite covering of $k$-regular graphs $X'\to X$ is Ramanujan in dimension $0$
        if and only if it is completely Ramanujan in the sense of Remark~\ref{RM:ramanujan-covers}.
    \end{prp}

    \begin{proof}
        Let $\Gamma\leq_{\calX}G$.
        If $\Gamma\leftmod\calX$ is completely Ramanujan, then it is Ramanujan in dimension $0$
        by definition.  Assume the converse, and let $K$ be the stabilizer of some vertex in $\calX$.
        By Example~\ref{EX:ramanujan-equivalence-tree}, every irreducible $V\wc \LL{\Gamma\leftmod G_k}$
        with $V^K\neq 0$ is tempered of has finite action. Suppose now that $V\wc\LL{\Gamma\leftmod G_k}$
        is irreducible with $V^K=0$. By Olshanski's classification of the irreducible
        representations of $G_k$  \cite{Olshan77}, $V$ is either \emph{special} or \emph{supercuspidal},
        and in both cases it is tempered. (In more detail, in both cases, for all $u,v\in V$,
        the \emph{matrix coefficient} $\psi_{u,v}:=[g\mapsto \Trings{g^{-1}u,v}]$ is in $\LL{G}$,
        and hence $V\leq \LL{G}$; see \cite[Pr.~9.6]{Knapp86}, for instance.)
        By Theorem~\ref{TH:Ramanujan-criterion}(ii), this means $\Gamma\leftmod\calX$ is completely Ramanujan.
        The assertion about covers is shown similarly.
    \end{proof}

    Proposition~\ref{PR:Raman-graph-equiv-conds} allows us to slightly strengthen  Corollary~\ref{CR:ramanujan-subgroups-of-tree}
    by dropping the assumption $V^K\neq 0$ in the end.

    \begin{prp}\label{PR:PGLiii-special-case}
        Let $G=\nPGL{F}{d}$ and $\calX=\calB_d(F)$ be as in  Chapter~\ref{sec:ramanujan-complexes},
        and suppose $d\in\{2,3\}$. Let $\Gamma\leq_{\calX}G$.
        Then $\Gamma\leftmod \calX$ is Ramanujan in dimension $0$
        if and only if $\Gamma\leftmod\calX$ is flag Ramanujan
        (cf.\ \ref{subsec:Ramanujan-quotients}, Example~\ref{EX:classical-Ramamanujan}(ii)).
    \end{prp}

    \begin{proof}
        Being flag Ramanujan implies being Ramanujan in all dimensions (Proposition~\ref{PR:Ramanujan-well-behaved}(i),
        Example~\ref{EX:summand-functor}), so we need to show the converse.
        Suppose $\Gamma\leftmod \calB_d(F)$ is Ramanujan in dimension $0$.
        Then by Example~\ref{EX:ramanujan-equivalence-Bd} (see there for   notation), any irreducible
        $V\wc \LL{\Gamma\leftmod G}$ with $V^K\neq 0$ is tempered or has finite action,
        and we need to show that the same holds under the milder assumption $V^{I_0}\neq 0$.
        It is therefore enough to show that any irreducible unitary representation $V$ of $\nGL{F}{d}$
        with $V^{I}\neq 0$ and $V^{\nGL{\calO}{d}}=0$ is tempered.
        However, when $d\in\{2,3\}$ this follows from the classification of irreducible representations of
        $\nGL{F}{d}$. Specifically, when $d=2$, $V$ is necessarily a \emph{Steingberg representation}, and hence tempered.
        The case $d=3$ is analyzed in \cite[\S2.2, Table~2]{KaLiWa10}. (A description
        of the unitary irreducible representations of $\nGL{F}{d}$ with a nonzero $I$-invariant
        vector can be given using \cite[Pr.~2.6]{Casselman80} together with the classification
        of all  unitary irreducible representations of $\nGL{F}{d}$ by Tadi\'c \cite{Tadic86}.
        See \cite{Zelevinsky80} for the
        classification of the  \emph{smooth} irreducible representations of  $\nGL{F}{d}$.)
    \end{proof}

    \begin{remark}
    	In \cite[Th.~2]{KaLiWa10}, the authors show that the
    	$\Omega_0^+$-Ramanujan property
    	of quotients of $\calB_3(F)$ can be described using adjacency operators of $1$-cells
    	or adjacency operators of
    	$2$-cells. This agrees with
    	Proposition~\ref{PR:PGLiii-special-case}.
    	
        The proof of Proposition~\ref{PR:PGLiii-special-case} breaks when $d\geq 4$ since
        then $\nGL{F}{d}$ has non-tempered irreducible unitary representations $V$ with $V^I\neq 0$ and $V^K=0$;
        see \cite{Tadic86} or \ref{subsec:reps-of-GLd} below.
    \end{remark}

    We now give an application to isospectral complexes.

    \begin{example}\label{EX:isospectral-complexes}
        Let $G=\nPGL{F}{d}$ and $\calX=\calB_d(F)$ and assume $d\neq 6$.
        In  \cite{LubSamVi06} (see also \cite{LubSamVish06B}), Lubotzky, Samuels and Vishne
        construct arbitrarily large families of non-isomorphic $G$-quotients of $\calX$ which are isopectral in the sense that
        their $0$-dimensional spectrum is the same. They further note that the high-dimensional
        Laplacians of these complexes also have the same spectrum. Their example
        is based on constructing non-commensurable cocompact lattices $\Gamma_1,\Gamma_2,\dots,\Gamma_m\leq G$
        such that $\LL{\Gamma_1\leftmod G}\cong\dots\cong\LL{\Gamma_m\leftmod G}$
        as representations of $G$.
        In this case, Theorem~\ref{TH:spectrum-correspondence} implies that the
        quotients $\{\Gamma_i\leftmod\calX\}_{i=1}^m$ have the same $F$-spectrum for any elementary
        (and hence semi-elementary) functor $F:\catC(G,\calX)\to\catPHil$, that is, the $G$-quotients
        $\{\Gamma_i\leftmod\calX\}_{i=1}^m$ are \emph{completely isospectral}.
        For example, it follows that the spectrum of any high-dimensional adjacency operator is the same
        on $\{\Gamma_i\leftmod\calX\}_{i=1}^m$. Nevertheless, $\Gamma_1\leftmod \calB_d(F),\dots,\Gamma_m\leftmod\calB_d(F)$
        are pairwise non-isomorphic.
    \end{example}
    
    Finally, we  use Theorem~\ref{TH:spectrum-correspondence} to give a partial converse
    to Proposition~\ref{PR:trivial-spectrum-containment}. 
    We will use notation introduced in \ref{subsec:trivial-spec}.
    
    \begin{prp}\label{PR:converse-trivial-spectrum-containment}
        Let $G$, $\calX$, $F$ and $x_1,\dots,x_t$ be as in Theorem~\ref{TH:spectrum-correspondence},
        and write $K_n=\Stab_G(x_n)$ ($1\leq n\leq t$) and $A=\Alg{\catC(G,\calX),F}{}$. 
		Let $\Gamma\leq_{\calX}G$ be such that $\Gamma\leftmod\calX$ is finite,
		and let $N$ be a \emph{normal} finite-index subgroup of $G$. 
		Recall from \ref{subsec:trivial-spec} that  $\frakT_{A,N}=\Spec_A((F\calX)_N)$.        
        \begin{enumerate}
        	\item[(i)] If $N$ contains $K_n$ for some $n$,
        	then $\Gamma\subseteq N$ $\iff$ $\mSpec_A((F\calX)_N)\subseteq \mSpec_A(\Gamma\leftmod\calX)$
        	(as multisets).
        	In this case, every $[V]\in\frakT_{A,N}$ has the same multiplicity 
        	in $\mSpec_A((F\calX)_N)$ and $\mSpec_A(\Gamma\leftmod\calX)$.
        \end{enumerate}
        When $G/N$ is abelian we also have:
        \begin{enumerate}
        	\item[(ii)] $\frakT_{A,N}=\bigcup_{n=1}^t\frakT_{A,NK_n}$.
        	\item[(iii)] If $N$ contains $K_n$ for some $n$,
        	then $\Gamma\subseteq N$ $\iff$ $\frakT_{A,N}\subseteq \Spec_A(\Gamma\leftmod\calX)$.
        	In this case, every $[V]\in\frakT_{A,N}$ has multiplicity
        	$1$ in $\mSpec_A(\Gamma\leftmod\calX)$.
        \end{enumerate} 
    \end{prp}
    
    \begin{proof}
		(i) 
		The direction ($\derives$) is shown exactly as in the proof of
		Proposition~\ref{PR:trivial-spectrum-containment}, so we turn to prove
		the other direction.
		Without loss of generality, $K_1\subseteq N$.
		By Theorem~\ref{TH:spectrum-correspondence}(iv), $\what{\calF}$
		induces an isomorphism of multisets $\mSpec(\LL{\Gamma\leftmod G})\to \mSpec(\Gamma\leftmod\calX)$
		(notice that $\LL{\Gamma\leftmod G}$ is admissible by Proposition~\ref{PR:unimodularity-of-G}
		and
        Example~\ref{EX:LL-Gamma-mod-G-admissible}).
		Since $\mSpec_A((F\calX)_N)\subseteq \mSpec_A(\Gamma\leftmod\calX)$
		and $\calF(C^\infty(N\leftmod G))=(F\calX)_N$ (see the proof of Theorem~\ref{TH:spectrum-correspondence}(v)), and
		since every irreducible subrepresentation
		of $C^\infty(N\leftmod G)$ contains nonzero $K_1$-invariant
		vectors (because $K_1\subseteq N$), the representation $\LL{\Gamma\leftmod G}$ contains
		a copy of $C^\infty(N\leftmod G)$. Since $N\normalin G$,
		we have $C^\infty(N\leftmod G)^N=C^\infty(N\leftmod G)$
		and $\LL{\Gamma\leftmod G}^N=\LL{\Gamma N\leftmod G}$. 
		Dimension considerations now imply that $N=\Gamma N$ and that
		that every irreducible subrepresentation of $C^\infty(N\leftmod G)$
		has the same multiplicity in $C^\infty(N\leftmod G)$ and
		$\LL{\Gamma\leftmod G}$. The latter gives the assertion about
		the multiplicity of elements of $\frakT_{A,N}$, thanks to
		Theorem~\ref{TH:spectrum-correspondence}(v).
		
		(ii)
		Since $G/N$ is abelian and $N\normalin G$,  we have $NK_n\normalin G$ for all $1\leq n \leq t$.
		This easily implies that
		$C^\infty(N\leftmod G)^{K_n}=C^\infty(NK_n\leftmod G)$, and hence
		$\Spec (\llf(N\leftmod G))^{(K_1,\dots,K_t)}=
		\bigcup_{n=1}^t\Spec(\llf (NK_n\leftmod G))$.
		The assertion now follows from 	Theorem~\ref{TH:spectrum-correspondence}(v).
		
		(iii) Since $G/N$ is abelian, every irreducible subrepresentation
		of $C^\infty(G/N)$ occurs with multiplicity $1$. Invoking this fact in the proof
		of (i) gives all  claims.
    \end{proof}
    
    \begin{example}
    	(i) Taking $\calX=\calT_k$, $G=G_k$, $F=\Omega_0^+$ and $N=H_k$ (notation as in 
    	Corollary~\ref{CR:ramanujan-subgroups-of-tree}) in Proposition~\ref{PR:converse-trivial-spectrum-containment}, we recover
    	the well-known fact that a $k$-regular \emph{connected} graph $X$ is bipartite 
    	(i.e.\ $X\cong \Gamma\leftmod \calX$ for $\Gamma\leq_{\calX} H_k$) if and only
    	if $-k$ is an eigenvalue of its adjacency matrix $a_{0,X}$ (cf.\ Example~\ref{EX:trivial-spec-tree}).
    	We also see that the eigenvalues $k$ and $-k$ can occur with multiplicity at most $1$
    	in $\Spec(a_{0,X})$.
    	
    	(ii) Let $\calX=\calB_d(F)$, $G=\nPGL{F}{d}$ and $K=\nPGL{\calO}{d}$ by
    	as in Chapter~\ref{sec:ramanujan-complexes}. 
    	We observed in Example~\ref{EX:trivial-spec-complex} that the smallest
    	finite-index subgroup of $G$ is $N:=\im(\nSL{F}{d}\to\nPGL{F}{d})$.
    	Since $N\normalin G$ and $G/N$ is abelian, Proposition~\ref{PR:converse-trivial-spectrum-containment}
    	implies that the trivial $0$-dimensional spectrum $\frakT_{\Alg{\catC(G,\calX),\Omega_0^+}{}}$
    	coincides with
    	$\frakT_{\Alg{\catC(G,\calX),\Omega_0^+}{},NK}$, and for any $NK\leq N'\leq G$ 
    	and cocompact $\Gamma\leq_{\calX}G$,
    	we have $\frakT_{\Alg{\catC(G,\calX),\Omega_0^+}{},N'}\subseteq \Spec_0(\Gamma\leftmod \calX)$
    	$\iff$ $\Gamma\leq N'$. 
    \end{example}

\subsection{Finite Index Subgroups}

    Let $\calX$  be an almost transitive $G$-complex with $G$ unimodular, let
    $F:\catC(G,\calX)\to \catPHil$ be an elementary functor with $F\cong \llf\circ S$
    as in Definition~\ref{DF:elementary-functor}, and let $H$ be a finite-index
    open subgroup of $G$. Then $\calX$ is an almost transitive $H$-complex,
    and for any $\Gamma\leq_{\calX}H$, we can consider $\Gamma\leftmod \calX$
    both as a $G$-quotient and as an $H$-quotient of $\calX$, giving rise to two
    possible  notions of $F$-Ramanujan-ness for $\Gamma\leftmod \calX$.
    In this  section, we use Theorem~\ref{TH:Ramanujan-criterion} to
    show that, in some cases, these two notions are equivalent.

    \begin{thm}\label{TH:finite-index-subgroup}
        Let $\Gamma\leq_{\calX}H$.
        \begin{enumerate}
            \item[(i)] If $\Gamma\leftmod \calX$ is $F$-Ramanujan
            as an $H$-quotient of $\calX$, then it is $F$-Ramanujan as a $G$-quotient of $\calX$.
            \item[(ii)] The converse of (i) holds if there are representatives $x_1,\dots,x_t$ for the $G$-orbits
            in $S\calX$ such that $\Stab_G(x_n)\subseteq H$ for all $1\leq n\leq t$.
            \item[(iii)] $\Gamma\leftmod \calX$ is completely Ramanujan as an $H$-quotient of $\calX$
            if and only if $\Gamma\leftmod\calX$ is completely Ramanujan as a $G$-quotient of $\calX$.
        \end{enumerate}
    \end{thm}

    For the proof, observe that the Haar measure $\mu$ of $G$ restricts to a Haar
    measure on $H$. We can extend any function  $\vphi\in\CSLC{H}$
    to a function in $\CSLC{G}$ by setting $\vphi$ to be $0$
    on $G-H$. This allows us to view $\Hecke{H}$ as a $*$-subalgebra of $\Hecke{G}$.
    If $V$ is a unitary representation of $G$, then the $\Hecke{H}$-module structure
    of $V$ obtained by considering $V$ as a representation of $H$ coincides with the $\Hecke{H}$-module structure
    obtained by restricting the $\Hecke{G}$-module structure of $V$ to $\Hecke{H}$.

    \begin{proof}
        (i) Write $\catC_G=\catC(G,\calX)$ and $\catC_H=\catC(H,\calX)$ (Definition~\ref{DF:G-complex-category}).
        Since $\catC_H\subseteq \catC_G$, we have a homomorphism
        of unital $*$-algebras $\Phi:\Alg{\catC_G,F}{}\to  \Alg{\catC_H,F}{}$ given
        by $\Phi(\{a_X\}_{X\in\catC_G})= \{a_X\}_{X\in\catC_H}$.
        The map $\Phi$ injective by Theorem~\ref{TH:Adj-algebra-descrition}(i), so we may view
        $\Alg{\catC_G,F}{}$ as an algebra of $(\catC_H,F)$-operators.
        Now, if $\Gamma\leftmod\calX$ is $F$-Ramanujan as an $H$-quotient, i.e.\ $\Alg{\catC_H,F}{}$-Ramanujan, then
        it is  $\Alg{\catC_G,F}{}$-Ramanujan by Proposition~\ref{PR:Ramanujan-well-behaved}(ii),
        and hence $F$-Ramanujan as a $G$-quotient.

        (ii) Write $K_n:=\Stab_G(x_n)\subseteq H$.
        Suppose that $\Gamma\leftmod\calX$ is $F$-Ramanujan as a $G$-quotient. Then by
        Theorem~\ref{TH:Ramanujan-criterion}(i), any irreducible $V\wc \LL{\Gamma\leftmod G}$
        with $V^{K_1}+\dots+ V^{K_t}\neq 0$ is tempered or has finite action.

        Let $U$ be an irreducible unitary representation of $H$ such that $U\wc \LL{\Gamma\leftmod H}$
        and $U^{K_1}+\dots+ U^{K_t}\neq 0$.
        We need to show that $U$ is tempered or has finite action. Since $\LL{\Gamma\leftmod H}$
        is an $H$-subrepresentation of $\LL{\Gamma\leftmod G}$, we have $U\wc \LL{\Gamma\leftmod G}$.
        By Theorem~\ref{TH:subalgebra-spectrum-I} and the comment preceding our proof,
        there is an irreducible unitary representation $V$ of $G$ such that $V\wc\LL{\Gamma\leftmod G}$
        and $U\leq V$. Since $V^{K_1}+\dots +V^{K_t}\supseteq U^{K_1}+\dots+ U^{K_t}\neq 0$,
        either $V$ is tempered, or $V$ has finite action. When the latter holds,  $U$
        clearly has finite action. On the other hand, if $V$ is tempered,
        then $V\wc \LL{1\leftmod G}$, and hence $U\wc \LL{1\leftmod G}$ (as representations of $H$).
        Since $\LL{1\leftmod G}\cong \LL{1\leftmod H}^{[G:H]}$ as representations of $H$,
        Theorem~\ref{TH:direct-sum-spectrum} implies that $U\wc \LL{1\leftmod H}$, so $U$ is tempered, as required.

        (iii) One direction follows from (i). The proof of the other direction is similar to the proof of (ii) --- simply
        drop the condition $U^{K_1}+\dots+ U^{K_t}\neq 0$.
    \end{proof}

    \begin{example}
        Let $\calT_k$ be a $k$-regular tree, and let $G=\Aut(\calT_k)$.
        Choose a  coloring $C_0:\calT_k^{(0)}\to \{0,1\}$ such that
        each edge in $\calT_k$ consists of vertices of different colors, and let $H$ be
        the index $2$ subgroup of $G$ preserving the coloring $C_0$.
        Then $H$-quotients of $\calT_k$ can be regarded as bipartite $k$-regular graphs with a $\{0,1\}$-coloring.
        This coloring gives rise to a richer spectral theory since one can consider the spectra of associated operators
        (\ref{subsec:adjacency-operators}) taking colors into account.
        However, despite the extra associated operators, an $H$-quotient of $\calT_k$ is completely Ramanujan
        if and only if it is Ramanuajan in dimension $0$ as a $G$-quotient (i.e.\ Ramanujan as a $k$-regular graph in the
        classical sense). This follows from Corollary~\ref{PR:Raman-graph-equiv-conds} and Theorem~\ref{TH:finite-index-subgroup}(iii).
    \end{example}

    \begin{example}
        Let $G=\nPGL{F}{d}$ and $\calX=\calB_d(F)$  be as in Chapter~\ref{sec:ramanujan-complexes}.
        Let $H$ be the index $d$ subgroup of $G$
        preserving the vertex coloring $C_0:\calB_d(F)^{(0)}\to \Z/d\Z$ (explicitly,
        $H=\ker(c:G\to \Z/d\Z)$ where $c$ is as in Chapter~\ref{sec:ramanujan-complexes}). The  $H$-quotients
        of $\calB_d(F)^{(0)}$ inherit the  coloring $C_0$  and hence admit a richer spectral theory.
        However, since   $G$-stabilizers of maximal flags in $\calB_d(F)$
        are contained in $H$,
        Theorem~\ref{TH:finite-index-subgroup}(ii) implies that
        an $H$-quotient of $\calB_d(F)$
        is flag Ramanujan as an $H$-quotient
        if and only if it is flag Ramanujan as a $G$-quotient.

        Furthermore, we mentioned in Remark~\ref{RM:auto-group-of-Bd} that when $d>2$,
        the group $G$ is of index $2$ in $\tilde{G}:=\Aut(\calB_d(F))$.
        It is easy to check that
        the pointwise $\tilde{G}$-stabilizer of any maximal flag in $\calB_d(F)$ still lies in $H$,
        and hence an $H$-quotient of $\calB_d(F)$ is flag Ramanujan as $H$-quotient if and only if it is flag Ramanujan
        as an $\tilde{G}$-quotient. This remains true if we replace $H$ with any subgroup between $H$ and $\tilde{G}$.

        We also observed in Example~\ref{EX:trivial-spec-complex}
        that the smallest open finite-index subgroup of $G$ is $G_0:=\im(\nSL{F}{d}\to\nPGL{F}{d})$. (In fact,
        $G_0=H$ if $d$ is coprime to $q(q-1)$, where $q$ is the cardinality
        of the residue field of $F$.)
        By Theorem~\ref{TH:finite-index-subgroup}(iii), for any $H'$ between $G_0$ and $\tilde{G}$ and any $H'$-quotient $\Gamma\leftmod \calB_d(F)$,
        the condition that $\Gamma\leftmod \calB_d(F)$ is a
        completely Ramanujan $H'$-quotient is independent of the choice of $H'$, so long as $\Gamma\subseteq H'$.
    \end{example}


\section{Automorphic Representations}
\label{sec:existence}

    Let $k$ be a global field,  let $\bfG$ be a simple algebraic group over $k$,
    let $\nu$ be a non-archimedean place of $k$, and let $F$ be the completion of $k$ at $\nu$.
    Let $\calB$ be the affine Bruhat-Tit building of $G:=\bfG(F)$. Then $G$ is an $\ell$-group and
    $\calB$ is a $G$-complex (\ref{subsec:G-complex}). Let $\Gamma$ be a \emph{congruence subgroup} of $\bfG(k)$
    such that $\Gamma\leq_{\calB} G$
    (\ref{subsec:quotienst-of-simp-comps}).
    In this final chapter, we relate certain  properties of automorphic representations
    of $\bfG$
    with the condition that the $G$-quotient $\Gamma\leftmod \calB$
    is  Ramanujan (\ref{subsec:Ramanujan-quotients}).
    This is used, together with deep results about automorphic representations,
    to give examples of infinite families of $G$-quotients which are \emph{completely Ramanujan} when the group
    $\bfG$ is an inner form of $\uPGL_n$
    and $\Char k>0$.

\medskip

    We note that
    such ideas were applied, sometimes implicitly, in \cite{LubPhiSar88}, \cite{Marg88}, \cite{Morg94}, \cite{LubSamVi05}, \cite{Li04} and \cite{BallCiub11}
    to construct infinite families of Ramanujan regular graphs, Ramanujan complexes (in the sense
    of Chapter~\ref{sec:ramanujan-complexes}), and Ramanujan biregular graphs (see \cite{BallCiub11}).
    All of these constructions rely on powerful results about automorphic representations, e.g.\ the proof of the Ramanujan--Petersson conjecture for
    $\uGL_n$
    in positive characteristic by
    Lafforgue  \cite{Laff02} (see
    also \cite{Drinfel88}, \cite{Deligne74}).

    Continuing this approach, we
    Lafforgue's work
    together with recent results about the Jacquet--Langlands correspondence (\cite{BadulRoch14} and
    related works)
    to show that for any central division algebra $D$ over $F$, the  affine Bruhat--Tits building
    of $\nPGL{D}{d}$ has infinitely many $\nPGL{D}{d}$-quotients which are {completely Ramanujan} (\ref{subsec:Ramanujan-quotients}).
    When $D=F$, our construction gives the Ramnujan complexes constructed by Lubotzky, Samuels and Vishne in \cite{LubSamVi05}.
    Thus, the Ramanujan complexes of \cite{LubSamVi05}, which are Ramanujan in dimension $0$ (Example~\ref{EX:classical-Ramamanujan}(ii)),
    are in fact completely Ramanujan.

\medskip

    We alert the reader that
    in sections \ref{subsec:cuspidal-and-residual} and \ref{subsec:existence},  it is essential that the base
    field $k$ has positive characteristic.

\subsection{Adeles}
\label{subsec:adeles}

    Throughout, $k$ denotes a global field and $\calV$ is the set of places of $k$.
    For the sake of simplicity, \emph{we shall assume that $k$ has positive characteristic}, and hence
    $\calV$ consists entirely of non-archimedean places. All results in this chapter excluding
    those in
    \ref{subsec:cuspidal-and-residual} and \ref{subsec:existence} hold when $k$ is a number field after suitable  modifications.

\medskip

    For $\nu\in\calV$, let $k_\nu$ denote the completion of $k$ at $\nu$.
	We also let $\nu$ stand for the additive valuation
    $\nu:k_\nu\onto \Z\cup\{\infty\}$.
    The integer ring of $k_\nu$ is denoted  $\calO_\nu$ and we fix a generator $\pi_\nu$ of the maximal ideal
    of $\calO_\nu$. For every $a\in k_\nu$, let
    \[
    \abs{a}_\nu=q_\nu^{-\nu(a)},
    \]
    where $q_\nu:=\abs{\calO_\nu/\pi_\nu\calO_\nu}$.

\medskip

    Let $S\subseteq \calV$ be finite set of places. We denote by $\bbA^S$ the ring of adeles over $k$ away
    from $S$, that is,
    \[
    \bbA^S = \prod'_{\nu\in \calV-S}k_\nu :=\left\{(a_\nu)_{\nu}\in\prod_{\nu\in\calV-S}k_\nu\suchthat \text{$a_\nu\in \calO_\nu$ for almost
    all $\nu$}\right\}\ .
    \]
    We also write
    $
    \bbA=\bbA^\emptyset
    $.
    The field $k$ is a subring of $\bbA^S$ via the diagonal embedding,
    and each of the fields $k_\nu$ ($\nu\in\calV-S$) embeds as summand of $\bbA^S$. We endow
    $\prod_{\nu\in\calV-S}\calO_\nu$ with the product topology, and topologize $\bbA^S$ by viewing it
    as a disjoint union of (additive) cosets of $\prod_{\nu\in\calV}\calO_\nu$. The ring $\bbA^S$ is therefore
    a locally compact topological ring. By the {product formula}, $k$ is discrete in $\bbA$.

\medskip

    Let $\bfG$ be an algebraic group over $k$ with a fixed closed embedding $j:\bfG\to \bGL_n$.
    If $R$ is a commutative domain whose fraction field $F$ contains $k$, we set
    \[\bfG(R)=j(\bfG(F))\cap \bGL_n(R)\ .\]
    When $R$ is a topological ring, we embed $\bGL_n(R)$ in $\bSL_{n+1}(R)$ via
    $(a_{ij})\mapsto (a_{ij})\oplus (\det (a_{ij})^{-1})$ and give $\bfG(R)$ the topology
    induced from $\bSL_{n+1}(R)\subseteq\nMat{R}{n+1}\cong R^{(n+1)^2}$.
    This makes $\bfG(R)$ into a topological group, which is an $\ell$-group if $R$ is totally disconnected
    and locally compact.
    In particular, $\bfG(k_\nu)$ and $\bfG(\bbA)$ are $\ell$-groups.
    We note that the (topological) group $\bfG(R)$ is independent of
    the embedding $j$
    when $R$  contains $k$.

    For  $\nu\in\calV$ and $I\idealof \calO_\nu$,  the $I$-congruence subgroup
    of  $\bfG(k_\nu)$ is
    \[
    \bfG(\calO_\nu,I):=\ker\big(\bfG(\calO_v)\xrightarrow{j} \bGL_n(\calO_v)\to \bGL_n(\calO_v/I)\big) \ .
    \]
    The subgroups $\{\bfG(\calO_\nu,\pi_\nu^n\calO_\nu)\}_{n\geq 0}$
    form a basis of compact open neighborhoods at the identity.
    The group
    \[
    \bfG(\bbA)=\prod'_{\nu\in \calV}\bfG(k_\nu) :=
    \left\{(g_\nu)_{\nu}\in\prod_{\nu\in\calV}\bfG(k_n)\suchthat \text{$g_\nu\in \bfG(\calO_\nu)$ for almost
    all $\nu$}\right\}
    \]
    has  $\prod_\nu \bfG(\calO_\nu)$ as a compact open subgroup.
    Thus, the collection
    \[\left\{\prod_{\nu}\bfG(\calO_\nu,\pi_\nu^{n_\nu}\calO_\nu)\suchthat
    \text{$n_\nu\in\N\cup\{0\}$ and $n_\nu=0$ for almost all $\nu$}\right\}\]
    is basis of compact open neighborhoods
    at the identity in $\bfG(\bbA)$.
    Since $k$ is discrete in $\bbA$, the group $\bfG(k)$ is a discrete subgroup of $\bfG(\bbA)$.

\medskip

    When $\bfG$ is reductive, the groups $\bfG(\bbA)$ and $\bfG(k_\nu)$ ($\nu\in\calV$)
    are unimodular, and the quotient $\bfG(k)\leftmod \bfG(\bbA)$ is compact precisely
    when $\bfG$ is $k$-anisotropic; see \cite[Th.~5.5]{PlatRapin94} and \cite[Cr.~2.2.7]{Harder69}.

\subsection{Smooth Representations Revised}
\label{subsec:admissible-revised}

    Let $G$ be an $\ell$-group and let $P$ be a closed subgroup of $G$. As usual, given a smooth $P$-module
    $V$ (see~\ref{subsec:unitary-reps-G}), we let
    \[\Ind_P^G(V)\]
    denote the vector space of locally constant functions $\vphi:G\to V$ satisfying
    $\vphi(pg)=p\vphi(g)$ for all $p\in P$, $g\in G$. We make
    $\Ind_P^G(V)$
    into a left $G$-module by
    defining $(g\cdot \vphi)h=\vphi(hg)$ for all $g,h\in G$. The $G$-module $\Ind_P^G(V)$ is smooth
    when $P\leftmod G$ is compact, which will always be the case in the sequel.

\medskip

    Throughout, a \emph{character} of $G$ means a continuous group homomorphism $\xi:G\to \units{\C}$.
    Since $G$ is an $\ell$-group, its characters are locally constant.
    The character $\xi$ is unitary if $|\xi(g)|=1$ for all $g\in G$.
    If $V$ is a smooth $G$-module, we let $\xi V$ denote the vector space $V$ endowed with the $G$-action
    $(g,v)\mapsto \xi(g)gv$. The $G$-module  $\xi V$ is clearly smooth.
    Characters are  in correspondence
    with $1$-dimensional smooth representations of $G$, considered up to isomorphism. In the sequel, we shall
    freely interchange between characters and their corresponding representations.

\medskip

	We proceed by recalling several facts about admissible representations
	that will be used in the sequel, often without comment.
	Note in particular that if $\bfG$ is a reductive algebraic group over $k_\nu$, then all irreducible
    smooth
	(resp.\ unitary)  representations of $\bfG(k_\nu)$ are admissible
	(\cite{Bern74}, for instance).
	Here, as usual, a {smooth} representation $V$ of $G$ is said to
	be irreducible if it has no nonzero proper $G$-submodules. The class of irreducible
    smooth $G$-representations is denoted $
    \Irr[sm]{G}
    $.
	
	\begin{thm}[Schur's Lemma]\label{TH:Schurs-Lemma-for-admissible}
		Let $V$ be an admissible irreducible smooth  representation of $G$.
		Then $\End_G(V)=\C\cdot\id_V$.
	\end{thm}
	
	\begin{proof}
		Let $\psi\in\End_G(V)$. Since $V$ is smooth, there is $K\co G$
		such that $V^K\neq 0$, and since $V$ is admissible, $\dim V^K<\infty$.
		Clearly $\psi(V^K)=V^K$, hence there exists $\lambda\in\C$ and $0\neq v\in V^K$
		such that $v\in\ker(\psi-\lambda)$. Since $\ker(\psi-\lambda)$ is a $G$-submodule of $V$
		and $V$ is irreducible,
		we must have $\ker(\psi-\lambda)=V$, so $\psi=\lambda\id_V$.
	\end{proof}
	
	Let $Z$ be the center of $G$ and let $V$ be an admissible irreducible smooth representation of $G$. By Schur's Lemma,
	there is a character $\chi_V:Z\to \units{\C}$, called the \emph{central character of $V$}, satisfying
	$gv=\chi_V(g)v$ for all $g\in Z$, $v\in V$.

\medskip

    Let $V$ be a smooth $G$-module. The group $G$ acts on $\Hom_{\C}(V,\C)$ via $(g\phi)v=\phi(g^{-1}v)$. Let $V^*$ denote the smooth
	part of $\Hom_{\C}(V,\C)$, and let
    \[\check{V}\]
    denote the $G$-module obtained from $V^*$
	by twisting the  $\C$-vector-space structure via the complex conjugate.
    The representation $\check{V}$ is  called the \emph{contragradient} representation of $V$.	
    It is admissible if and only if $V$ is admissible.
    If $U$ is another smooth representation
    and $f:U\to V$ is $G$-equivariant,
    then we define
    $\check{f}:\check{V}\to\check{U}$
    by $\check{f}(\vphi)=\vphi\circ f$.
    This makes $V\mapsto \check{V}$
    into a contravariant
    functor which restricts
    to
    a duality on the full subcategory of admissible representation of $G$.


	\begin{prp}\label{PR:unitarizability}
		Let $V$ be an irreducible  admissible smooth representation. Then, up
		to scaling, there exists
		at most one inner product $\Trings{~,~}:V\times V\to \C$ such that $\Trings{gu,gv}=\Trings{u,v}$ for all $g\in G$ and $u,v\in V$.
	\end{prp}
	
	\begin{proof}
        Both $V$ and $\check{V}$ are admissible and irreducible,
        hence
		any inner product $\Trings{~,~}:V\times V\to\C$ induces an isomorphism $V\to \check{V}$ mapping $v\in V$
		to $[u\mapsto \Trings{v,u}]$. Schur's Lemma
        implies that any
		two such isomorphisms $V\to \check{V}$ must be the same up to scaling.
        (It should be noted that we cannot use Proposition~\ref{PR:iso-implies-unitary-iso} with $A=\Hecke{G}$ because
		there is no  assumption of continuity.)
	\end{proof}
	
	An admissible irreducible smooth $G$-representation $V$ admitting an inner product as  in Proposition~\ref{PR:unitarizability}
	is called \emph{unitarizable}. In this case, we can view $V$ as a pre-unitary
	representation, well-defined up to isomorphism.

	A pre-unitary representation of $G$ is called irreducible if it is irreducible as a smooth
    representation. The class of irreducible pre-unitary representations of $G$ is denoted $\Irr[pu]{G}$.
    By Theorem~\ref{TH:admissible-category-equivalence} (applied with
    $A=\Hecke{G}$, cf.~\ref{subsec:G-vs-Hecke}), an admissible pre-unitary representation
    $V$ is irreducible if and only if its completion $\quo{V}$ is irreducible as a unitary representation.

\medskip

	Let $V$ be an irreducible admissible pre-unitary
	representation of $G$. For any $u,v\in V$, define $\vphi_{u,v}:G\to \R$ by
	\[
	\vphi_{u,v}(g)=\abs{\Trings{gu,v}}\ .
	\]
	Since the central character $\chi_V$ is unitary,
	$\vphi_{u,v}$ is  $Z$-invariant and hence we may view it as a function on $G/Z$.
	Recall that $V$ is called \emph{supercuspidal}\footnote{
		Some texts use the term \emph{absolutely cuspidal}.
	} if $\vphi_{u,v}$ is compactly supported as function on $G/Z$ for all $u,v\in V$,
	and
	$V$ is called \emph{square-integrable} if $\vphi_{u,v}\in\LL{G/Z}$
	for all $u,v\in V$.
	When $G=Z$, all irreducible pre-unitary representations are supercuspidal.

	supercuspidal representations are clearly square-integrable.
    When $Z=1$ or $G$ is the group of $k_\nu$-points of a reductive algebraic group $\bfG$ over $k_\nu$, the square-integrable
	representations of $G$ are tempered; see \cite[Pr.~9.6]{Knapp86} for the case $Z=1$ and \cite[\S2.4]{Oh02} or \cite[p.~265]{Walds03}
    for the other case.\footnote{
        We believe this statement is correct for a general $\ell$-group $G$. 
    }

\subsection{Local Factors}
\label{subsec:local-factors}

    Let  $\bfG$ be a reductive algebraic group over $k$,
    and write $K_\nu=\bfG(\calO_\nu)$ (see \ref{subsec:adeles}).
    The facts stated here are proved in \cite{Flath79}.

\medskip

    For every $\nu\in\calV$, choose an irreducible smooth representation
    $V_\nu\in\Irr[sm]{\bfG(k_\nu)}$. As noted in \ref{subsec:admissible-revised},
    the representations $\{V_\nu\}_{\nu\in\calV}$
    are admissible. Suppose that for almost all $\nu$ we have
    $V_\nu^{K_\nu}\neq 0$. When this holds, choose a nonzero vector $v_\nu\in V^{K_\nu}$,
    and otherwise, choose an arbitrary nonzero vector $v_\nu\in V_\nu$.
    Let
    \[
    \bigotimes'_{\nu\in\calV}V_\nu=\underset{S}{\dirlim} \bigotimes_{\nu\in S}V_\nu
    \]
    where $S$ ranges over the finite subsets of $\calV$, and for $S\subseteq S'$,
    we embed $\bigotimes_{\nu\in S}V_\nu$ in $\bigotimes_{\nu\in S'}V_\nu$ via
    $\bigotimes_{\nu\in S}w_\nu\mapsto\bigotimes_{\nu\in S}w_\nu\otimes\bigotimes_{\nu\in S'-S}v_\nu$.
    It turns out that $V^{K_\nu}$ is $1$-dimensional for almost all $\nu$, and hence
    $\bigotimes'_{\nu}V_\nu$ is independent of the vectors $v_\nu\in V_\nu$,  up to isomorphism.

    The space $\bigotimes'_{\nu\in\calV}V_\nu$ carries an obvious
    $\bfG(\bbA)$-action making it into an admissible smooth representation of $\bfG(\bbA)$.
    The representation $\bigotimes'_{\nu\in\calV} V_\nu$ is irreducible, and conversely,
    every   $V\in \Irr[sm]{\bfG(\bbA)}$ can be factored as
    $
    V=\bigotimes'_{\nu\in\calV} V_\nu
    $,
    where the factors $\{V_\nu\}_{\nu\in\calV}$, called the \emph{local factors} of $V$, are unique up to isomorphism.

    A necessary and sufficient condition for $U\in \Irr[sm]{\bfG(k_\nu)}$ to be
	isomorphic to the  $\nu$-local-factor of $V\in\Irr[sm]{\bfG(\bbA)}$ is to have $U\leq V$ when $V$ is viewed as a $\bfG(k_\nu)$-module.
	Furthermore, if we are given
	for all $\nu\in\calV$ a compact open subgroup $L_\nu\co \bfG(k_\nu)$
    such that $L_\nu=K_\nu$ for almost all $\nu\in\calV$, then
	$V=\bigotimes'_\nu V_\nu$ has a nonzero $\prod_\nu L_\nu$-invariant vector
    if and only if each factor
	$V_\nu$ has a nonzero $L_\nu$-invariant
	vector.
	
\medskip

	Everything stated above remains correct if one replaces smooth representations with pre-unitary
	representations. One should then chose the vectors $\{v_\nu\}_{\nu\in\calV}$ to be unit vectors.

\subsection{Automorphic Representations}
\label{subsec:automorphic-reps}

    Let $\bfG$ be a reductive algebraic group over $k$ and let $\bfZ$ be the center of $\bfG$.
    Fix a smooth unitary  character
    \[\omega:\bfZ(\bbA)/\bfZ(k)\to\units{\C}\ .\]
    We let $\LL[\omega]{\bfG(k)\leftmod\bfG(\bbA)}$ denote the space of measurable functions $\vphi:\bfG(k)\leftmod\bfG(\bbA)\to \C$
    satisfying $\vphi(ag)=\omega(a)\vphi(g)$ for all $a\in \bfZ(\bbA)/\bfZ(k)$ and whose absolute value $|\vphi|$ is square-integrable
    when viewed as a function on $\bfZ(\bbA)\bfG(k)\leftmod\bfG(\bbA)$. The space $\LL[\omega]{\bfG(k)\leftmod\bfG(\bbA)}$
    is a Hilbert space with respect to the inner product
    \[
    \Trings{\vphi,\psi}=\int_{x\in \bfZ(\bbA)\bfG(k)\leftmod\bfG(\bbA)} \vphi x\cdot\quo{\psi x}\,\,\mathrm{d}\mu_{\bfZ(\bbA)\bfG(k)\leftmod\bfG(\bbA)}\ .
    \]
    Here, $\mu_{\bfZ(\bbA)\bfG(k)\leftmod\bfG(\bbA)}$ is a fixed
    right  $\bfG(\bbA)$-invariant measure on $\bfZ(\bbA)\bfG(k)\leftmod\bfG(\bbA)$,
    and $\vphi x\cdot\quo{\psi x}$ means $\vphi x'\cdot\quo{\psi x'}$ where $x=\bfZ(\bbA)x'$ (this expression
    is independent of the choice of the representative $x'$). Note  that  $\vphi x\cdot\quo{\psi x}$
    is right $\bfZ(\bbA)\bfG(k)$-invariant and can hence  be regarded as a function on $\bfZ(\bbA)\bfG(k)\leftmod\bfG(\bbA)$.

	We let $\bfG(\bbA)$
	act on $\LL[\omega]{\bfG(k)\leftmod\bfG(\bbA)}$
	via $(g\vphi)x=\vphi(xg)$, making
	$\LL[\omega]{\bfG(k)\leftmod\bfG(\bbA)}$ into
	a unitary representation.
    Recall that an automorphic representation
    of $\bfG$ with central character $\omega$ is an irreducible pre-unitary representation $V\in\Irr[pu]{\bfG(\bbA)}$
    such that
    \[
    V\wc \LL[\omega]{\bfG(k)\leftmod\bfG(\bbA)}\ .
    \]
    The representation $V$ is said to be in the \emph{discrete spectrum} or just \emph{discrete} if
    $V\leq \LL[\omega]{\bfG(k)\leftmod\bfG(\bbA)}$. In this case,
    $V$ is admissible and can therefore be written as
    a product of local factors $\bigotimes'_\nu V_\nu$  as in \ref{subsec:local-factors}.

\medskip

	When $\bfZ$ is finite (as an algebraic group), the group $\bfZ(\bbA)$ is profinite,
	and in particular compact. Viewing
	$\LL{\bfG(k)\leftmod\bfG(\bbA)}$ as a representation of $\bfZ(\bbA)$ and applying the
	Peter-Weyl Theorem shows that
	\begin{equation*}\label{EQ:automorphic-Hilbert-sum-decomp}
	\LL{\bfG(k)\leftmod\bfG(\bbA)}=\hat{\bigoplus_\omega}\,\LL[\omega]{\bfG(k)\leftmod\bfG(\bbA)}\ ,
	\end{equation*}
	where $\omega$ ranges over the unitary characters of $\bfZ(\bbA)/\bfZ(k)$
	(one has to choose the Haar measure of $\bfZ(\bbA)/\bfZ(k)$ to have total measure
	$1$ and the right $\bfG(\bbA)$-invariant measures on $\bfG(k)\leftmod \bfG(\bbA)$
	and $\bfZ(\bbA)\bfG(k)\leftmod \bfG(\bbA)$ such that
	\[\int\vphi\, \mathrm{d}\mu_{\bfG(k)\leftmod \bfG(\bbA)}=
	\int_{x\in\bfZ(\bbA)\bfG(k)\leftmod \bfG(\bbA)}\int_{y\in \bfZ(\bbA)/\bfZ(k)}\vphi(xy)\,\mathrm{d}\mu_{\bfZ(\bbA)/\bfZ(k)}\,
	\mathrm{d}\mu_{\bfZ(\bbA)\bfG(k)\leftmod \bfG(\bbA)}\]
	for all $\vphi\in \CSLC{\bfG(k)\leftmod\bfG(\bbA)}$).
	We may therefore  define automorphic representations of $\bfG$
	as irreducible pre-unitary representations $V\in\Irr[pu]{\bfG(\bbA)}$
	such that
	$
    V\wc \LL{\bfG(k)\leftmod\bfG(\bbA)}
    $.

\medskip

    Assume henceforth that  $\bfZ$ is finite and $\bfG(k)\leftmod\bfG(\bbA)$
    is compact.
    This holds when $\bfG$ is semisimple and $k$-anisotropic.
    Then $\LL{\bfG(k)\leftmod\bfG(\bbA)}$ is admissible (Example~\ref{EX:LL-Gamma-mod-G-admissible}), and hence any
    automorphic representation is in the discrete spectrum (Theorem~\ref{TH:admissible-pu-are-completely-red}).
    We now determine under what conditions a pre-unitary representation of $\bfG(k_\nu)$ is a local factor
    of a discrete automorphic representation.

    Fix $\eta\in\calV$. We shall  view $\bfG(\bbA)$ as $\bfG(k_\eta)\times \bfG(\bbA^{\{\eta\}})$ (notation as in \ref{subsec:adeles}).
    Choose $K_\eta\co \bfG(k_\eta)$ and $K^\eta\co \bfG(\bbA^{\{\eta\}})$, and let
    $
    K:=K_\eta\times K^\eta\co \bfG(\bbA)$.
	The double coset space
	\[\bfG(k)\leftmod\bfG(\bbA)/(\bfG(k_\eta)\times K^{\eta})\]
	is compact and discrete, hence finite. Let $(1,g_1),\dots,(1,g_t)\in \bfG(k_\eta)\times \bfG(\bbA^{\{\eta\}})$ be representatives
	for the double cosets.
	For each $1\leq i\leq t$, define
	\[
	\Gamma_i=\bfG(k)\cap (\bfG(k_\eta)\times g_i K^{\eta}g_i^{-1})
	\]
	and view $\Gamma_i$ as a subgroup of $\bfG(k_\eta)$. It is a standard fact that there is an
	isomorphism of topological (right) $\bfG(k_\eta)$-spaces
	\[
	\bigsqcup_{i=1}^t\Gamma_i\leftmod \bfG(k_\eta)\to \bfG(k)\leftmod\bfG(\bbA)/(1\times K^{\eta})
	\]
	given by sending $\Gamma_i g$ to $\bfG(k)(g,g_i)(1\times K^{\eta})$.
	In particular, $\Gamma_i$ is a cocompact lattice in $\bfG(k_\eta)$ for all $1\leq i\leq t$.

	\begin{prp}\label{PR:local-factor-embedding}
		In the previous setting, the following holds:
		\begin{enumerate}
			\item[(i)] Let $U$ be an irreducible pre-unitary subrepresentation
			of $\LL{\Gamma_i\leftmod \bfG(k_\eta)}$ such that $U^{K_\eta}\neq 0$. Then $U$ is
			the $\eta$-local-factor of an automorphic representation $V$ of $\bfG$ with $V^{K}\neq 0$.
			\item[(ii)] Conversely, if $V=\bigotimes'_\nu V_\nu$ is an automorphic representation of
			$\bfG$ such that $V^{K}\neq 0$,
			then there is $1\leq i\leq t$ such that $V_\eta$ is isomorphic to a subrepresentation
			of $\LL{\Gamma_i\leftmod\bfG(k_\eta)}$ and $V_\eta^{K_\eta}\neq 0$.
		\end{enumerate}
	\end{prp}
	
	\begin{proof}
		(i) By the previous discussion, we may view $U$ as a $\bfG(k_\eta)$-submodule of
		$\LL{\bfG(k)\leftmod\bfG(\bbA)/(1\times K^\eta)}$.
		Since
		$\LL{\bfG(k)\leftmod\bfG(\bbA)}$
		decomposes as a direct sum
		of irreducible unitary representations of $\bfG(\bbA)$, the projection of $U$ onto
		one of those representations,
		call it $V_1$, must be nonzero, so $U$ must be a local factor
		of $V:=\sm{(V_1)}$. Since $V^{K}$ contains
		a copy of $U^{K_\eta}\subseteq \LL{\bfG(k)\leftmod\bfG(\bbA)/ K}$, we have $V^{K}\neq 0$.
		
		(ii) Let $0\neq \vphi\in V^{K}$.
		Since $V$ is in the discrete spectrum,
		we may view $\vphi$ as a  function in $\LL{\bfG(k)\leftmod\bfG(\bbA)/(1\times K^\eta)}$,
		which is isomorphic to $\bigoplus_{i=1}^t\LL{\Gamma_i\leftmod \bfG(k_\eta)}$ as  $\bfG(k_\eta)$-modules.
		Let $U_1$ be the $\bfG(k_\eta)$-module generated by $\vphi$. Then there is $i$
		such that the projection of $U_1$ onto $\LL{\Gamma_i\leftmod \bfG(k_\eta)}$ is nonzero.
		Let $U$ be an irreducible smooth $\bfG(k_\eta)$-submodule of this image (it exists
		because $\LL{\Gamma_i\leftmod \bfG(k_\eta)}$ is admissible and hence completely reducible,
        cf.\ Theorem~\ref{TH:admissible-pu-are-completely-red}).
		Then we must have $U=V_\eta$. That $V_\eta^{K_\eta}\neq 0$ follows
		from $V^K\neq 0$; see \ref{subsec:local-factors}.
	\end{proof}

    Suppose now that $\bfG$ is almost simple, let $G=\bfG(k_\eta)/\bfZ(k_\eta)$ and write
    \[
    \quo{\Gamma}_i=\im(\Gamma_i\to G)\ .
    \]
    Then $G$ acts faithfully on the affine Bruhat-Tits building $\calB$ of $\bfG(k_\eta)$, making it into
    an almost transitive $G$-complex (Example~\ref{EX:G-complex-building-III}). Recall from \ref{subsec:quotienst-of-simp-comps}
    that for $\Gamma\leq G$, we write $\Gamma\leq_\calB G$
    to denote that $\Gamma\leftmod \calB$ is a simplicial complex and that the quotient
    map $\calB\to \Gamma\leftmod\calB$
    is a cover map.
    Combining Proposition~\ref{PR:local-factor-embedding} with Theorem~\ref{TH:Ramanujan-criterion}
    and Remark~\ref{RM:non-faithful-action}, we get:
	
	\begin{thm}\label{TH:automorphic-ramanujan}
        Let $F:\catC(G,\calB)\to \catPHil$ be an elementary functor (e.g.\
        $\Omega_i^+$, $\Omega_i^\pm$ or $\llFlag$). Write $F\cong \llf \circ S$
        as in Definition~\ref{DF:elementary-functor},  let $x_1,\dots,x_s$ be representatives
        for the $G$-orbits in $S\calB$, and let $L_j=\Stab_{\bfG(k_\eta)}(x_j)\times K^\eta$ ($1\leq j\leq s$).
        Assume that for any automorphic representation $V=\bigotimes'_{\nu}V_\nu$
        of $\bfG$
        with $V^{L_1}+\dots V^{L_s}\neq 0$
        (resp.\ $V^{\bfZ(k_\eta)\times K^\eta}\neq 0$), the local factor $V_\eta$ is tempered or finite-dimensional.
        Then $\quo{\Gamma}_i\leftmod \calB$ is $F$-Ramanujan
        (resp.\ completely Ramanujan)
        for every $1\leq i\leq t$ such that $\quo{\Gamma}_i\leq_{\calB} G$.
        The converse holds when $\quo{\Gamma}_i\leq_\calB G$ for all $1\leq i\leq t$.
    \end{thm}

    \begin{remark}\label{RM:lattice-subgroup}
        (i)
        Fix one of the representatives $g_i$ and
        choose some $\rho\in\calV-\{\eta\}$ and $K^{\{\eta,\rho\}}\co \bfG(\bbA^{\{\eta,\rho\}})$.
        For every $n\geq 0$, let  $K^\eta(n)=\bfG(\calO_\rho,\pi_\rho^{n}\calO_\rho)\times K^{\{\eta,\rho\}}$,
        $\Gamma_i(n)=\bfG(k)\cap (\bfG(k_\eta)\times g_iK^\eta(n)g_i^{-1})$ and $\quo{\Gamma}_i(n)=\im(\Gamma_i(n)\to G)$.
        Then  $\{\quo{\Gamma}_i(n)\}_{n\geq 0}$ is a decreasing family of normal
        subgroups of $\quo{\Gamma}_i(0)$ with trivial intersection.
        Since  $\quo{\Gamma}_i(0)\leftmod\calB$
        is finite,  Corollary~\ref{CR:finite-subsect-away} implies that there is $n_0\in\N$ such that $\quo{\Gamma}_i(n_0)\leq_G\calB$.
        In particular, we have $\quo{\Gamma}\leq_G\calB$ for $\Gamma=\bfG(k)\cap (\bfG(k_\eta)\times g_iK^\eta g_i^{-1})$
        whenever $K^\eta\subseteq K^\eta(n_0)$.
        Notice, however, that replacing $K^\eta$ with $K^\eta(n_0)$  may increase the number
        of double cosets in  $\bfG(k)\leftmod\bfG(\bbA)/(\bfG(k_\eta)\times K^{\eta})$ and $n_0$ depends a priori  on
        the representative $g_i$.

        (ii) When $\bfG$ is simply-connected and $k_\eta$-isotropic, strong approximation
        (\cite{Marg77}, \cite{Prasad77})
        implies that $\bfG(k)\leftmod\bfG(\bbA)/(\bfG(k_\eta)\times K^{\eta})$ consists of a single double coset
        for every $K^\eta$. Thus, by (i), there is $K^\eta_0\co \bfG(\bbA^{\{\eta\}})$
        such that $\quo{\Gamma}_i=\quo{\Gamma}_1\leq_\calB \bfG(k_\eta)$ whenever
        $K^\eta\subseteq K^\eta_0$, in which case the last statement of Theorem~\ref{TH:automorphic-ramanujan}
        can be applied.
    \end{remark}

    Given information about the automorphic spectrum of $\bfG$, one can apply
    Theorem~\ref{TH:automorphic-ramanujan} to show  existence of Ramanujan $G$-quotients of $\calB$.
    We mention here several places where such ideas were applied in the literature, sometimes implicitly
    or in an equivalent formulation:
    \begin{itemize}
        \item Lubotzkly, Phillips and Sarnak  \cite{LubPhiSar88}, and independently
        Margulis \cite{Marg88}, constructed infinite families
        of Ramanujan $(p+1)$-regular graphs for every prime $p$ using results of Eichler \cite{Eichler54} and Igusa \cite{Igusa59}
        about modular forms. (See also Delinge's proof of the Ramanujan--Petersson conjecture
        for modular forms \cite{Deligne74}.)
        In our setting, this corresponds to taking $k=\Q$
        and $\bfG$ to be an inner form of $\bPGL_2$ which  splits over $k_\eta$.
        \item Morgenstern \cite{Morg94} used Drinfeld's proof of the Ramanujan--Petersson conjecture
        for $\bGL_2$ when $\Char k>0$ \cite{Drinfel88}
        to construct  infinite families of Ramanujan ${(q+1)}$-regular graphs for every prime power $q$. Again, the
        corresponding group $\bfG$ is an inner form of $\bPGL_2$.
        \item Lubotzky, Samuels and Vishne \cite{LubSamVi05} applied Lafforgue's proof of the Ra\-ma\-nu\-jan--Petersson conjecture
        for $\bGL_d$ when $\Char k>0$ \cite{Laff02} to construct infinite families of Ramanujan complexes
        (in the sense of Chapter~\ref{sec:ramanujan-complexes}). The corresponding group $\bfG$ is an inner form of $\bPGL_n$
        which splits over $k_\eta$.
        \change{32}{
        \item Li \cite{Li04} independently gave similar constructions of Ramanujan complexes, using
        results of Laumon, Rapoport and Stuhler, who proved a special case of the Ramanujan--Petersson
        conjecture for anisotropic inner forms of $\uGL_n$  \cite[Th.~14.12]{LaRaSt93}. (In fact, in \cite{Li04} it is only shown
        that the complexes are $\C[a_i,a_i^*]$-Ramanujan for all $0<i<d$ (notation as in Chapter~\ref{sec:ramanujan-complexes}
        and \ref{subsec:Ramanujan-quotients}; cf.\ Proposition~\ref{PR:unitary-dual-topology}).
        However, they are in fact Ramanujan in the sense of Chapter~\ref{sec:ramanujan-complexes} by
        \cite [Pr.~1.5]{LubSamVi05} or Example~\ref{EX:ramanujan-equivalence-Bd}.)
        }
        \item Ballantine and Ciubotaru  \cite{BallCiub11} constructed infinite families of Ramanujan $(q+1,q^3+1)$-biregular graphs
        for every prime power $q$. The corresponding group $\bfG$ is an inner form of $\mathbf{SU}(3)$, and they
        use the classification of the automorphic spectrum of $\bfG$ due to
        Rogawski \cite{Roga90}.
    \end{itemize}
    We hope  our work will facilitate further results of this kind.

\medskip

    Let $D$ be a central division algebra over $k_\eta$.
    The rest of this chapter concerns with showing that when $\Char k>0$, the
    affine building  of $\nPGL{D}{d}$ admits infinitely many non-isomorphic $\nPGL{D}{d}$-quotients which
    are  completely Ramanujan.
    The proof uses
    Theorem~\ref{TH:automorphic-ramanujan} to transfer the problem to a question about automorphic
    representations, and then applies Lafforgue's work on the Ramanujan conjecture for $\uGL_n$ \cite{Laff02},
    together with the Jacquet--Langlands correspondence in positive characteristic, established by Badulescu and Roche \cite{BadulRoch14}.

\subsection{The Affine Building of $\nPGL{D}{d}$}
\label{subsec:building-of-GLdD}

    Fix $\nu\in\calV$ and let $F=k_\nu$.
    Given a central simple $F$-algebra $A$,
    let $\deg A$ denote the degree of $A$, and let  $\Nrd_{A/F}:A\to F$ denote the reduced norm map;
    see \cite[\S1]{InvBook} for the relevant definitions and further details.
    There is a reductive
    algebraic group $\uGL_{n,A}$ over $F$, unique up to isomorphism, such that for every commutative $F$-algebra $R$,
	the groups $\uGL_{n,A}(R)$ and $\nGL{A\otimes_F R}{n}$ are naturally isomorphic.
	The topology on $F$ induces a topology on $\nGL{A}{n}=\uGL_{n,A}(F)$, making it into an $\ell$-group
	(\ref{subsec:adeles}). We further let $\uPGL_{n,A}=\uGL_{n,A}/\bfZ$, where $\bfZ\cong \nGm{F}$ is the center of $\uGL_{n,A}$.

\medskip

    Let $D$ be a finite dimensional central division $F$-algebra of degree $r$.
    By \cite[\S12]{MaximalOrders}, the additive valuation $\nu:F\onto \Z\cup\{\infty\}$ extends uniquely
    to an additive valuation  $\nu_D:D\to \R\cup\{\infty\}$ given by:
    \[
    \nu_D(x)=r^{-1}\nu(\Nrd_{D/F}(x))\ .
    \]
    Since the residue field of $F$ is finite, $\im(\nu_D)=\frac{1}{r}\Z$ and the residue
    division ring of $(D,\nu_D)$ is the Galois field of cardinality $q_\nu^r$ \cite[Th.~14.3]{MaximalOrders}. We
    fix an element $\pi_D\in D$ with $\nu_D(\pi_D)=\frac{1}{r}$ and write
    \[\calO_D=\{x\in D\suchthat \nu_D(x)\geq 0\}\ .\]
    The topology on $D=\uGL_{1,D}(F)$ coincides with the topology induced by $\nu_D$ and the topology
    on $\nMat{D}{d}=\uGL_{d,D}(F)$ coincides with the topology induced from $\nMat{D}{d}\cong D^{d^2}$.

\medskip

    The affine Bruhat-Tits building of $\nPGL{D}{d}:=\nGL{D}{d}/\units{F}$, denoted $\calB_d(D)$, is a simplicial
    complex of dimension $(d-1)$. It is constructed
    exactly as the affine Bruhat-Tits building of $\nPGL{F}{d}$ described in Chapter~\ref{sec:ramanujan-complexes} with
    the following modifications: Take $G=\nPGL{D}{d}$, let $K$ be the subgroup generated
    by the images of $\nGL{\calO_D}{d}$ and
    \[
    \left[\DDotsArr{\pi_D}{\pi_D}\right]
    \]
    in $G$,
    and replace $\pi=\pi_\nu$ with $\pi_D$. See \cite[\S3]{AbramNebe02} for further details and an alternative construction.
    The group $\nPGL{D}{d}$ acts on $\calB_d(D)$ on the left via its action on $\calB_d(D)^{(0)}=\nPGL{D}{d}/K$, making
    $\calB_d(D)$ into an almost transitive $\nPGL{D}{d}$-complex.
    The building $\calB_2(D)$ is a $(q_\nu^r+1)$-regular tree.

\medskip

    One can define a vertex coloring $C_0$ and a directed-edge coloring $C_1$ on $\calB_d(D)$ as in Chapter~\ref{sec:ramanujan-complexes}
    using the map $c:\nPGL{D}{d}\to \Z/d\Z$ given by
    \[c(g\units{F})={ \nu(\Nrd_{\nMat{D}{d}/F} (g))}+d\Z\ .\]
    One can then define the operators $a_1,\dots,a_{d-1}$ of Chapter~\ref{sec:ramanujan-complexes}  for $\nPGL{D}{d}$-quotients
    of $\calB_d(D)$. It can be shown using the building axioms that $a_1,\dots,a_{d-1}$ still
    commute among themselves and that $a_i^*=a_{d-i}$. Furthermore, \gap{}it seems correct that $a_1,\dots,a_{d-1}$
    generate  $\Alg{\catC(\nPGL{D}{d},\calB_d(D)),\Omega_0^+}{}$, and hence their common spectrum
    is equivalent to the $0$-dimensional spectrum (cf.\ Example~\ref{EX:zero-dim-spec-of-Bd}).
    We will not need this fact, however.
	
\subsection{Representations of $\nGL{D}{d}$}
\label{subsec:reps-of-GLd}

	Fix $\nu\in\calV$, let $F=k_\nu$, and let $D$ be a finite dimensional central division $F$-algebra
    of degree $r$.
    This section  recalls various facts about representations of
    $\nGL{D}{d}$.

\medskip
	
	Let $n_1,\dots,n_t\in\N$ and let $n=n_1+\dots+n_t$. Denote by
	$\bfP_{(n_1,\dots,n_t)}$ and $\bfM_{(n_1,\dots,n_t)}$ the closed algebraic subgroups
	of $\uGL_{n,D}$ consisting of block matrices of the form
    \[
    \left[\UTDotsArr{*}{*}{*}\right]\qquad\text{and}\qquad\left[\DDotsArr{*}{*}\right]\,,
    \]
    where the $i$-th  $*$ on the diagonal  stands for an $n_i\times n_i$ matrix.
    The group $\bfP_{(n_1,\dots,n_t)}$ is a \emph{standard parabolic subgroup}
    of $\uGL_{n,D}$ and
    $\bfM_{(n_1,\dots,n_t)}$ is its \emph{standard Levi factor}.
    We write $P_{(n_1,\dots,n_t)}=\bfP_{(n_1,\dots,n_t)}(F)$ and
    $M_{(n_1,\dots,n_t)}=\bfM_{(n_1,\dots,n_t)}(F)$.
   	Let $\delta_{(n_1,\dots,n_t)}$ denote the unimodular character of $P_{(n_1,\dots,n_t)}$.
   	Given smooth
   	representations $V_i\in \Rep[sm]{\nGL{D}{n_i}}$ ($1\leq i\leq t$),  let
   	\[
   	V_1\times V_2\times \dots \times V_t = \mathrm{Ind}_{P_{(n_1,\dots,n_t)}}^{\nGL{D}{n}}
    \left(\delta_{(n_1,\dots,n_t)}^{1/2}(V_1\otimes V_2\otimes \dots\otimes V_t)\right)
   	\]
	Here, $V_1\otimes V_2\otimes \dots\otimes V_t$ is viewed
    as an $M_{(n_1,\dots,n_t)}$-module, which is in turn viewed as a $P_{(n_1,\dots,n_t)}$-module
    via the homomorphism $P_{(n_1,\dots,n_t)}\to M_{(n_1,\dots,n_t)}$ removing the blocks above the diagonal.

    The operation $\times$ is associative up to a natural isomorphism \cite[Pr.~1.1(b)]{Zelevinsky80} but not commutative in general.
    However, when $V_1,\dots,V_t$ are of finite length, semisimplification
    of $V_1\times\dots\times V_t$ does not depend on the order of terms
    \cite[Th.~1.9]{Zelevinsky80}. In particular, if $V_1\times \dots \times V_t$ is irreducible,
    then $V_1\times \dots \times V_t\cong V_{\sigma 1}\times \dots\times V_{\sigma t}$
    for any permutation $\sigma$ on $\{1,\dots,t\}$. 
    When $V_1,\dots,V_t$ are admissible, there is a canonical-up-to-scaling isomorphism
    $(V_1\times \dots\times V_t)\check{~} \cong \check{V}_1\times\dots \times \check{V}_t$
    (this is similar to \cite[Pr.~2.25c]{BerZel77}). If  $V_1,\dots,V_t$ are  pre-unitary and
    one identifies $V_i$ with $\check{V}_i$ using the inner product on $V_i$
    (cf.\ the proof
    of Theorem~\ref{TH:Schurs-Lemma-for-admissible}), then
    the previous isomorphism gives rise to an inner product on $V_1\times\dots\times V_t$, making it
    into a pre-unitary representation.

    \begin{thm}\label{TH:tempred-classification}
    	For all $1\leq i\leq t$, let $V_i\in \Irr[pu]{\nGL{D}{n_i}}$ be  square-integrable.
    	Then $V_1\times \dots \times V_t$ is an irreducible pre-unitary tempered representation of
    	$\nGL{D}{n}$, where $n=n_1+\dots+n_t$. Any irreducible pre-unitary tempered representation of
    	$\nGL{D}{n}$ is obtained in this manner with $(V_1,n_1),\dots,(V_t,n_t)$ uniquely determined
    	up to isomorphism and reordering.
    \end{thm}

    \begin{proof}
    	By \cite[Pr.~III.4.1]{Walds03}, every irreducible tempered pre-unitary representation of $\nGL{D}{n}$
        is a subquotient of $V_1\times \dots\times V_t$, where  $V_i\in \Irr[sm]{\nGL{D}{n_i}}$
        is square-integrable, and $(V_1,n_1),\dots,(V_t,n_t)$ are uniquely determined up to isomorphism and
        reordering. It is therefore enough to show that $V_1\times \dots\times V_t$ is irreducible.
        This is shown in \cite[Th.~B.2.d]{DelKazhVign84} when $\Char F=0$ and in \cite[Th.~1.1]{Badul04}
        when $\Char F>0$. (In fact, $V_1\times \dots\times V_t$ is irreducible under the milder assumption
        that each $V_i$ is pre-unitary; see \cite{Seche09} and \cite{BadulHennLema10} for the cases
        $\Char F=0$ and $\Char F\neq 0$, respectively.)
    \end{proof}
	
	Let
	$\xi_n:\nGL{D}{n}\to \units{\C}$ be defined by
	$
	\xi_n(g)=|\Nrd_{\nMat{D}{n}/F}(g)|_\nu
	$. For $V\in\Rep[sm]{\nGL{D}{n}}$ and $\alpha\in\R$, we write
	\[
	\nu^\alpha V=\xi_n^\alpha V\ .
	\]
	Given $V_i\in\Rep[sm]{\nGL{D}{n_i}}$ ($1\leq i\leq t$) with $n=\sum_in_i$,
	there is an obvious canonical isomorphism
	\[\nu^\alpha(V_1\times\dots\times V_t)\cong\nu^{\alpha}V_1\times\dots\times \nu^{\alpha}V_t\ .\]
	
	\begin{thm}\label{TH:Silberger}
		For all $1\leq i\leq t$, let $V_i\in \Irr[pu]{\nGL{D}{n_i}}$   be tempered,
		and let $\alpha_1,\dots,\alpha_t\in \R$ satisfy
		$\alpha_1> \dots> \alpha_t$.
		Then $\nu^{\alpha_1}V_1\times\dots \times \nu^{\alpha_t}V_t$
		has a unique irreducible quotient, called the
		\emph{Langlands quotient}.
        When unitarizable, the Langlands quotient of $\nu^{\alpha_1}V_1\times\dots \times \nu^{\alpha_t}V_t$ is not
        tempered if $t>1$.
	\end{thm}
	
	\begin{proof}
    	The first part follows from \cite[Thm.~4.1(1)]{Silberger78}. The second part
    	follows from the uniqueness claim
        in \cite[Thm.~4.1(2)]{Silberger78}.
    \end{proof}

    For the sake of simplicity, we henceforth specialize to the case $D=F$. The  results
    to follow
    extend to general $D$ after some modifications; see \cite{Tadic90} and \cite{BadulHennLema10}.

\medskip

    Let $V\in\Irr[pu]{\nGL{F}{r}}$ be tempered, and let $s\in\N$. By Theorem~\ref{TH:Silberger}, the representation
    $\nu^{\frac{s-1}{2}}V\times \nu^{\frac{s-3}{2}}V\times\dots\times \nu^{\frac{1-s}{2}}V$
    has a unique irreducible quotient, which we denote by
    \[
    u(V,s)\ .
    \]
    We will occasionally use the following alternative characterization of $u(V,s)$:

    \begin{prp}\label{PR:def-of-u}
        Let $V\in\Irr[pu]{\nGL{F}{r}}$ be tempered and assume $u(V,s)$ is unitarizable. Then
        \[I_s(V):=\mathrm{Ind}_{P_{(r,\dots,r)}}^{\nGL{F}{rs}}(\,\underbrace{V\otimes\dots\otimes V}_{\text{$s$ times}}\,)\ .\]
        has a unique irreducible subrepresentation  which is isomorphic to $u(V,s)$.
    \end{prp}

    \begin{proof}
        Since $u(V,s)$ is unitarizable, $u(V,s)\cong u(V,s)\check{~}$.
        It is therefore enough to show that $I_s(V)\cong (\nu^{\frac{s-1}{2}}V\times \dots\times \nu^{\frac{1-s}{2}}V)\check{~}$.
        Indeed,
        \begin{align*}
        I_s(V)&=
        \mathrm{Ind}_{P_{(r,\dots,r)}}^{\nGL{F}{rs}}(V\otimes\dots\otimes V)\\
        &=\mathrm{Ind}_{P_{(r,\dots,r)}}^{\nGL{F}{rs}}(\nu^{\frac{s-1}{2}}\nu^{\frac{1-s}{2}}V\times \dots
        \times \nu^{\frac{1-s}{2}}\nu^{\frac{s-1}{2}}V)\\
        &=\Ind_{P_{(r,\dots,r)}}^{\nGL{F}{rs}}(\delta_{(r,\dots,r)}^{1/2}(\nu^{\frac{1-s}{2}}V\times \dots\times \nu^{\frac{s-1}{2}}V))\\
        &=\nu^{\frac{1-s}{2}}V\times \nu^{\frac{3-s}{2}}V\times\dots\times \nu^{\frac{s-1}{2}}V\\
        &\cong(\nu^{\frac{s-1}{2}}\check{V})\check{~}\times (\nu^{\frac{s-3}{2}}\check{V})\check{~}\times\dots\times (\nu^{\frac{1-s}{2}}\check{V})\check{~}\\
        &\cong(\nu^{\frac{s-1}{2}}V\times \dots\times \nu^{\frac{1-s}{2}}V)\check{~}\ .
        \end{align*}
        We used the fact that $V\cong \check{V}$ in the last isomorphism.
    \end{proof}

    \begin{example}[{cf.\ \cite[Ex.~3.2]{Zelevinsky80}}]\label{EX:square-integrable-char}
    	Let $\chi:\nGL{F}{1}=\units{F}\to \units{\C}$ be a  unitary
    	character.
    	Then $u(\chi,s)\in\Irr[sm]{\nGL{F}{s}}$ is the character $\chi\circ \Nrd_{\nMat{F}{s}/F}$.
        To see this, observe that the function $g\mapsto \chi(\Nrd_{\nMat{F}{s}/F}( g))$
        is in $I_s(\chi)$, so the character $\chi\circ \Nrd_{\nMat{F}{s}/F}$ must be the unique
        irreducible $\nGL{F}{s}$-submodule of $I_s(\chi)$.
    \end{example}

    \begin{thm}
        If $V\in\Rep[pu]{\nGL{F}{r}}$ is square-integrable, then $u(V,s)$ is unitarizable.
    \end{thm}

    \begin{proof}
        This follows from \cite[Th.~D]{Tadic86}.
    \end{proof}

    \begin{thm}\label{TH:Steinberg}
        Let $r,s\in\N$, let $n=rs$, and let $V$ be a pre-unitary supercuspidal representation of $\nGL{F}{r}$.
        Then:
        \begin{enumerate}
            \item[(i)] $I_s(V)$ of Proposition~\ref{PR:def-of-u} has a unique irreducible quotient, denoted $T(V,s)$.
            \item[(ii)] $T(V,s)$ is unitarizable and square-integrable. Moreover,
            any square-integrable representation of $\nGL{F}{n}$ is obtained in this manner,
            with $V$ and $s$ uniquely determined up to isomorphism.
        \end{enumerate}
    \end{thm}

    \begin{proof}
        See \cite[Prp.~2.10]{Zelevinsky80} for (i) and \cite[\S9.3, Thm.~6.1]{Zelevinsky80} for (ii).
    \end{proof}

	Let $K:=\nGL{\calO_\nu}{n}$.
    An irreducible smooth $\nGL{F}{n}$-module $V$ is called \emph{unramified} if
    $V^K\neq 0$.
    In this case, $\dim V^{K}=1$. (This holds because $V^K$ is an irreducible
    representation of $\Hecke[K]{\nGL{F}{n}}$, which is well-known to  be commutative.)
    For example, a character $\chi$
    of $\nGL{F}{1}=\units{F}$ is unramified when $\units{\calO_\nu}\subseteq\ker \chi$.

    Suppose we are given  $V_i\in \Irr[sm]{\nGL{F}{n_i}}$  ($1\leq i\leq t$) with $n=\sum_in_i$.
    If each $V_i$ is unramified,
    then
    $V_1\times\dots\times V_t$ has a unique irreducible unramified subquotient, and conversely, if
    such a subquotient exists, then $V_1,\dots,V_t$ must all be unramified. This follows
    easily from the Iwasawa decomposition $P_{(n_1,\dots,n_t)}K=\nGL{F}{n}$.

    \begin{prp}\label{PR:Jn-unramified}
        Let $V\in\Irr[pu]{\nGL{F}{r}}$ be tempered and unramified. Then:
        \begin{enumerate}
        \item[(i)] There are unramified unitary characters $\chi_1,\dots,\chi_r:\nGL{F}{1}\to \units{\C}$
        such that $V\cong \chi_1\times\dots\times \chi_r$.
        \item[(ii)]
        $u(V,s)$ is  unramified for all $s\geq 1$.
        \end{enumerate}
    \end{prp}

    \begin{proof}
    	(i) By Theorem~\ref{TH:tempred-classification}, we may assume $V$ is square-integrable.
    	By Theorem~\ref{TH:Steinberg}, we can write $V=T(U,t)$ for $t\in\N$ and $U$ a supercuspidal representation
    	of $\nGL{F}{r/t}$. Then $U$ is unramified. In this case, \cite[Pr.~2.6]{Casselman80} implies
        that there are characters $\psi_1,\dots,\psi_{r/t}:\nGL{F}{1}\to \units{\C}$ such that
        $U$ is a subrepresentation of $\psi_1\times \dots\times \psi_{r/t}$. Since $U$ is supercuspidal,
        we must have $r/t=1$ \cite[Pr.~1.10]{Zelevinsky80}, so  $U$ is
        an unramified character of $\nGL{F}{1}$. Now, by Example~\ref{EX:square-integrable-char},
        $u(U,t)$ is the unique unramified subquotient of $I_t(U)$, so $u(U,t)\cong T(U,t)$.
        Since $u(U,t)$ is non-tempered when $t>1$ (Theorem~\ref{TH:Silberger}) and $T(U,t)$ is square-integrable
        (Theorem~\ref{TH:Steinberg}), we must have $t=1$, hence $V=U$ is an unramified character of $\nGL{F}{1}$.

        (ii)
        Write $V=\chi_1\times\dots\times \chi_r$ as in (i). Then $u(V,s)$ 
        is the unique irreducible subrepresentation of 
        \[
        (\nu^{\frac{1-s}{2}}\chi_1\times\dots\times \nu^{\frac{1-s}{2}}\chi_r)\times \dots\times(\nu^{\frac{s-1}{2}}\chi_1\times
        \dots\times\nu^{\frac{s-1}{2}}\chi_r)\ .
        \]
        We claim that this subrepresentation is isomorphic to 
        $u(\chi_1,s)\times\dots\times u(\chi_r,s)$, which is irreducible by \cite[Th.~A]{Tadic86}
        (or \cite[Th.~4.2]{Zelevinsky80}). By Example~\ref{EX:square-integrable-char},
        $u(\chi_i,n)$ is unramified for all $i$, so this would show that $u(V,s)$ is  unramified.

\smallskip

		Suppose first that $\chi_1=\dots=\chi_r=\chi$. It is convenient to introduce
		$Z(\chi,s):=\nu^{\frac{s-1}{2}}u(\chi,s)$. Alternatively, $Z(\chi,s)$ is the unique irreducible
		subrepresentation of $\chi\times \nu\chi\times\dots\times \nu^{s-1}\chi$.
		For $U\in\Irr[sm]{\nGL{F}{n}}$, we write $U^{\times r}$ to abbreviate
		$U\times\dots\times U$ ($r$ times).
		We  claim that
		\begin{enumerate}
			\item[(1)] $Z(\chi,s)\times Z(\chi,s+1)\cong Z(\chi,s+1)\times Z(\chi,s)$ and
			\item[(2)] $Z(\chi,s+1)^{\times r}$ is isomorphic to a subrepresentation of
			$Z(\chi,s)^{\times r}\times (\nu^s\chi)^{\times r}$. 
		\end{enumerate}
		To prove (1), it enough to show that $Z(\chi,s)\times Z(\chi,s+1)$
		is irreducible, and this  follows from \cite[Th.~4.2]{Zelevinsky80} (in 
		\cite{Zelevinsky80}, $Z(\chi,s)$ is denoted  by $\Trings{a}$ where $a$ is the \emph{segment}
		$\{\chi,\nu\chi,\dots,\nu^{s-1}\chi\}$).
        We prove (2) by induction on $r$. The case $r=1$ is immediate from the characterization of
        $Z(\chi,s)$ as the unique irreducible subrepresentation of $\chi\times \nu\chi\times\dots\times \nu^{s-1}\chi$.
        For $r>1$, the induction hypothesis implies that $Z(\chi,s)^{\times r}\times (\nu^s\chi)^{\times r}$
        contains a copy of $Z(\chi,s)\times Z(\chi,s+1)^{\times(r-1)}\times \nu^s\chi$. By
        (1), the latter is isomorphic to $Z(\chi,s+1)^{\times(r-1)}\times Z(\chi,s)\times \nu^s\chi$,
        which contains a copy of $Z(\chi,{s+1})^{\times (r-1)}\times Z(\chi,s+1)=Z(\chi,s+1)^{\times r}$.
        
        Now, applying (2) repeatedly, we see that
      	$
        \chi^{\times r}\times(\nu\chi)^{\times r}\times\dots\times (\nu\chi^{s-1})^{\times r}
        $ contains a copy of $
        Z(\chi,s)^{\times r}$.
        This implies that
        $u(\chi,s)^{\times r}=(\nu^{\frac{1-s}{2}}Z(\chi,s))^{\times r}\cong\nu^{\frac{1-s}{2}}(Z(\chi,s))^{\times r}$
        is isomorphic to a subrepresentation
        of $\nu^{\frac{1-s}{2}}(\chi^{\times r}\times(\nu\chi)^{\times r}\times\dots\times (\nu^{s-1}\chi)^{\times r})\cong
        (\nu^{\frac{1-s}{2}}\chi)^{\times r}\times\dots\times (\nu^{\frac{s-1}{2}}\chi)^{\times r}$, as required.
        
\medskip

		Suppose now that $\chi_1,\dots,\chi_r$ are arbitrary,
		and write $\{\rho_1,\dots,\rho_t\}=\{\chi_1,\dots,\chi_r\}$ where
		$\rho_1,\dots,\rho_t$ are distinct. Let $r_i$ denote the number of indices $j$ with $\rho_i=\chi_j$.
		Since $\rho_1,\dots,\rho_t$ are unitary, \cite[Pr.~1.11(a)]{Zelevinsky80}
		implies that $\rho_i\times\rho_j$ is irreducible for all $i\neq j$,
		hence $\rho_i\times \rho_j\cong \rho_j\times \rho_i$. Using this, we can rearrange
		the terms in
		\[
		(\nu^{\frac{1-s}{2}}\chi_1\times\dots\times \nu^{\frac{1-s}{2}}\chi_r)\times \dots\times(\nu^{\frac{s-1}{2}}\chi_1\times
        \dots\times\nu^{\frac{s-1}{2}}\chi_r)
		\]
		to get
		\[
		((\nu^{\frac{1-s}{2}}\rho_1)^{\times r_1}\times\dots\times (\nu^{\frac{s-1}{2}}\rho_{1})^{\times r_1})
		\times \dots\times((\nu^{\frac{1-s}{2}}\rho_t)^{ \times r_t}\times
        \dots\times(\nu^{\frac{s-1}{2}}\rho_t)^{\times r_t})\ .
		\]
		By the previous paragraphs, this representation contains a copy of
		\[
		u(\rho_1,s)^{\times r_1}\times\dots\times u(\rho_t,s)^{\times r_t}
		\cong u(\chi_1,s)\times\dots\times u(\chi_r,s)\ ,
		\]
		so we are done.
\rem{
		
		we can rearrange $\chi_1,\dots,\chi_r$
		such that $\chi_1=\dots=\chi_{r_1}$, $\chi_{r_1+1}=\dots=\dots_{r_2}$, etc.\
		and $\chi_{r_1},\chi_{r_2},\dots$ are distinct.

        is the unique irreducible
        quotient of
        \[
        (\nu^{\frac{s-1}{2}}\chi_1\times\dots\times \nu^{\frac{s-1}{2}}\chi_r)\times \dots\times(\nu^{\frac{1-s}{2}}\chi_1\times
        \dots\times\nu^{\frac{1-s}{2}}\chi_r)\ .
        \]
        By \cite[Th.~9.10]{PrasRagh01}, this quotient is unramified. \gap{}(The source we have quoted states this result without proof;
        we were unsuccessful in finding a proof in the literature.
        However, according experts we have discussed with, this result is well-established.
        When $\chi_1,\dots,\chi_r$ are distinct, it follows from \cite[Pr.~3.6]{Casselman80}, which states
        that the representation above is generated by its
        $\nGL{\calO}{rs}$-invariant vectors.)
}
    \end{proof}

\subsection{The Jacquet-Langlands Correspondence}
\label{subsec:JL}

	Let $\nu$, $F$, $D$ be as in \ref{subsec:reps-of-GLd},
	and let  $r=\deg D$.
    For all $n\in\N$, write
    \[
    G_n=\nGL{F}{nr}\qquad\text{and}\qquad G'_n=\nGL{D}{n}\ .
    \]
    We choose Haar measures $\mu_{G_n}$, $\mu_{G'_n}$
    for $G_n$, $G'_n$ respectively.
	Two elements $g\in G_n$, $g'\in G'_n$
	are said to be \emph{in correspondence}, denoted
	$g\leftrightarrow g'$, if they have the same {reduced characteristic polynomial} (see \cite[\S9a]{MaximalOrders}).
	The element $g$ (resp.\ $g'$) is \emph{regular} if the roots of its reduced characteristic polynomial
	in an algebraic closure of $F$ are distinct.
    Denote by $\tilde{G}_n$ (resp.\ $\tilde{G}'_n$) the set of regular elements in $G_n$ (resp.\ $G'_n$).
	An element $g\in G_n$ is called \emph{$D$-compatible} if there is $g'\in G'_n$ such that
	$g\leftrightarrow g'$. This is equivalent to  saying that the degrees of the irreducible factors of the reduced characteristic
    polynomial of $g$ are divisible by $r$ \cite[Lm.~2.1]{BadulRoch14}.

\medskip

    Let $V$ be an irreducible smooth representation of $G'_n$. Harish-Chandra \cite{Harish81}
    showed that there
    exists a unique function $\psi_V:\tilde{G}'_n\to \C$ such that $\psi_V$ is locally constant, stable under conjugation,
    and for any $\vphi\in \Hecke{G'_n}$ supported on $\tilde{G'}$ one has
    \[
    \mathrm{Tr}(\vphi|_V)=\int_{g'\in  \tilde{G}'_n} \psi_V(g')\cdot\vphi(g')\,\mathrm{d}\mu_{G'_n}\ .
    \]
    (In fact, this holds for any $\vphi\in \Hecke{G'_n}$; see  \cite[\S2.5]{BadulRoch14}.)
    Notice that $\psi_V$ is independent of
    the Haar measure $\mu_{G'_n}$.
    The function $\psi_V$ is called the \emph{function character} of $V$. It determines $V$ up to isomorphism.
    A representation $V\in\Irr[sm]{G_n}$ is
	called \emph{$D$-compatible} if $\psi_V$ does not vanish on the set of regular $D$-compatible elements.

    Let $\what{G}_n^{(D)}$ denote the isomorphism classes
    of $D$-compatible irreducible pre-unitary representations of $G_n$ and let $\what{G}'_n$
    denote the isomorphism classes of all irreducible pre-unitary representations of $G'_n$.
    To avoid cumbersome statements, we will occasionally identify irreducible
	representations with their isomorphism classes.

    \begin{thm}\label{TH:local-JL}
		There exists a unique map $\LJ_\nu:\what{G}_n^{(D)}\to \what{G}'_n$, called
        \emph{local Jacquet--Langlands correspondence},  with the property
        that
        for all $V\in \what{G}_n^{(D)}$, there is $\veps=\veps(V)\in\{\pm1\}$ such that
		\[
		\psi_{V}(g)=\veps\cdot\psi_{\LJ_\nu(V)}(g')
		\]
		for all  $g\in G_n$, $g'\in G'_n$ with $g\leftrightarrow g'$.
		The map $\LJ_\nu$ has the following additional properties:
		\begin{enumerate}
            \item[(i)] $\LJ_\nu$ restricts to a bijection between the square-integrable representations
            of $G_n$ and the square integrable representations of $G'_n$ (considered up to isomorphism).
            In particular, all square-integrable representations of $G_n$ are $D$-compatible.
			\item[(ii)] Let $n_1,\dots,n_t\in\N$ and let $V_i\in\Rep[pu]{\nGL{F}{n_i}}$ ($1\leq i\leq t$).
            Assume that $\sum_in_i=nr$ and $V:=V_1\times \dots\times V_t$ is irreducible. If $r\mid n_i$ and   $V_i\in \what{G}_{n_i}^{(D)}$
            for all $i$, then $V$ is $D$-compatible and
            \[
            \LJ_\nu(V)=\LJ(V_1)\times\dots\times \LJ(V_t)
            \]
            Otherwise, $V$ is not $D$-compatible.
            \item[(iii)] If $[V]\in\what{G}_n^{(D)}$ is tempered, then so does  $\LJ_\nu(V)$.
		\end{enumerate}
	\end{thm}

    \begin{proof}
        The existence of $\LJ_\nu$ was established in \cite{Badul08} and \cite[Cor.~4.5]{BadulHennLema10}
        in the cases $\Char F=0$ and $\Char F>0$, respectively. 

        (i) See \cite{DelKazhVign84} ($\Char F=0$) and \cite{Badul02} ($\Char F>0$).

        (ii) This follows from the formula for the function character of an induced representation
        in \cite[Th.~3]{vanDij72} (see also the comment following
        that theorem).
        See also \cite[Pr.~3.4]{Badul07}.\rem{\gap{Sign issues?}}

        (iii) This follows from (i), (ii) and Theorem~\ref{TH:tempred-classification}.
    \end{proof}

    \begin{remark}
        (i) The map $\LJ_\nu$ is neither injective nor surjective in general (\cite[Rm.~3.2]{Badul07}
        and Proposition~\ref{PR:JL-image} below).
        When $D=F$, the map $\LJ_\nu$ is just the identity map.

        (ii) The Jacquet-Langlands correspondence is often defined to
        be the inverse of $\LJ_\nu$, denoted $\JL_\nu$.
    \end{remark}

	\begin{prp}\label{PR:JL-image}
		Let $\chi:\units{F}\to\units{\C}$ be a unitary character.
        Then $\LJ_\nu(u(\chi,r))=\LJ_\nu(T(\chi,r))\cong\chi\circ \Nrd_{D/F}$.
        Furthermore, any $V\in \what{G}_1^{(D)}$ with $\LJ_\nu(V)\cong \chi\circ\Nrd_{D/F}$
        is isomorphic to $u(\chi,r)$ or $T(\chi,r)$. (See \ref{subsec:reps-of-GLd}
        for the definition $u(\chi,r)$ and $T(\chi,r)$.)
	\end{prp}

    \begin{proof}
        Recall from Example~\ref{EX:square-integrable-char} that $u(\chi,r)=\chi\circ\Nrd_{\nMat{F}{r}/F}$.
        It is easy to check that $\psi_{\chi\circ\Nrd_{D/F}}=\chi\circ\Nrd_{D/F}|_{\tilde{G}'_1}$
        and $\psi_{\chi\circ\Nrd_{\nMat{F}{r}/F}}=\chi\circ\Nrd_{\nMat{F}{r}/F}|_{\tilde{G}_1}$,
        hence  $\LJ_\nu(\chi\circ \Nrd_{\nMat{F}{r}/F})=\chi\circ\Nrd_{D/F}$
        (if $g\leftrightarrow g'$,
        then $\Nrd_{\nMat{F}{r}/F}(g)=\Nrd_{D/F}(g')$). 
        By \cite[Th.~5.2]{BadulHennLema10}, this means that $\LJ_\nu (T(\chi,r))=\chi\circ\Nrd_{D/F}$.
        (In more detail, the map $\LJ_\nu$ is compatible with the Zelevinsky--Aubert involution, which takes $u(\chi,r)$ to $T(\chi,r)$
        and leaves the supercuspidal representation $\chi\circ\Nrd_{D/F}$ fixed.)

        Suppose now that $\LJ_\nu(V)=\chi\circ\Nrd_{D/F}$. By Tadic's classification of unitary representations of $\nGL{F}{r}$
        \cite{Tadic86}, we can write
        \begin{align*}V&=u(V_1,s_1)\times\dots\times u(V_\ell,s_\ell)\\
        &\quad\times (\nu^{\alpha_1}u(U_1,t_1)\times\nu^{-\alpha_1}u(U_1,t_1))\times\dots\times
        (\nu^{\alpha_m}u(U_m,t_m)\times\nu^{-\alpha_m}u(U_m,t_m))
        \end{align*}
        where $V_1,\dots,V_\ell,U_1,\dots,U_m$ are square-integrable and $\alpha_1,\dots,\alpha_m\in (0,\frac{1}{2})$.
        By Theorem~\ref{TH:local-JL}(ii), we must have $\ell=1$ and $m=0$, so $V=u(V_1,s_1)$.
        By \cite[Th.~5.2(3)]{BadulHennLema10}, $V$ is not $D$-compatible unless $s_1=1$ or $s_1=r$.
        In the first case, $V=V_1$ is square integrable and hence $V\cong T(\chi,r)$ by Theorem~\ref{TH:local-JL}(i).
        In the second case, $V_1$ is a character $\xi:\nGL{F}{1}\to\units{\C}$ and $V=u(\xi,r)=\xi\circ \Nrd_{\nMat{F}{r}/F}$.
        This means that $\xi\circ  \Nrd_{\nMat{F}{r}/F}$ and $\chi\circ  \Nrd_{\nMat{F}{r}/F}$ coincide on $\tilde{G}_1$,
        so $\xi=\chi$.
    \end{proof}
	
	Let $E$ be a division algebra  of degree $r$ over $k$. For every
	$\nu\in\calV$, write $E_\nu=E\otimes_kk_\nu$. Then $E_\nu\cong \nMat{D_\nu}{m_\nu}$
	where $D_\nu$ is a central division $k_\nu$-algebra.
	Let $T$ denote the set of places $\nu\in\calV$  at which $E$ ramifies, i.e.\ $m_\nu<r$.
	It is well-known that $T$ is finite \cite[Thm.~32.1]{MaximalOrders}.
	
	Let $\bfG=\uGL_{r,k}$ and let $\bfG'=\uGL_{1,E}$. Notice that for $\nu\notin T$,
	we have $\bfG'(k_\nu)\cong \nGL{k_\nu}{r}=\bfG(k_\nu)$.
	A discrete automorphic representation $V=\bigotimes'_\nu V_\nu$ of $\bfG$
	is said to be $E$-compatible if each local factor $V_\nu$ is $D_\nu$-compatible.
	
	\begin{thm}\label{TH:global-JL}
		In the previous setting, there exists a unique map $\LJ$ from the $E$-compatible discrete automorphic representations
		of $\bfG$ to the discrete automorphic representations of $\bfG'$
		with the property that the $\nu$-local-component
        of $\LJ(\bigotimes'_\nu V_\nu)$ is $\LJ_\nu(V_\nu)$ for all $\nu\in\calV$.
		The map $\LJ$ is bijective.
	\end{thm}

    \begin{proof}
        See \cite[Th.~18.1]{BadulRena10} ($\Char k=0$) and \cite[Th.~3.2]{BadulRoch14} ($\Char k>0$).
    \end{proof}

\subsection{Cuspidal and Residual Representations}
\label{subsec:cuspidal-and-residual}

	Let $\bfG$, $\bfZ$  and $\omega$ be as in \ref{subsec:automorphic-reps}.
    Discrete automorphic representations of $\bfG$ are  classified as {cuspidal} or {residual}.
    A discrete automorphic representation $V\leq\LL[\omega]{\bfG(k)\leftmod\bfG(\bbA)}$ is \emph{cuspidal} if it
    consists of functions $\vphi\in\LL[\omega]{\bfG(k)\leftmod\bfG(\bbA)}$ satisfying
    \[\int_{n\in\bfN(k)\leftmod\bfN(\bbA)} \vphi(ng)\,d\mu_{\bfN(k)\leftmod\bfN(\bbA)}=0\]
    for any $g\in\bfG(\bbA)$ and $\bfN$, where $\bfN$ ranges over the unipotent radicals
    of the parabolic subgroups of $\bfG$. Otherwise, $V$  is called \emph{residual}.

\medskip

    \emph{From now and until the end of \ref{subsec:existence},
    we assume that $\Char k>0$.} Under this assumption,
    Lafforgue  proved the following result, known the as Ra\-ma\-nu\-jan--Petersson conjecture
    for $\bGL_d$.

    \begin{thm}[Lafforgue; {\cite[Th.~VI.10]{Laff02}}]\label{TH:Lafforgue}
        Let $V=\bigotimes'_\nu V_\nu$ be a cuspidal automorphic representation of $\bGL_d$.
        Then $V_\nu$ is tempered for all $\nu\in\calV$.
    \end{thm}

    When $\bfG=\bGL_d$ with $d$ prime, the residual spectrum consists of
    $1$-dimensional representations.
    For general $d$, the residual spectrum of  $\bGL_d$ can be described as follows:
    Suppose $d=rs$ and let $\bfP_{(r,\dots,r)}$ and $\bfM_{(r,\dots,r)}$ be as in \ref{subsec:reps-of-GLd}.
    Let $U$
    be a cuspidal automorphic representation of $\bGL_{r}$ and, in analogy with Proposition~\ref{PR:def-of-u}, let
    \[
    I_{s}({U})=
    \mathrm{Ind}_{\bfP_{(r,\dots,r)}(\bbA)}^{\bGL_d(\bbA)}(\,\underbrace{U\otimes \dots\otimes U}_{\text{$s$ times}}\,)
    \]
    (the induction is not normalized and $I_s(U)$ has no a priori  structure of a pre-unitary
    representation).

    \begin{thm}[M\oe{}glin, Waldspurger; {\cite{MoegWald89}}]\label{TH:residual-spectrum}
    	For every cuspidal automorphic representation $U$ of $\bGL_{r}$,
        the representation $I_s(U)$ has a unique irreducible subrepresentation, denoted $u(U,s)$.
    	The representation $u(U,s)$ is unitarizable.
        Furthermore, every residual automorphic representation of $\bGL_d$
        is isomorphic to $u(U,s)$ for unique $1<s\mid d$ and $U$ as above.
    \end{thm}

    \begin{cor}
    	Let $V$ be a residual automorphic representation of $\bGL_d$ and write $V=u(U,s)$
    	as in Theorem~\ref{TH:residual-spectrum}. Let $\{U_\nu\}_{\nu\in \calV}$ be the local factors
    	of $U$. Then $V\cong \bigotimes'_{\nu\in\calV}u(U_\nu,s)$.
    	In particular, the local factors of $V$ are not tempered.
    \end{cor}

    \begin{proof}
    	By Theorem~\ref{TH:Lafforgue}, the local factors of $U$ are tempered, and hence $u(U_\nu,s)$ is well-defined
    	and non-tempered (Theorem~\ref{TH:Silberger}).
    	Since $U_\nu$ is unramified for almost all $\nu$, the representation $u(U_\nu,s)$ is unramified
    	for almost all $\nu$ (Proposition~\ref{PR:Jn-unramified}(ii)). Thus, the restricted
    	product $\bigotimes'_{\nu\in\calV}u(U_\nu,s)$ is well-defined.
    	
    	Since $I_s(U)$ has a unique irreducible submodule, it is enough to give a nonzero
    	homomorphism from $\bigotimes'_{\nu\in\calV}u(U_\nu,s)$
    	to $I_s(U)$. Recall from Proposition~\ref{PR:def-of-u} that $u(U_\nu,s)$ is a subrepresentation of
    	$I_s(U_\nu)=\Ind_{\bfP_{(r,\dots,r)}(k_\nu)}^{\nGL{k_\nu}{d}}(U_\nu\otimes\dots\otimes U_\nu)$ ($s$ times).
    	Choose a unit vector $u_\nu\in U_\nu$ which is $\nGL{\calO_\nu}{r}$-invariant if $U_\nu$ is unramified.
    	When $U_\nu$ is unramified, $I_s(U_\nu)$ contains a unique function $f_\nu:\nGL{k_\nu}{d}\to U_\nu\otimes\dots\otimes U_\nu$
    	with $f_\nu(\nGL{\calO_\nu}{d})=u_\nu\otimes\dots\otimes u_\nu$.
        The function $f_\nu$ is the only $\nGL{\calO_\nu}{d}$-invariant element in $I_s(U_\nu)$ up to scaling, hence $f_\nu\in u(U_\nu,s)$.
    	If $U_\nu$ is ramified, then choose an arbitrary nonzero $f_\nu\in u(U_\nu,s)$.
    	We use the vectors $\{u_\nu\}$ to form  $\bigotimes'_{\nu}U_\nu=U$ and
    	the functions $\{f_\nu\}$ to form  $\bigotimes'_{\nu}u(U_\nu,s)$.
    	We further identify $\bigotimes'_{\nu}(U_\nu\otimes\dots\otimes U_\nu)$ with
    	$U\otimes\dots\otimes U$ (as representations of $\bfP_{(r,\dots,r)}(\bbA)$) in the obvious way,
    	using the vectors $\{u_\nu\otimes\dots\otimes u_\nu\}_{\nu}$ when forming the restricted product.
    	Now, define $\Phi:\bigotimes'_{\nu\in\calV}u(U_\nu,s)\to I_s(U)$
    	by sending $\bigotimes'_\nu h_\nu$ (with $h_\nu=f_\nu$ for almost all $\nu$) to the function $h\in I_n(U)$
    	given by $h((g_\nu)_{\nu\in\calV})=\bigotimes'_\nu f_\nu(g_\nu)$. It is straightforward
    	to check that $\Phi$  is a well-defined
    	nonzero $\nGL{\bbA}{d}$-homomorphism.
    \end{proof}

    \begin{cor}\label{CR:dichotomy-cor}
        Let $V=\bigotimes'_\nu V_\nu$ be a discrete automorphic representation of $\bGL_d$ and let $\theta\in\calV$.
        Then $V$ is cuspidal if and only if $V_{\theta}\in \Irr[pu]{\nGL{k_{\theta}}{d}}$
        is tempered.
    \end{cor}

	We finally record  the following well-known fact. 
	
	\begin{prp}\label{PR:automorphic-form-dim-one}
		Let $V=\bigotimes'_\nu V_\nu$ be a discrete automorphic representation of $\bGL_d$.
		If one of the local factors $V_\nu$ is finite dimensional, then $V$ is $1$-dimensional.
	\end{prp}
	
	\begin{proof}
		We first observe that
		$\nSL{k_\nu}{d}$ acts trivially on $V_\nu$.
		Indeed, since $V_\nu$ is smooth and finite-dimensional,
		the group $L=1+\pi_\nu^r\nMat{\calO_\nu}{d}$
		acts trivially on $V_\nu$ for $r$ sufficiently large. The normal subgroup generated
		by $L$ is easily seen to contain $\nSL{k_\nu}{d}$ (look on elementary matrices),
		hence our claim.

		We view $V_\nu$ as a $\nGL{k_\nu}{d}$-subrepresentation of $V$.
		Since $V$ is irreducible as a $\nGL{\bbA}{d}$-module,  it is generated by $V_\nu$,
		and hence
		$\nSL{k_\nu}{d}$  acts trivially on $V$ (because $\nSL{k_\nu}{d}\lhd \nGL{\bbA}{d}$).
		The functions in $V$ are therefore $\nGL{k}{d}$-invariant on the left
		and $\nSL{k_\nu}{d}$-invariant on the right. Since $\nSL{k_\nu}{d}\lhd \nGL{\bbA}{d}$ and $V$
		consists of locally constant functions (it is smooth), the functions
		in $V$ are $\overline{\nGL{k}{d}\nSL{k_\nu}{d}}$-invariant on the left.
		By strong approximation (for $\bSL_d$), the group $\overline{\nGL{k}{d}\nSL{k_\nu}{d}}$  contains $\nSL{\bbA}{d}$,
		so $\nSL{\bbA}{d}$ acts trivially
		on $V$. Therefore, we may view $V$ as an irreducible admissible pre-unitary representation of $\nGL{\bbA}{d}/\nSL{\bbA}{d}$, which
		is abelian, so $V$ is one-dimensional (cf.\ Corollary~\ref{CR:commutative-algs}, Theorem~\ref{TH:Schurs-Lemma-for-admissible}).
	\end{proof}

\subsection{Ramanujan Complexes}
\label{subsec:existence}

    Let $F$ be a local non-archimedean field \emph{of positive characteristic}, let $D$ be a central
    division $F$-algebra and let $d>1$.
    We finally prove that $\calB_d(D)$, the affine Bruhat-Tits building of $\nPGL{D}{d}$,
    admits infinitely many $\nPGL{D}{d}$-quotients which are completely Ramanujan (\ref{subsec:Ramanujan-quotients}).
    The Ramanujan complexes  constructed in \cite{LubSamVi05}, which are $\nPGL{F}{d}$-quotients of $\calB_d(F)$
    that
    are Ramanujan in dimension $0$ (Example~\ref{EX:classical-Ramamanujan}(ii)), will arise as special cases
    of our construction. In particular, they
	are completely Ramanujan.


\medskip

    Let the global field $k$ and $\eta\in\calV$ be chosen such that $F=k_\eta$.
    Write $r=\deg D$, choose  a central division $k$-algebra  $E$ of degree $dr$, and define
    $E_\nu$, $D_\nu$, $m_\nu$ and $T$ as in \ref{subsec:JL}. We assume that:
    \begin{itemize}
        \item $E_\eta\cong \nMat{D}{d}$, i.e.\ $D_\eta\cong D$ and $m_\eta=d$.
        \item There is $\theta\in T$ such that $E_\theta$ is a division ring, i.e.\ $m_\theta=1$.
    \end{itemize}
    Existence of a suitable $E$ for any prescribed $D$ and $d$ follows from the
	Albert--Brauer--Hasse--Noether Theorem (\cite[Rem.~32.12(ii)]{MaximalOrders}, for instance).
    We further write
    $\bfG=\uGL_{dr,k}$, $\bfG'=\uGL_{d,E}$ and $\bfH'=\uPGL_{d,E}$.
    Since $E$ is a division ring, $\bfH'$ is  $k$-anisotropic.
    Note also that
    \begin{align*}
    &\bfH'(k_\eta)=\nPGL{D}{d}\\
    &\bfH'(k_\theta)=\units{E_\theta}/\units{k_\theta}\\
    &\bfH'(k_\nu)=\nPGL{k_\nu}{dr}& \text{when $\nu\notin T$}
    \end{align*}
    Finally, let
    \[K_\theta =\im\left(\units{\calO_{E_\theta}}\to\bfH'(k_\theta)\right)=\units{k_\theta}\units{\calO_{E_\theta}}/\units{k_\theta}\]
    (see~\ref{subsec:building-of-GLdD}
    for the definition of $\calO_{E_\theta}$).
    We shall view the $k$-adeles
    $\bbA$ as $k_\eta\times  \bbA^{\{\eta\}}$
    and $\bbA^{\{\eta\}}$ as $k_\theta\times\bbA^{\{\eta,\theta\}}$
    (notation as in~\ref{subsec:adeles}).

\medskip
	
	We are now ready to state our main result. When $D=F$, it  is just Theorem~1.2 of \cite{LubSamVi05} with the difference
	that we show complete Ramanujan-ness whereas  \cite{LubSamVi05} shows Ramanujan-ness in dimension $0$.
    (Also, in \cite{LubSamVi05}, it is assumed that $m_\nu=1$ for all $\nu\in T$, but this assumption can
    be dropped thanks to \cite{BadulRoch14}.)
	
	\begin{thm}\label{TH:ram-quo-exists}
        Let $K^\eta$ be a compact open subgroup of $\bfH'(\bbA^{\{\eta\}})$ and let
        $\Gamma=\bfH'(k)\cap(\bfH'(k_\eta)\times  K^{\eta})$ be viewed as a cocompact
        lattice in $\bfH'(k_\eta)=\nPGL{D}{d}$ (cf.\ \ref{subsec:automorphic-reps}). Assume that either
        \begin{enumerate}
        	\item[(1)] $K^\eta$ contains $K_\theta\times 1$ (view $\bfH'(\bbA^{\{\eta\}})$ as
        	$\bfH'(k_\theta)\times \bfH'(\bbA^{\{\eta,\theta\}})$), or
        	\item[(2)] $D=F$ and $d$ is prime,
        \end{enumerate}
        and $\Gamma\leq_{\calB_d(D)}\nPGL{D}{d}$
        (cf.\ \ref{subsec:quotienst-of-simp-comps}). Then $\Gamma\leftmod \calB_d(D)$
        is completely Ramanujan.
	\end{thm}

    \begin{example}
        For every $\Gamma$ as in the theorem, the spectrum
        of the $i$-dimensional Laplacian $\Delta_i$ of $\Gamma\leftmod\calB_d(D)$ is contained
        in the union of the spectrum of the $i$-dimensional Laplacian of $\calB_d(D)$  with the trivial spectrum
        $\frakT_{\Delta_i}$ (apply Proposition~\ref{PR:Ramanujan-well-behaved}
        with $A=\Alg{\catC(\nPGL{D}{d},\calB_d(D)),\Omega_i^-}{}$
        and $B=\C[\Delta_i]$; see also Proposition~\ref{PR:unitary-dual-topology}).
        This also holds for the adjacency operators considered in Example~\ref{EX:higher-dimensional-adj-ops}.
        Therefore,
        when $d=2$, the quotient $\Gamma\leftmod \calB_2(D)$ is a Ramanujan $(q_\eta^r+1)$-regular
        graph.
    \end{example}

	We shall need the following lemmas for the proof.
	
	\begin{lem}\label{LM:A-theta-quotient}
		The group $\units{k_\theta}\units{\calO_{E_\theta}}$ is normal in $\units{E_\theta}$
		and $\units{E_\theta}/\units{k_\theta}\units{\calO_{E_\theta}}\cong \Z/rd\Z$.
	\end{lem}
	
	\begin{proof}
        Define $\Phi:\units{E_\theta}\to \frac{1}{rd}\Z/\Z$ by $\Phi(g)=\nu_{E_\theta}(g)+\Z$.
        Then $\Phi$ is onto (\ref{subsec:building-of-GLdD}) and its kernel is easily
        seen to be $\units{k_\theta}\units{\calO_{E_\theta}}$.
	\end{proof}
	
	\begin{lem}\label{LM:temperedness-transfer}
		Let $V$ be an irreducible pre-unitary  representation of $\nPGL{D}{d}$. Then $V$
		is tempered as a representation of $\nPGL{D}{d}$ if and only if it is tempered
		as a representation of $\nGL{D}{d}$.
	\end{lem}
	
	\begin{proof}
		Use the equivalent conditions for temperedness in \cite[\S2.4]{Oh02}.
	\end{proof}
	
	\begin{proof}[Proof of Theorem~\ref{TH:ram-quo-exists}]
        Notice that $\Gamma$ is just $\quo{\Gamma}_i$ in
        the setting of \ref{subsec:automorphic-reps} when $(1,g_i)$ represents the trivial
        double coset $\bfH'(k)\cdot (\bfH'(k_\eta)\times  K^{\eta})$.
        Thus,
        by Theorem~\ref{TH:automorphic-ramanujan}, it is enough to show that for every automorphic
        representation $V'=\otimes'_\nu V'_\nu$ of $\bfH'$ with $V'^{1\times  K^{\eta}}\neq0$,
        the local component $V'_\eta$ is tempered or finite-dimensional.
        View $V'$ as an automorphic representation of $\bfG'$ with trivial central character.
        By Lemma~\ref{LM:temperedness-transfer}, it is enough to show that $V'_\eta$ is  tempered or finite-dimensional
        as a representation of $\bfG'(k_\eta)=\nGL{D}{d}$.
        By Theorem~\ref{TH:global-JL}, there is a discrete automorphic
        representation $V=\bigotimes'_\nu V_\nu$
        of $\bfG$ such that $\LJ_\nu(V_\nu)=V'_\nu$ for all $\nu\in\calV$.

\smallskip

        Suppose first that $D=F$ and $d$ is prime. Then $\eta\notin T$
        and hence $V'_\eta=V_\eta$. If $V$ is cuspidal, then $V_\eta$ is tempered
        by Theorem~\ref{TH:Lafforgue}. Otherwise, $V$ is $1$-dimensional (\ref{subsec:cuspidal-and-residual}), so
        $V_\eta$ is finite-dimensional.

        Suppose now that $K^\eta$ contains $K_\theta\times 1$.
        Then $V'^{K_\theta}_\theta\neq 0$, or rather ${V'_\theta}^{\units{k_\theta}\units{\calO_{E_\theta}}}\neq 0$.
        By Lemma~\ref{LM:A-theta-quotient}, we may view $V'_\theta$ as an irreducible unitary
        representation of $\units{E_\theta}/\units{k_\theta}\units{\calO_{E_\theta}}$, which is finite and abelian,
        hence $\dim V'_\theta=1$.
        Since every element of $\units{E_\theta}$ with reduced norm $1$ is a commutator
		\cite{TadaMats43}, there is a unitary character $\chi:\units{k_\theta}\to\units{\C}$ such
        that
        $V'_\theta\cong \chi\circ\Nrd_{E_\theta/k_\theta}$.
        Since $\LJ_\theta(V_\theta)=V'_\theta\cong\chi\circ\Nrd_{E_\theta/k_\theta}$, Proposition~\ref{PR:JL-image}
        implies that $V_\theta\cong u(\chi,rd)=\chi\circ \Nrd_{\nMat{k_\theta}{rd}/k_\theta}$ (cf.\ Example~\ref{EX:square-integrable-char})
        or $V_\theta\cong T(\chi,rd)$ (cf.\ Theorem~\ref{TH:Steinberg}).

        If $V_\theta\cong \chi\circ \Nrd_{\nMat{k_\theta}{rd}/k_\theta}$, then
        $V$ is $1$-dimensional (Proposition~\ref{PR:automorphic-form-dim-one}).
        In particular, $V_\eta$ is $1$-dimensional, and as explained previously, we can
        write $V_\eta=\xi\circ\Nrd_{\nMat{F}{rd}/F}$ for a unitary character $\xi:\units{F}\to \units{\C}$.
        Thus, by Proposition~\ref{PR:JL-image}, $V'_\eta=\LJ_\eta(V_\eta)=\xi\circ\Nrd_{\nMat{D}{d}/F}$
        and $\dim V'_\eta=1$.

        If $V_\theta\cong T(\chi,rd)$, then $V_\theta$ is tempered (Theorem~\ref{TH:Steinberg}(ii)),
        and hence so is $V_\eta$ (Corollary~\ref{CR:dichotomy-cor}). Now, by Theorem~\ref{TH:local-JL}(iii),
        $V'_\eta=\LJ_\eta(V_\eta)$ is also tempered. This completes the proof.
	\end{proof}

    \begin{remark}
        (i) As explained in Remark~\ref{RM:lattice-subgroup}(i),
        there exists $K^{\eta}_0\co \bfH'(\bbA^{\{\eta\}})$
        containing $K_\theta\times 1$
        such that $\Gamma$ of Theorem~\ref{TH:ram-quo-exists} satisfies
        $\Gamma\leq_{\calB_d(D)}\nPGL{D}{d}$
        whenever $K^{\eta}\subseteq K^{\eta}_0$.

        (ii) One can choose the implicit closed embedding
        $j:\bfH'\to \uGL_{m,k}$ such that $\bfH'(\calO_\theta)=K_\theta$; see \cite[\S5]{LubSamVi05}.
        Assuming this, let
        \[R=\{x\in k\suchthat \text{$\nu(x)\geq 0$ for all $ \nu\in\calV-\{\eta\}$}\}\ ,\]
        and choose non-negative integers $\{n_{\nu}\}_{\nu\in\calV-\{\eta\}}$ such that $n_\nu=0$ almost always
        and $n_\theta=0$. Taking $K^{\eta}=\prod_{\nu\neq \eta}\bfH'(\calO_\nu,\pi_\nu^{n_\nu}\calO_\nu)$
        and writing $I=\prod_{\nu\neq\eta}\pi_\nu^{n_\nu}R$,
        the lattice $\Gamma=\bfH'(k)\cap(\bfH'(k_\eta)\times K^{\eta})$ is just the congruence
        subgroup
        \[
        \bfH'(R,I)=\ker(\bfH'(R)\xrightarrow{j} \uGL_m(R)\to\uGL_m(R/I))\ .
        \]
        By Remark~\ref{RM:lattice-subgroup}(i)
        or Corollary~\ref{CR:finite-subsect-away}, for every $\nu_0\in\calV-\{\eta\}$, there is $n_0$ such that
        $\Gamma\leq_{\calB_d(D)}\nPGL{D}{d}$ whenever $n_{\nu_0}\geq n_0$.

        (iii) When $D=F$, Lubotzky, Samuels and Vishne  \cite{LubSamVish05B} gave  an explicit
        description of some of the quotients $\Gamma\leftmod \calB_d(F)$ as \emph{Cayley complexes} of certain finite groups.
        It is interesting to find  similar descriptions for $\Gamma\leftmod \calB_d(D)$.
    \end{remark}

    \begin{remark}
        Let $\uSL_{d,D}=\ker(\Nrd:\uGL_{d,D}\to \nGm{F})$ and let $H$ be the image of $\uSL_{d,D}(F)$ in $\nPGL{D}{d}$.
        If the lattice $\Gamma$ of Theorem~\ref{TH:ram-quo-exists} is contained in $H$, then $\Gamma\leftmod \calB_d(D)$
        is completely Ramanujan as an $H$-quotient of $\calB_d(D)$. This follows from Theorem~\ref{TH:finite-index-subgroup}(iii).
    \end{remark}

\subsection{A Remark about The Spectrum}

    While we have shown the
    existence of infinitely many completely Ramanujan quotients of $\calB_d(D)$,  we did not compute the spectrum
    of even a single associated operator! As most combinatorial applications require explicit bounds on spectra
    of operators of interest, we now comment  about how such bounds might be obtained.

	We retain the general setting of \ref{subsec:being-Ramanujan}: $G$ is an $\ell$-group,
	$\calX$ is an almost transitive $G$-complex, $\catC=\catC(G,\calX)$, and $F:\catC\to\catPHil$
	is an elementary functor (semi-elementary functors can be treated
	using Remark~\ref{RM:summand-transition}).
	Proposition~\ref{PR:unitary-dual-topology}
    allows us to shift the discussion from $(\catC,F)$-operators
    to algebras of $(\catC,F)$-operators (\ref{subsec:spectrum-of-simp}).
	Given such an algebra $A_0$,
    embed it in the algebra $B$ of Theorem~\ref{TH:Hecke-description-of-A}, and let $\calF$ and $\what{\calF}$ be defined
    as in Theorem~\ref{TH:spectrum-correspondence}.
	If a classification of the the
    irreducible unitary representations of $G$ is known,
    then one can compute the $A_0$-spectrum of $\calX$ by
    finding
    $\Spec_{A_0}\what{\calF}([V])$ for every tempered
    $[V]$ in the domain of $\what{\calF}$ and taking
    the union.
    The trivial $A_0$-spectrum can be computed similarly by letting $[V]$
    range over the irreducible unitary subrepresentations of $A(F\calX)_N$
    (see \ref{subsec:trivial-spec}), where $N$ ranges over the finite-index
    open subgroups of $G$.

    When $G=\nGL{D}{d}$, a classification of the
    irreducible unitary representations is known; see \cite{Tadic86}
    (for $D=F$) and \cite{Tadic90}, \cite{Seche09}, \cite{BadulHennLema10}.
    Thanks to \cite[Pr.~2.6]{Casselman80}, the classification simplifies significantly if one  is only interested in representations
    admitting a nonzero vector fixed under an Iwahori subgroup.
    The approach we sketched was successfully applied in
    \cite{KaLiWa10} and \cite{GolPar14} with $G=\nPGL{F}{3}$, in \cite{FaLiWa13} with $G=\mathrm{PGSp}_4(F)$,
    and also in \cite[Th.~2.11, \S2.4]{LubSamVi05}, to some
    extent. (In these sources, the functor $\calF$ is implicit and turns out to take
    $V$ to $V^L$ where $L$ is a certain parahoric subgroup  of $G$.)
    As can be seen from \cite{KaLiWa10}, \cite{GolPar14}, \cite{FaLiWa13}, one must engage in lengthy case-by-case analysis, and
    the number of cases increases with the split-rank of $G$.
    It is therefore interesting to look for more economical approaches.

\bibliographystyle{plain}
\bibliography{MyBib_16_05}

\newpage

\thispagestyle{empty}

\vspace*{8cm}

\begin{chapquote}{Doge}
``So math. Very algebra. Many logic. Much variable. Wow.''
\end{chapquote}

\end{document}